\documentclass[12pt, reqno,oneside]{book}
\usepackage [totalwidth=400pt, totalheight=600pt, headsep=1cm, vmarginratio=1:1] {geometry}
\usepackage{titlepic}
\usepackage{graphicx}
\usepackage{amsmath}
\usepackage{amsfonts}
\usepackage{amssymb}
\usepackage{caption}
\usepackage[nooneline, figbotcap]{subfigure}
\usepackage [section]{placeins}
\usepackage{mathrsfs} 
\usepackage{multirow}
\usepackage{cancel}
\usepackage{yhmath}
\usepackage{hyperref}
\usepackage{pstricks-add}
\usepackage{enumerate}
\usepackage{xypic}
\usepackage{float}
\usepackage{rotating}

\usepackage{fancyhdr}
\newtheorem{theorem}{Theorem}[chapter]

\newtheorem{claim}{Claim}
\numberwithin{claim}{theorem}

\newtheorem{corollary}[theorem]{Corollary}

\newtheorem{definition}[theorem]{Definition}

\newtheorem{exercise}[theorem]{Exercise}
\newtheorem{lemma}[theorem]{Lemma}

\newtheorem{proposition}[theorem]{Proposition}

\newtheorem{remark}[theorem]{Remark}

\newenvironment{proof}[1][Proof]{\textbf{#1.} }{\ \rule{0.5em}{0.5em}}

\subfiguretopcaptrue

\newcommand{\brs}[1]{
\left| #1 \right|
}
\newcommand{\SC}[1]{
\mathscr{#1}
}

\newcommand{\wt}[1]{
\widetriangle{#1}
}

\newcommand{\ov}[1]{
\overline{#1}
}

\newcommand{\VRT}[1]{
\left\Vert #1 \right\Vert
}

\newcommand{\TT}[0]{
\mathscr{T}
}

\newcommand{\QQ}[0]{
\mathscr{Q}
}

\newcommand{\MM}[0]{
\mathscr{M}
}

\newcommand{\PP}[0]{
\mathbb{P}
}

\newcommand{\CC}[0]{
\mathbf{C}
}

\renewcommand{\SS}[0]{
\mathbf{S}
}

\newcommand{\TAN}[0]{
\mathbf{T}
}

\newcommand{\cc}[0]{
\mathbf{c}
}

\newcommand{\tia}[0]{
\tilde{a}
}
\newcommand{\tib}[0]{
\tilde{b}
}
\newcommand{\tic}[0]{
\tilde{c}
}
\newcommand{\tid}[0]{
\tilde{d}
}

\newcommand{\tiA}[0]{
\tilde{A}
}
\newcommand{\tiB}[0]{
\tilde{B}
}
\newcommand{\tiC}[0]{
\tilde{C}
}
\newcommand{\tiD}[0]{
\tilde{D}
}
\newcommand{\tiE}[0]{
\tilde{E}
}
\newcommand{\tiF}[0]{
\tilde{F}
}

\newcommand{\tiT}[0]{
\widetilde{\mathscr{T}}
}

\renewcommand{\ss}[0]{
\mathbf{s}
}

\begin{document}

\title{Non-euclidean shadows of classical projective theorems}

\titlepic{\includegraphics[width=0.9\textwidth]{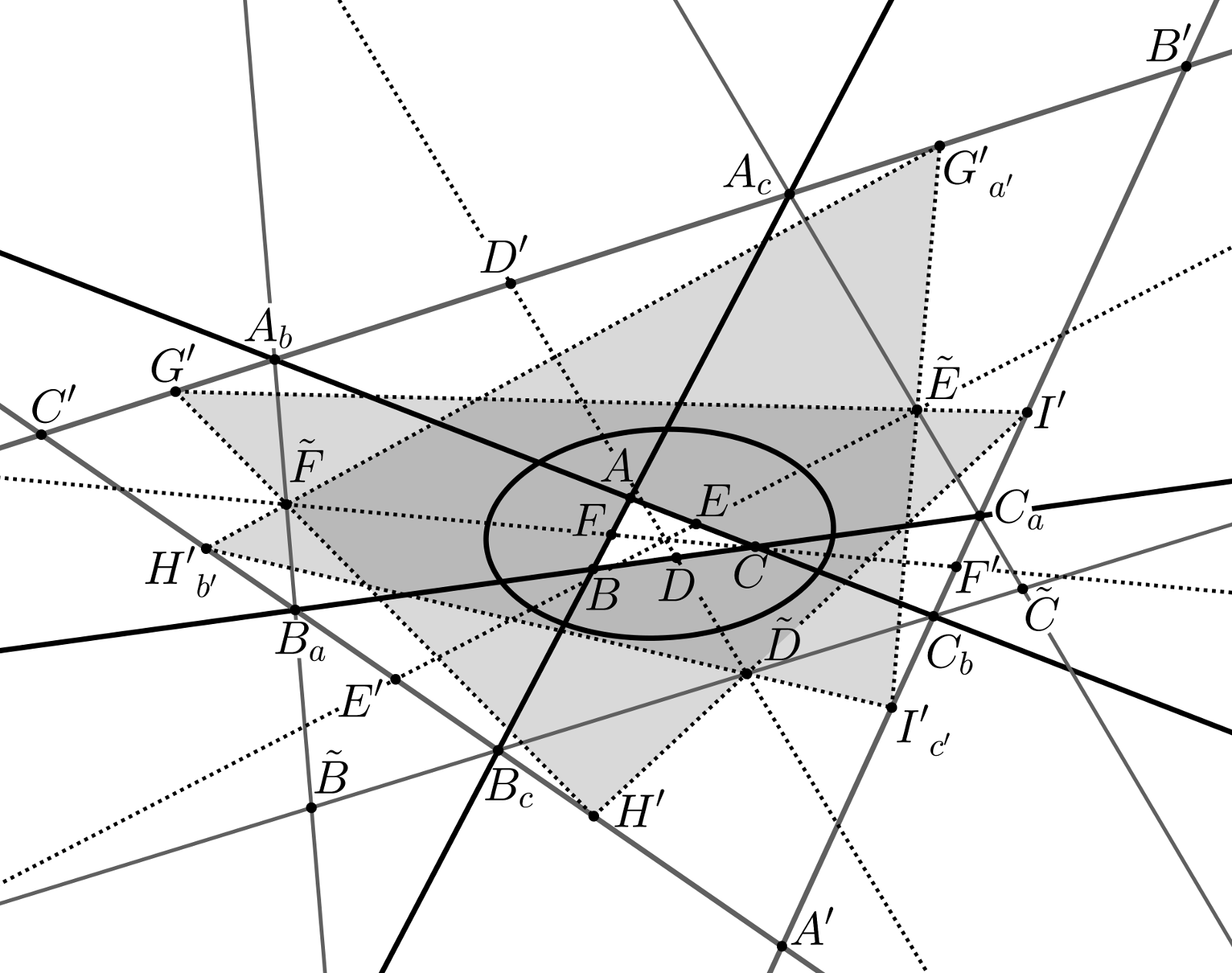}}

\author{Ruben Vigara \\
Centro Universitario de la Defensa - Zaragoza\\
I.U.M.A. - Universidad de Zaragoza\\}
\maketitle

\tableofcontents

%
%

\chapter{Introduction}\label{sec:Introduction.}

Cayley-Klein projective models for hyperbolic and elliptic (spherical)
geometries have the disadvantage of being non-conformal. However,
they have some interesting virtues: geodesics are represented
by straight lines, allowing to easily visualize incidence properties,
and they give an unified treatment of the three classic planar geometries:
euclidean, spherical and hyperbolic. They can be introduced
into any course on projective geometry, and so they can give to undergraduate students an
early contact with non-euclidean geometry before learning more advanced
topics such as riemannian geometry or calculus in a complex variable.
Within a course on projective geometry, the (extremely beautiful in itself)
projective theory of conics can be enhanced by introducing Cayley-Klein
models.

Almost any projective theorem about conics might have multiple interpretations
as different theorems of non-euclidean geometry.
For example, 
\hyperref[thm:Chasles-Polar-triangle]{Chasles' polar triangle Theorem}
(Theorem~\ref{thm:Chasles-Polar-triangle}) asserts
that a projective triangle and its polar triangle with respect to
a conic are perspective: the lines joining the corresponding vertices
are concurrent. This theorem has, at least, the following corollaries
in non-euclidean geometry:
\begin{itemize}
\item the three altitudes of a spherical triangle are concurrent
\item the three altitudes of a hyperbolic triangle are: (i) concurrent;
or (ii) parallel (they are asymptotic through the same side); or (iii)
ultraparallel (they have a common perpendicular).
\item the common perpendiculars to opposite sides of a hyperbolic right-angled
hexagon are concurrent.
\end{itemize}
Following \cite{Richter-Gebert} (and \cite{Plato}, indeed!), we say
that these three results are different non-euclidean \emph{shadows}
of \hyperref[thm:Chasles-Polar-triangle]{Chasles' Theorem}.
\begin{figure}[H]
\centering
\subfigure[elliptic triangle]
{\includegraphics[width=0.6\textwidth]{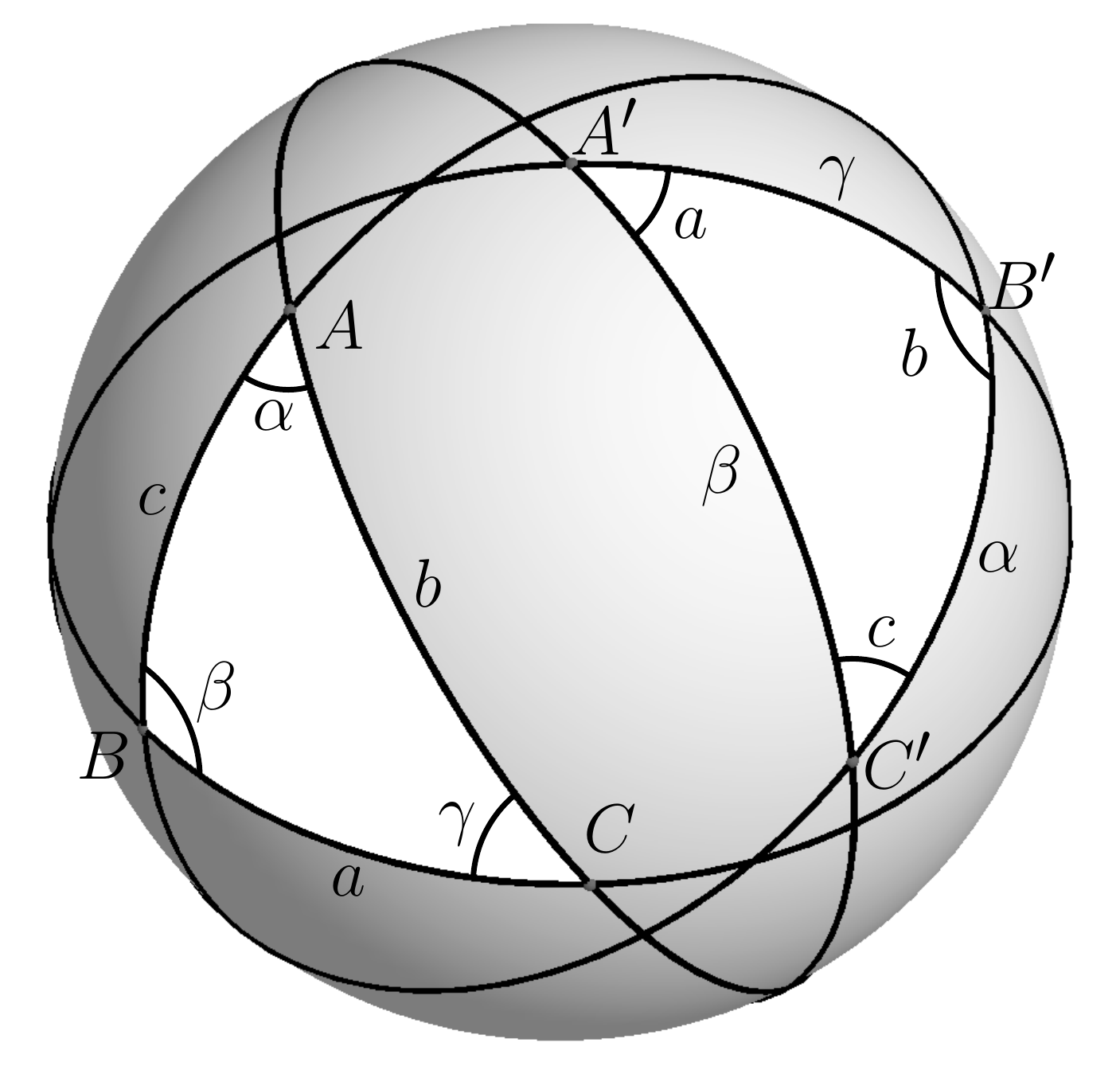}
\label{Fig:general_spherical_triangle}}\\
\subfigure[hyperbolic triangle]{
\includegraphics[width=0.8\textwidth]{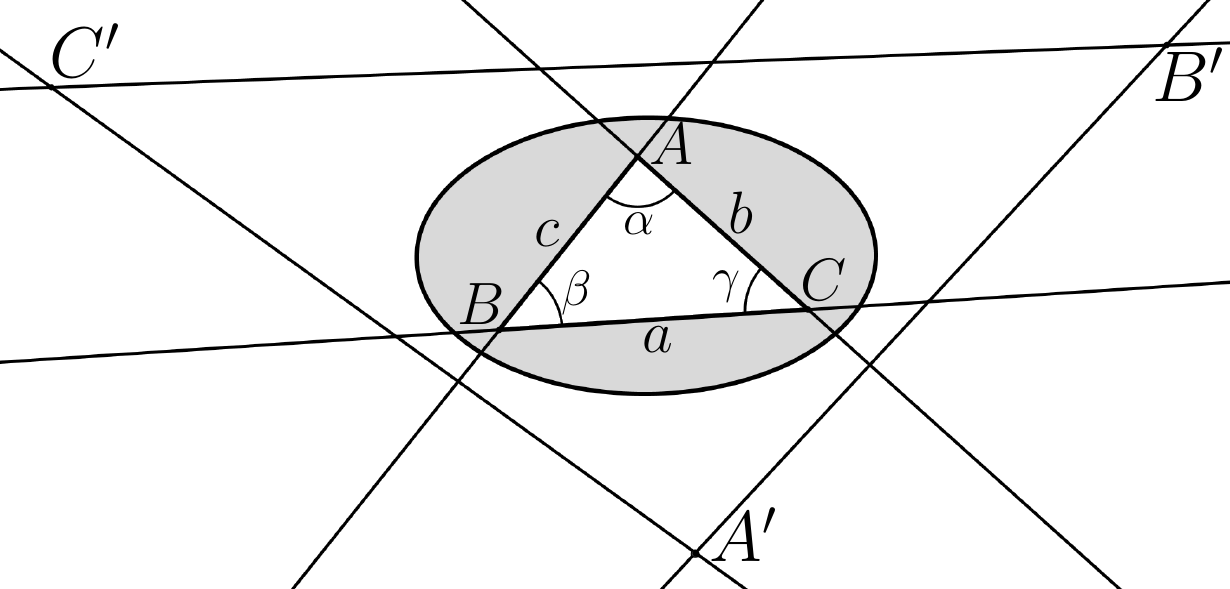}
\label{Fig:general_hyperbolic_triangle}
}\\
\subfigure[right-angled hexagon]{
\includegraphics[width=0.8\textwidth]{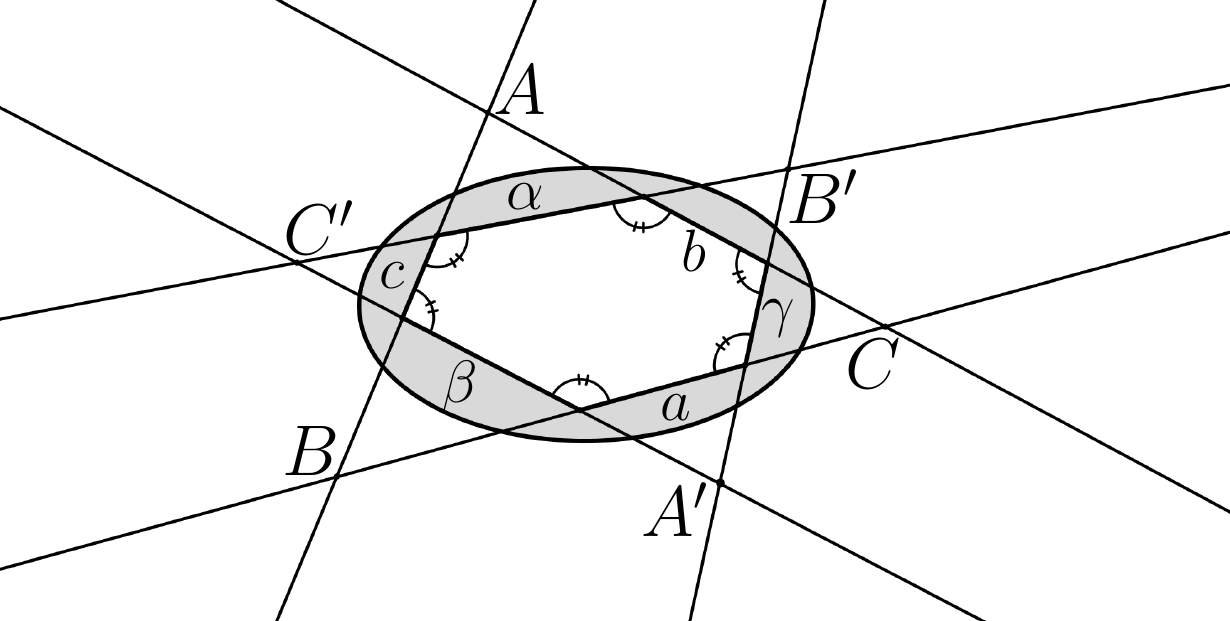}
\label{Fig:general_hyperbolic_hexagon}}
\caption{Generalized triangles}
\label{Fig:hyperbolic_generalized_triangles_I}
\end{figure}

\begin{figure}[H]
\centering
\subfigure[quadrangle I]
{\includegraphics[width=0.8\textwidth]
{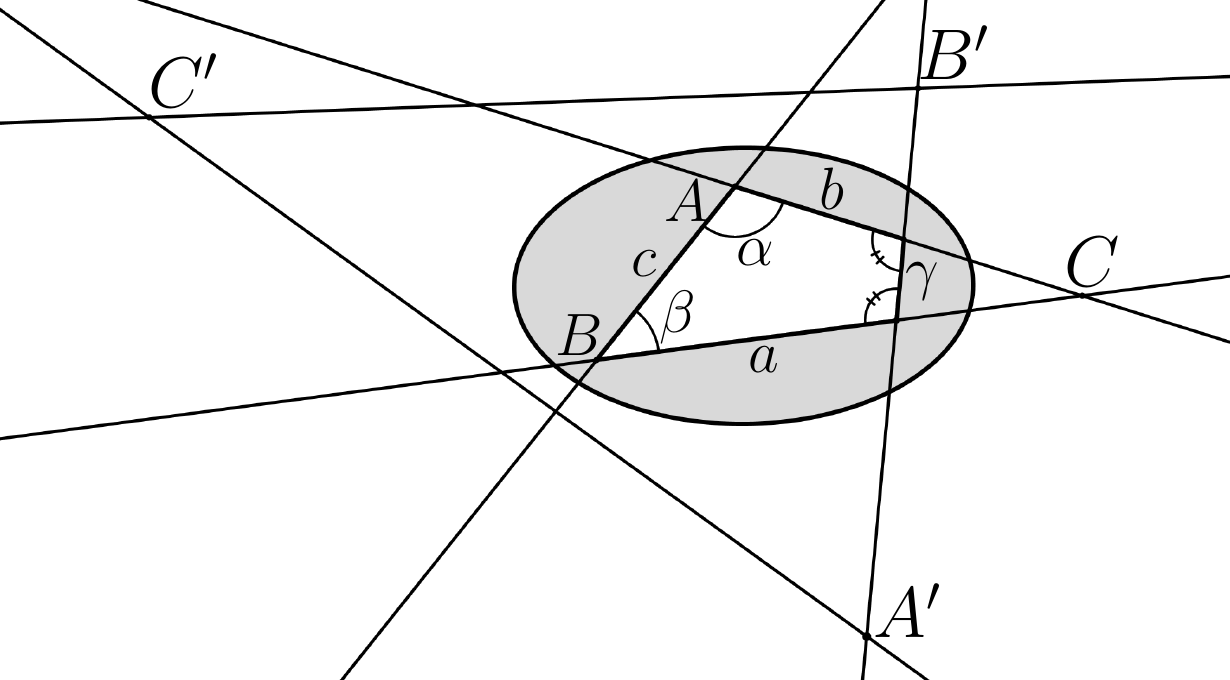}
\label{Fig:general_hyperbolic_quadrangle_01}}\\
\subfigure[quadrangle II]
{\includegraphics[width=0.8\textwidth]
{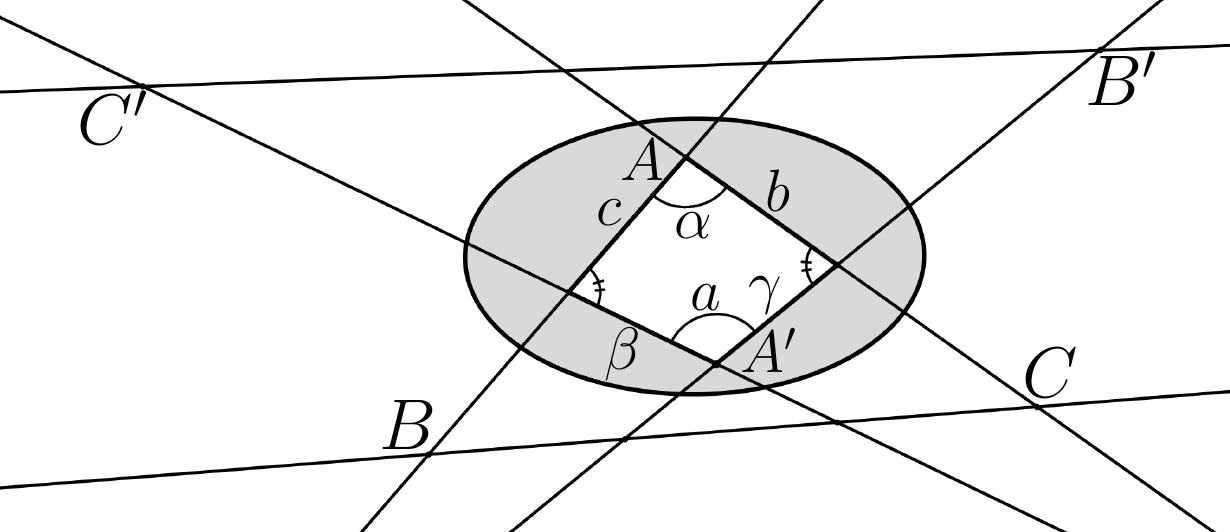}
\label{Fig:general_hyperbolic_quadrangle_02}
}\\
\subfigure[pentagon with four right angles]
{\includegraphics[width=0.8\textwidth]
{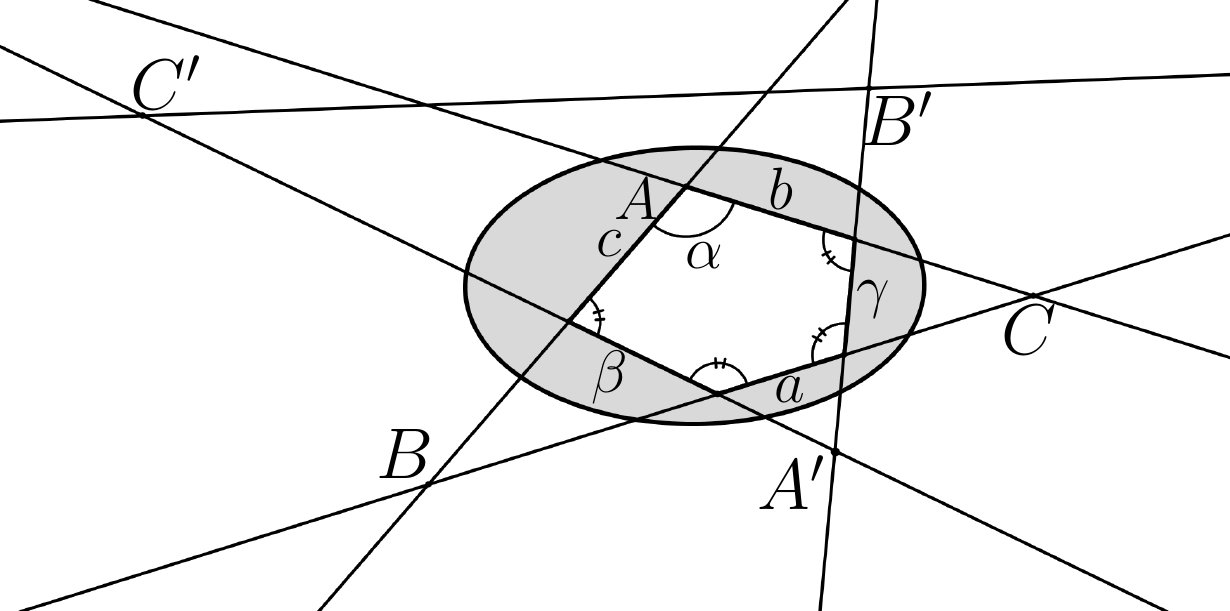}
\label{Fig:general_hyperbolic_pentagon}}
\caption{Generalized triangles II}
\label{Fig:hyperbolic_generalized_triangles_II}
\end{figure}

Aside from spherical and hyperbolic triangles, in the
hyperbolic plane there are many other figures which verify similar
trigonometric relations as the triangles do, such as Lambert and Saccheri
quadrilaterals, right-angled pentagons and hexagons, etc.. Following
\cite{Buser}, we will refer as \emph{generalized triangles} to all
these elliptic and hyperbolic figures which have trigonometry.

All generalized triangles are exactly the same figure when we
look at them wearing ``projective glasses'': in Cayley-Klein models they are
the result of intersecting
a projective triangle $\TT=\wt{ABC}$ with its polar triangle $\TT'=\wt{A'B'C'}$
with respect to the
absolute
conic $\Phi$ of the model (see Figures
\ref{Fig:hyperbolic_generalized_triangles_I}--\ref{Fig:hyperbolic_generalized_triangles_star_II}, 
where right angles are denoted with the symbol
\includegraphics[height=1em]{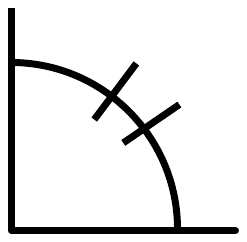}). With some general
position assumptions, projective theorems
involving a triangle and its polar triangle with respect to a conic will not
depend on the relative position of both triangles with respect to a conic,
and so they will have some different shadows as non-euclidean theorems
about different generalized triangles. In this work we will explore the
non-euclidean shadows of
some classical theorems from planar geometry. We will adopt a classic
(old-fashioned?) point of view. For more modern approaches to this subject  see
\cite{Richter-Gebert} or \cite{Wildberger, Wildberger2}, for example.

\paragraph{Pascal's and Chasles' theorems and classical triangle centers}

In the previous example about 
\hyperref[thm:Chasles-Polar-triangle]{Chasles' Theorem} and the
concurrency of altitudes, we can see that the relation of a triangle with its
polar one with respect to a conic is of key importance in the non-euclidean 
treatment of triangles. The set of midpoints of the sides of a 
triangle has a particular structure that is deduced essentially from 
\hyperref[thm:Pascal]{Pascal's Theorem}, 
and this structure allows to prove
the concurrency of medians, of side bisectors and
of angle bisectors of a triangle.

\paragraph{Desargues' Theorem and the non-euclidean Euler line}

For an euclidean triagle, the orthocenter, the
circumcenter and the barycenter are collinear. The line passing through
these three points is the \emph{Euler line} of the triangle, and it
contains also many other interesting points such as the center of the
\emph{nine-point circle}. The nine-point circle
is the circle passing through the midpoints of the sides of the triangle,
and it contains also the feet of the altitudes of the triangle and the midpoints
of the segments joining the orthocenter with the vertices of the triangle.

The Euler line and the nine-point circle have no immediate analogues
in non-euclidean gometry because, in general, for a hyperbolic or
elliptic triangle, the orthocenter, the circumcenter and the barycenter
are non-collinear, and the midpoints of the sides and the feet
of the altitudes are not concyclic. Thus, it is usually said that
in non-euclidean geometry the Euler line and the nine-point circle 
do not exist. Nevertheless,
we claim that they exist and that they are not unique. 
We will propose a non-euclidean
version of these two objects, and a different version of them is proposed
in \cite{Akopyan-other-proposal}. The line that we propose as Euler 
line is the line which is called \emph{orthoaxis} in \cite{Wildberger2}. 
We claim that there are enough reasons for this line to deserve 
the name \emph{Euler line}. 
Giving credit to \cite{Wildberger2}, we will call it \emph{Euler-Wildberger
line}.

Some of the classical centers of an euclidean triangle can be defined
in multiple ways, all of them equivalent. However, two such definitions,
equivalent under the euclidean point of view, could be non-equivalent
in the non-euclidean world. We will illustrate this fact for the barycenter
and the circumcenter. Using alternative definitions of these points,
we will show how every non-euclidean triangle has an \emph{alternative
barycenter} and an \emph{alternative circumcenter}, different from
the standard ones. These alternative centers are collinear with the orthocenter
of the triangle, and we will say that the line passing through them is the
\emph{Euler-Wildberger line} of the triangle. All these constructions have been
introduced before in \cite{Wildberger, Wildberger2} with a different notation.
We give new proofs of them based on a reiterated
application of \hyperref[thm:Desargues]{Desargues' Theorem}.

Beyond the Euler line, we construct a 
\emph{nine-point conic} which is a non-euclidean version of the nine-point
circle.

\paragraph{Menelaus' Theorem and non-euclidean trigonometry}

Hyperbolic trigonometry is a recurrent topic in most treatments 
on non-euclidean geometry since the early works of N. I. Lobachevsky and J.
Bolyai (see \cite{Coolidge, Cox Non-euc}). Due to its connection
with different topics as Riemann sufaces, low-dimensional topology or special
relativity, it has been treated also from different viewpoints in more recent
works such as \cite{Beardon, Buser, Fenchel, Thurston}. There exists a close
connection between elliptic and hyperbolic
trigonometry. Trigonometric relations between sides and angles of
spherical and hyperbolic triangles look much more the same, with some
replacements between sines and cosines into hyperbolic sines and cosines
and vice versa.

Generalized triangles are characterized by having exactly six defining
magnitudes among sides and non-right angles. The value (segment length
or angular measurement) of each of these magnitudes is related with a
side of $\TT$ or $\TT'$: the cross ratio of four
points on this side provides the square power of a (circular or hyperbolic)
trigonometric function of the corresponding magnitude.

There are four kinds of generalized triangles whose trigonometric
relations are simpler than in the general case: elliptic and hyperbolic
right-angled triangles (Figures~\ref{sph_triangle} and~\ref{hyp_triangle}
respectively), Lambert quadrilaterals (Figure~\ref{hyp_Lambert})
and right-angled pentagons (Figure~\ref{hyp_pentagon}). They appear
when we force one of the non-right angles of a generalized triangle
to be a right angle. For this reason, we will refer to them as \emph{generalized
right-angled triangles}. 
In \S\ref{sec:Menelaus-Theorem}
we show that all the trigonometric relations for generalized right-angled
triangles are non-euclidean shadows of 
\hyperref[thm:Menelaus-affine]{Menelaus' Theorem}.

A generalized triangle can be constructed by pasting together two
generalized right-angled triangles. This decomposition allows us to
deduce the trigonometric relations of generalized triangles from
the trigonometric relations of the right-angled ones.
Thus, the whole non-euclidean trigonometry can be deduced from 
\hyperref[thm:Menelaus-affine]{Menelaus' Theorem}.
\begin{figure}[t]
\centering
\subfigure[stellate quadrangle I]
{\includegraphics[width=0.7\textwidth]
{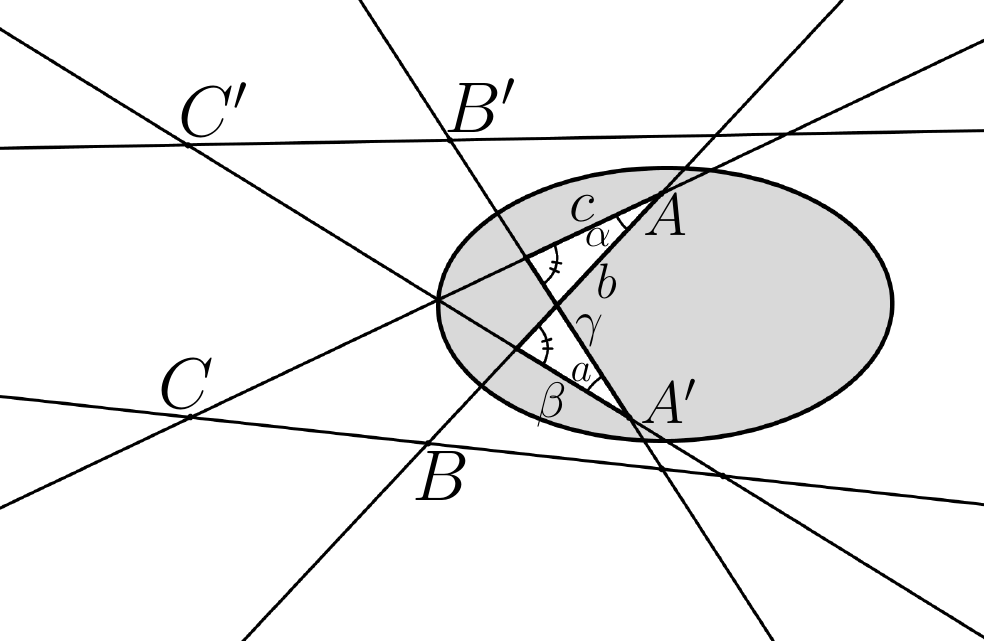}
\label{Fig:general_hyperbolic_quadrangle_star-01}}
\\
\subfigure[stellate quadrangle II]
{\centering\includegraphics[width=0.7\textwidth]
{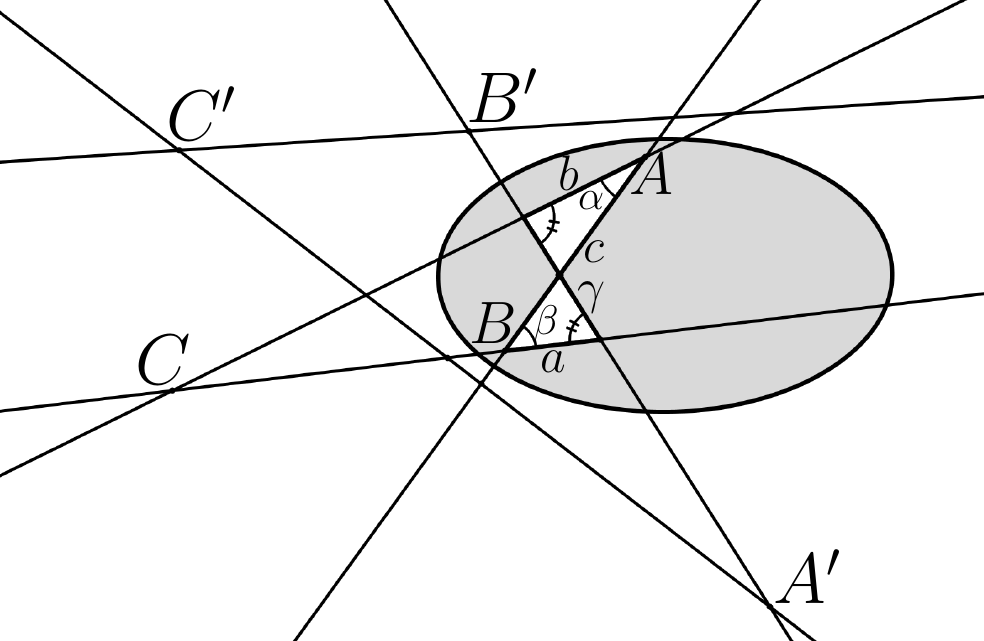}
\label{Fig:general_hyperbolic_quadrangle_star-02}}
\caption{hyperbolic stellate generalized
triangles}
\label{Fig:hyperbolic_generalized_triangles_star_I}
\end{figure}

\begin{figure}[t]
\centering
\subfigure[stellate pentagon with four right angles]
{\includegraphics[width=0.7\textwidth]
{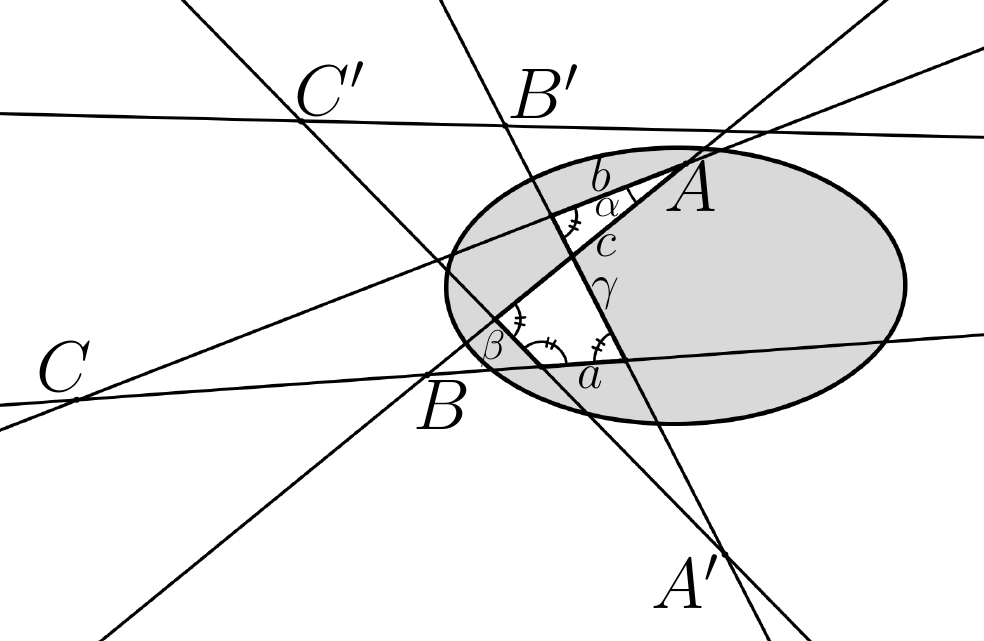}
\label{Fig:general_hyperbolic_pentagon_star}}
\\
\subfigure[stellate right-angled hexagon]
{\includegraphics[width=0.7\textwidth]
{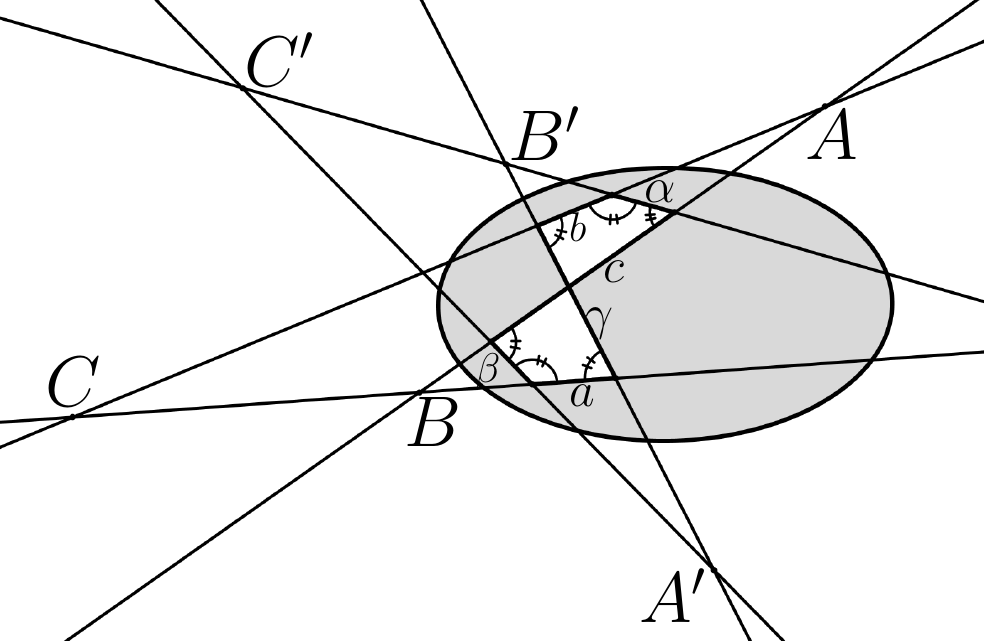}
\label{Fig:general_hyperbolic_hexagon_star}}
\caption{hyperbolic stellate generalized
triangles II}
\label{Fig:hyperbolic_generalized_triangles_star_II}
\end{figure}

\paragraph{Carnot's Theorem and... Carnot's
Theorem?}
In \S\ref{sec:Carnot's-theorem-and} we extend the arguments applied
in \S\ref{sec:Menelaus-Theorem} to 
\hyperref[thm:Menelaus-affine]{Menelaus' Theorem}
to a theorem of Carnot on affine triangles (Theorem~\ref{thm:Carnot-affine}).
We obtain that its non-euclidean shadows are the non-euclidean versions (for
generalized triangles)
of another classical theorem of Carnot (Theorem
\ref{thm:Carnot_theorem_euclidean}) for euclidean triangles!

\paragraph{Where do laws of cosines come from?}
If the projective figure behind every generalized triangle is the
same (triangle and polar triangle) and the measurements of its sides and angles
are given in projective
terms (cross ratios), it is natural to expect that the trigonometric relations
of generalized triangles have
indeed a projective basis (this viewpoint has been proposed also in
\cite{Richter-Gebert}). In \S\ref{sec:Menelaus-Theorem}, 
\hyperref[thm:Menelaus-affine]{Menelaus' Theorem}
succeeds in providing all the trigonometric relations of right-angled figures
and the law of sines for any, not necessarily right-angled, generalized
triangle in a straightforward way. However, 
\hyperref[thm:Menelaus-affine]{Menelaus' Theorem} fails with the
law of cosines.

In \S
\ref{sec:Menelaus-Theorem}, we introduce projective versions of the law of sines
and the law of cosines of generalized
triangles, but the non-euclidean trigonometric formulae that we obtain as
translations of projective formulae are \emph{squared}:
the trigonometric functions appearing in each formula are always raised
to the square power. In order to obtain the standard (\emph{unsquared})
trigonometric formulae, we must proceed with an \emph{unsquaring} process
involving some choices of $\pm$ signs. The unsquaring process is
straightforward for all the projective trigonometric formulae with
the only exception of the law of cosines. In this sense,
the projective law of cosines given in \S\ref{sec:Menelaus-Theorem} is
unsatisfactory, and it is natural to ask if
there exists 
a \emph{better projective law of cosines}: a formula relating
cross ratios of points of a triangle and its polar triangle with respect
to $\Phi$ such that it translates in a straightforward way into the
different laws of cosines for any generalized triangle.

Looking for the answer to this question, in \S\ref{sec:Cosine-Rule} we reverse
the
strategy of earlier chapters. While in the rest of the book we
have tried to find interesting non-euclidean theorems that can be deduced
from classical projective ones, in \S\ref{sec:Cosine-Rule} we started
with some different non-euclidean results with the conviction that there must be
a unique projective theorem (perhaps a non-classical one) hidden behind them.
The projective formula behind all the laws of cosines is found after a
detailed study of the set of midpoints of a triangle and its polar one. The midpoints allow
to define a concept of ``orientation'' of a triangle based in purely projective techniques,
and to obtain the actual, unsquared, trigonometric functions required.
Although we don't recognize any classical theorem as the essence of
this projective 
law of cosines, in its proof we will need another classical theorem of affine geometry:
\hyperref[thm:Van-Aubel]{Van Aubel's Theorem on cevians}\footnote{The existence of
this theorem was pointed out to the author by M. Avendano.}.

\paragraph{Appendix: Laguerres's formula for rays}

One of the starting points of this subject, prior to the work of Cayley~\cite{Cayley},
is Laguerre's formula \eqref{eq:Laguerre-formula} expressing an euclidean angle
in terms of the cross-ratio of four concurrent lines. As this formula uses a cross-ratio between \emph{lines}, it cannot distinguish an angle from its supplement. In the appendix
we propose a slight modification of Laguerre's formula in such a way that it is valid
for computing angles between \emph{rays}.

In order to make the paper more self-contained, we will introduce in 
\S\ref{sec:Basics-of-projective}
some basics from projective geometry,
and in \S\ref{sec:Cayley-Klein-models-for} we will present 
Cayley-Klein models for planar hyperbolic and
elliptic geometries. Apart from the notation that we introduce in them, a reader
familar with this topic can skip \S\ref{sec:Basics-of-projective} and
\S\ref{sub:Cayley-Klein-models-distances-angles}.

\section*{Acknowledgements}

I am very grateful to professors J. M. Montesinos, M. Avendano and A. M. Oller-Marc\'{e}n for their interesting suggestions and comments during the writing of this manuscript.

This research has been strongly influenced by \cite{Santalo}. It has been partially supported by the European Social Fund and Diputaci\'on General de Arag\'on (Grant E15 Geometr{\'\i}a). 

\chapter{Basics of planar projective geometry.}\label{sec:Basics-of-projective}

We will asume that the reader is familiar with the basic concepts
of real and complex planar projective geometry: the projective plane
and its fundamental subsets (points, lines, pencils of lines, conics),
and their projectivities (collineations, correlations). Nevertheless,
we will review some concepts and results needed for understanding
the rest of the paper. For rigurous definitions and proofs we
refer to \cite{Cox Proj,V - Y}, for example.

We consider the real projective plane $\mathbb{RP}^{2}$ standardly
embedded in the complex projective plane $\mathbb{CP}^{2}$. We consider
the objects lying in $\mathbb{CP}^{2}$, and they can be real or imaginary
depending on how they intersect with $\mathbb{RP}^{2}$. A point in
$\mathbb{CP}^{2}$ is real if it lies in $\mathbb{RP}^{2}$, and it
is imaginary otherwise. A line $l$ in $\mathbb{CP}^{2}$ is real
if $l\cap\mathbb{RP}^{2}$ is a real projective line, and it is imaginary
otherwise. A nondegenerate conic $\Phi$ in $\mathbb{CP}^{2}$ is
real if $\Phi\cap\mathbb{RP}^{2}$ is a nondegenerate real projective
conic, and it is imaginary otherwise.

For two points $A,B$ in $\mathbb{CP}^{2}$, let $AB$ denote the
line joining them. For two lines $a,b$ in $\mathbb{CP}^{2}$, let
$a\cdot b$ denote the intersection point between them. For a line
$a$ and a conic $\Phi$, let $a\cdot\Phi$ denote de set of intersection
points between them. In $\mathbb{CP}^{2}$, $a\cdot\Phi$ has one
or two (perhaps imaginary) points, while in $\mathbb{RP}^{2}$ the
intersection set $a\cdot\Phi$ can consist of $0$, $1$ or $2$ real points.
For a real line $a$ and a real conic $\Phi$, we say that $a$ is
\emph{exterior}, \emph{tangent} or \emph{secant} to $\Phi$ if the
intersection set $a\cdot\Phi$ has $0$, $1$ or $2$ real points, respectively.

\section{Cross ratios}\label{sub:Cross-ratios.}

Given a line $r\subset\mathbb{CP}^{2}$ and four different collinear
points $A,B,C,D\in r$, their \emph{cross ratio} is given by
\begin{equation}
\left(ABCD\right)=
\dfrac{\left| AC \right|}{\brs{BC}}:\dfrac{\brs{AD}}{\brs{BD}},
\label{eq:definition-of-cross-ratio}
\end{equation}
where $\brs{XY}=Y-X$ once we have chosen a set of nonhomogeneous
coordinates in $r$ such that none of the points $A,B,C,D$ is at
infinity. This formula does not depend on the chosen set of coordinates, 
and so the cross ratio is well-defined.

The cross ratio $\left(ABCD\right)$ depends on the ordering of the
four points $A,B,C,D$. By simple computations it can be checked that
it fulfills the following relations:
\begin{subequations}\label{eq:cross_ratio_identities}
\begin{align}
\left(ABDC\right) & =\left(BACD\right)=\left(ABCD\right)^{-1},\label{eq:CR1}\\
\left(ACBD\right) & =\left(DBCA\right)=1-\left(ABCD\right),\label{eq:CR2}
\end{align}
\end{subequations}
and for any point $E$ collinear with $A,B,C,D$ it is 
\begin{equation}
\left(ABCD\right)=\left(ABED\right)\left(ABCE\right)=\left(EBCD\right)\left(AECD
\right).\label{eq:CR3}\tag{2.2c}
\end{equation}

\begin{proposition}
\label{prop:cross-ratio-four-different-points}The cross ratio of
four different points is a number different from $0$ and $1$.\end{proposition}
\begin{proof}
If the four points $A,B,C,D$ are different, the four numbers in the
right-hand side of~\eqref{eq:definition-of-cross-ratio} are different
from $0$ and thus $(ABCD)\neq0$. Applying~\eqref{eq:CR2}, we have
also $(ABCD)\neq1$.
\end{proof}

Cross ratio can be defined even if some of the four points coincide,
or if one of the points lies at infinity in the chosen set of coordinates.
When one of the four points is the point at infinity of the line, the cross
ratio coincides with a harmonic ratio of the three remaining points.
For example, if $D$ is at infinity we have
\begin{equation}
\left(ABCD\right)=\dfrac{\brs{AC}}{\brs{BC}}.\label{eq:harmonic ratio equals cross ratio}
\end{equation}

The fundamental property of cross ratio is that it is invariant under
the operations of projection from a point and section with a line
\footnote{
The invariance of cross-ratio under projection and section can be checked
experimentally: (i) draw four collinear points on a blackboard; (ii)
take a ruler and put it in front of you (with one eye closed!) making it
coincide in your sight with the line containing the four points depicted;
(iii) write down the numbers on the ruler that correspond to the four points in your sight;
(iv) compute the cross-ratio of these four numbers. 
Repeat the experiment from another place of the classroom, 
with another ruler (cm., inches,...), etc.
The resulting cross-ratio will be approximately the same.

}.
Let $A,B,C,D$ be four points on a line $r$, let $s$ be another
line, and let $P$ be a point not incident with $r$ or $s$. Take
the lines $a,b,c,d$ joining the point $P$ with $A,B,C,D$ respectively,
and take the points $A',B',C',D'$ in $s$ given by (Figure~\ref{Fig:prespectivity-invariance-of-cross-ratios})
\[
A'=a\cdot s,\quad B'=b\cdot s,\quad C'=c\cdot s,\quad D'=d\cdot s.
\]
Then, the following relation holds:%
\footnote{If we take one of the four lines $a,b,c,d$ as the line at infinity
and we apply~\eqref{eq:harmonic ratio equals cross ratio}, 
then~\eqref{eq:invariance_of_Cross-ratio}
is Thales' theorem.%
}
\begin{equation}
\left(ABCD\right)=\left(A'B'C'D'\right).\label{eq:invariance_of_Cross-ratio}
\end{equation}

\begin{figure}
\centering
\includegraphics[width=0.5\textwidth]{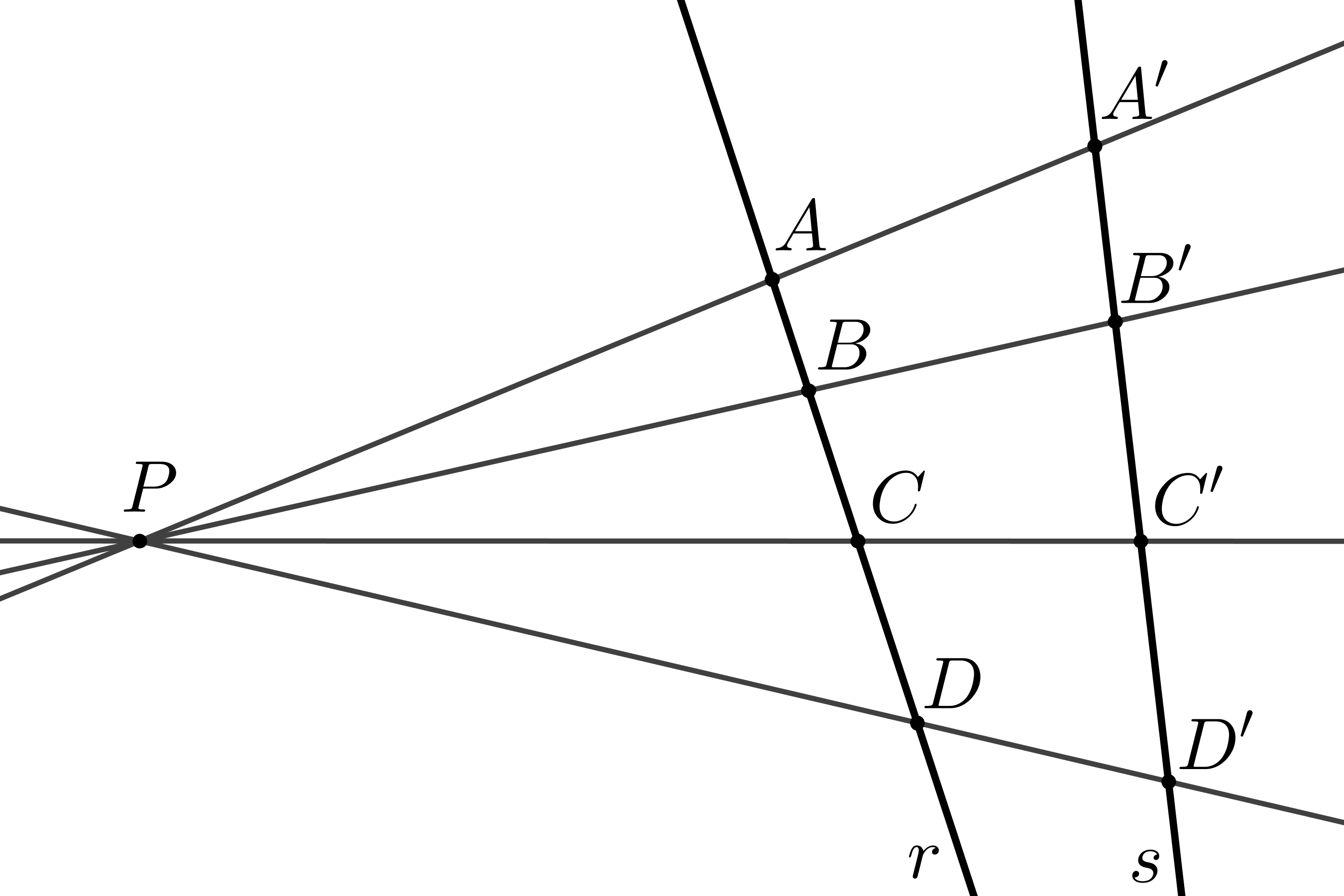}
\caption{cross ratio is invariant under projection and section}
\label{Fig:prespectivity-invariance-of-cross-ratios} 
\end{figure}
Property~\eqref{eq:invariance_of_Cross-ratio} allows us to define
the cross ratio of four concurrent lines. If $a,b,c,d$ are four concurrent
lines and $r$ is a line not concurrent with them, by taking the points
$A,B,C,D$ where $r$ intersects $a,b,c,d$ respectively we can define
\[
\left(a\,b\,c\,d\right):=\left(ABCD\right).
\]

An interesting property of cross ratio for real points is the following
(see
\cite[vol. II, Theorem 17]{V - Y}):
\begin{lemma}\label{lemma:ABCD-negative-iff-AB-separate-CD}
If $A,B,C,D$ are four points in a real line $p$, then
$$(ABCD)<0\iff A,B \text{ separate } C,D\,.$$
\end{lemma}

For the sake of simplicity, in the projective context we will consider
segments and triangles just as the subsets of the set of points of the
projective plane composed by their vertices, 
without entering deeper discussions about separation
or convexity. This will not be the same in the (non-euclidean) geometric
context: segments and triangles will retrieve their usual meaning
when considered in the hyperbolic or elliptic plane.

\section{Segments}

A \emph{segment} is a pair $\{P,Q\}$ of different points of the projective
plane. The segment $\{P,Q\}$ is denoted by $\overline{PQ}$. The points $P,Q$ are the \emph{endpoints} of the segment $\overline{PQ}$. Sometimes we will use the same name
for a segment and for the line that contains it.

\section{Triangles}

A \emph{triangle} is a set $\{P,Q,R\}$ composed by three noncollinear
points (the \emph{vertices} of the triangle). Because they are noncollinear,
in particular the three vertices of a triangle are different. The
triangle $\TT=\{P,Q,R\}$ is denoted by
$\widetriangle{PQR}$. A line joining two vertices of the triangle
is a \emph{side} of $\TT$.
The vertex $P$ and the side $p$ of a triangle are \emph{opposite}
to each other if $P\notin p$. If $p,q,r$ are three nonconcurrent lines, then
we denote by $\widetriangle{pqr}$ the triangle having $p,q,r$ as its sides.

\section{Quadrangles}

A \emph{quadrangle} (see Figure~\ref{Fig:quadrangle}) is the figure
composed by four points (the \emph{vertices} of the quadrangle), no
three of which are collinear, and all the lines 
joining any two of
them (the \emph{sides} of the quadrangle). If the intersection point
$P$ of two sides $r,s$ of the quadrangle is not a vertex of the
quadrangle, we say that $P$ is a \emph{diagonal point} of the quadrangle
and that $r$ and $s$ are \emph{opposite sides} to each other. A
quadrangle has six sides, arranged in three pairs of opposite sides,
and thus it has three diagonal points. The three diagonal points of
the quadrangle are noncollinear: they are the vertices of the \emph{diagonal
triangle} of the quadrangle.

\begin{figure}
\centering
\includegraphics[width=0.6\textwidth]{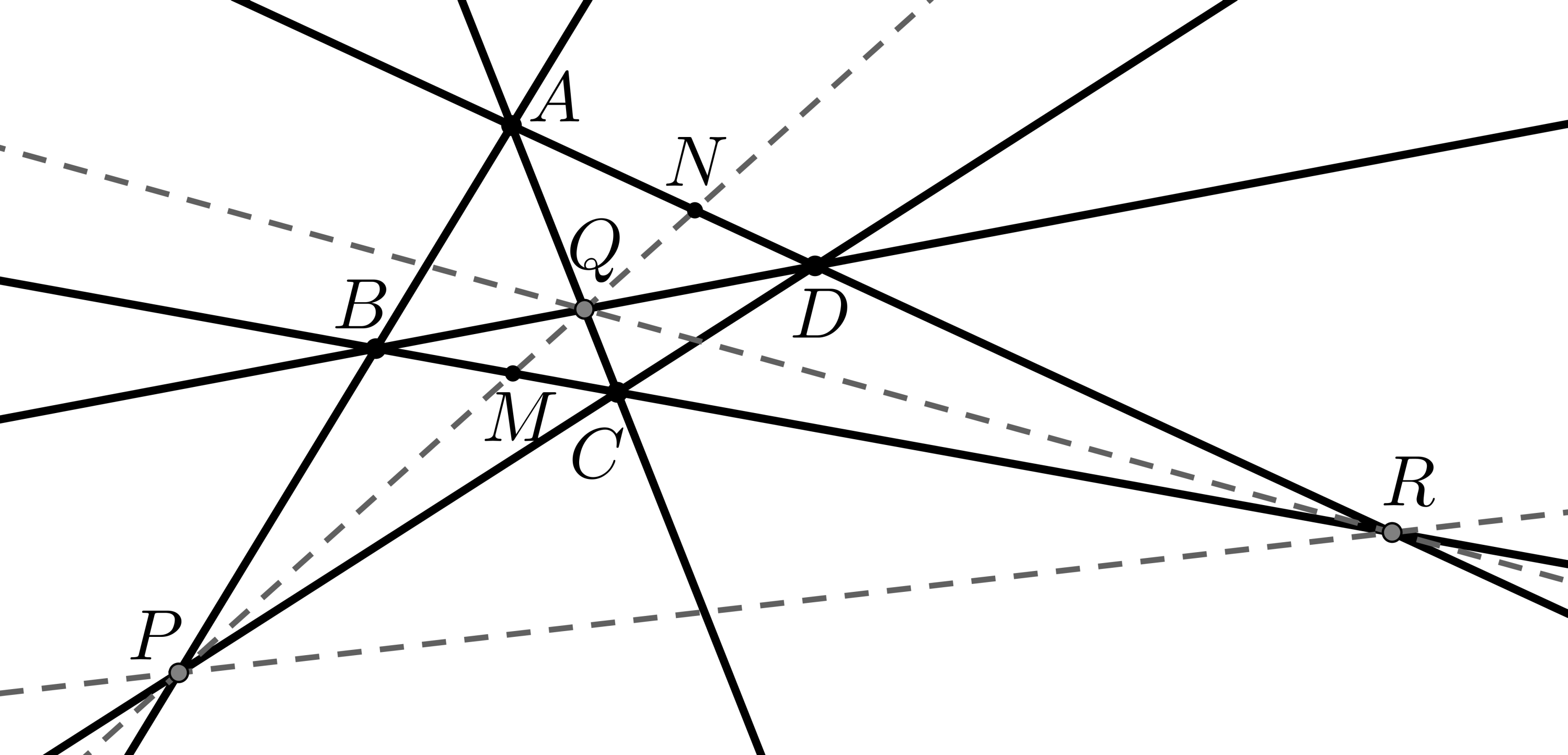}
\caption{quadrangle with vertices $A,B,C,D$ and diagonal points $P,Q,R$.}\label{Fig:quadrangle}
\end{figure}

\section{Poles and polars}

Let $\Phi$ be a nondegenerate conic.

The conic $\Phi$ allows us to associate to each point $P$ of the
plane a line $\rho(P)$ which is called the \emph{polar line} of $P$
with respect to $\Phi$. If $P$ does not belong to $\Phi$, there
are two (perhaps imaginary) lines $u,v$ through $P$ which are tangent
to $\Phi$. If $U,V$ are the two contact points with $\Phi$ of the
lines $u,v$ respectively, then it is $\rho(P)=UV$ (see 
Figure~\ref{Fig:drawing-polars-P-exterior}).
If $P$ lies on $\Phi$, then $\rho(P)$ is the unique line tangent
to $\Phi$ through $P$.

The map $\rho$ is the \emph{polarity} induced by $\Phi$, and
it is in fact a bijection: every line $p$ of the plane has a unique
point $P$ such that $\rho(P)=p$. If $p$ is not tangent to $\Phi$,
then $p\cdot\Phi$ is composed by two (perhaps imaginary) points $U,V$.
In this case, if $u,v$ are the two lines tangent to $\Phi$ at $U,V$
respectively, the point $P=u\cdot v$ verifies that $\rho(P)=p$.
If $p$ is tangent to $\Phi$, then $p\cdot\Phi$ contains only one
point $P$, and it is $\rho(P)=p$. If $\rho(P)=p$, we say that $P$
is the \emph{pole} of $p$ with respect to $\Phi$ and we denote also
$\rho(p)=P$.
\begin{proposition}
\label{prop:polarity-preserves-incidence}The polarity $\rho$ induced
by $\Phi$ is a \emph{correlation}: a bijection between the set of
points and the set of lines of the projective plane that preserves
incidence. In particular, the poles of concurrent lines are collinear and
the polars of collinear points are concurrent.
\end{proposition}
See \cite[vol. I, p.124]{V - Y} for a proof of this statement.

A quadrangle $\mathscr{Q}$ is \emph{inscribed into} $\Phi$ if the
four vertices of $\mathscr{Q}$ belong to~$\Phi$. Quadrangles are
important in the theory of conics due to the following theorem (see \cite[vol. I, p.123]{V - Y}).
\begin{theorem}
\label{thm:polars-quadrangles}The diagonal triangle of a quadrangle
inscribed into $\Phi$ is \emph{self-polar} with respect to $\Phi$:
each side is the polar of its opposite vertex with respect to $\Phi$.
Conversely, every self-polar triangle with respect to $\Phi$ is the
diagonal triangle of a quadrangle inscribed into $\Phi$.
\end{theorem}
If $\Phi$ is a real conic and $P$ is a real point, Theorem~\ref{thm:polars-quadrangles}
provides an algorithm for drawing $\rho(P)$ even if $P$ is interior
to $\Phi$ (the tangent lines to $\Phi$ through $P$ are imaginary
in this case, so we can't draw them!): (i) draw two secant lines $a,b$
to $\Phi$ through $P$ and take the points $A_{1},A_{2}$, and $B_{1},B_{2}$
lying in $a\cdot\Phi$ and $b\cdot\Phi$ respectively; (ii) for the
quadrangle with vertices $A_{1},A_{2},B_{1},B_{2}$, find the two
diagonal points $Q,R$ different from $P$; and (iii) the polar $\rho(P)$
of $P$ with respect to $\Phi$ is the line $QR$ (see 
Figure~\ref{Fig:drawing-polars}).
In order to find the pole of a line $p$, take two points $A,B\in p$
and draw their polars $\rho(A),\rho(B)$: the pole of $p$ is the
intersection point $\rho(A)\cdot\rho(B)$.

\begin{figure}\centering
\subfigure[$P$ exterior to $\Phi$]
{\includegraphics[width=0.48\textwidth]
{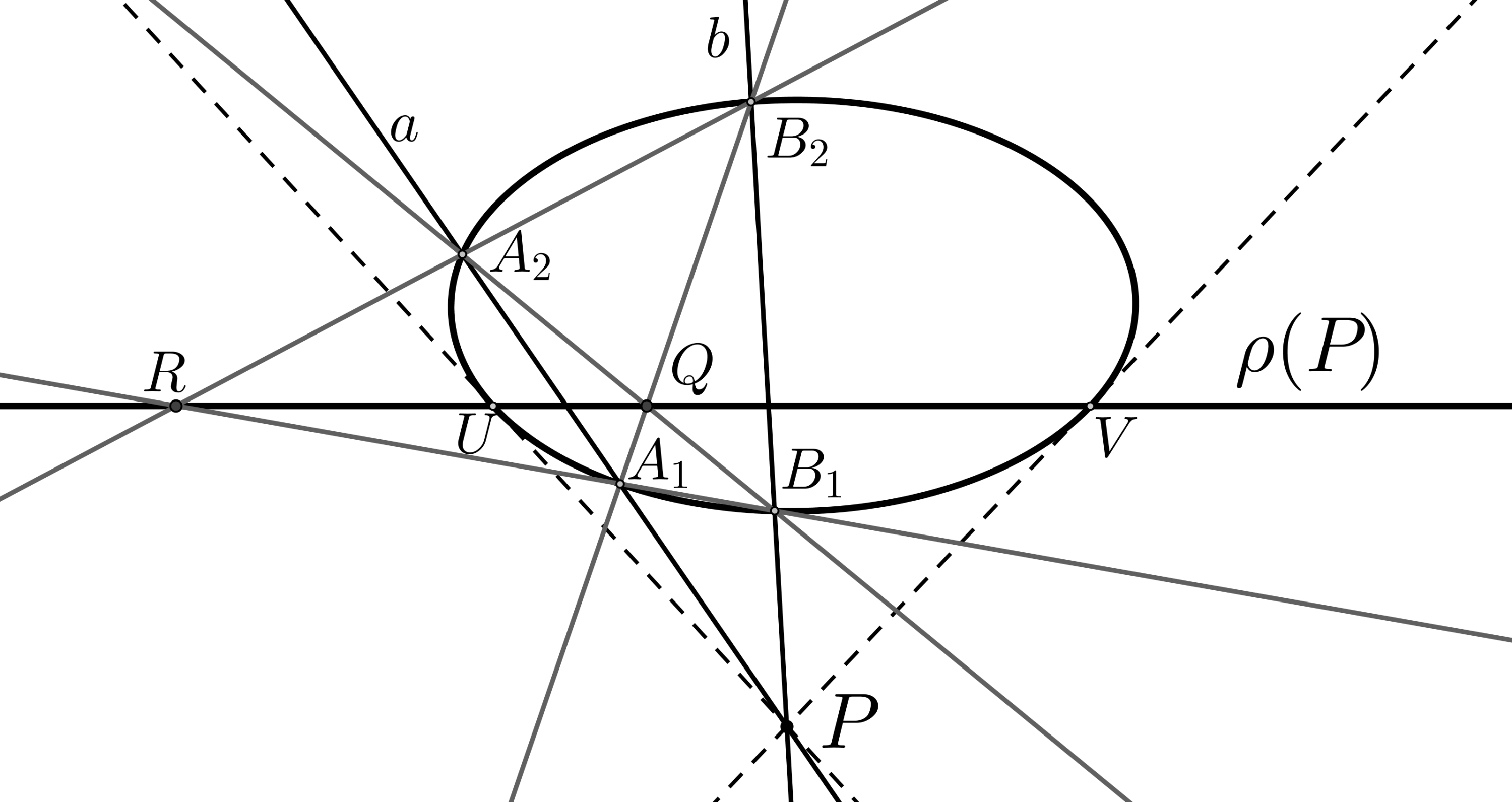}
\label{Fig:drawing-polars-P-exterior}}
\hfill
\subfigure[$P$ interior to $\Phi$]
{\includegraphics[width=0.48\textwidth]
{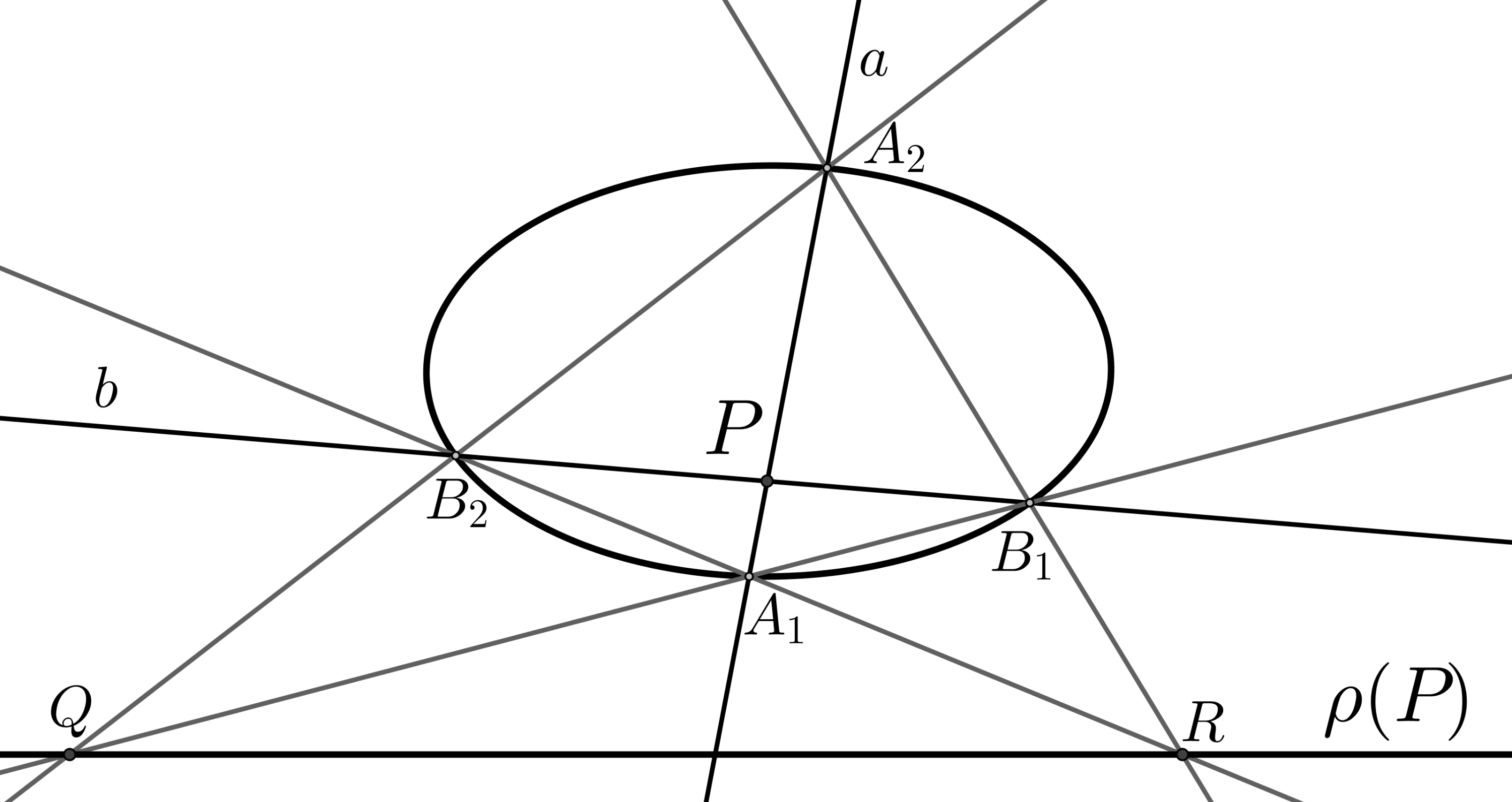}
\label{Fig:drawing-polars-P-interior}}
\caption{drawing polars}\label{Fig:drawing-polars}
\end{figure}

\section{Conjugate points and lines}

Let $p$ be a line not tangent to $\Phi$. The polarity with respect to $\Phi$ defines a 
\emph{conjugacy map} $\rho_p:p\rightarrow p$ given by
$$\rho_p(Q)=p\cdot \rho(Q)\quad\text{for all}\quad Q\in p\,.$$
We say that $\rho_p(Q)$ is \emph{the conjugate point} of $Q$ in $p$ \emph{with respect to}
$\Phi$, and we denote it by $Q_{p}$ (see Figure
\ref{Fig:conjugate-point}). If $P$ is a point not lying on $\Phi$, 
we define the map $\rho_P$ from the pencil of lines through $P$ onto itself given by
$$\rho_P(q)=P \rho(q)\,,$$
for any line $q$ with $P\in q$.
We say that $\rho_p(q)$ is \emph{the conjugate line} of $q$
\emph{with respect to} $\Phi$, and we denote it by $q_{P}$ (see 
Figure~\ref{Fig:conjugate-line}). Note that:
(i) if $Q\in p$ lies on $\Phi$ (resp. $q$ incident with $P$ is
tangent to $\Phi$), then it is $Q_{p}=Q$ (resp. $q_{P}=q$); and
(ii) the conjugacy maps $\rho_p,\rho_P$ are involutive, that is: $\left(Q_{p}\right)_{p}=Q$
and $\left(q_{P}\right)_{P}=q$.
\begin{figure}
\centering
\subfigure[conjugate of a point]{
\includegraphics[width=0.45\textwidth]
{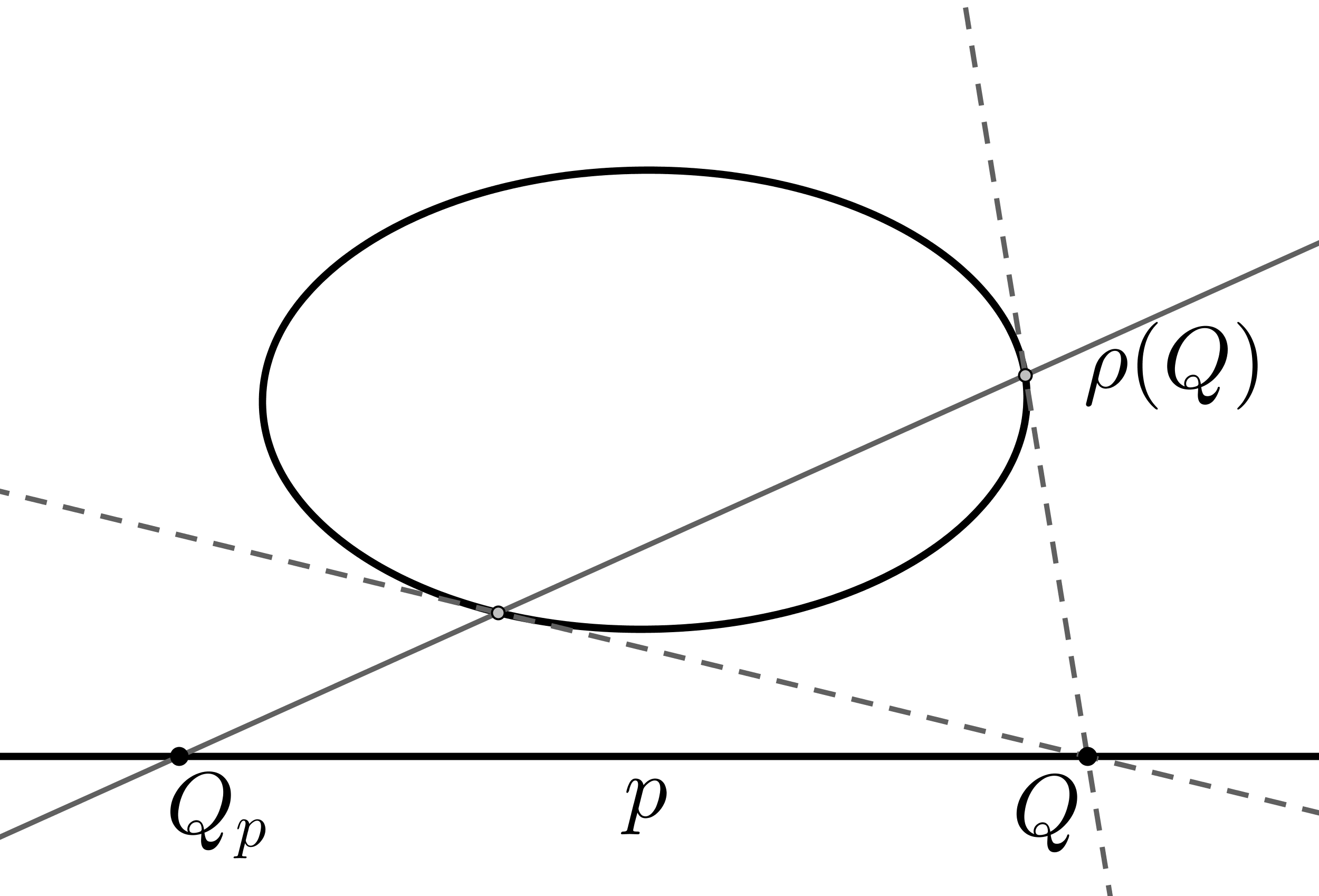}
\label{Fig:conjugate-point}}\hfill
\subfigure[conjugate of a line]{
\includegraphics[width=0.45\textwidth]
{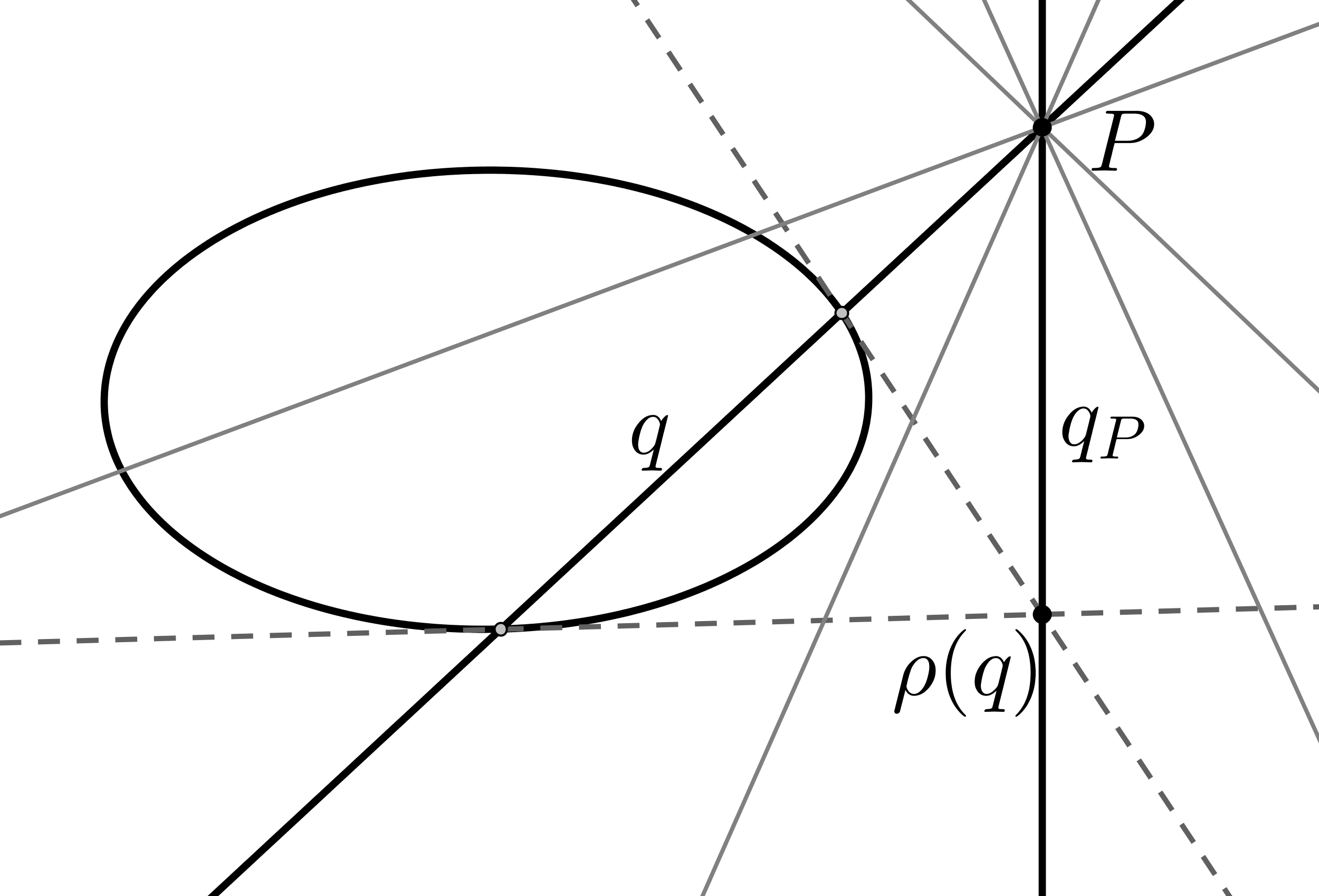}
\label{Fig:conjugate-line}}\caption{conjugacy with respect to $\Phi$}\label{Fig:conjugacy}
\end{figure}

Polarity and conjugacy are projective pransformations and therefore they
preserve cross ratios: if $p$ is
a projective line and $A,B,C,D$ are four points lying on $p$ it
is
\begin{equation}
\left(ABCD\right)=\left(\rho(A)\rho(B)\rho(C)\rho(D)\right)=\left(A_{p}B_{p}C_{p}D_{p}\right),\label{eq:invariance_cross-ratio-I}
\end{equation}
and if $P$ is a point and $a,b,c,d$ are four lines incident with
$P$ it is
\begin{equation}
\left(a\,b\,c\,d\right)=\left(\rho(a)\rho(b)\rho(c)\rho(d)\right)=\left(a_{P}\,b_{P}\,c_{P}\,d_{P}\right).\label{eq:invariance_cross-ratio-II}
\end{equation}

\section{Harmonic sets and harmonic
conjugacy}\label{sec:Harmonic-sets-harmonic-conjugacy}

Two collinear segments $\overline{AB}$ and $\overline{CD}$ form an \emph{harmonic
set} if
$$(ABCD)=-1.$$
Quadrangles are intimately
related with harmonic sets: if $P,Q$ are two diagonal points of a
quadrangle $\mathscr{Q}$, the two lines of the quadrangle not incident
with $P$ or $Q$ intersect the line $PQ$ in two points $M,N$ verifying
$(PQMN)=-1$. Indeed, if $\mathscr{Q}$ is the quadrangle
of Figure~\ref{Fig:quadrangle}, taking the points $Q_{1},Q_{2}$
given by
\[
Q_{1}=QR\cdot AB,\qquad Q_{2}=QR\cdot CD,
\]
by successive projections and sections we have
\[
(PQMN)=(PQ_{1}BA)=(PQ_{2}DC)=(PQNM),
\]
and thus by~\eqref{eq:CR1} it is $(PQMN)=(PQMN)^{-1}$. This only
could happen if $(PQMN)=\pm1$, but by 
Proposition~\ref{prop:cross-ratio-four-different-points}
it cannot be $(PQMN)=1$. Therefore it must be $(PQMN)=-1$.

Given three different collinear points $A,B,C$, the \emph{harmonic
conjugate of} $C$ \emph{with respect to} $A,B$ is the unique point
$D$, also collinear with $A,B,C$, such that $(ABCD)=-1$. 
By~\eqref{eq:CR2},
if $D$ is the harmonic conjugate of $C$ with respect to $A,B$,
then $D$ is the harmonic conjugate of $C$ with respect to $A,B$.
If $p$ is the line containing all these points, the \emph{harmonic conjugacy
map} 
$\tau_{AB}:p\rightarrow p$ that leaves $A$ and $B$ fixed and maps each point of
$p$ different from $A$ and $B$ to its harmonic conjugate with respect
to $A$ and $B$ is an involution of $p$ that is, indeed, a projectivity.
Every projective involution on a line has two (perhaps imaginary) fixed points,
and it coincides with the
harmonic conjugacy on the line with respect to its fixed points. In
particular, if $p$ is not tangent to $\Phi$ and $U,V$ are the intersection
points of $p$ with $\Phi$, the conjugacy $\rho_p$ on $p$ with respect to
$\Phi$ coincides with harmonic conjugacy $\tau_{UV}$ on $p$ with respect
to $U,V$ and for any point $A\in p$ it is 
\begin{equation*}
(UVAA_{p})=-1.
\end{equation*}

An interesting property of harmonic sets, wich will be useful later, is the
following.
\begin{lemma}\label{lem:harmonic-sets-on-the-sides}
Let $\TT=\wt{ABC}$ be a projective triangle, and let $F_1,F_2$ and $E_1,E_2$ be
two points on the side $AB$ and on the side $CA$ respectively such that
$$(ABF_1F_2)=(CAE_1E_2)=-1\,.$$
The diagonal points $D_1,D_2$ different from $A$ of the quadrilateral
$\QQ=\{E_1,E_2,F_1,F_2\}$
lie on $BC$ and $(BCD_1D_2)=-1$.
\end{lemma}
\begin{proof}
See Figure~\ref{Fig:quadrangle_02-harmonic-sets-on-the-sides}.
Consider the point $P=BE_2\cdot CF_2$, and consider also
$D_1=AP\cdot BC$. By considering the quadrangle $\{C,P,E_2, D_1\}$ we have that
$E_2D_1$ intersects the line $AB$ at the harmonic conjugate of $F_2$ with
respect to $A,B$. That is, $E_2D_1$ passes through $F_1$. In the same way, using
the quadrangle $\{B,P,F_2, D_1\}$ for example, it is proved that $E_1$ is
collinear with $F_2$ and $D_1$.

Consider now the point $Q=BE_1\cdot CF_2$, and take $D_2=AQ\cdot BC$.
By considering the quadrangle $\{C,Q,E_1, D_2\}$ we have that $E_1D_2$
intersects the line $AB$ at $F_1$. A similar argument can be used to conclude
that $D_2,E_2, F_2$ are collinear.

The points $D_1,D_2$ so constructed are the diagonal points
different from $A$ of the quadrangle $\QQ=\{E_1,E_2,F_1,F_2\}$
and it is $(BCD_1D_2)~=-1$.

\begin{figure}
\centering
\includegraphics[width=0.9\textwidth]
{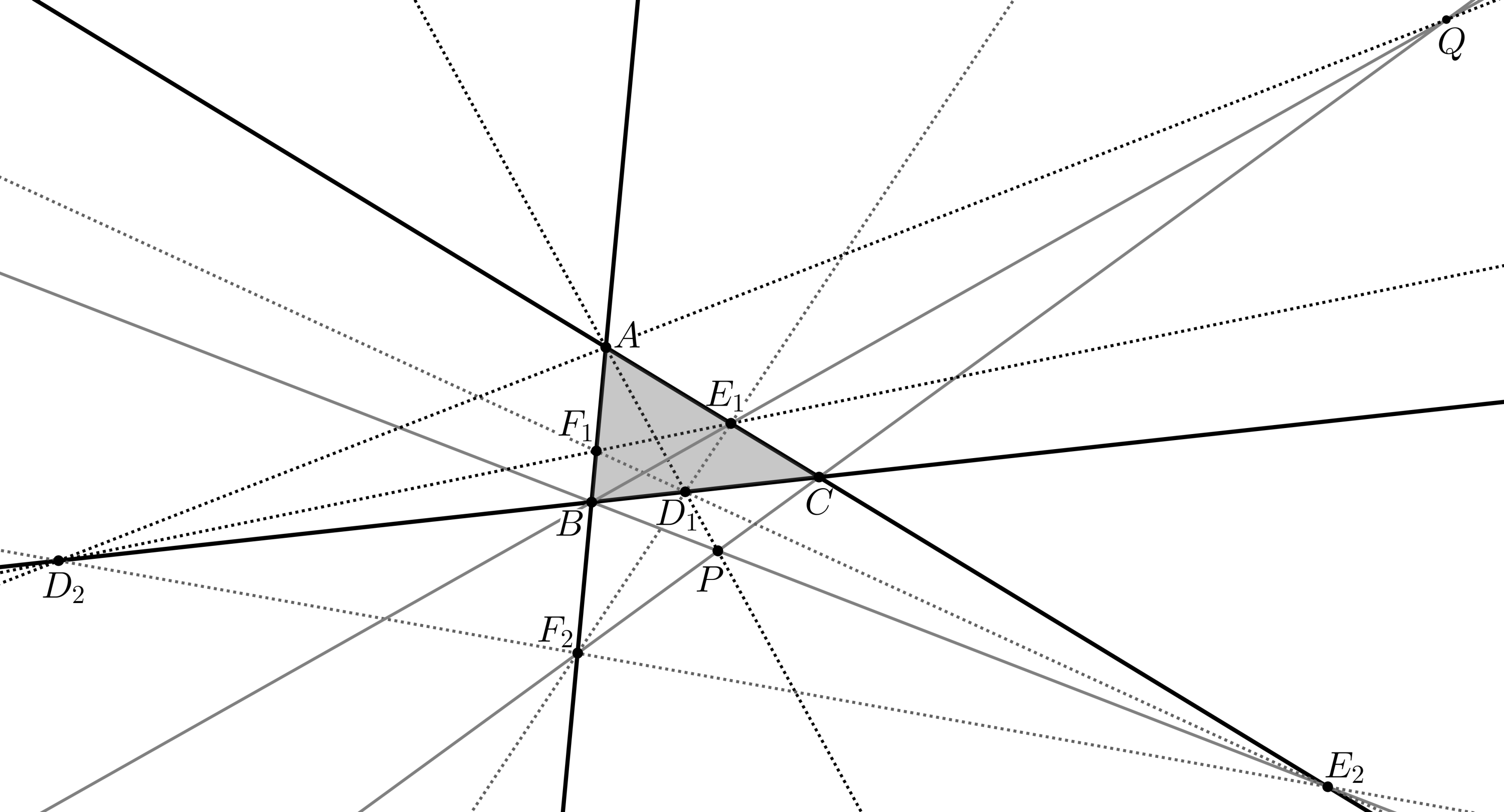}
\caption{Harmonic sets on the sides of a
triangle.}\label{Fig:quadrangle_02-harmonic-sets-on-the-sides}
\end{figure}
\end{proof}

\section{Conics whose polarity preserves real elements}

For the rest of the paper, we will consider a nondegenerate conic
$\Phi$ with respect to which the polar of a real point is always
a real line%
\footnote{The conics with this property are those that can be 
expressed by means
of a 2nd degree equation with real coefficients.%
}. This property is satified by all real conics and also by a set of
imaginary conics. Moreover, two such conics are equivalent under the
group of real projective collineations if and only if both are real
or both are imaginary (see \cite[vol. II, p. 186]{V - Y}). Through
the whole paper, when we were talking about polar lines or points
and conjugate lines or points, it must be understood that we are talking
about polar lines or points, and conjugate lines or points \emph{with respect
to} $\Phi$.

\chapter{Cayley-Klein models for hyperbolic and elliptic planar geometries}\label{sec:Cayley-Klein-models-for}

In 1871, in his paper \cite{Klein} Felix Klein presented his projective
interpretation of the geometries of Euclid, Lobachevsky and Riemann
and he introduced for them the names \emph{parabolic}, \emph{hyperbolic} and \emph{elliptic},
respectively. Klein completed the study \cite{Cayley} just made in 1859 by Arthur
Cayley by adding the geometry of Lobachevsky and
studying also the three-dimensional cases (Cayley only paid attention
to euclidean and spherical planar geometries). For the planar models,
Klein considers the non-degenerate conic $\Phi$ in the projective
plane $\mathbb{RP}^{2}\subset\mathbb{CP}^{2}$ (\emph{the absolute
conic}). When $\Phi$ is a real conic, the interior points of $\Phi$
compose the hyperbolic plane, and when $\Phi$ is an imaginary conic
the whole $\mathbb{RP}^{2}$ compose the elliptic plane. When $\Phi$
degenerates into a single line $\ell_{\infty}$, then $\mathbb{RP}^{2}\backslash
\ell_{\infty}$
gives a model of the euclidean (parabolic) plane.

We will use the common term \emph{the plane} $\mathbb{P}$ either
for the hyperbolic plane (when $\Phi$ is a real conic) or for the elliptic
plane (when $\Phi$ is an imaginary conic). 
Geodesics in these models are given by the intersection with $\mathbb{P}$ of
real projective lines, and
rigid motions are given by the set of real collineations
that leave invariant the absolute conic. 

In the hyperbolic plane, we still use expressions as ``the intersection point of
two lines'' or ``concurrent lines'' in a projective sense, even if the referred
point of intersection or concurrency is not interior to $\Phi$. A common
notation in hyperbolic geometry is to call ``parallel'' to those lines
intersecting on $\Phi$ and ``ultraparallel'' to those lines intersecting
outside $\Phi$. If $a,b$ are hyperbolic lines intersecting outside $\Phi$,
then $\rho(a\cdot b)$ is the common perpendicular to $a$ and $b$.

\section{Distances and angles}\label{sub:Cayley-Klein-models-distances-angles}

\begin{figure}
\centering
\hfill
\subfigure[]{\includegraphics[width=0.48\textwidth]
{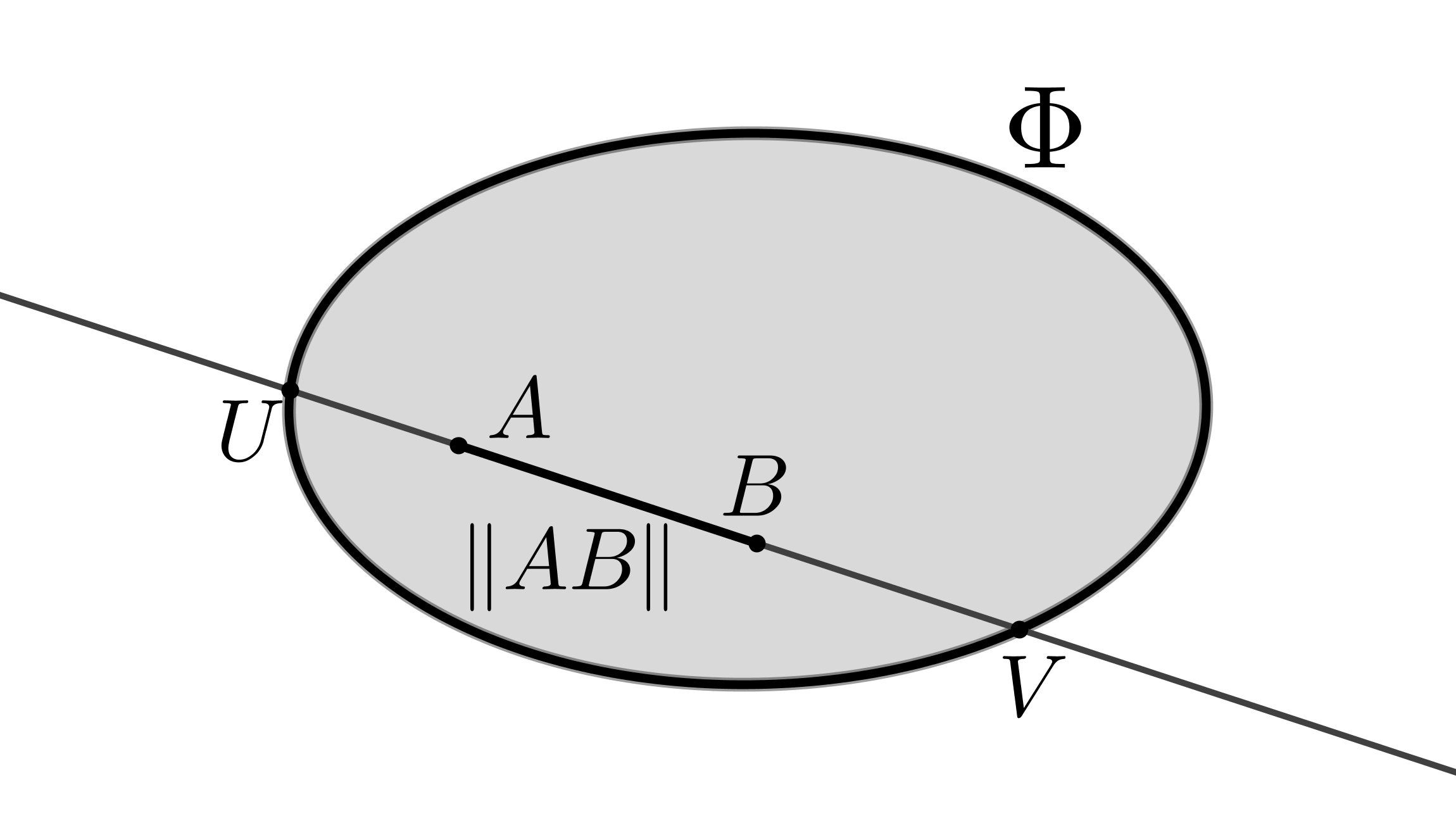}
\label{Fig:distance-hyperbolic}}\hfill
\subfigure[]{\includegraphics[width=0.48\textwidth]
{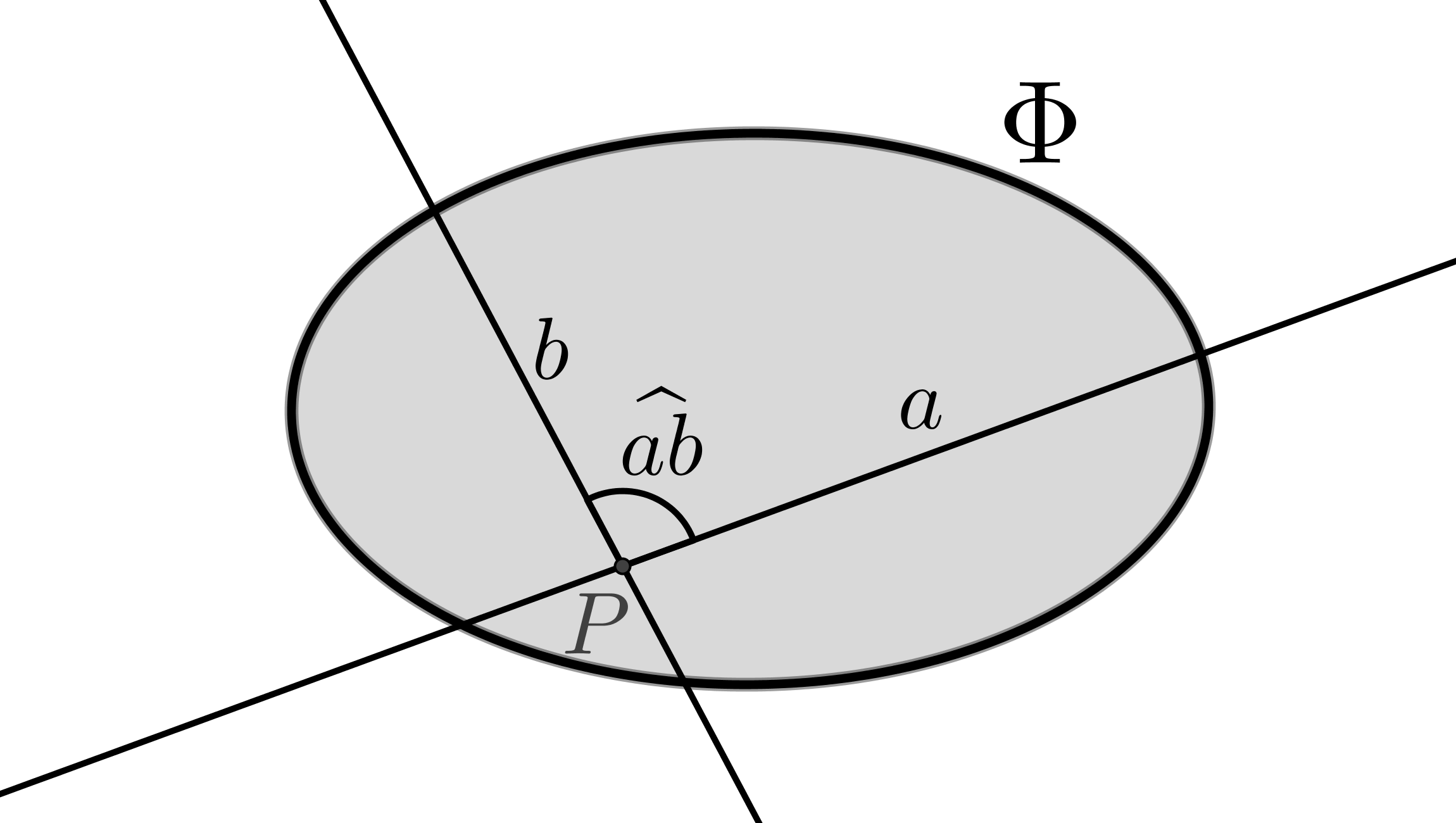}
\label{Fig:angle-hyperbolic}}\hfill\hfill
\caption{distances and angles in the hyperbolic plane}\label{Fig:distances-angles-hyperbolic}
\end{figure}

For two points $A,B\in\mathbb{P}$, let $U,V$ be the two intersection
points of the line $AB$ with $\Phi$. The points $U,V$ will be two
real points if the absolute conic is real (Figure~\ref{Fig:distance-hyperbolic}), and they will be imaginary
if the absolute conic is imaginary. The non-euclidean distance $\left\Vert AB\right\Vert $
between them is given by
\begin{equation}\label{eq:distances-hyperbolic}
	\left\Vert AB\right\Vert :=\dfrac{1}{2}\log\left(UVAB\right)
\end{equation}
if the absolute conic is real, and by
\begin{equation}\label{eq:distances-elliptic}
\left\Vert AB\right\Vert :=\dfrac{1}{2i}\log\left(UVAB\right)
\end{equation}
if the absolute conic is imaginary. The factors $\frac{1}{2},\frac{1}{2i}$
before the natural logarithms of the above expressions are arbitrary,
and they are taken in order to make the hyperbolic and elliptic
planes' curvature equal to $-1$ or $+1$ respectively.
In the elliptic case the total length of a line turns out to be $\pi$.

For two lines $a,b$ such that $a\cdot b\in\mathbb{P}$, let $u,v$
be the two tangent lines to $\Phi$ through $a\cdot b$. The lines
$u,v$ are both imaginary because of the hypothesis $a\cdot b\in\mathbb{P}$ 
(Figure~\ref{Fig:angle-hyperbolic}).
The angle $\widehat{ab}$ between the lines $a,b$ at their intersection
point $a\cdot b$ is given in the hyperbolic and elliptic cases by
the same formula
\begin{equation}
\widehat{ab}:=\dfrac{1}{2i}\log\left(u\,v\,a\,b\right).\label{eq:Laguerre-formula}
\end{equation}
This formula is similar to the formula given in \cite{Laguerre} for
computing euclidean angles, and for this reason we will refer to it
as  \emph{Laguerre's formula}. Note that $\widehat{ab}$ is an angle
between \emph{lines}, not an angle between \emph{rays} as usual (see
\cite{Forder,Johnson, Picken}). In particular, $\widehat{ab}$ does not distinguish
an angle from its supplement and it takes values between $0$ and
$\pi$.
We'll come back on this discussion in \S\ref{sec:Cosine-Rule}
and in the Appendix.

In our pictures, we will depict the hyperbolic plane as the interior
points of an ellipse in $\mathbb{R}^{2}=\mathbb{RP}^{2}\backslash
\ell_{\infty}$,
and the elliptic plane as a round 2-sphere of unit radius where antipodal
points are identified. In this picture of the elliptic plane it is
easier to visualize the pole-polar relation: lines are represented
by great circles of the sphere, and the pole of a line is the pair
of antipodal points representing the north and south poles when the
line is chosen as the equator of the 2-sphere. This elliptic model
has the virtue also of being locally isometric. In particular, angles
and lengths of segments lower than $\pi$ are the same in the modelling
sphere as in the elliptic plane.

Conjugacy of lines with respect to the absolute characterizes perpendicularity in both models 
(Figure~\ref{Fig:orthogonality}, see \cite[p.36]{Santalo}).
\begin{proposition}
\label{prop:perpendicular-iff-conjugate}The lines $a,b$ such that
$a\cdot b\in\mathbb{P}$ are perpendicular if and only if they are conjugate.
\end{proposition}

\begin{figure}
\centering
\hfill
\subfigure[]{\includegraphics[width=0.45\textwidth]
{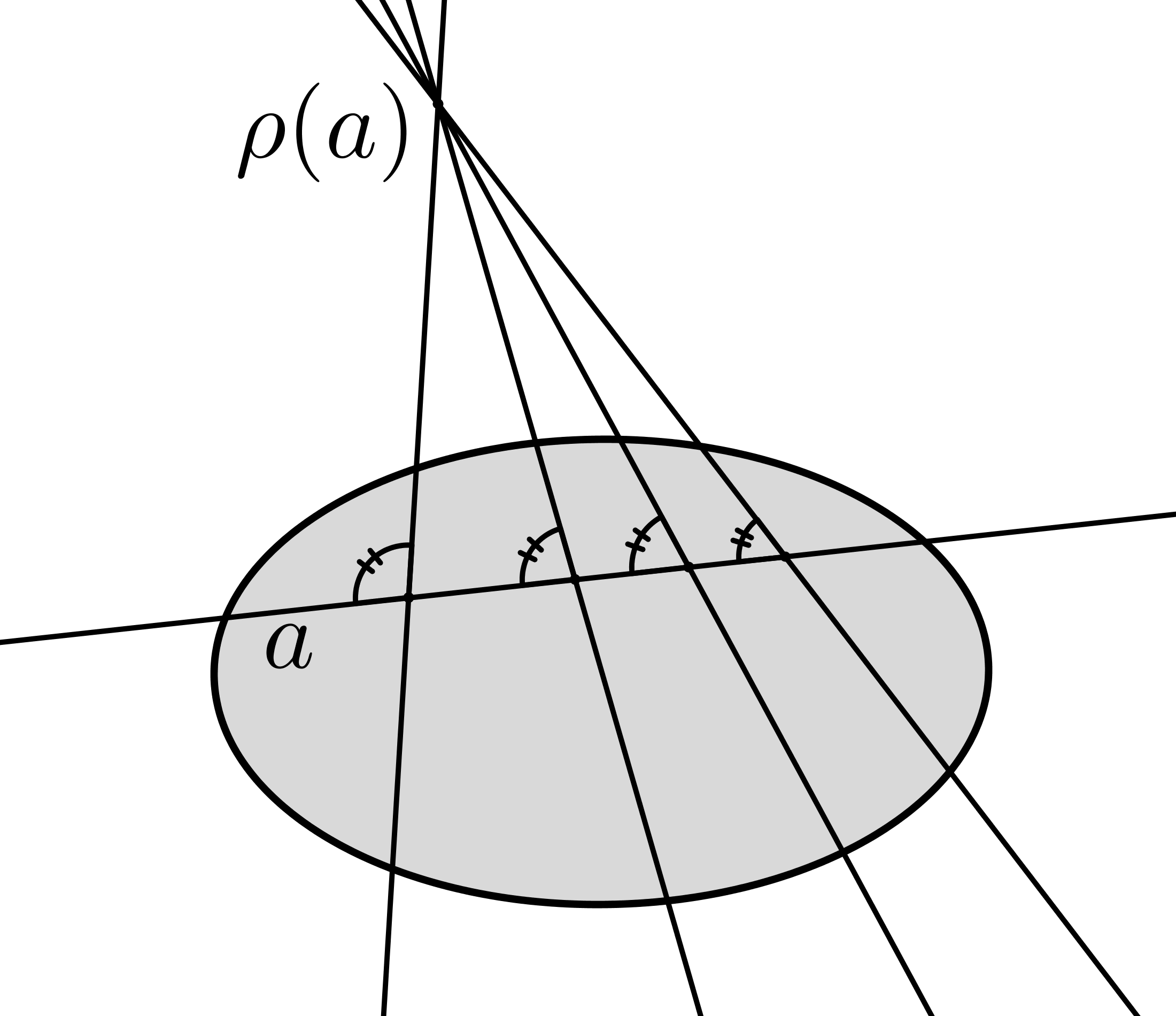}
\label{Fig:orthogonality-hyperbolic}}
\hfill
\subfigure[]{\includegraphics[width=0.42\textwidth]
{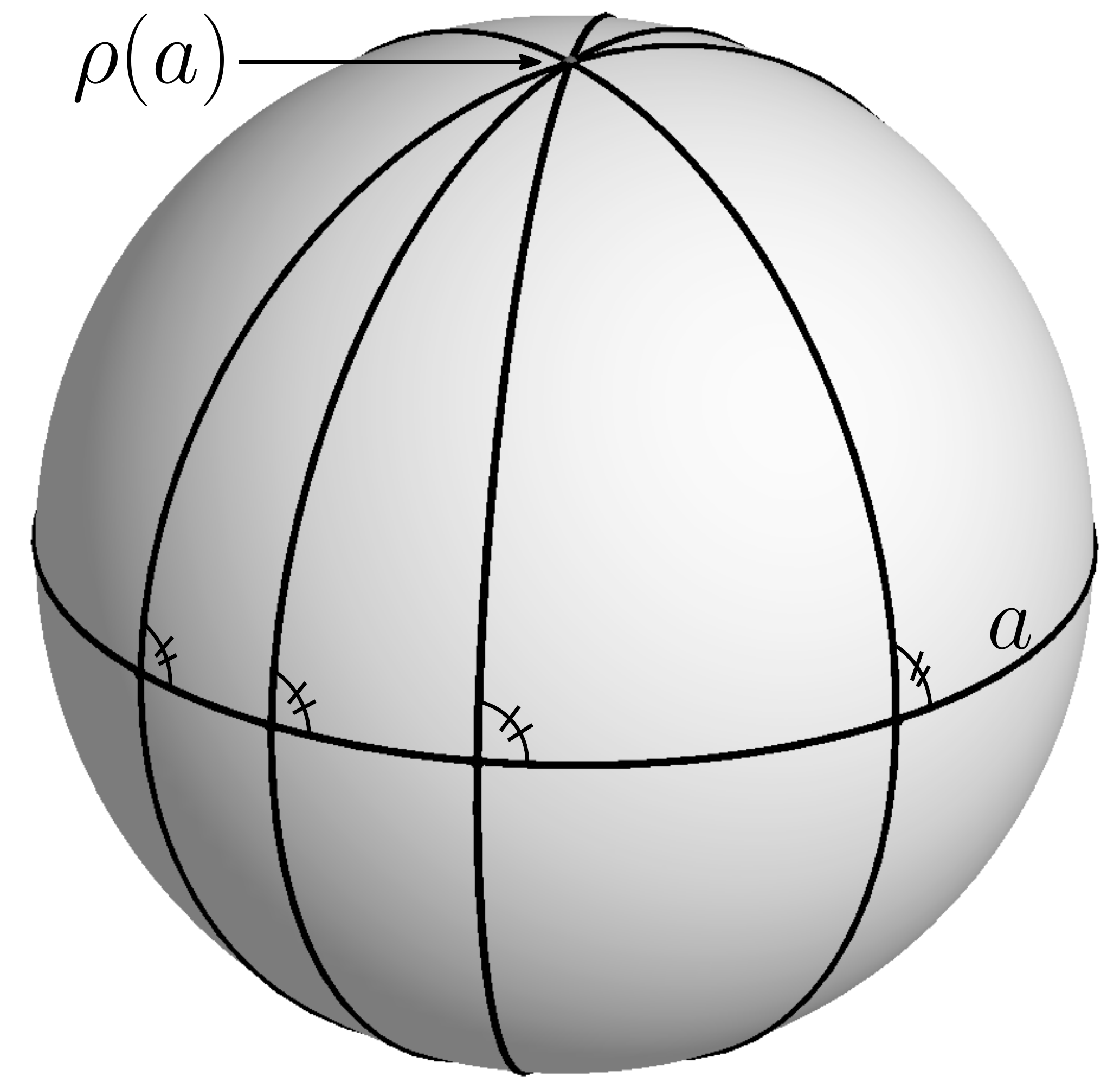}
\label{Fig:othogonality-elliptic}}
\hfill\hfill
\caption{orthogonality in Cayley-Klein models}\label{Fig:orthogonality}
\end{figure}

\section{Midpoints of a segment}\label{sec:midpoints-of-a-segment}

Take a segment $\ov{AB}$ in the projective plane whose endpoints $A,B$
do not lie on $\Phi$ and such that the line $p=AB$ is not tangent to
$\Phi$. Let $P$ be the pole of $p$, let $a,b$ be the lines joining
$P$ with $A,B$ respectively, and let $A_{1},A_{2}$ and $B_{1},B_{2}$
be the intersection points with $\Phi$ of $a$ and $b$ respectively (Figure
\ref{Fig:midponts-of-a-segment-1}).
The points $A_{1},A_{2},B_{1},B_{2}$ are the vertices of a quadrangle
$\mathscr{Q}$ inscribed in $\Phi$ having $P$ as a diagonal point.
By Theorem~\ref{thm:polars-quadrangles}, the other two diagonal points
$Q,R$ of $\mathscr{Q}$ lie in $p$ and it is $R=Q_{p}$. The segments
$\overline{AB}$ and $\overline{QQ_{p}}$ form a harmonic set and so it is
$(ABQQ_{p})=-1$.

Let $U,V$ be the two intersection points of $p$ with $\Phi$. As
$$(UVQQ_{p})=(ABQQ_{p})=-1\,,$$
the harmonic conjugacy of $p$ with respect
to $Q,Q_{p}$ sends the points $U,V,A,B$ into the points $V,U,B,A$
respectively. In particular, it is $(UVAQ)=(VUBQ)$, and by~\eqref{eq:CR1}
it is $(VUBQ)=(UVQB)$. By \eqref{eq:CR3}, we have
\begin{equation}
(UVAB)=(UVAQ)(UVQB)=(UVAQ)^{2}.\label{eq:midpoints-01}
\end{equation}
In the same way, it can be checked that
\begin{equation}
(UVAB)=(UVAQ_{p})^{2}=(UVQB)^{2}=(UVQ_{p}B)^{2}.\label{eq:midpoints-02}
\end{equation}
If the points $A,B,Q$ belong to $\mathbb{P}$, from the distance 
formulae~\eqref{eq:distances-hyperbolic} and~\eqref{eq:distances-elliptic}, we have
$$\Vert AB\Vert=2\Vert AQ\Vert=2\Vert BQ\Vert.$$ This property suggest the following definition:
\begin{definition}
\label{def:midpoints}The points $Q,Q_{p}$ are the \emph{midpoints}
of the segment $\overline{AB}$
.
\end{definition}
If $A,B\in\mathbb{P}$, when $\Phi$ is a real conic there is exactly
one of the midpoints of the projective segment $\ov{AB}$, say $Q$, lying in the hyperbolic segment bounded
by $A$ and $B$. In this case, $Q$ is the geometric midpoint of
the segment  and $Q_{p}$ is the pole of the perpendicular bisector of
this segment. When $\Phi$ is an imaginary conic, then $A$ and $B$
divide the line $AB$ into two elliptic segments whose respective
midpoints are $Q$ and $Q_{p}$.

It is clear that the lines $PQ,PQ_p$ are the polars of $Q_p,Q$ respectively. As
these lines are orthogonal to $p$ through the midpoints of $\ov{AB}$, they are
the \emph{segment bisectors} or simply \emph{bisectors} of $\ov{AB}$. After
dualizing:

\begin{remark}\label{rem:midpoints-bisectors}
The polars of the midpoints of
$\ov{AB}$ are the \emph{angle bisectors} or \emph{bisectors}  of
$\widehat{\rho(A)\rho(B)}$.
\end{remark}

An interesting characterization of midpoints is the following:
\begin{lemma}\label{lem:midpoints-harmonic-UV-AB}
Let $p$ be a line not tangent to $\Phi$, let $p\cdot\Phi=\{U,V\}$,
and take two points $A,B\in p$ different from $U,V$. If $C,D$ are two points of $AB$ such that
\[
(ABCD)=(UVCD)=-1,
\]
then $C,D$ are the midpoints of $\ov{AB}$.
\end{lemma}

\begin{proof}
Let $C,D$ be two points of $p$ such that $(ABCD)=(UVCD)=-1$.
We have seen that the midpoints $Q,Q_p$ of $\ov{AB}$ verify also the identity
$(UVQQ_{p})=(ABQQ_{p})=-1$. 
If we take the harmonic involutions $\tau_{QQ_p}$ and $\tau_{CD}$, of $p$ with fixed
points $Q,Q_p$ and $C,D$ respectively, the composition $\tau_{QQ_p}\circ\tau_{CD}$
fixes the points $U,V,A,B$ and so it must be the identity on $p$. This implies
that $\{Q,Q_p\}=\{C,D\}$.
\end{proof}

An immediate consequence of this lemma is:
\begin{lemma}\label{lem:midpoints-AB-midpoints-ApBp}
Let $A,B$ be two points not in $\Phi$ such that $p=AB$ is not tangent to
$\Phi$. Then, the midpoints of $\ov{AB}$ are also the midpoints of
$\ov{A_pB_p}$.
\end{lemma}
\begin{proof}
Let $U,V$ be the intersection points of $p$ with $\Phi$, and let $Q,Q_p$ be the
midpoints of $\ov{AB}$. It is clear that $(UVQQ_p)=-1$. If we apply on $p$
the conjugacy $\rho_p$ with respect to $\Phi$, we get
$$(ABQQ_p)=(A_pB_pQ_pQ)=-1\,,$$
and thus by Lemma~\ref{lem:midpoints-harmonic-UV-AB} $Q,Q_p$ are the midpoints
of $\ov{A_pB_p}$.
\end{proof}

\begin{figure}
\centering
\subfigure[]{\includegraphics[width=0.8\textwidth]
{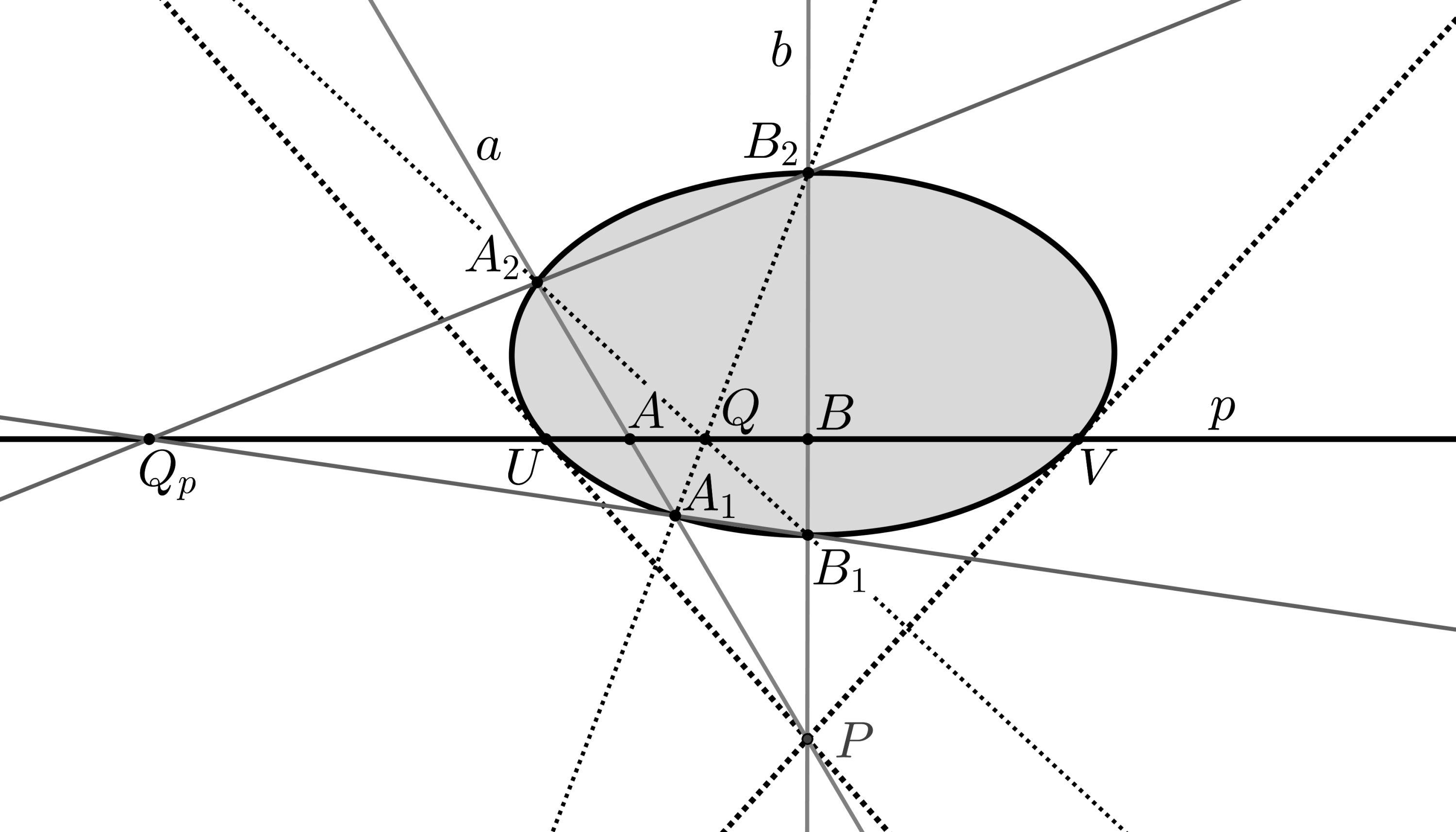}
\label{Fig:midponts-of-a-segment-1}}\\
\subfigure[]{\includegraphics[width=0.8\textwidth]
{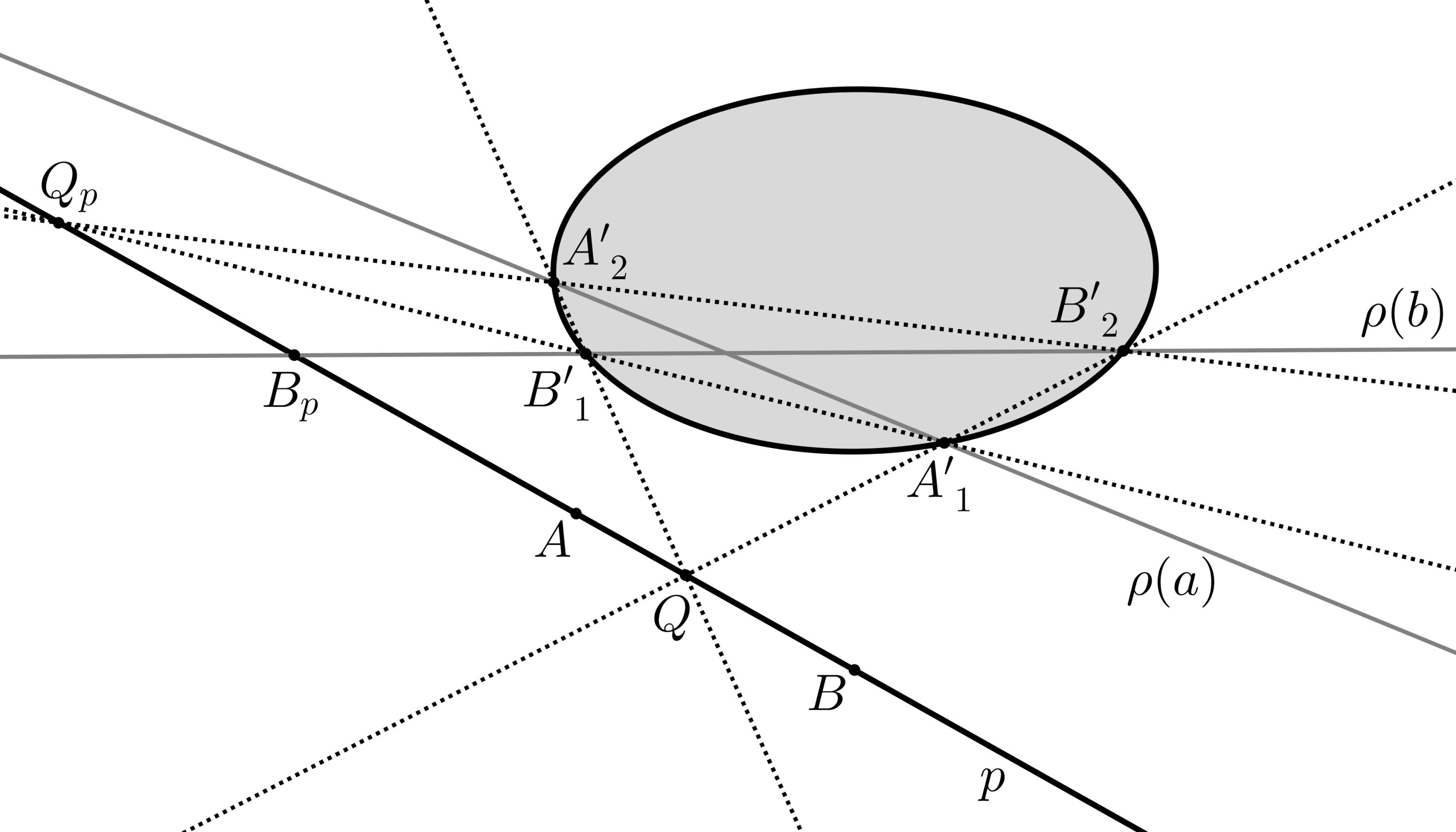}
\label{Fig:midponts-of-a-segment-2}}
\caption{midpoints of a segment}\label{Fig:midponts-of-a-segment}
\end{figure}

The previous lemma gives an alternative way for constructing the midpoints of
$\ov{AB}$ using polars (Figure~\ref{Fig:midponts-of-a-segment-2}): (i) take the
polars $\rho(A),\rho(B)$; (ii) take the intersection points $A'_1,A'_2$ and
$B'_1,B'_2$ of $\rho(A)$ and $\rho(B)$ respectively with $\Phi$; and (iii) the
midpoints of $\ov{AB}$ are the diagonal points of the quadrangle
$\{A'_1,A'_2,B'_1,B'_2\}$ different from $\rho(AB)$. This is the result of
applying the original construction of the midpoints to the segment
$\ov{A_pB_p}$.

In the general case, we will work with segments such that the line they belong
to is \emph{in general position} with respect to $\Phi$, that is, not tangent to
$\Phi$. Nevertheless, sometimes we will talk about the midpoints of a
segment whose underlying line could be tangent to $\Phi$. For dealing with
such limit cases, we extend the concept of ``midpoint'' to such
segments in the natural way. If in Figure
\ref{Fig:midponts-of-a-segment-2} we move continuously the points $A,B$ in
order to make $p$ tangent to $\Phi$ while $A,B$ remain out of $\Phi$, one of
the points $A'_1,A'_2$ (say $A'_1$) become coincident with one of the points
$B'_1,B'_2$ (say $B'_1$) and with one of the points $Q,Q_p$ (say $Q$) at the
contact point of $p$ with $\Phi$ (see Figure
\ref{Fig:midponts-of-a-segment-3}). The other midpoint $Q_p$ will become the
point $p\cdot A'_2B'_2$ and the identity $(ABQQ_p)=-1$ remains unchanged.

    \begin{definition}\label{def:midpoints-tangent}
	Let $A,B$ be two points not lying on $\Phi$ such that the line $AB$ is
	tangent to $\Phi$ at a point $Q$. We say that the \emph{midpoints} of
	the
	segment $\ov{AB}$ are $Q$ and the harmonic conjugate of $Q$ with respect
	to $A,B$.
    \end{definition}

\begin{figure}
\centering
\includegraphics[width=0.7\textwidth]
{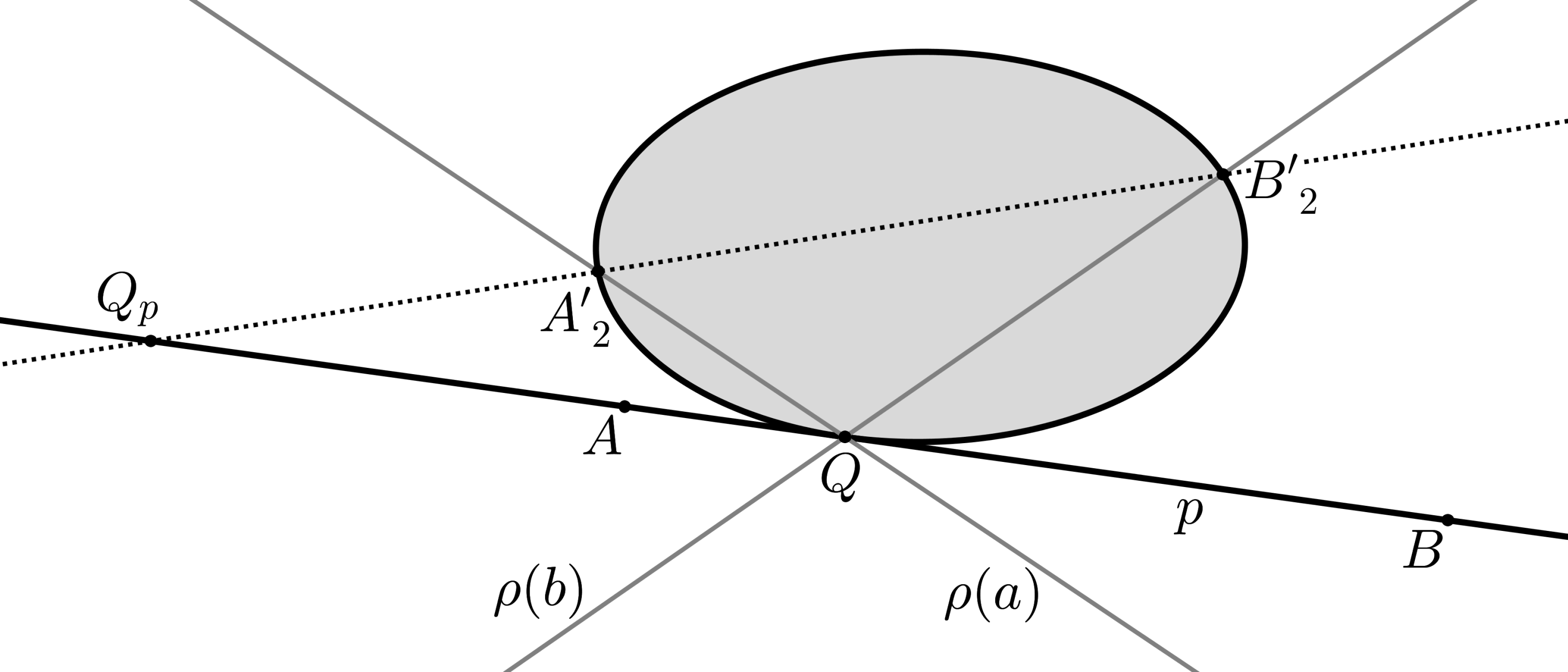}
\caption{Midpoints on a line tangent to
$\Phi$}\label{Fig:midponts-of-a-segment-3}
\end{figure}

Lemma~\ref{lem:midpoints-harmonic-UV-AB}
inspires the following notation. If $Q$ is a point not in $\Phi$ and $p$ is a
line through $Q$ not tangent to $\Phi$, we will say that the harmonic involution
$\tau_{QQ_p}$ of $p$ with respect to $Q$ and $Q_p$ is the \emph{symmetry of $p$
with respect to $Q$}. If we consider the
symmetry with respect to $Q$ acting on every line through $Q$ not tangent to
$\Phi$ at once, we obtain the \emph{symmetry of $\PP$ with respect to $Q$},
which corresponds projectively to the harmonic homology of the projective plane
with center $Q$ and axis $\rho(Q)$.

\section{Proyective trigonometric ratios}

%

The points $U,V$ of~\eqref{eq:distances-elliptic} and the lines $u,v$ 
of~\eqref{eq:Laguerre-formula} are imaginary. This is the reason why it is more usual to use another cross ratios when working with distances and angles.
    \begin{theorem}\label{thm:ABBpAp-angles-and-distances}
	Let $A,B$ be two points in $\PP$, and let $p$ be the line joining them.
	If $\PP$ is the hyperbolic plane, then
	    \begin{equation}
		(ABB_p A_p)=\cosh^{2}\VRT{AB};\label{eq:ABBpAp-hyperbolic}
	    \end{equation}
	and if $\PP$ is the elliptic plane, then
	    \begin{equation}
		(ABB_p A_p)=\cos^{2}\VRT{AB}.\label{eq:ABBpAp-elliptic}
	    \end{equation}
	Let $a,b$ be two projective lines such that their intersection point $P$ lies in $\PP$. Then,
	    \begin{equation}
		(a\,b\,b_{P}\,a_{P})=\cos^{2}\widehat{ab}.\label{eq:abbPaP-hyperbolic-elliptic}
	    \end{equation}
    \end{theorem}
For proving this theorem, we need the following lemma:
    \begin{lemma}\label{lem:ABBpAp-UVAB-UVBA}
	Let $p$ be a projective line not tangent to $\Phi$, let $U,V$ be the two intersection points 
	of $p$ with $\Phi$, and let $A,B$ be two points of $p$ different from $U,V$. Then
	    \begin{equation}
		4(ABB_pA_p)=(UVAB)+(UVBA)+2.\label{eq:lemma-4ABBpAp-UVAB}
	    \end{equation}
    \end{lemma}
Identity~\eqref{eq:lemma-4ABBpAp-UVAB} appears in \cite[p. 24]{Santalo}. We
give a different proof.

\begin{proof}
We will make a reiterated use of cross ratio 
identities~\eqref{eq:cross_ratio_identities}. We have
    \begin{align*}
	(ABB_pA_p)&=(B_pA_pAB)=(UA_pAB)(B_pUAB)=\\
	&=(VA_pAB)(UVAB)(B_pUAB).
    \end{align*}

If $Q,Q_p$ are the two midpoints of the segment $\ov{AB}$, the symmetry
$\tau_{QQp}$ sends $U,A,A_p$ into $V,B,B_p$ respectively and vice versa. So it
is
$$(VA_pAB)=(A_pVBA)=(B_pUAB),$$
and therefore
$$(ABB_pA_p)=(UVAB)(B_pUAB)^2.$$
The cross ratio $(B_pUAB)$ can be splitted as
$(B_pUVB)(B_pUAV)$,
and because of the identity $(UVB_pB)=-1$, we have
\begin{align*}
(B_pUVB)&=\dfrac{1}{(UB_pVB)}=\dfrac{1}{1-(UVB_pB)}=\dfrac{1}{2};\\
(B_pUAV)&=(UB_pVA)=1-(UVB_pA)=1-(UVBA)(UVB_pB)=\\
&=1+(UVBA);
\end{align*}
which finally gives
$$(ABB_pA_p)=(UVAB)(B_pUAB)^2=\dfrac{1}{4}(UVAB)[1+(UVBA)]^2.$$
The proof now follows from~\eqref{eq:CR1}.
\end{proof}

\begin{proof}[Proof of Theorem~\ref{thm:ABBpAp-angles-and-distances}]
We will prove only~\eqref{eq:ABBpAp-hyperbolic}.
The identities~\eqref{eq:ABBpAp-elliptic} 
and~\eqref{eq:abbPaP-hyperbolic-elliptic} can be proved in a similar way.

Let assume that $\PP$ is the hyperblic plane, 
and let denote $\VRT{AB}$ simply by $d$. Then,
$$\cosh d=\dfrac{1}{2}(e^d+e^{-d})=\dfrac{1}{2}[(UVAB)^{1/2}+(UVAB)^{-1/2}],$$
and so by Lemma~\ref{lem:ABBpAp-UVAB-UVBA} it is
$$\cosh^2 d=\dfrac{1}{4}[(UVAB)+(UVBA)+2]=(ABB_pA_p).$$
\end{proof}

\begin{figure}[t]
\centering
\subfigure[elliptic segment and its polar angle]
{\includegraphics[width=0.35\textwidth]
{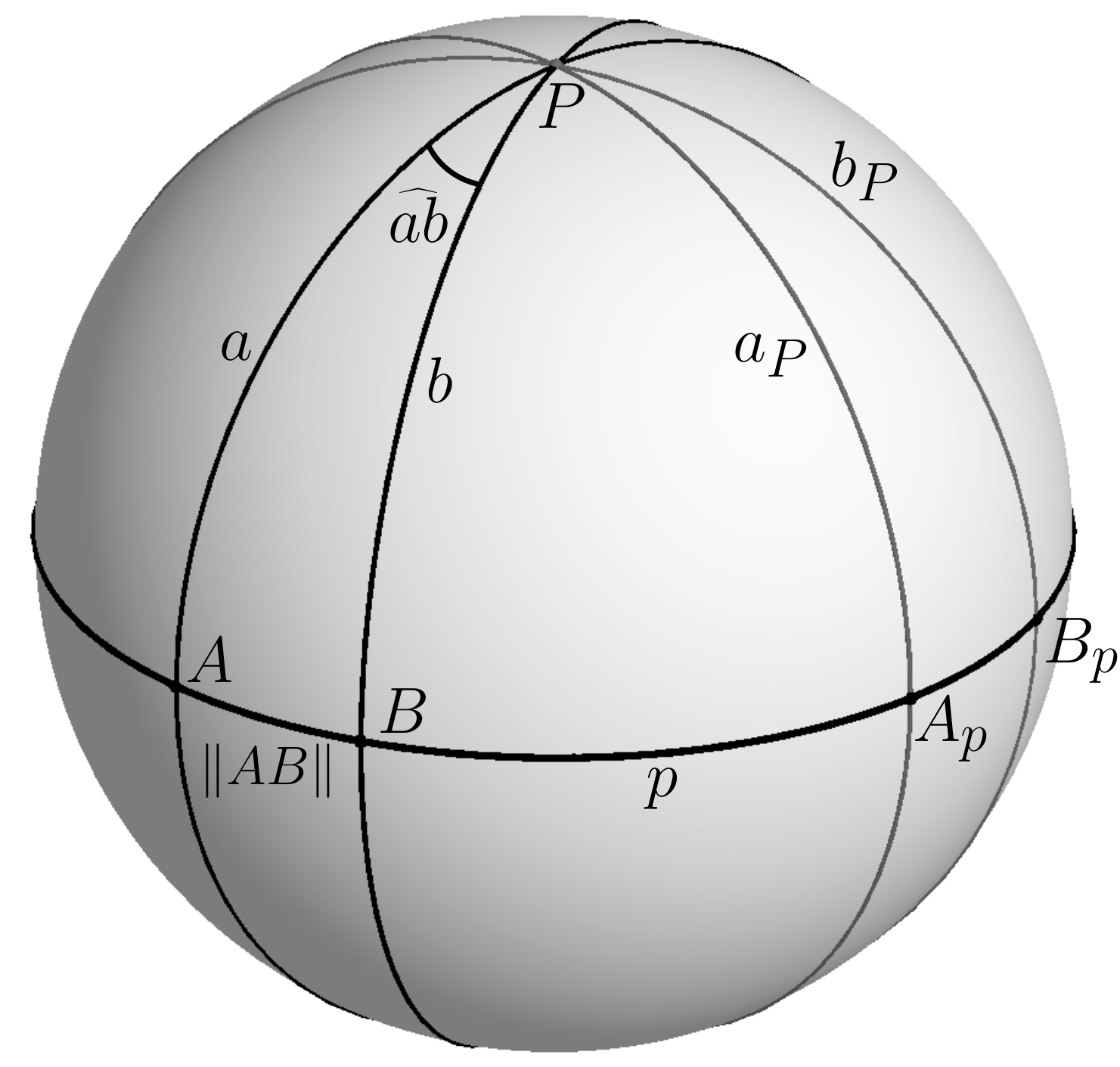}
\label{Fig:segment-spherical}}
\\
\begin{tabular}{cc}
\subfigure[hyperbolic segment 1]
{\includegraphics[width=0.48\textwidth]
{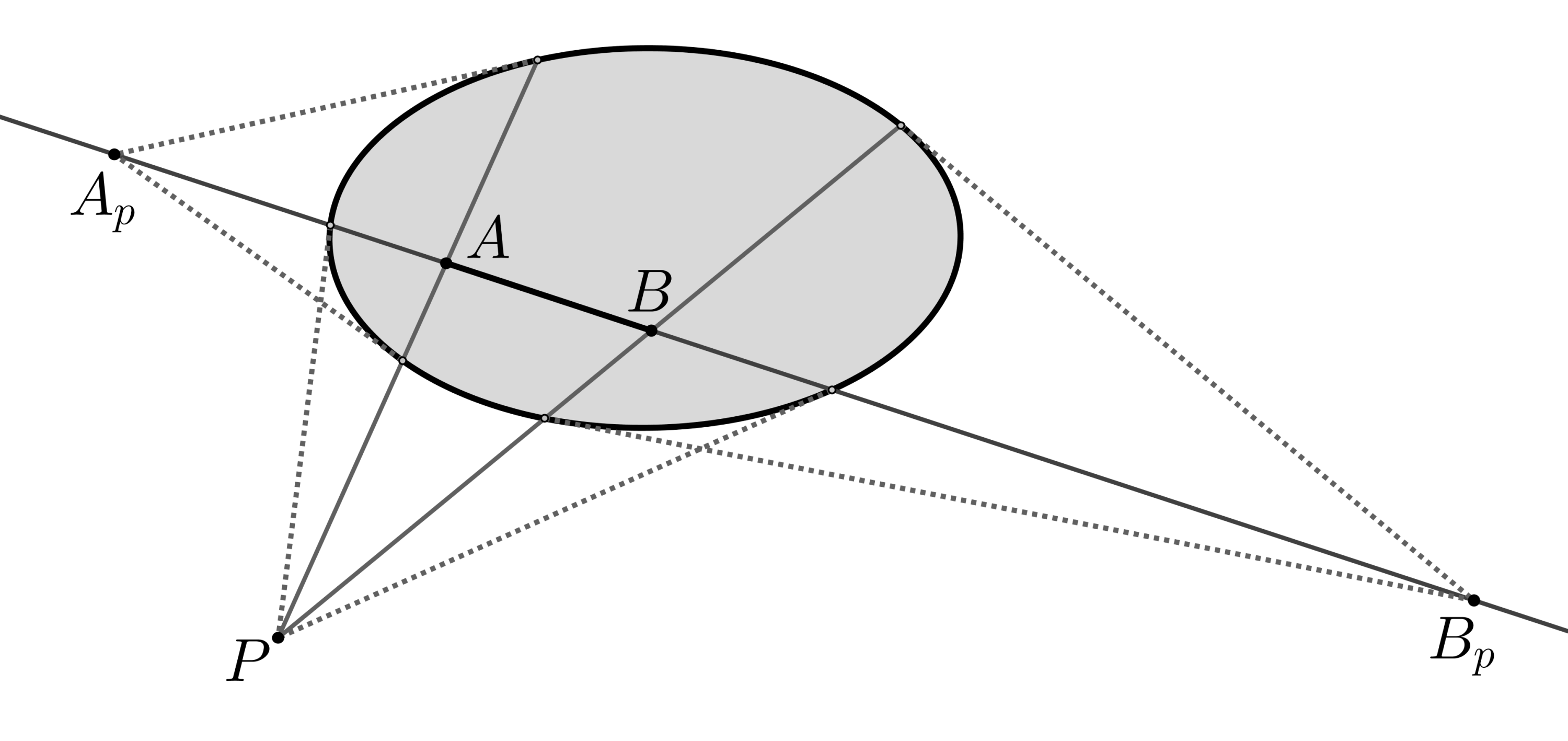}
\label{Fig:segment-hyperbolic-01}}
&
\subfigure[hyperbolic segment 2]
{\includegraphics[width=0.48\textwidth]
{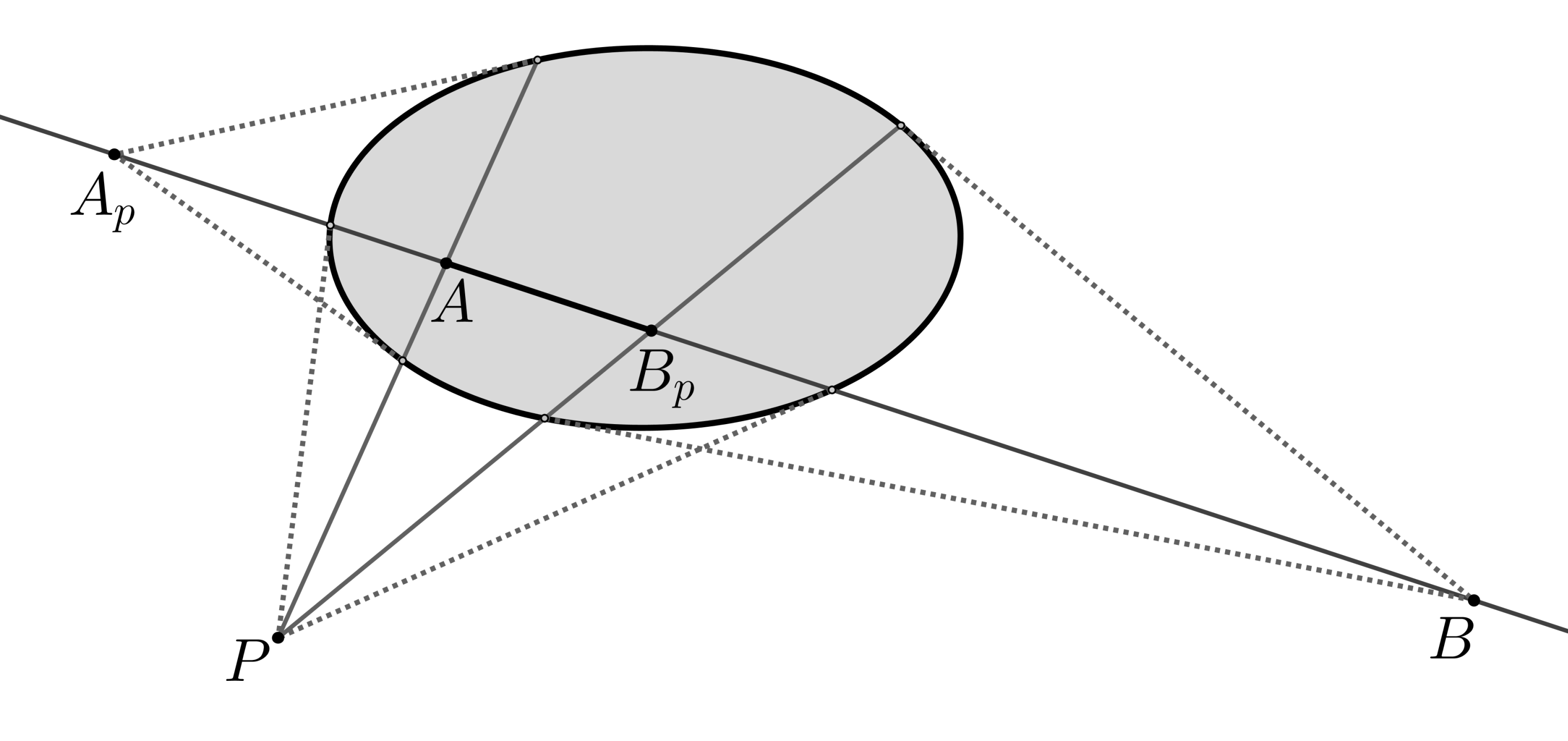}
\label{Fig:segment-hyperbolic-02}}
\\
\subfigure[hyperbolic segment 3]
{\includegraphics[width=0.48\textwidth]
{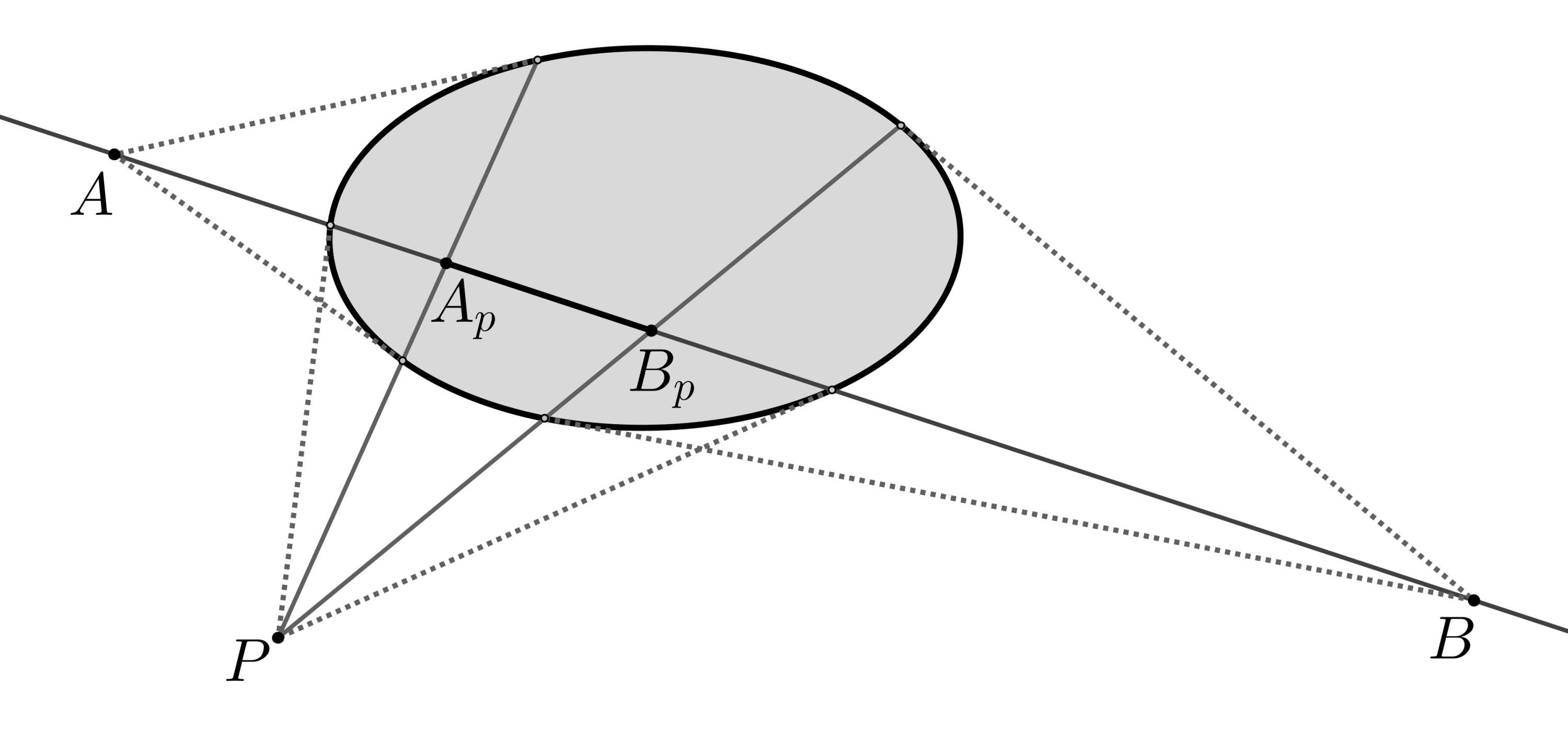}
\label{Fig:segment-hyperbolic-03}}
&
\subfigure[hyperbolic angle]
{\includegraphics[width=0.48\textwidth]
{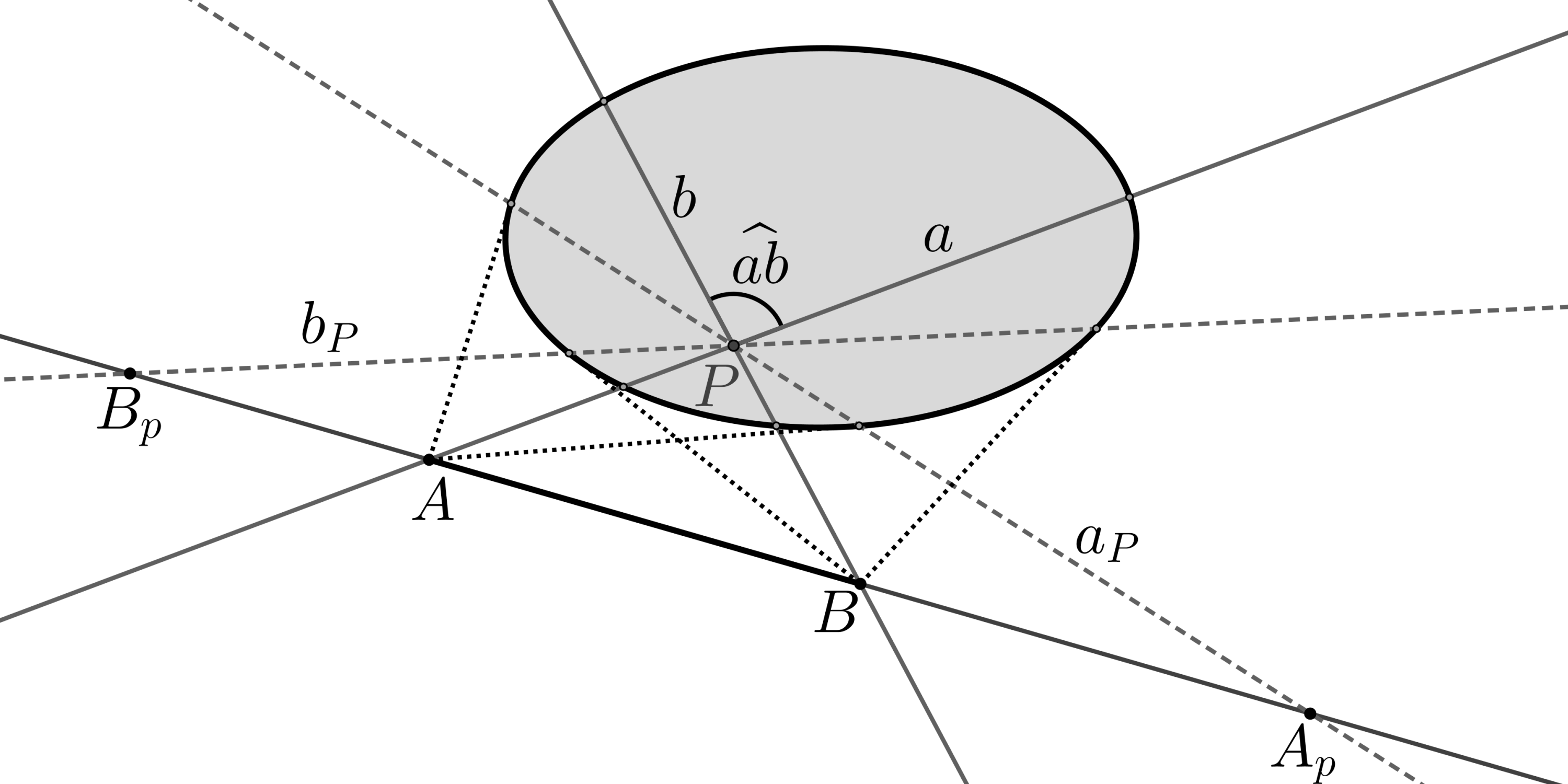}
\label{Fig:segment-hyperbolic-04}}
\end{tabular}
\caption{Positions of a projective segment with respect to the conic}
\label{Fig:non-euclidean_segments}
\end{figure}
\begin{table}
\renewcommand{\arraystretch}{1.5}
\centering
\footnotesize{\begin{tabular}{c|c|c|c|ccc}
\multicolumn{1}{c}{$\Phi$} & \multicolumn{1}{c}{$p$} & \multicolumn{1}{c}{$A$} & \multicolumn{1}{c}{$B$}&  $\mathbf{C}(c)$ & $\mathbf{S}(c)$ & $\mathbf{T}(c)$\tabularnewline
\hline 
\hline 
 im & $-$ & \multicolumn{2}{c|}{$-$} & $\begin{array}{c} \cos^{2}\left\Vert AB\right\Vert \\ \cos^{2}\widehat{ab} \end{array}$ & $\begin{array}{c} \sin^{2}\left\Vert AB\right\Vert \\ \sin^{2}\widehat{ab} \end{array}$  & $\begin{array}{c} \tan^{2}\left\Vert AB\right\Vert  \\ \tan^{2}\widehat{ab} \end{array}$\tabularnewline
\hline 
\multirow{5}{*}{real} & \multirow{4}{*}{sec} & \multirow{2}{*}{int} & int &  $\cosh^{2}\left\Vert AB\right\Vert $ & $-\sinh^{2}\left\Vert AB\right\Vert $ & $-\tanh^{2}\left\Vert AB\right\Vert $\tabularnewline
\cline{4-4} 
&  &  & \multirow{2}{*}{ext} & $-\sinh^{2}\left\Vert AB_{p}\right\Vert $ & $\cosh^{2}\left\Vert AB_{p}\right\Vert $ & $-\coth^{2}\left\Vert AB_{p}\right\Vert $\tabularnewline
\cline{3-3}
&  & \multirow{2}{*}{ext} &   & $\cosh^{2}\left\Vert A_{p}B_{p}\right\Vert $ & $-\sinh^{2}\left\Vert A_{p}B_{p}\right\Vert $ & $-\tanh^{2}\left\Vert A_{p}B_{p}\right\Vert $\tabularnewline
\cline{4-4}
&  &  & int &   $-\sinh^{2}\left\Vert A_{p}B\right\Vert $ & $\cosh^{2}\left\Vert A_{p}B\right\Vert $ & $-\coth^{2}\left\Vert A_{p}B\right\Vert $\tabularnewline
\cline{2-4} 
& ext & \multicolumn{2}{c|}{ext} & $\cos^{2}\widehat{ab}$ & $\sin^{2}\widehat{ab}$ & $\tan^{2}\widehat{ab}$\tabularnewline
\end{tabular}}
\smallskip

\caption{non-euclidean translations of projective trigonometric ratios.}
\label{table:non-euclidean-segments}

\end{table}

Expresions (\ref{eq:ABBpAp-hyperbolic}$-$\ref{eq:abbPaP-hyperbolic-elliptic})
suggest the following notation. Take a segment $c=\overline{AB}$ whose endpoints
$A,B\in\mathbb{RP}^{2}$ don't lie on $\Phi$ and such that the line
$AB$ containing $c$ is not tangent to $\Phi$. We define the \emph{projective
trigonometric ratios} of the segment $c$ (compare \cite{Wildberger})
as:
\begin{align*}
\mathbf{C}(c) & :=\left(ABB_{c}A_{c}\right)
\\
\mathbf{S}(c) & :=1-\mathbf{C}(c)=(AB_{c}BA_{c})
\\
\mathbf{T}(c) & :=
\dfrac{\mathbf{S}(c)}{\mathbf{C}(c)}=
\dfrac{1}{\left(ABB_{c}A_{c}\right)}-1=
\\
 & =\left(ABA_{c}B_{c}\right)-1=-\left(AA_{c}BB_{c}\right).\nonumber 
\end{align*}

In the limit case when $B=A_c$, we take
$$\CC(c)=0\,,\quad\SS(c)=1\quad\text{and}\quad \TT(c)=\infty\,,$$
and we say that the segment $\ov{AA_c}$ is \emph{right}. The segments 
$\ov{A_cB}$ and $\ov{AB_c}$ are \emph{complementary segments} of $\ov{AB}$.
The identities~\eqref{eq:CR1} and~\eqref{eq:CR2} imply that $\mathbf{C}(c)$,
$\mathbf{S}(c)$ and $\mathbf{T}(c)$ do not depend on the chosen ordering of the endpoints
$A,B$ of the segment $c$. Proposition~\ref{prop:geometric-translations}
below follows directly from Theorem~\ref{thm:ABBpAp-angles-and-distances} and
from the properties~\eqref{eq:CR1} and~\eqref{eq:CR2} of cross ratios (see 
Figure~\ref{Fig:non-euclidean_segments}).

\begin{proposition}
\label{prop:geometric-translations}Depending on the type of absolute conic
(real or imaginary) we are working with and, when $\Phi$ is a real
conic, on the relative positions of the endpoints $A,B$ of the segment
$c$ with respect to $\Phi$ (Figure \ref{Fig:non-euclidean_segments}), the projective trigonometric ratios
$\mathbf{C}$, $\mathbf{S}$ and $\mathbf{T}$ have the non-euclidean trigonometric translations listed in Table~\ref{table:non-euclidean-segments}.
\end{proposition}
In Table~\ref{table:non-euclidean-segments}, \emph{im}, \emph{ext}, \emph{sec} and \emph{int} mean \emph{imaginary}, \emph{exterior to} $\Phi$, \emph{secant to} $\Phi$ and \emph{interior to} $\Phi$,
respectively.

\chapter{Pascal's and Chasles' theorems and classical triangle centers}\label{sec:Desargues}

We will deal with the most common
triangle centers for generalized triangles: orthocenter, 
circumcenter, barycenter and incenter. Although they are very 
well-known for triangles, their projective interpretation provides
equivalent centers for any generalized triangle.

\section{Quad theorems}\label{sub:quad-theorems}

\begin{figure}
\centering
\hfill
\subfigure[]{\includegraphics[width=0.48\textwidth]
{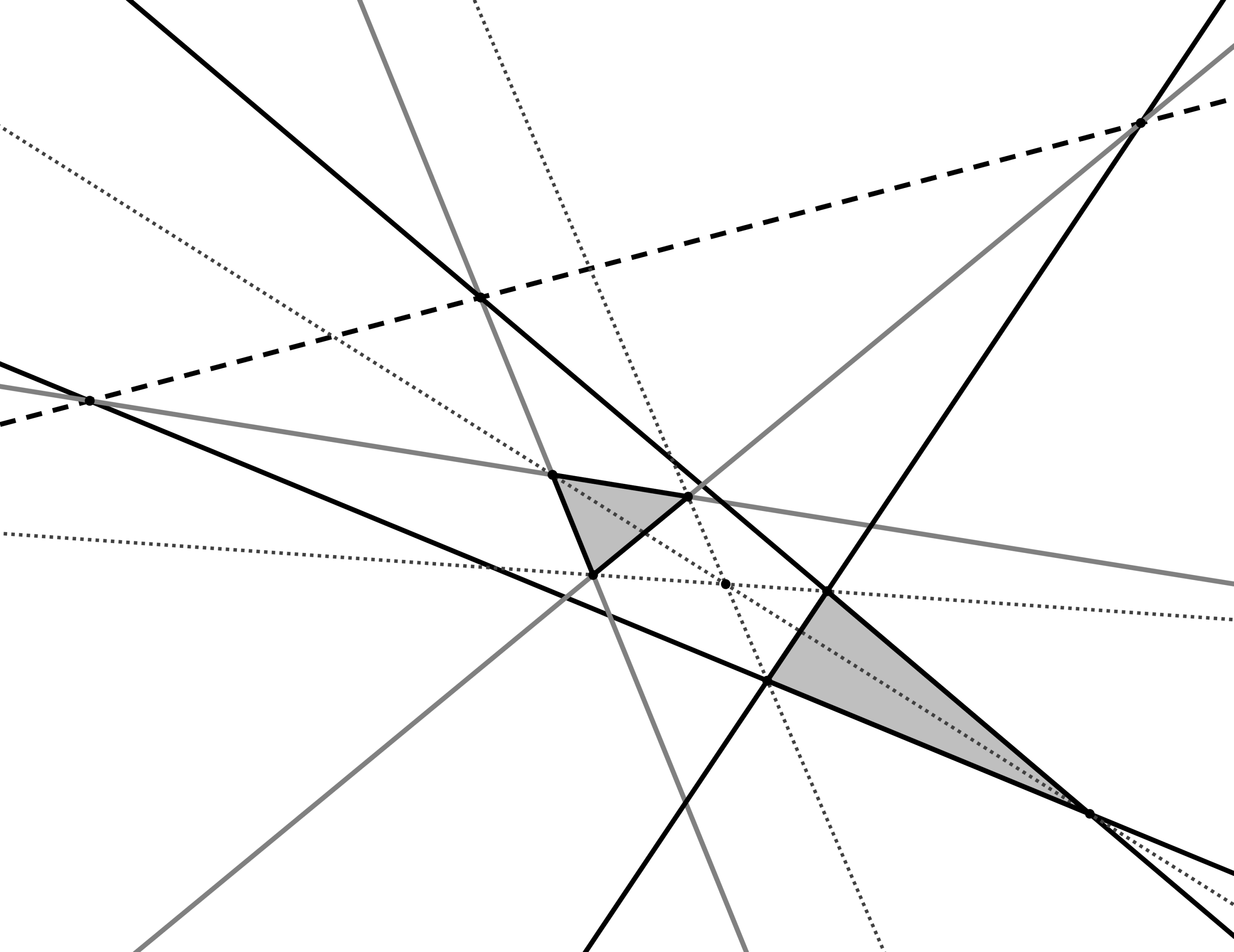}
\label{Fig:Desargues}}\hfill
\subfigure[]{\includegraphics[width=0.48\textwidth]
{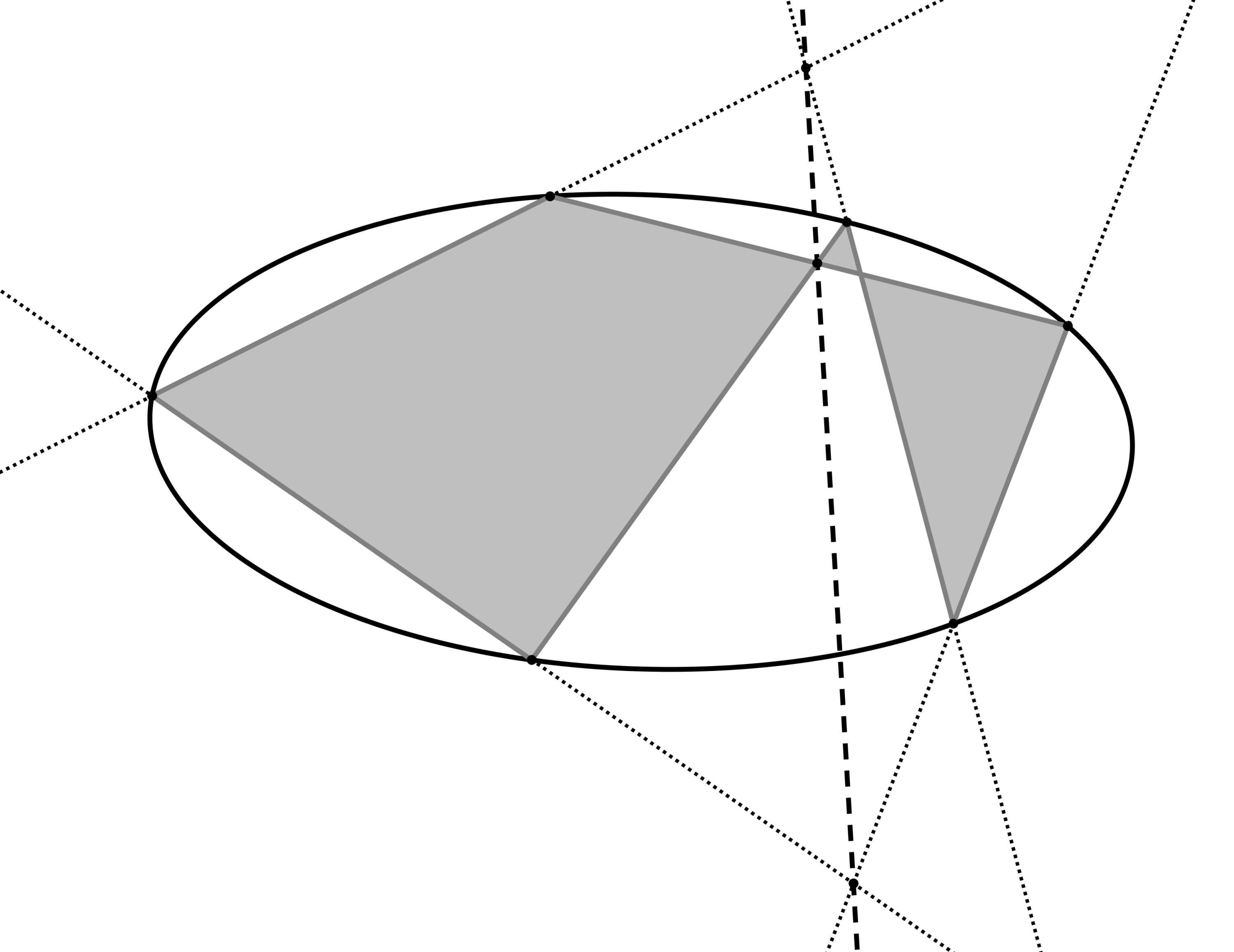}
\label{Fig:Pascal}}\hfill\hfill
\caption{Desargues' and Pascal's Theorems}\label{Fig:Desargues-Pascal}
\end{figure}

Let us state without proof four classical theorems of projective planar
geometry.

\begin{theorem}[Desargues' Theorem]\label{thm:Desargues}
Let  $\widetriangle{ABC}$ and $\widetriangle{A'B'C'}$ be two projective
triangles. 
The lines $AA',BB',CC'$ are concurrent if and only if the points
$$AB\cdot A'B',\quad BC\cdot B'C',\quad CA\cdot C'A'$$
are collinear.	
\end{theorem}
\begin{theorem}[Pascal's Theorem]\label{thm:Pascal}
If an hexagon is inscribed in a conic, the three pairs of opposite sides
meet in collinear points.
\end{theorem}
\begin{theorem}[Chasles' polar triangle
Theorem]\label{thm:Chasles-Polar-triangle}
A triangle and its polar triangle with respect to a conic are perspective.
\end{theorem}
\begin{theorem}[Pappus' Involution Theorem]\label{thm:Pappus-involution}
The three pairs of opposite sides of a quadrangle meet any line (not
through a
vertex) in three pairs of an involution.
\end{theorem}
Theorem~\ref{thm:Pappus-involution} is a partial version of 
Desargues' Involution Theorem (see \cite[p. 81]{Cox Proj}).
Using this theorem, a given quadrangle in the projective plane
determines a \emph{quadrangular involution} on every line not through a vertex.
If $\QQ$ is a quadrangle, we usually denote by $\sigma_\QQ$ the quadrangular
involution that it induces on the line considered. In the limit case where the
line passes through a vertex of $\QQ$, this vertex is a double point of
$\sigma_\QQ$. For the sake of simplicity, 
\hyperref[thm:Chasles-Polar-triangle]{Chasles' polar triangle Theorem}
will be called ``Chasles' Theorem''.

\section{Triangle notation}\label{sec:triangle-notation}

Along the whole text we will deal with a projective triangle $\TT$ with
vertices $A,B,C$
and sides
\[
a=\overline{BC}\,,\quad b=\overline{CA}\,,\quad c=\overline{AB}\,.
\]
The polar triangle of $\TT$ will be denoted by $\TT'$, its sides $a',b',c'$
are
the polars of $A,B,C$ respectively with respect to the absolute conic
$\Phi$, and its vertices
$$A'=b'\cdot c'\,,\quad B'=c'\cdot a'\,,\quad C'=a'\cdot b'\,,$$
are the poles of $a,b,c$ respectively. We will assume always that $\TT$ and
$\TT'$ are \emph{in general position}:
the vertices and sides of $\TT$ and $\TT'$ are all different and no vertex from
$\TT$ or $\TT'$  lies on $\Phi$. This last condition
implies also that  no side of $\TT$ or $\TT'$ is tangent to $\Phi$. Depending on
the type of conic
(real or imaginary) we are working with and, in the case of a real
conic, on the relative position of the triangle $\TT$ 
with respect to the conic $\Phi$, the triangles $\TT$ and $\TT'$ can produce in
$\mathbb{P}$ any of the generalized triangles of Figures
\ref{Fig:hyperbolic_generalized_triangles_I}--\ref{Fig:hyperbolic_generalized_triangles_star_II}.

Inspired by Proposition~\ref{prop:perpendicular-iff-conjugate}, we
say that $\TT$ is \emph{right-angled} if two of its sides
are conjugate to each other. In this case, only two of its sides can be
conjugate because if the sides $a$ and $b$ are both conjugate to $c$, it turns
out that $a\cdot b$ is the pole of $c$, in contradiction with the general
position assumptions.

The conjugate points of the vertices of $\TT$ at the sides they belong to are
$$
\begin{array}{ccc}
A_b=b\cdot a'& B_c=c\cdot b' & C_a=a\cdot c'\\
A_c=c\cdot a'& B_a=a\cdot b' & C_b=b\cdot c'\,.\\
\end{array}
$$
In the same way, the conjugate lines of the sides of $\TT$ are
$$
\begin{array}{ccc}
a_B=BA'& b_C=CB' & c_A=AC'\\
a_C=CA'& b_A=AB' & c_B=BC'\,.\\
\end{array}
$$
Note that it is
$$
\begin{array}{cccccc}
A_b=B'_{a'}& B_c=C'_{b'} & C_a=A'_{c'}& 
A_c=C'_{a'} & B_a=A'_{b'} a& C_b=B'_{c'}\\
a_B=b'_{A'}&b_C=c'_{B'}& c_A=a'_{C'}&
a_C=c'_{A'}& b_A=a'_{B'}&c_B=b'_{C'}\,.
\end{array}
$$

The midpoints of the segments $\ov{BC}$, $\ov{CA}$ and $\ov{AB}$ are
$D,D_a$, $E,E_b$ and
$F,F_c$ respectively. Equivalently, the midpoints of
$\ov{B'C'},\ov{C'A'}$ and $\ov{A'B'}$ are denoted $D',D'_{a'}$, $E',E_{b'}$ and
$F',F'_{c'}$ respectively.

Most of all these points and lines are depicted in 
Figure~\ref{Fig:triangle_midpoints_model}.

\begin{figure}[h]
\centering
\includegraphics[width=0.9\textwidth]
{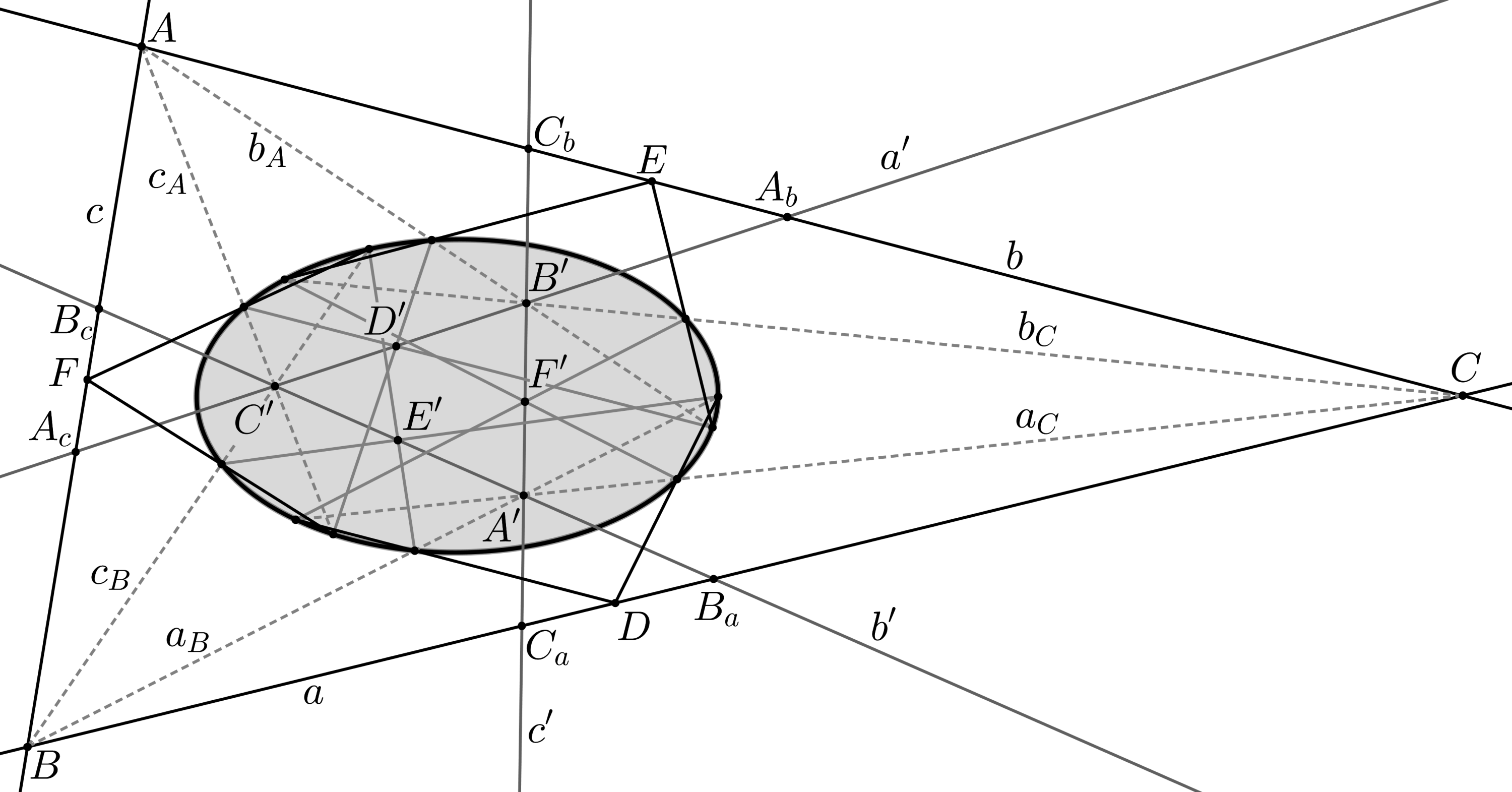}
\caption{Triangle and its polar one. Midpoints and conjugate points and lines}
\label{Fig:triangle_midpoints_model}
\end{figure}

As the lines $a,a'$ are
different, among the points $D,D_a,D',D'_{a'}$ , there must be at least three
different points. In some
cases, a midpoint of $\ov{BC}$ and a midpoint of
$\ov{B'C'}$ could coincide at the intersection point $A_0$ of $a$ with $a'$.
We want to clarify what happens in this limit situation. 
Note that $A_0$ is also the pole of the line $AA'$.

\begin{lemma}\label{lem:equivalences-DD'-coincide-1}
Assume that $D$ and $D'$ are both different from $A_0$.
If one of the points $D_a,D'_{a'}$ coincide with $A_0$, the other one
coincides too, and this happens if and only if the lines $AA'$ and $DD'$
coincide.
\end{lemma}
\begin{proof}
Assume that $D'_{a'}=A_0$. This implies that the polar $AA'$ of $A_0$ passes
through $D'$. The point $H_A=AA'\cdot a$ is the conjugate point of $A_0$ in
$a$. Projecting the line $a'$ onto $a$ since $A'$ we get
$$-1=(B'C'D'A_0)=(C_aB_aH_AA_0)\,.$$
By Lemmas~\ref{lem:midpoints-AB-midpoints-ApBp} and~\ref{lem:D-Da-midpoints}
this implies that $A_0,H_A$ are the midpoints of $\ov{BC}$.

On the other hand, if $AA'=DD'$, the
polar of $D$ passes through the pole of $AA'$, which is $A_0$ and so it is
$D_a=A_0$, and equivalently it is $D'_{a'}=A_0$.
\end{proof}

The line $AA'$ is the line orthogonal to $a$ through $A$, and therefore it is
the \emph{altitude} of $\TT$ through $A$. In the situation of previous lemma,
in the limit case when $D_a$ and $D'_{a'}$ coincide, the altitude $AA'$ is also
a \emph{segment bisector} of $\ov{BC}$ (as it passes through $D$), and an
\emph{angle bisector} of $\widehat{bc}$ (as it passes through $D'$, cf.  Remark
\ref{rem:midpoints-bisectors}). This is the reason why in this situation we
say that the triangle $\TT$ is \emph{isosceles} at $A$. The triangle $\TT$ is
\emph{equilateral} if it is isosceles at its three vertices. As we can expect,
if $\TT$ is isosceles at $A$ the sides incident with $A$ have the same
``length'': let $B_1,B_2$ and $C_1,C_2$ be the intersection points of $b$ and
$c$ with the absolute conic $\Phi$; then

\begin{proposition}
The triangle $\TT$ is isosceles at $A$ if and only if we can label
$B_1,B_2,C_1,C_2$ such that
\begin{equation}\label{eq:prop-isosceles}
(B_1B_2AC)=(C_1C_2AB)\,.	
\end{equation}
\end{proposition}
\begin{proof}
	If $\TT$ is isosceles at $A$, the midpoint $D'_{a'}$ of $\ov{B'C'}$
coincides with $A_0$. By \S\ref{sec:midpoints-of-a-segment}, the
midpoints of $\ov{B'C'}$ are the diagonal points of the quadrangle
$\{B_1,B_2,C_1,C_2\}$ different from $A$, so we can label $B_1,B_2,C_1,C_2$
such that $B_1C_1\cdot B_2C_2=D'_{a'}=A_0$, and~\eqref{eq:prop-isosceles}
follows from projection and section since $A_0$.

On the other hand, if~\eqref{eq:prop-isosceles} holds, we consider the midpoint
$$D'_*=B_1C_1\cdot B_2C_2$$
of $\ov{B'C'}$. If we take the point $C_*=BD'_*\cdot
b$, it turns out that 
$$(B_1B_2AC_*)=(C_1C_2AB)$$
by projection and section since $D'_*$, and then~\eqref{eq:prop-isosceles}
implies that $C_*=C$. Therefore, $D'_*$ belongs to $BC$, it coincides with
$A_0$ and the triangle $\TT$ is isosceles at $A$.
\end{proof}

\begin{proposition}\label{prop:2-isosceles-equilateral}
	If $\TT$ is isosceles at two vertices, it is equilateral. \hfill$\blacksquare$
\end{proposition}

As in the hypotheses of Lemma~\ref{lem:equivalences-DD'-coincide-1}, we will
always assume that $D$ and $D'$ are different from $A_0$. In
the same way, taking the points
$B_0=b\cdot b'$ and $C_0=c\cdot c'$, we will always assume that $E$ and $E'$
are different from $B_0$ and that $F$ and $F'$ are different from $C_0$.
In particular, it must be $D\neq D'$, $E\neq E'$ and $F\neq F'$.

\section{Chasles' Theorem and the orthocenter}\label{sub:Orthocenter}

\begin{figure}[h]
\centering
\includegraphics[width=0.9\textwidth]
{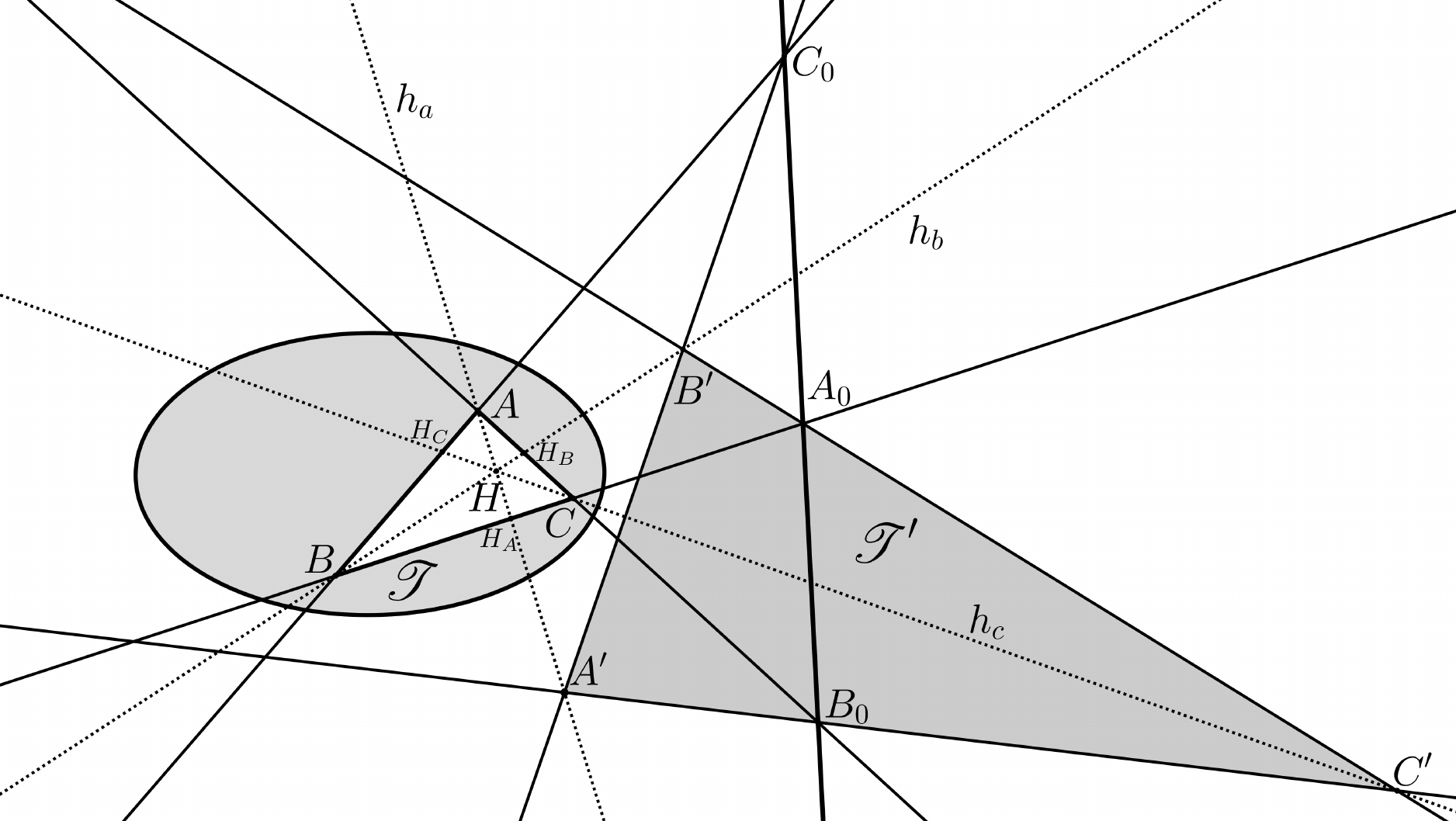}
\caption{orthocenter}\label{Fig:orthocenter}
\end{figure}

The altitudes of $\TT$ through the vertices
$A,B$ and $C$ are the lines
$$h_{a}=AA'\,,\quad h_{b}=BB'\quad \text{and}\quad h_{c}=CC'\,,$$
respectively. A straightforward consequence of 
\hyperref[thm:Chasles-Polar-triangle]{Chasles' Theorem} is:
\begin{theorem}[Concurrency of altitudes]
The lines $h_a,h_b$ and $h_c$ are concurrent. 
\end{theorem}

We say that the intersection point $H$ of $h_a,h_b$ and $h_c$ is the
\emph{orthocenter} of
$\TT$. Note that the orthocenter of $\TT$ is also the orthocenter of $\TT'$.

The poles of $h_{a},h_{b},h_{c}$ are the points $A_{0},B_{0},C_{0}$
respectively. Because $h_{a},h_{b},h_{c}$ intersect at the point
$H$, their poles $A_{0},B_{0},C_{0}$ lie on the polar $h$ of $H$. In our
constructions of \S\ref{sub:Baricenters} to \S\ref{sec:nine-point-conic} 
the line $h$ will have the same relation to the triangle $\TT$ as
the line at infinity has with any euclidean triangle.

\section{Pascal's Theorem and classical
centers}\label{sec:midpoint-quadrilateral-classical-centers}

\begin{figure}
\centering
\includegraphics[width=0.9\textwidth]
{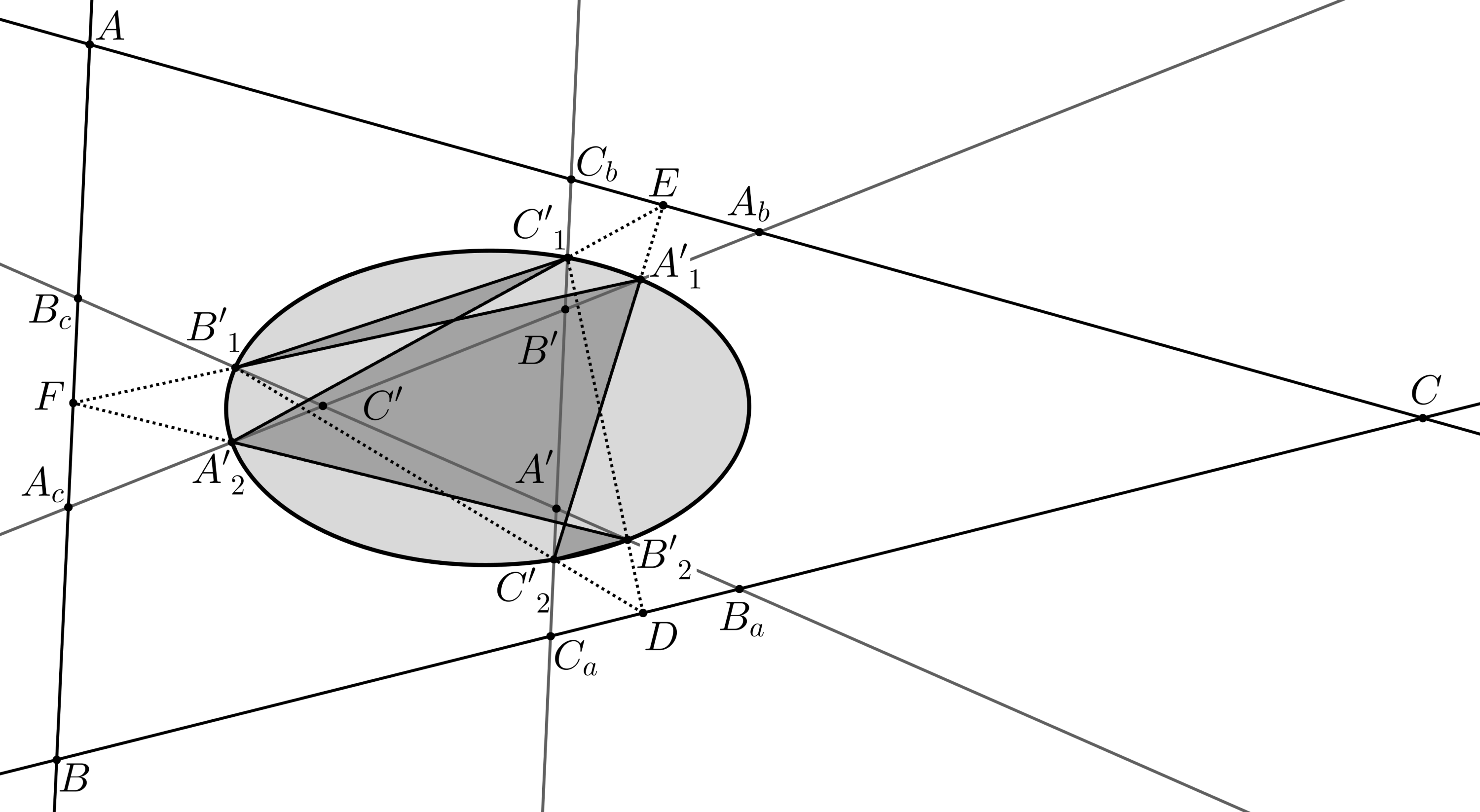}
\caption{Pascal's Theorem and the midpoints
of a
triangle}\label{Fig:Pascal_midpoints}
\end{figure}

First of all, we will introduce on the points 
$D,D',E,E',F,F'$ another assumption additional to those
of the last paragraph of \S\ref{sec:triangle-notation}: 
we will assume from now on that $D,E,F$ are non-collinear and that 
$D',E',F'$ are non-collinear too. We show how this always can be done
with the points $D,E,F$.
If $\TT$ is not isosceles at any vertex: (i)
we take as $D$ and $E$ any of the midpoints of $\ov{BC}$ and $\ov{CA}$
respectively; (ii) as $D,E$ cannot be collinear with both $F,F_c$, we choose as
$F$ a midpoint of $\ov{AB}$ not collinear with $D,E$; and (iii) this point $F$ 
will be different from $C_0$ because $\TT$ is not isosceles at $C$. 
If $\TT$ is isosceles only at $A$: (i) we choose as $D$ the midpoint of $\ov{BC}$
different from $A_0$; (ii) we choose as $E$ any of the midpoints of $\ov{CA}$; 
and (iii) repeat steps (ii) and (iii) of the non-isosceles case. Finally, if
$\TT$ is equilateral, we take as $D,E,F$ the midpoint of
$\ov{BC},\ov{CA},\ov{AB}$ different from $A_0,B_0,C_0$ respectively. In this
case, as $A_0,B_0,C_0$ are collinear it follows from Lemma 
\ref{lem:harmonic-sets-on-the-sides} that $D,E,F$ cannot be collinear.

By Lemma
\ref{lem:harmonic-sets-on-the-sides}, we know that $EF\cdot E_bF_c$ and
$EF_c\cdot E_bF$ are two points on $a$ which are harmonic conjugates with
respect to $B,C$. In fact, these two points are $D,D_a$.

\begin{lemma}\label{lem:midpoints-collinear}
The points $D_a,E,F$ are collinear. The points $D,D_a$ are the diagonal points
different from $A$ of the quadrangle $\{E,E_b,F,F_c\}$.
\end{lemma}
\begin{proof}
Let $A'_1,A'_2$, $B'_1,B'_2$ and $C'_1,C'_2$ be the intersection points of $a'$,
$b'$ and $c'$ respectively with $\Phi$.
By \S
\ref{sec:midpoints-of-a-segment}, the midpoints of $\ov{AB}$ are
the diagonal points different from
$C'$ of the quadrangle $\{A'_1,A'_2,B'_1,B'_2\}$. In the same way, the midpoints
of $\ov{BC}$ are the diagonal points different from $A'$ of the quadrangle
$\{B'_1,B'_2,C'_1,C'_2\}$, and the midpoints of $\ov{CA}$ are the diagonal
points different from $B'$ of the quadrangle $\{C'_1,C'_2,A'_1,A'_2\}$.

Imagine
that we have given the names $A'_1,A'_2$, $B'_1,B'_2$ and $C'_1,C'_2$ to the
points of $a'\cdot\Phi$, $b'\cdot\Phi$ and $c'\cdot\Phi$ respectively in such a
way that
$$D=B'_1C'_2\cdot B'_2C'_1\quad\text{and}\quad E=C'_1A'_2\cdot C'_2A'_1\,.$$
Taking the hexagon $A'_1B'_1C'_1A'_2B'_2C'_2$ inscribed in $\Phi$, by
\hyperref[thm:Pascal]{Pascal's Theorem}
the points $E,F$ are collinear with the point $B'_1C'_1\cdot B'_2C'_2$,
which is a midpoint of $\ov{BC}$ (see Figure~\ref{Fig:Pascal_midpoints}).
Because we have assumed that $D,E,F$ are non-collinear, it must be
$D_a=B'_1C'_1\cdot B'_2C'_2$.

Taking the hexagon $B'_1C'_1A'_1B'_2C'_2A'_2$, the points $E_b,F_c$ are
also collinear with $D_a$ and therefore $D_a$ is a diagonal point of the
quadrangle $\{E,E_b,F,F_c\}$. In the same way it can be proved that $D$ is
another diagonal point of $\{E,E_b,F,F_c\}$.
\end{proof}

As the polar of $D_a$ is the bisector of $\ov{BC}$ through $D$, 
straightforward shadows of this lemma are (compare \cite[Thm. 7]{Santalo}):

\begin{theorem}
Let $T$ be an elliptic or hyperbolic triangle, and let $\frac{1}{2}T$ be a 
\emph{medial triangle} of $T$, i.e. a triangle whose vertices are midpoints of the sides of $T$. The side bisectors of $T$ through the vertices of $\frac{1}{2}T$ are the altitudes of $\frac{1}{2}T$.
\end{theorem}

\begin{theorem}
Let $H$ be a hyperbolic right-angled hexagon, and let $\frac{1}{2}H$ be a triangle whose vertices are the midpoints of alternate sides of $H$. The side bisectors of $H$ through the vertices of $\frac{1}{2}H$ are the altitudes of $\frac{1}{2}H$.
\end{theorem}

The dual figure of a quadrangle is a \emph{quadrilateral}: the figure
composed by four lines not three of which are concurrent (the \emph{sides}
of the quadrilateral), and the six points at which these four lines
intersect in pairs (the \emph{vertices} of the quadrilateral). For
any vertex $P$ of a quadrilateral $\mathscr{M}$ there is exactly
another vertex $Q$ of $\mathscr{M}$ such that the line $PQ$ is
not a side of $\mathscr{M}$. We say that $P$ and $Q$ are \emph{opposite}
vertices of $\mathscr{M}$ and that $PQ$ is a \emph{diagonal line}
of $\mathscr{M}$. The three diagonal lines of $\mathscr{M}$ compose
the \emph{diagonal triangle} of $\mathscr{M}$.
A glance at Figure~\ref{Fig:quadrangle} gives the following result.

    \begin{lemma}\label{lem:quadrangle-quadrilateral}
	Let $\QQ$ be a quadrangle, let $P$ be a diagonal point of $\QQ$ and let
	$x,y,z,w$ be the four sides of $\QQ$ not passing through $P$. Then, the
	vertices of the quadrilateral $\{x,y,z,w\}$ are the vertices of $\QQ$
	together with the two diagonal points of $\QQ$ different from $P$. The
	diagonal triangle of $\{x,y,z,w\}$ has as sides the sides of $\QQ$
	passing
	through $P$ and the side of the diagonal triangle of $\QQ$ not passing
	through
	$P$.
    \end{lemma}

The following is an interesting
property of midpoints (compare \cite[Theorem 22.5]{Richter-Gebert}). Its proof
is straightforward from Lemmas
\ref{lem:midpoints-collinear} and~\ref{lem:quadrangle-quadrilateral}.
\begin{theorem}\label{thm:midpoints-quadrilateral}
The midpoints of the sides of $\TT$ are the vertices
of a quadrilateral $\mathscr{M}_{\TT}$ whose diagonal triangle is $\TT$.
\end{theorem}

We say that $\MM_{\TT}$ is the \emph{midpoint quadrilateral} of $\TT$.
This theorem has many important consequences as it allows to prove
the concurrence of medians, the concurrence of side bisectors and
(after dualizing) the concurrence of angle bisectors of generalized triangles.

\begin{figure}
\centering
\includegraphics[width=0.7\textwidth]
{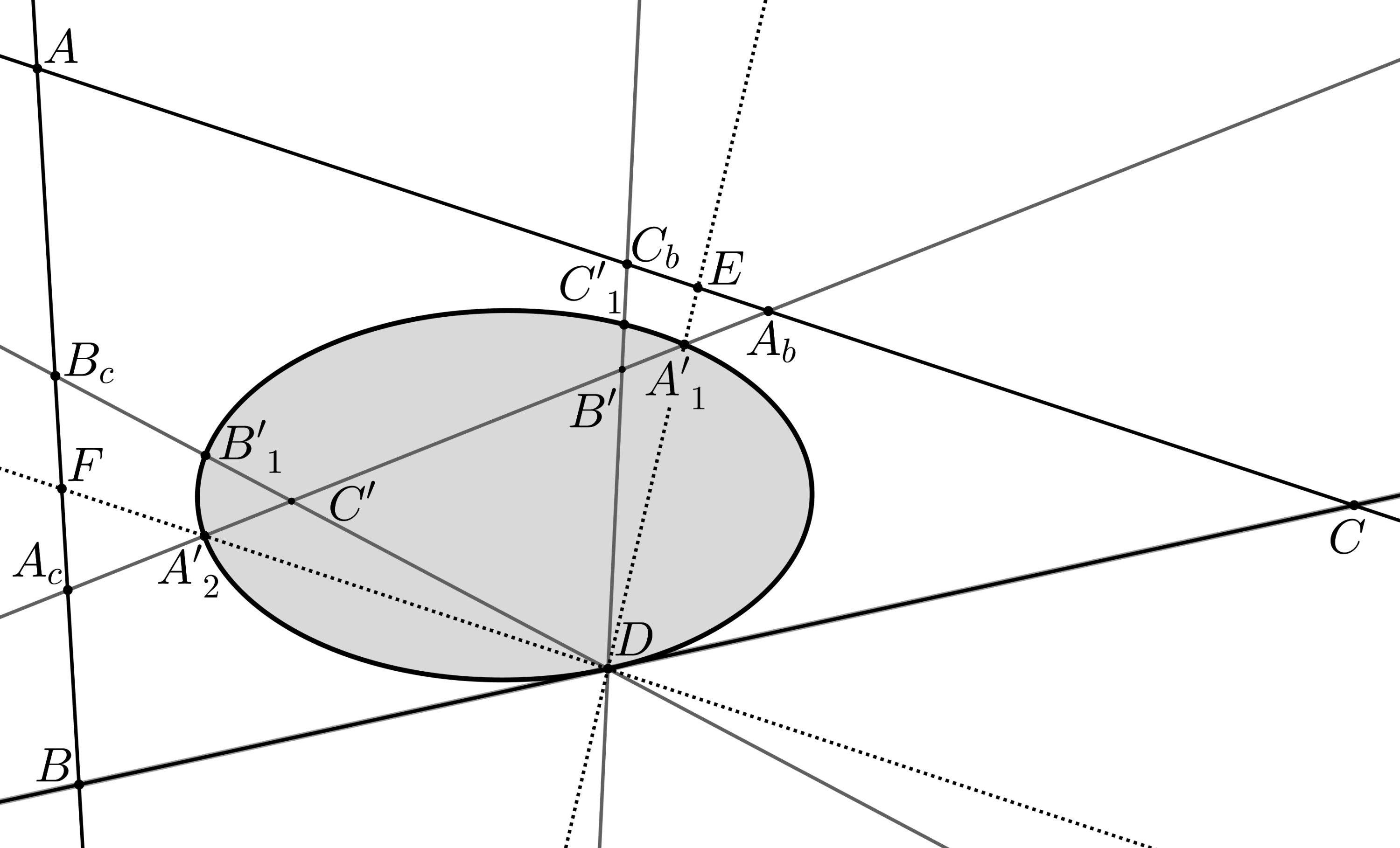}
\caption{Triangle with one side tangent to
$\Phi$}\label{Fig:Pascal_midpoints_03}
\end{figure}

Although we are interested in triangles $\TT$, $\TT'$ in general position with
respect to $\Phi$, in some of our constructions we will need to use
some
accessory triangles which could not verify the general position assumptions.
In particular, we wonder if Theorem
\ref{thm:midpoints-quadrilateral} is valid for triangles with sides
tangent to $\Phi$.
\begin{lemma}\label{lem:midpoints-quadrilateral-also-with-tangencies}
Theorem~\ref{thm:midpoints-quadrilateral} remains valid even if the triangle
$\wt{ABC}$ has sides tangent to $\Phi$.
\end{lemma}
\begin{proof}
Let assume that the side $a=BC$ is tangent to $\Phi$, while $AB$ and $CA$ are
not. Let $D$ be the point of tangency of $a$ with $\Phi$. The polars $b',c'$
of $B,C$ respectively pass through $D$. Let $A'_1,A'_2$ be the intersection
points of $a'$ with
$\Phi$, and let $B'_1,C'_1$ be the
intersection points of $b',c'$ with $\Phi$ respectively different from $D$ 
(Figure~\ref{Fig:Pascal_midpoints_03}).
The
midpoints of $\ov{AB}$ are the diagonal points different from $C'$ of the
quadrangle $\{A'_1,A'_2,B'_1,D\}$, and the midpoints of $\ov{CA}$ are the
diagonal points different from $B'$ of the quadrangle $\{A'_1,A'_2,C'_1,D\}$.
Thus, the line $DA'_1$ pass through a midpoint of $\ov{AB}$ and through a
midpoint of $\ov{CA}$, and so does the line $DA'_2$. In other words, $D$ is a
diagonal point of the quadrangle $\{E,E_b,F,F_c\}$. By Definition
\ref{def:midpoints-tangent} and Lemma~\ref{lem:harmonic-sets-on-the-sides} it
follows that the other midpoint $D_a$ of $\ov{BC}$ is the other diagonal point
of the quadrangle $\{E,E_b,F,F_c\}$ different from $A$.

\begin{figure}
\centering
\includegraphics[width=0.7\textwidth]
{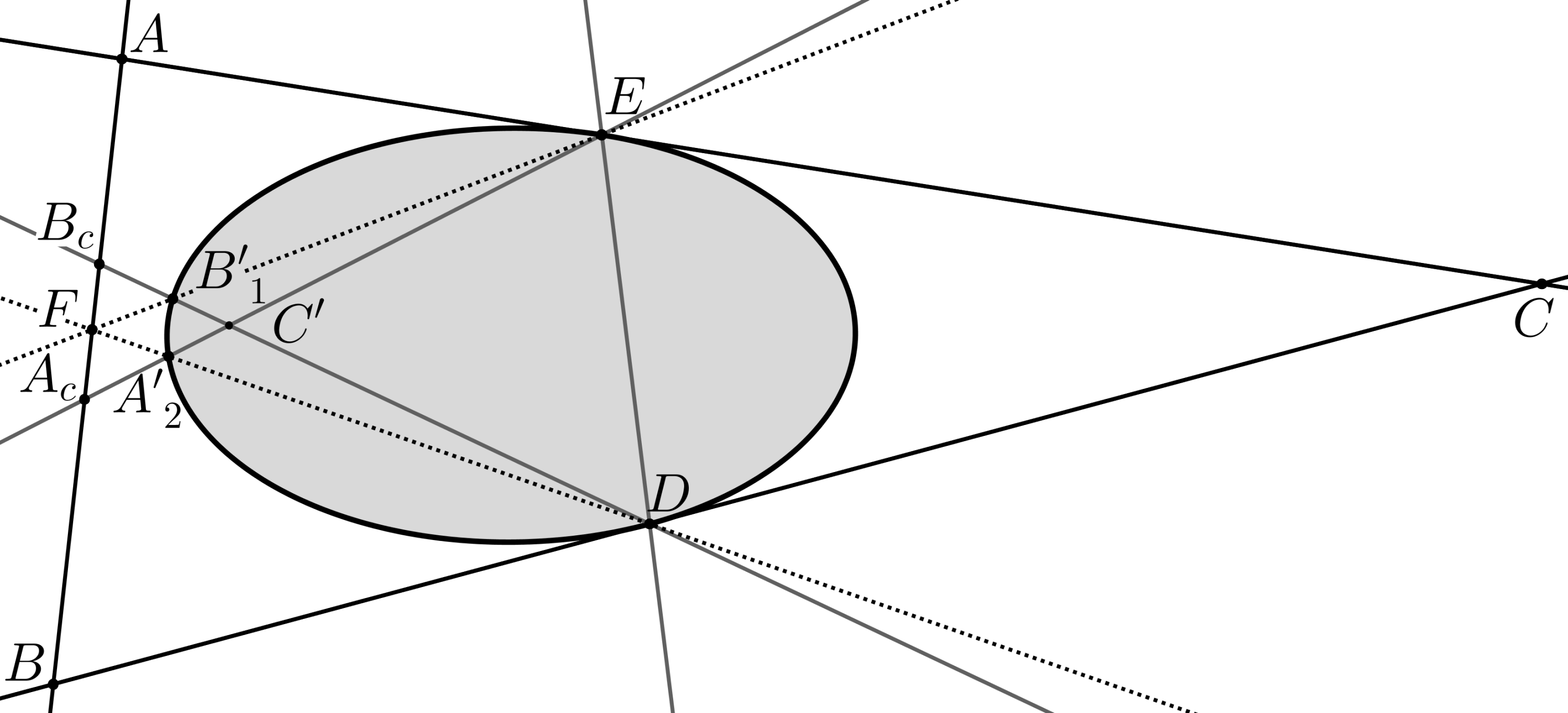}
\caption{Triangle with two sides tangent to
$\Phi$}\label{Fig:Pascal_midpoints_04}
\end{figure}

Let assume now that the sides $a,b$ are tangent to $\Phi$ while $c$ is not. Let
$D,E$ be the contact points of $a,b$ with $\Phi$  and let
$D_a,E_b$ be the midpoints of $\ov{BC},\ov{CA}$ different from $D,E$
respectively. The polar $c'$ of $C$ is the line $DE$. Let $A'_2,B'_1$ be
the intersection points of $a',b'$ with $\Phi$ different from $D,E$
respectively (Figure~\ref{Fig:Pascal_midpoints_04}). The points $F,F_c$ are the
diagonal points of the
quadrangle $\{D,E,A'_2,B'_1\}$ different from $C$.
Consider the points $D_a^*=EF\cdot a$ and $E_b^*=DF\cdot b$. By 
\hyperref[thm:Pascal]{Pascal's Theorem}
on the degenerate hexagon $DDEEB'_1A'_2$ the points $F_c$, $E_b^*$ and
$D_a^*$ are collinear. Taking the quadrangle $\{F_c,E_b^*,E,F\}$ it must be
$D_a^*=D_a$ and, equivalently, $E_b^*=E_b$. Therefore, $F,F_c$ are the diagonal
points of $\{D,D_a,E,E_b\}$ different from $C$.

In the two previous cases, Lemma~\ref{lem:quadrangle-quadrilateral} completes
the proof.

Finally, let assume that $a,b,c$ are tangent to $\Phi$ at the points $D,E,F$
respectively. In this case, $\wt{DEF}$ is the polar triangle $\TT'$ of $\TT$.
As $\TT$ and $\TT'$ are perspective by 
\hyperref[thm:Chasles-Polar-triangle]{Chasles' Theorem}, the points
$D_0=a\cdot EF$, $E_0=b\cdot FD$, $F_0=c\cdot DE$ belong to the line $h$.
In particular, $D,D_0$ are the diagonal points of the quadrangle
$\{E,E_0,F,F_0\}$ different from $A$. This implies that $(BCDD_0)=-1$, and thus
$D_0$ is the midpoint of $\ov{BC}$ different from $D$. In the same way,
$E_0,F_0$ are the midpoints of $\ov{CA},\ov{AB}$ different from $E,F$
respectively. The midpoints $D,D_0,E,E_0,F,F_0$ are the vertices of the
quadrilateral $\{a',b',c',d\}$.
\end{proof}

\begin{theorem}[Concurrence of medians]\label{thm:medians-concurrent}
The \emph{medians} $AD$, $BE$ and $CF$ of $\TT$ are concurrent.
\end{theorem}
\begin{proof}
By considering the triangles $\TT=\wt{ABC}$ and $\wt{DEF}$, by 
Theorem~\ref{thm:midpoints-quadrilateral} the intersection points of corresponding sides are
$$AB\cdot DE=F_c\,,\quad BC\cdot EF = D_a\,,\quad CA\cdot FD=E_b\,,$$
which are collinear. The result follows from 
\hyperref[thm:Desargues]{Desargues' Theorem}.
\end{proof}
\begin{theorem}[Concurrence of side bisectors]\label{thm:side-bisectors-concurrent}
The \emph{side bisectors} $A'D$, $B'E$ and $C'F$ of $\TT$ are concurrent.
\end{theorem}
\begin{proof}
The points $D_a,E_b$, and $F_c$ are, respectively, 
the poles of the lines $A'D$, $B'E$ and $C'F$. Because $D_a,E_b$, and $F_c$ are
collinear, their polars are concurrent.
\end{proof}

It is interesting to note that \emph{the side bisectors of $\TT$ are the angle
bisectors of $\TT'$} (cf. Remark~\ref{rem:midpoints-bisectors}). Thus, if
$D',E',F'$ are non-collinear midpoints of $\ov{B'C'},\ov{C'A'},\ov{A'B'}$,
respectively, we have (compare \cite[10.21]{Cox Non-euc}):
\begin{corollary}[Concurrence of angle bisectors]
The angle bisectors $AD',BE',CF'$ of $\TT$ are concurrent.
\end{corollary}

A point where three concurrent medians of $\TT$ intersect is a \emph{barycenter}
of $\TT$, 
a point where three concurrent side bisectors of $\TT$ intersect is a
\emph{circumcenter} of $\TT$, and a point where three angle bisectors of $\TT$
intersect is an \emph{incenter} of $\TT$. Each generalized triangle has four
barycenters, four circumcenters and four incenters. All these centers have
different geometric interpretations depending on the relative position of the
figure with respect to $\Phi$.

\begin{figure}
\centering
\includegraphics[width=0.7\textwidth]
{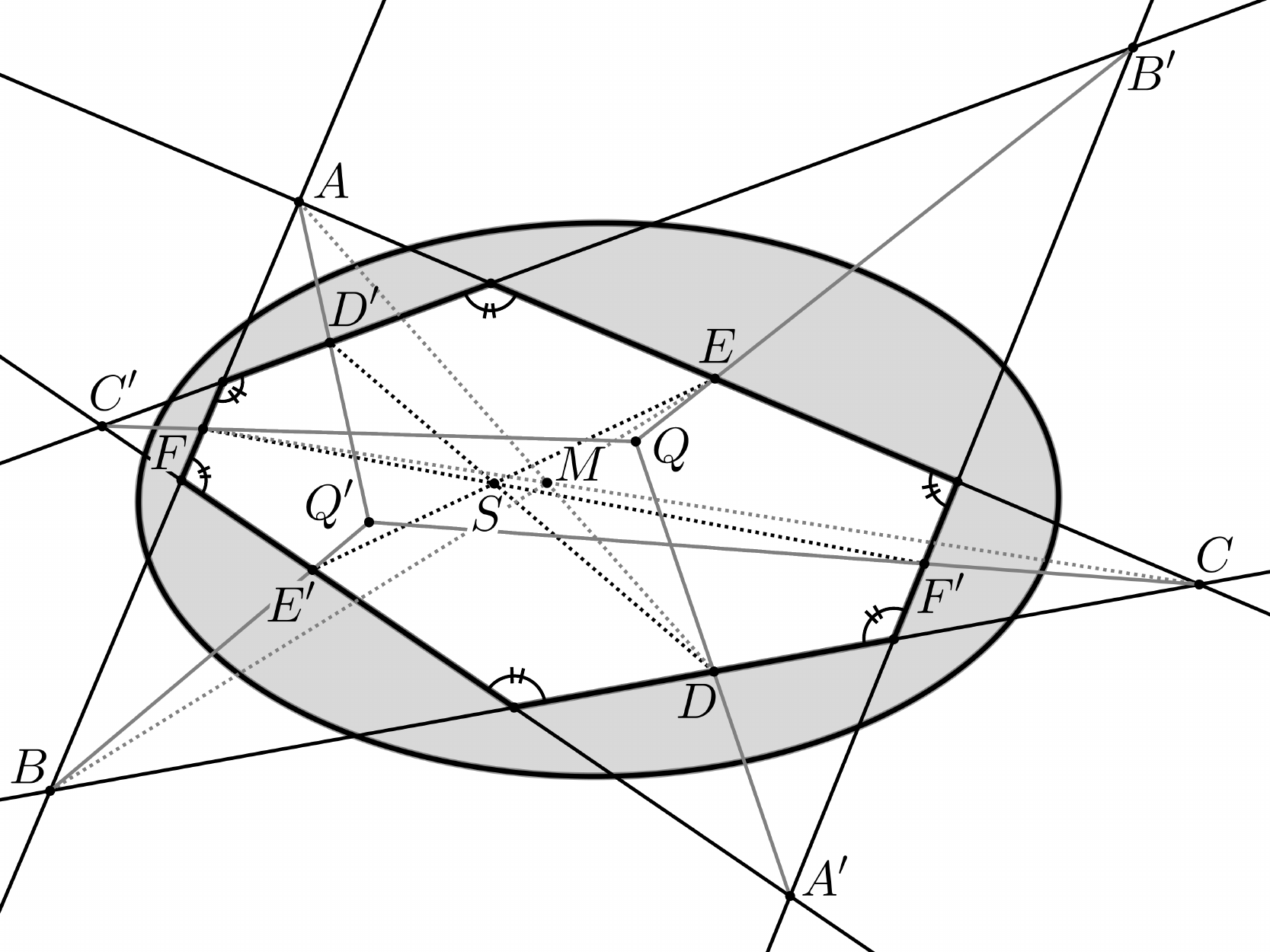}
\caption{Centers of a right-angled hexagon}
\label{Fig:hexagon-centers}
\end{figure}

Another center, which as the orthocenter is shared by $\TT$ and $\TT'$, is given by the following result.
\begin{theorem}\label{thm:DD'-EE'-FF'-concurrent}
The lines $DD',EE'$ and $FF'$ are concurrent.
\end{theorem}
This theorem fits perfect for a right-angled hexagon.
\begin{theorem}\label{thm:shadow-joins-midpoints-hexagon-concurrent}
	In a hyperbolic right-angled hexagon the lines joining the midpoints of
opposite sides are concurrent.
\end{theorem}
In Figure~\ref{Fig:hexagon-centers} it is depicted a right-angled 
hexagon with the intersection point $S$ of the lines joining opposite midpoints.
In the same figure, we have depicted also: three alternate \emph{medians}
of the hexagon, lines orthogonal to a side through the midpoint of its opposite side, 
and the intersection point $M$ of them; and its two circumcenters 
(or incenters, for this figure both concepts are the same),
the points $Q$ and $Q'$ where the orthogonal bisectors
of alternate sides intersect.

In order to interpret Theorem~\ref{thm:DD'-EE'-FF'-concurrent} 
for elliptic or hyperbolic triangles, note that the point $D'$ is the pole of a
bisector $d$ of $\widehat{bc}$, and therefore $DD'$ is
the line through $D$ orthogonal to $d$. As the points $D',E',F'$ are
non-collinear, their polars are non-concurrent. This leads us to the
shadow of Theorem~\ref{thm:DD'-EE'-FF'-concurrent} for hyperbolic or elliptic
triangles.
\begin{theorem}\label{thm:shadow-Spieker-center-elliptic-hyperbolic}
	Let $T$ be a hyperbolic or elliptic triangle with vertices $A,B,C$
and opposite sides $a,b,c$ respectively. Let $D,E,F$ be non-collinear midpoints
of
$a,b,c$ respectively, and let $d,e,f$ be non-concurrent bisectors of $T$
through $A,B,C$ respectively. The lines orthogonal to $d,e,f$ through $D,E,F$
respectively are concurrent.
\end{theorem}
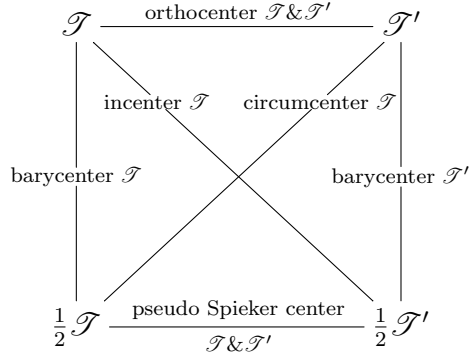
\begin{figure}
\centering
\[
\xymatrix@=0.25\textwidth{
\TT \ar@{-}[d]|{\text{barycenter}\;\TT} 
      \ar@{-}[r]^{\text{orthocenter}\;\TT \& \TT'}
    \ar@{-}[rd]|<<<<<<<<<<<<{\text{incenter}\;\TT} &
\TT'\ar@{-}[d]|{\text{barycenter}\;\TT'}\\
\frac{1}{2}\TT
    \ar@{-}[r]^{\text{pseudo Spieker center}}_{\TT\&\TT'}
    \ar@{-}[ru]|>>>>>>>>>>>>{\text{circumcenter}\; \TT} &\frac{1}{2}\TT'
      }
\]
\caption{concurrency graph of the triangles $\TT$, $\TT'$, $\frac{1}{2}\TT$ and $\frac{1}{2}\TT'$}
\label{Fig:perspectivity-graph}
\end{figure}

This theorem is true also in Euclidean geometry, where the lines orthogonal to
$d,e,f$ through $D,E,F$ are in fact angle bisectors of the medial triangle of
$T$ (the
triangle whose vertices are the midpoints of the sides of $T$). 
The concurrency point of the angle bisectors of the medial triangle of $T$ is
the \emph{Spieker center} of the triangle $T$. In the
non-euclidean case, the lines $DD'$, $EE'$ and $FF'$ are not necessarily
bisectors of $\wt{DEF}$, and because of this we say that the concurrency point
of $DD'$, $EE'$ and $FF'$ given by Theorem~\ref{thm:DD'-EE'-FF'-concurrent} is
the \emph{pseudo Spieker center} of $\TT$.

Surprisingly, the proof of Theorem~\ref{thm:DD'-EE'-FF'-concurrent} is harder
than expected, and it is left to \S
\ref{sec:magic-midpoints-oriented-triangles}
(p. \pageref{proof:thm:DD'-EE'-FF'-concurrent}).

The pseudo Spieker center closes an interesting
``concurrency graph''. Let denote by $\frac{1}{2}\TT$ and $\frac{1}{2}\TT'$ the
medial triangles $\wt{DEF}$ and $\wt{D'E'F'}$. The triangles $\TT$,
$\TT'$,
$\frac{1}{2}\TT$ and $\frac{1}{2}\TT'$ are perspective in pairs. If we
codify each 
perspectivity between triangles as an edge joining the 
corresponding triangles, we obtain the
complete graph of Figure~\ref{Fig:perspectivity-graph}.

\chapter{Desargues' Theorem, alternative triangle centers and the Euler-Wildberger line}\label{sec:Desargues2}

In \S\ref{sec:Desargues} we have presented the most common
centers of a triangle, using their standard definitions. In euclidean geometry, 
these points can be defined in multiple ways, all of them equivalent,
but these definitions which are equivalent in 
euclidean geometry could be non-equivalent in the 
non-euclidean context. Thus, definitions that in euclidean geometry produce
the same center in non-euclidean geometry produce different centers, all
of them having reminiscences of their euclidean analogue. 
\hyperref[thm:Desargues]{Desargues' Theorem} 
will allow us to construct a collection of such \emph{pseudocenters}
having some interesting properties.

\section{Pseudobarycenter}\label{sub:Baricenters}

\begin{figure}[h]
\centering
\includegraphics[width=0.9\textwidth]
{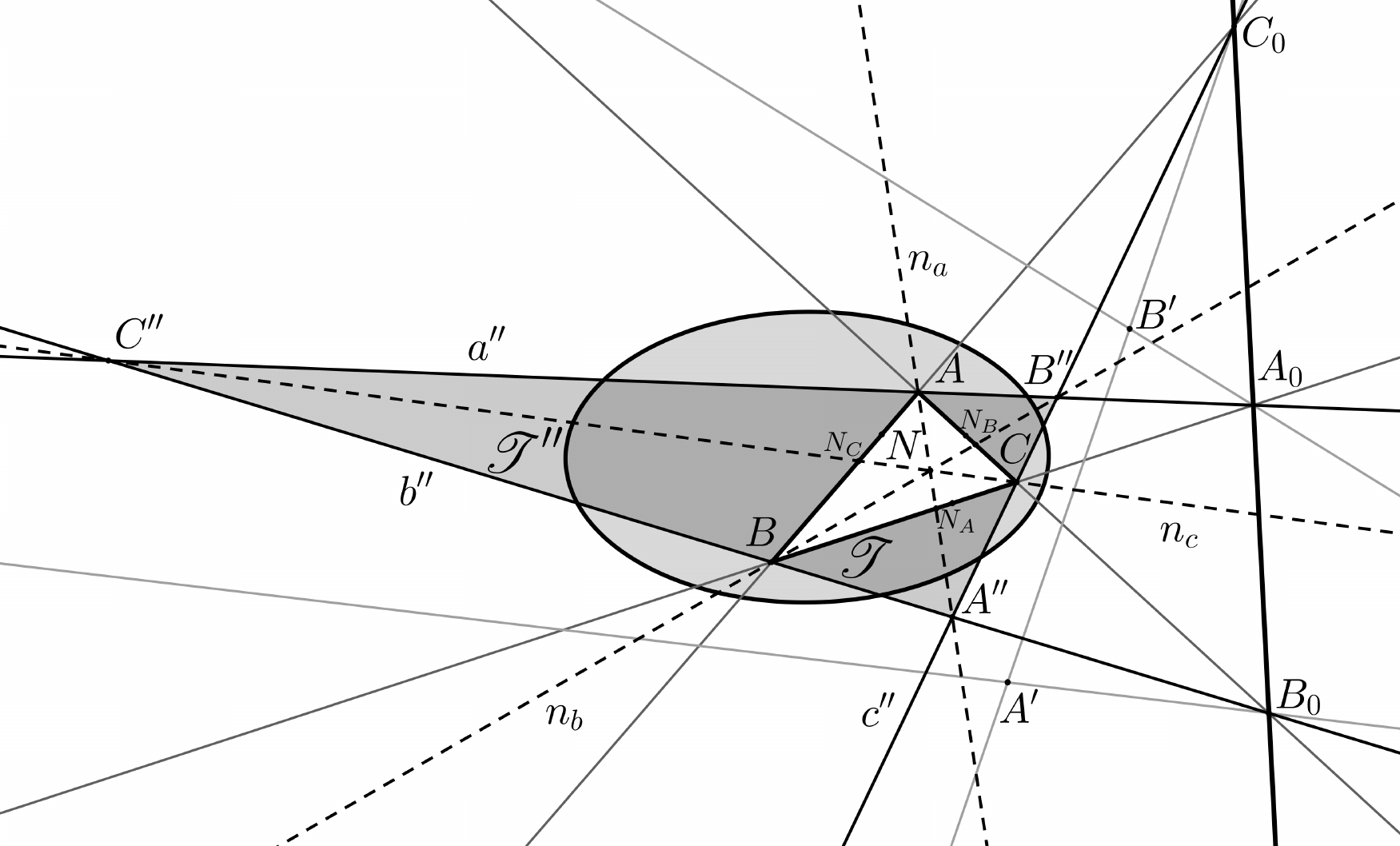}
\caption{Double triangle, pseudomedians and pseudobarycenter}
\label{Fig:pseudobarycenter}
\end{figure}

The barycenter of an euclidean or non-euclidean triangle is the intersection
point of the medians of the triangle. In euclidean geometry, the barycenter of a
triangle $T$ is also the barycenter of its medial triangle and the barycenter
of its double triangle (the triangle whose medial triangle is $T$). This is not
true in non-euclidean geometry.

We define the \emph{double triangle} $\TT''$ of $\TT$ as the triangle
whose sides
$$a''=A_0 A,\quad b''=B_0 B,\quad c''=C_0 C,$$
are the orthogonal lines to the altitudes
$h_{a},h_{b},h_{c}$ through the vertices $A,B,C$ of $\TT$ respectively.
Let $A''=b''\cdot c''$, $B''=c''\cdot a''$, $C''=a''\cdot b''$
be the vertices of $\TT''$. In euclidean geometry, the lines $a'',b'',c''$
are parallel to $a,b,c$ respectively and the lines $AA'',BB'',CC''$ coincide with
the medians of $\TT$. This does not happen in the non-euclidean geometries,
so we will call \emph{pseudomedians} of $\TT$ to the
lines $n_{a}=AA'',n_{b}=BB'',n_{c}=CC''$.
\begin{proposition}
The pseudomedians of $\TT$ are concurrent.\end{proposition}
\begin{proof}
The sides $a,b,c$ of the triangle $\TT$ intersect the sides $a'',b'',c''$
of $\TT''$ at the collinear points $A_{0},B_{0},C_{0}$ respectively.
The result follows from \hyperref[thm:Desargues]{Desargues' Theorem}.
\end{proof}

We say that the point $N$ of intersection
of the pseudomedians of $\TT$ is the \emph{pseudobarycenter} of $\TT$, and that the points $N_A=n_{a}\cdot a$,
$N_B=n_{b}\cdot b$, $N_C=n_{c}\cdot c$ where each pseudomedian
of $\TT$ intersects its opposite side are the \emph{pseudomidpoints} of $\TT$ (compare \cite{Akopyan-other-proposal}).

\begin{proposition}\label{prop:pseudomedial-triangle-orthogonal-to-altitudes}
The altitudes of $\TT$ are orthogonal to the sides of the pseudomedial
triangle $\wt{N_AN_BN_C}$.
\end{proposition}
\begin{proof}It suffices to prove that $h_a$ is orthogonal to $N_BN_C$.
The triangles $\widetriangle{N_BAN_C}$ and $\widetriangle{BA''C}$ are perspective with perspective center $N$. By \hyperref[thm:Desargues]{Desargues' Theorem}, 
the intersection points
$$N_BA\cdot BA'' =CA\cdot C''A''=B_0,\; AN_C\cdot A''C=AB\cdot A''B''=C_0, \; N_BN_C\cdot BC$$ are collinear. This implies that $N_BN_C\cdot BC=BC\cdot h=A_0$. The point $A_0$ is the pole of $h_a$, so the lines $N_BN_C$ and $h_a$ are conjugate.
\end{proof}

The following proposition, trivial in euclidean geometry, is Theorem 16 of
\cite{Wildberger2}.
\begin{proposition}
Even in the non-euclidean geometries, the points $A,B,C$ are
midpoints of the sides of $\TT''$.\end{proposition}
\begin{proof}
The quadrangle $BCB_{0}C_{0}$ has $A$ and $A_{0}$ as diagonal points,
and it is $AA_{0}\cdot BB_{0}=C''$ and $AA_{0}\cdot CC_{0}=B''$.
Thus, it is $(AA_{0}B''C'')=-1$. As the points $A$
and $A_{0}$ are conjugate with respect to $\Phi$, by Lemma
\ref{lem:midpoints-harmonic-UV-AB}
$A$ and $A_{0}$ are the midpoints of the segment $\overline{B''C''}$.
The same holds for $B,B_{0}$ and $C,C_{0}$, 
which are the midpoints of $\overline{C''A''}$ and $\overline{A''B''}$
respectively.
\end{proof}

\section{Pseudocircumcenter}

\begin{figure}[h]
\centering
\includegraphics[width=0.9\textwidth]
{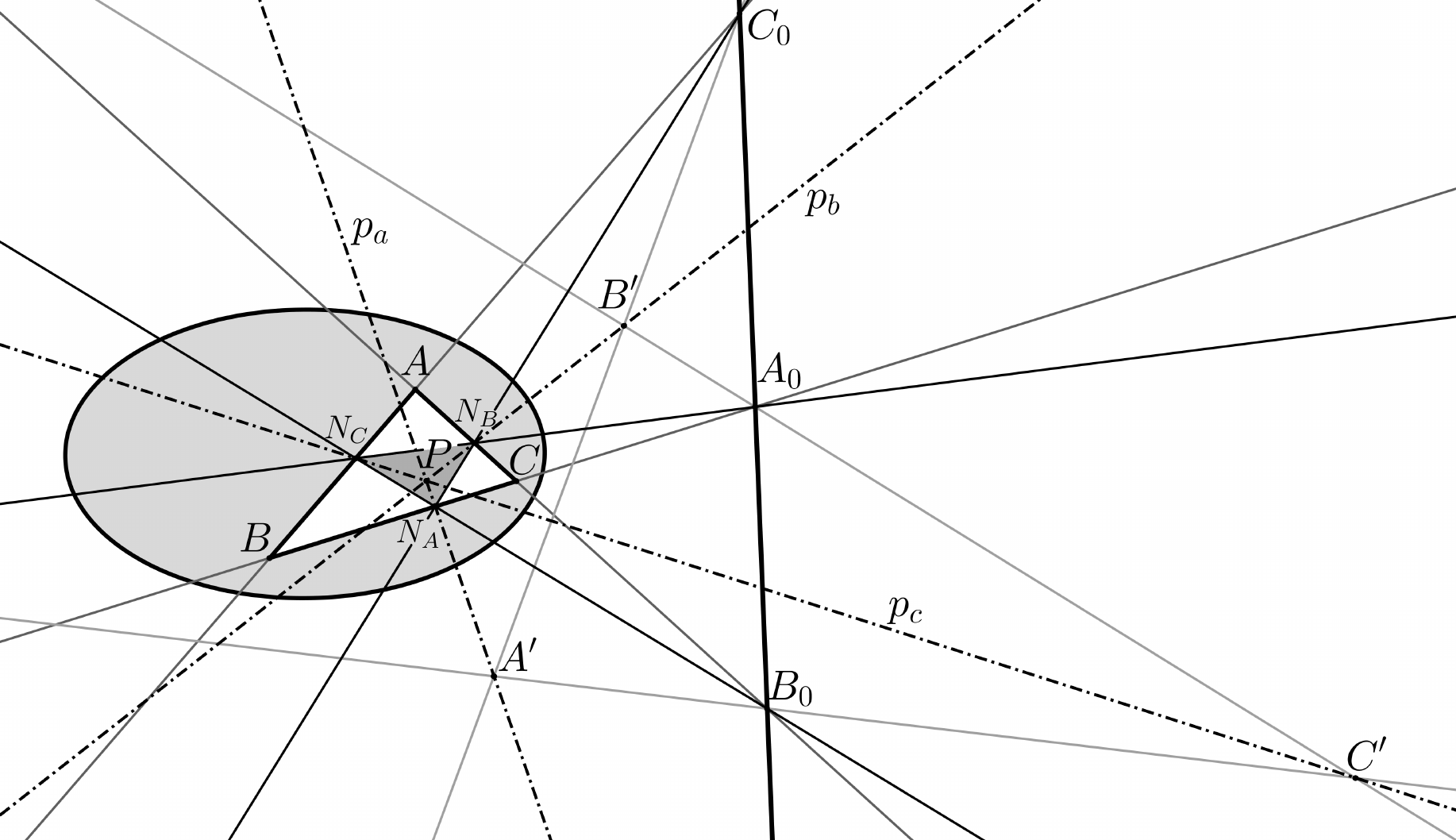}
\caption{Pseudobisectors and pseudocircumcenter}\label{Fig:pseudocircumcenters}
\end{figure}

Let $p_{a},p_{b},p_{c}$ be the orthogonal lines to $a,b,c$ through
the pseudomidpoints $N_A,N_B,N_C$ of $T$, respectively. We
will say that $p_{a},p_{b},p_{c}$ are the \emph{pseudobisectors}
of the sides $a,b,c$ of $\TT$ respectively.
\begin{proposition}
The pseudobisectors of $\TT$ are concurrent.\end{proposition}
\begin{proof}
By Proposition~\ref{prop:pseudomedial-triangle-orthogonal-to-altitudes}, 
the sides of the triangle $\wt{N_A N_B N_C}$ pass through $A_0,B_0$ and $C_0$.
The result follows by applying \hyperref[thm:Desargues]{Desargues' Theorem} to 
the triangles $\widetriangle{N_AN_BN_C}$ and $\TT'$.

%
%
%
%
\end{proof}

We will call \emph{pseudocircumcenter} to the point $P$ where the
pseudobisectors intersect.

The pseudocircumcenter has other analogies with the classical circumcenter
of euclidean geometry. Consider the lines
\begin{eqnarray*}
a_{B}=BA', & b_{C}=CB', & c_{A}=AC',\\
a_{C}=CA', & b_{A}=AB', & c_{B}=BC',
\end{eqnarray*}
and the points
\[
A_{1}=b_{C}\cdot c_{B},\qquad B_{1}=c_{A}\cdot a_{C},\qquad C_{1}=a_{B}\cdot b_{A},
\]
as in Figure~\ref{Fig:pseudocircumcenter-prop}.
In euclidean geometry,
the points $A_{1},B_{1},C_{1}$ lie on the circumcircle of $\TT$, and
they are the symmetric points of $A,B,C$ respectively through the
circumcenter. This does not happen in general in non-euclidean geometry, but
nevertheless we have:
\begin{proposition}
The lines $AA_{1},BB_{1},CC_{1}$ intersect at the pseudocircumcenter
$P$.\end{proposition}
\begin{proof}
It suffices to show that $P$ lies on $AA_{1}$.

We consider the triangles $\widetriangle{A_{1}B'C'}$ and $\widetriangle{AN_BN_C}$.
We have:
\[
A_{1}B'\cdot AN_B=C,\qquad B'C'\cdot N_BN_C=A_{0},\qquad C'A_{1}\cdot N_CA=B.
\]
Because the points $B,C,A_{0}$ are collinear, by 
\hyperref[thm:Desargues]{Desargues' Theorem}
the triangles must be perspective. Thus, the line $A_{1}A$ is concurrent
with the pseudobisectors $B'N_B=p_{b}$ and $C'N_C=p_{c}$, whose intersection point is $P$.
\end{proof}

\begin{figure}[h]
\centering
\includegraphics[width=0.9\textwidth]
{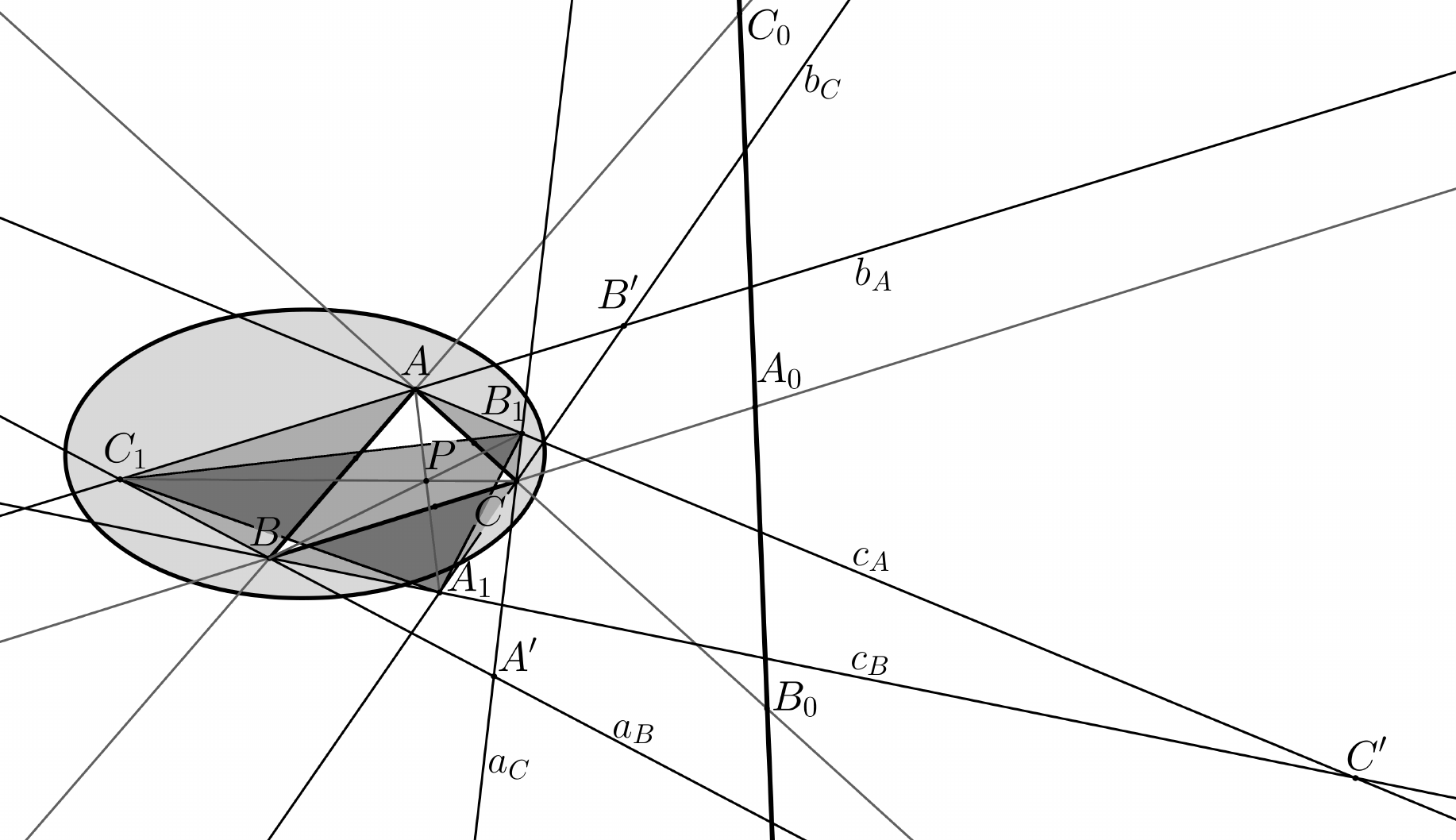}
\caption{more about the pseudocircumcenter}\label{Fig:pseudocircumcenter-prop}
\end{figure}

The definition of the points $A_{1},B_{1},C_{1}$ is symmetric with
respect to $\TT$ and $\TT'$, so we have also:
\begin{proposition}
The lines $A'A_{1},B'B_{1},C'C_{1}$ intersect at the pseudocircumcenter
$P'$ of $\TT'$.
\end{proposition}

The triangle $\wt{A_1B_1C_1}$ has a property similar to that of 
Proposition~\ref{prop:pseudomedial-triangle-orthogonal-to-altitudes} for the triangle $\wt{N_A N_B N_C}$:
\begin{proposition}\label{prop:pseudosymmetric-triangle-orthogonal-to-altitudes}
The altitudes of $\TT$ are orthogonal to the sides of $\wt{A_1B_1C_1}$.
\end{proposition}
\begin{proof}The proof is also identical to that of  Proposition~\ref{prop:pseudomedial-triangle-orthogonal-to-altitudes}. It suffices to prove that $h_a$ is orthogonal to $B_1C_1$.
The triangles $\widetriangle{C_1AB_1}$ and $\widetriangle{CA_1B}$ are perspective with perspective center $P$. By 
\hyperref[thm:Desargues]{Desargues' Theorem}, the intersection points
$$C_1A\cdot CA_1 =B',\quad AB_1\cdot A_1B=C', \quad B_1C_1\cdot BC,$$ are collinear. This implies that $B_1C_1\cdot BC=BC\cdot h=A_0$, and so the lines $B_1C_1$ and $h_a$ are conjugate.
\end{proof}

Another interesting property is the following:
\begin{proposition}
The points $A',A'',A_{1}$ are collinear.\end{proposition}
\begin{proof}
Take the triangles $\widetriangle{BC'B_{0}}$ and $\widetriangle{CB'C_{0}}$.
The lines $BC=a$, $C'B'=a'$and $B_{0}C_{0}=h$ concur at the point
$A_{0}$. Thus, the intersection points
\begin{eqnarray*}
BC'\cdot CB' & = & c_{B}\cdot b_{C}=A_{1}\\
C'B_{0}\cdot B'C_{0} & = & b'\cdot c'=A'\\
B_{0}B\cdot C_{0}C & = & b''\cdot c''=A''
\end{eqnarray*}
are collinear.
\end{proof}

The previous proposition has the following geometric translation:
\begin{proposition}
The lines $A''A_1,B''B_1,C''C_1$ are orthogonal to $a,b,c$ respectively.
\end{proposition}

\section{The Euler-Wildberger line}
\begin{theorem}
The orthocenter $H$, the pseudobarycenters $N$ and $N'$ of $\TT$
and $\TT'$ respectively, and the pseudocircumcenters $P$ and $P'$
of $\TT$ and $\TT'$ respectively are collinear.\end{theorem}

\begin{figure}[h]
\centering
\includegraphics[width=0.9\textwidth]{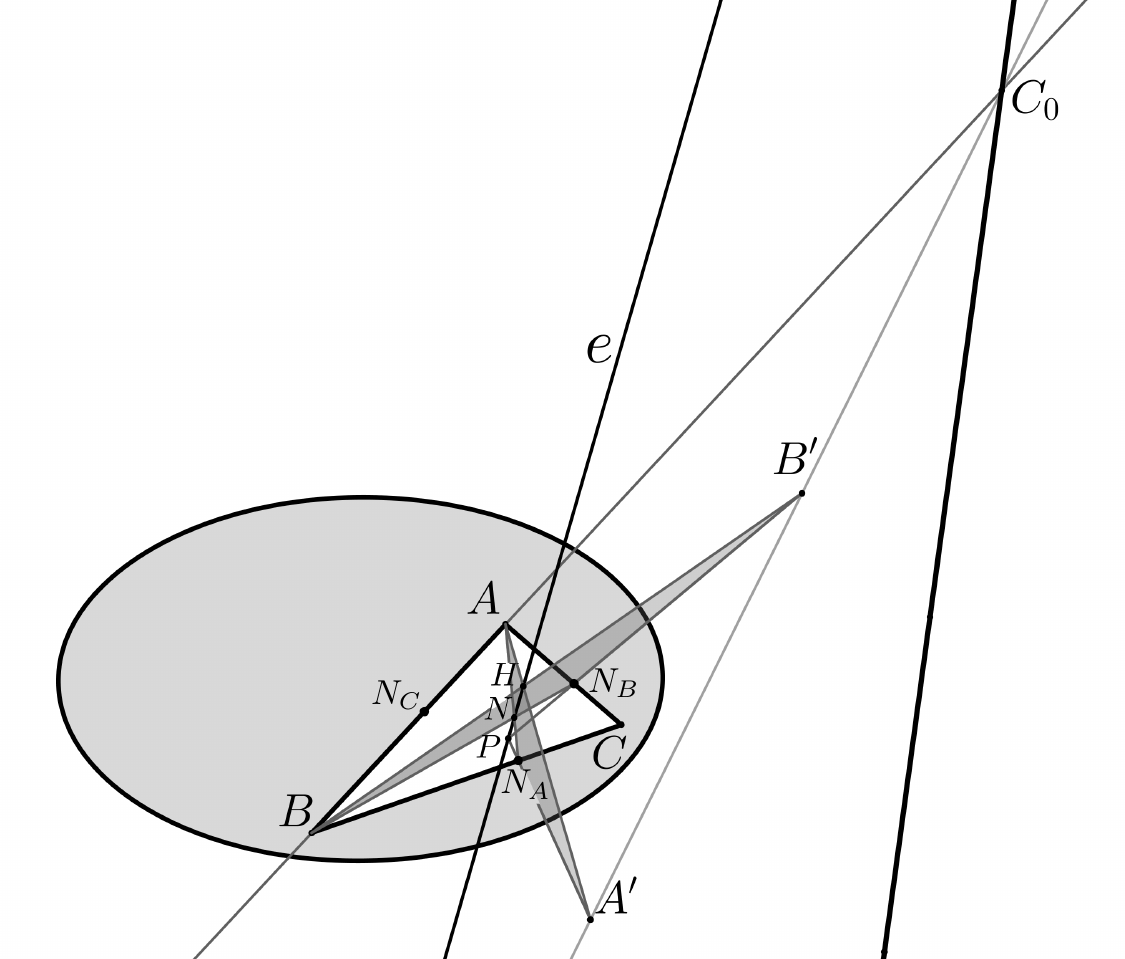}
\caption{Euler-Wildberger line. $H,N,P$ collinear}\label{Fig:thm-Euler-line-01}
\end{figure}

\begin{proof}
Consider the triangles $\widetriangle{AA'N_A}$ and $\widetriangle{BB'N_B}$ 
(Figure~\ref{Fig:thm-Euler-line-01}).
The lines $AB$, $A'B'$ and $N_AN_B$ concur at $C_{0}$. By
\hyperref[thm:Desargues]{Desargues' Theorem}, the intersection points 
\begin{eqnarray*}
AA'\cdot BB' & = & h_{a}\cdot h_{b}=H\\
A'N_A\cdot B'N_B & = & p_{a}\cdot p_{b}=P\\
N_AA\cdot N_BB & = & n_{a}\cdot n_{b}=N
\end{eqnarray*}
are collinear. Let $e$ be the line through $H,N,P$. In the same
way, it can be proved that $H,P'$ and $N'$ belong to a line $e'$.

Let consider now the triangles $\widetriangle{AA'A_{1}}$ and $\widetriangle{BB'B_{1}}$ (Figure~\ref{Fig:thm-Euler-line-02}).
The lines $AB=c,A'B'=c'$ and $A_{1}B_{1}$ concur at the point $C_{0}$.
The intersection points 
\begin{eqnarray*}
AA'\cdot BB' & = & H\\
A'A_{1}\cdot B'B_{1} & = & P'\\
AA_{1}\cdot BB_{1} & = & P
\end{eqnarray*}
are collinear and so it must be $e=e'$.
\end{proof}

\begin{figure}[h]
\centering
\includegraphics[width=0.9\textwidth]{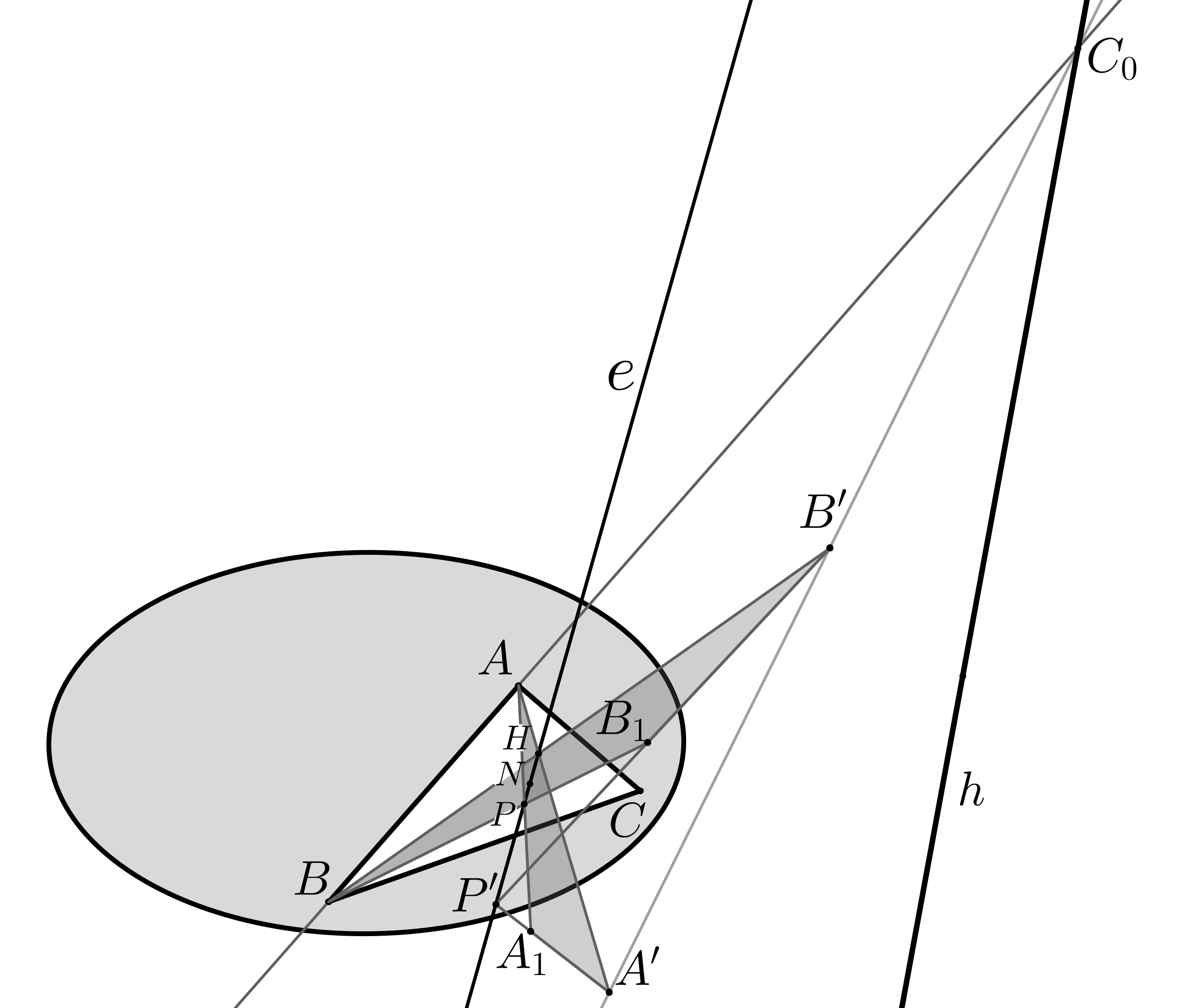}
\caption{Euler-Wildberger line. $H,P,P'$ collinear}\label{Fig:thm-Euler-line-02}
\end{figure}

Thus, the line $e$  joining $H,N,N',P$ and $P'$ is a non-euclidean version of 
the Euler line. This is the line which is called orthoaxis in \cite{Wildberger2}, 
and because of this we say that $e$ is the
\emph{Euler-Wildberger line} of $\TT$. Another interesting 
non-euclidean version
of the Euler line is the
\emph{Euler-Akopyan line} given in \cite{Akopyan-other-proposal}.

    \begin{theorem}
	 The pseudobarycenter $N'$ of $\TT'$ is the pole of the orthic axis of $\TT$.
    \end{theorem}

The \emph{orthic axis} $\mathfrak{o}$ of $\TT$ is the trilinear polar of the
orthocenter $H$ with respect to the triangle $\TT$. If
$$H_A=a\cdot h_a,\quad H_B=b\cdot h_b,\quad H_C=c\cdot h_c,$$
are the feet of the altitudes of $\TT$, by 
\hyperref[thm:Desargues]{Desargues' Theorem} the points
$$a\cdot H_BH_C,\quad b\cdot H_CH_A, \quad c\cdot H_AH_B$$
are collinear, and $\mathfrak{o}$ is the line joining them.

\begin{proof}
The polar of $H_A$ is the line $a'''=\rho(a)\rho(h_a)=A'A_0$. 
Hence, the polar triangle of the orthic triangle $\widetriangle{H_A H_B H_C}$ of
$\TT$ is the double triangle $\TT'''$ of $\TT'$. If $b'''=\rho(H_B)$ and
$c'''=\rho(H_C)$ are the other two sides of $\TT'''$ and
$$A'''=b'''\cdot c''',\quad B'''=c'''\cdot a''',\quad C'''=a'''\cdot b'''$$
are the vertices of $\TT'''$, then we have also
$$A'''=\rho(H_B H_C),\quad B'''=\rho(H_C H_A),\quad C'''=\rho(H_A H_B).$$
The polar line of $a\cdot H_BH_C$ is the line joining $A'$ and $A'''$. Thus, it
is a pseudomedian of $\TT'$, and the pole of $\mathfrak{o}$ is the intersection
of the pseudomedians of $\TT'$, which is the pseudobarycenter $N'$ of $\TT'$.
\end{proof}

Because $e$ passes through $N'$, we have the following corollary, which is also true in the euclidean case.
\begin{corollary}
The Euler-Wildberger line is the line orthogonal to the orthic axis of $\TT$ through the orthocenter $H$.
\end{corollary}
The previous corollary can be used to define the line $e$ before introducing the pseudocenters $N,P$. This is the way followed in \cite{Wildberger}.

\section{The nine-point conic}\label{sec:nine-point-conic}

The nine-point circle of euclidean triangle is a particular version of a more general construction from projective geometry: the \emph{eleven point conic}. Let us describe briefly this object.

Let $\QQ$ be a quadrangle in the projective plane with vertices $A,B,C,D$, and
let $\ell$ be a line not through a vertex of $\QQ$. 
Let $q_1,q_2,\ldots,q_6$ be the six sides of $\QQ$, and assume that $q_1$ and
$q_2$, $q_3$ and $q_4$, $q_5$ and $q_6$ are pairs of opposite sides of $\QQ$.
For each $i=1,2,\ldots,6$, take the intersection point $Q_i$ of $q_i$ with
$\ell$, and consider also the harmonic conjugate $L_i$ of $Q_i$ with respect to
the two vertices of $\QQ$ lying on $q_i$.
Let $\sigma_\QQ$ be the quadrangular involution that $\QQ$ induces in $\ell$,
and let $I,J$ be the two
fixed points of $\sigma_\QQ$.
\begin{theorem}[The eleven-point conic]\label{thm:eleven-point-conic}
The points, $I,J$, the diagonal points of $\QQ$ and the points $L_1,L_2,\ldots,L_6$ lie on a conic.
\end{theorem}
\begin{proof}
See \cite[vol. II, pp. 41-42]{Baker}
\end{proof}

We say that the conic given by this theorem is the \emph{eleven-point conic} of
the quadrangle $\QQ$ and the line $\ell$.

When $\wt{ABC}$ is an euclidean triangle, $D$ is its orthocenter 
and $\ell$ is the line at infinity, the conic given by Theorem
\ref{thm:eleven-point-conic} is the nine-point circle of $\wt{ABC}$.

Let recover our projective triangle $\TT=\wt{ABC}$ as before. By analogy with
the euclidean case, we consider the cuadrangle $\QQ=\{A,B,C,H\}$ and the line
$h$ polar of the orthocenter $H$ with respect to $\Phi$, and we will study the
eleven-point conic $\Gamma$ of $\QQ$ and $h$.

The diagonal points of $\QQ$ are the feet of the altitudes of $\TT$, so it is $H_A,H_B,H_C\in\Gamma$.

\begin{figure}[h]
\centering
\includegraphics[width=0.9\textwidth]
{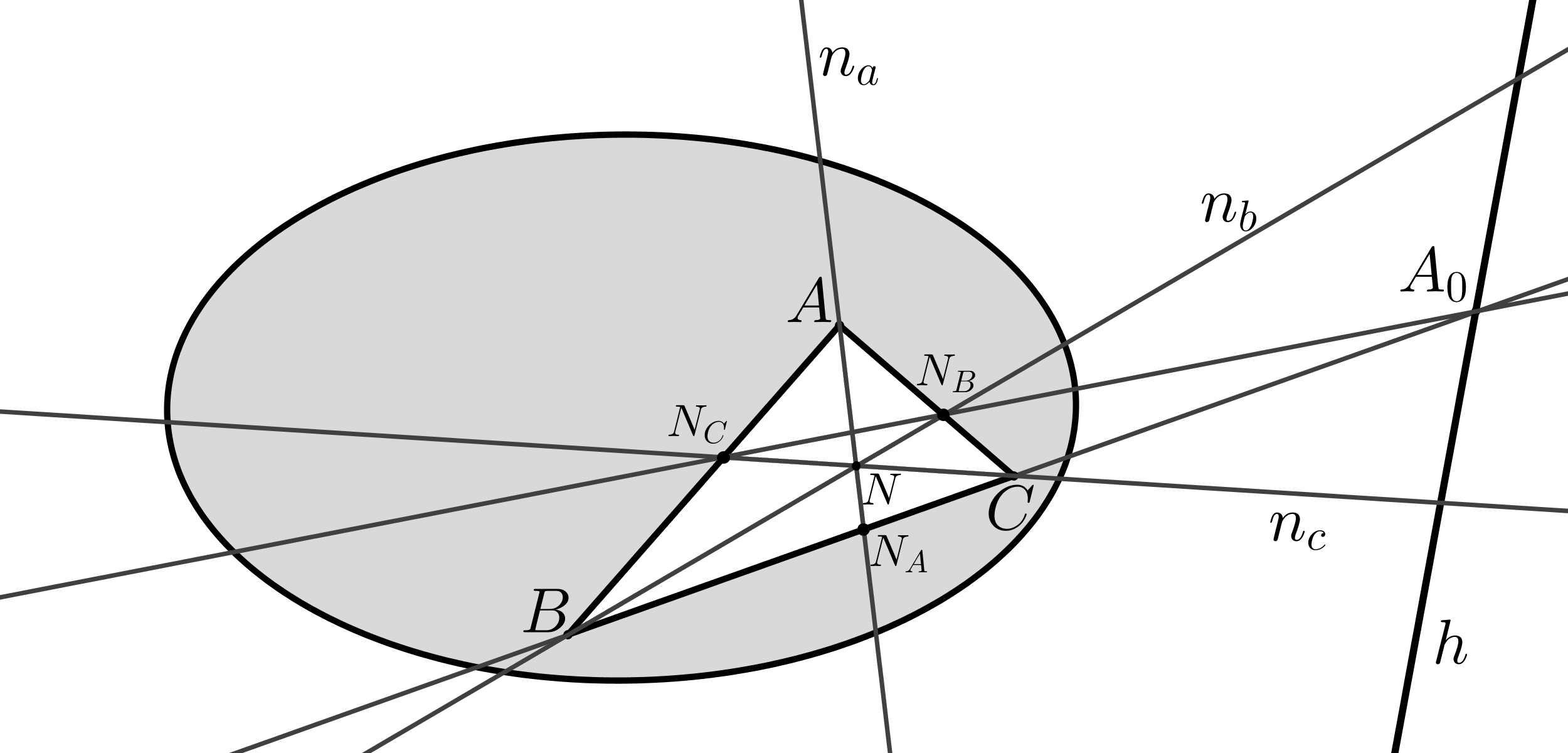}
\caption{$N_A$ is the harmonic conjugate of $A_0$ with respect to $B,C$}
\label{Fig:nine-point-conic-pseudomidpoints}
\end{figure}

Let consider now the sides $a=BC,b=CA,c=ab$ of $\QQ$. The intersection points $a\cdot h,b\cdot h,c\cdot h$ are $A_0,B_0,C_0$ respectively.
\begin{lemma}
The harmonic conjugate of $A_0$ with respect to $B$ and $C$ is $N_A$
respectively.
\end{lemma}
\begin{proof}
By considering the quadrangle $\{N,A,N_B,N_C\}$ 
(Figure~\ref{Fig:nine-point-conic-pseudomidpoints}) and using that 
$A_0\in N_BN_C$, it can be seen that $N_A$ is the harmonic conjugate of $A_0$
with respect to $B,C$.
\end{proof}

Thus, we have $N_A\in\Gamma$ and in the same way $N_B,N_C\in\Gamma$.

\begin{figure}[h]
\centering
\includegraphics[width=0.9\textwidth]
{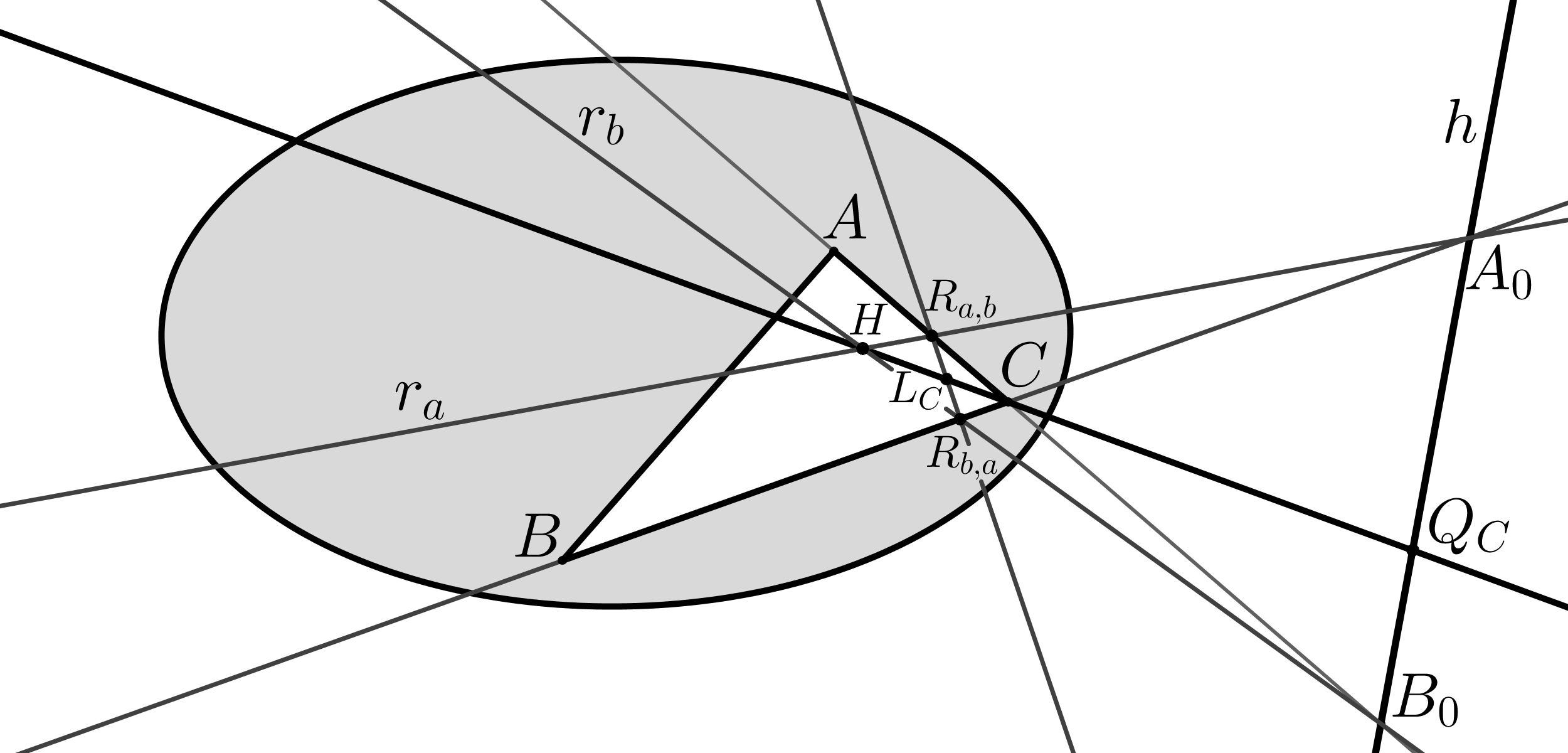}
\caption{Finding the harmonic conjugate of $h_c\cdot h$ with respect to $C,H$}
\label{Fig:nine-point-conic-other-three}
\end{figure}

Finally, consider the side $CH=h_c$ of the quadrangle $\QQ$. 
Let $Q_C$ be the intersection point of $h_c$ with $h$. We need to find the
harmonic conjugate $L_C$ of $Q_C$ with respect to $C,H$. Take the lines
$r_a=A_0H$ and $r_b=B_0H$, and the points $R_{a,b}=r_a\cdot b$ and
$R_{b,a}=r_b\cdot a$. If we take the quadrangle
$\SC{R}_C=\{A_0,B_0,R_{a,b},R_{b,a}\}$, we have that $C$ and $H$ are diagonal
points of $\SC{R}_C$ and thus $L_C=R_{a,b}R_{b,a}\cdot h_c$ (Figure
\ref{Fig:nine-point-conic-other-three}). In a similar way, using also the line
$r_c=C_0H$ and the points $R_{b,c}=r_b\cdot c$, $R_{c,b}=r_c\cdot b$,
$R_{c,a}=r_c\cdot a$ and $R_{a,c}=r_a\cdot c$, we can find other two points
$L_A,L_B$ of $\Gamma$ as in Figures
\ref{Fig:Euler-line-nine-point-conic-hyperbolic} and
\ref{Fig:Euler-line-nine-point-conic-elliptic}. The lines $r_a,r_b,r_c$ are the
perpendicular lines to the altitudes through the orthocenter, and so in
euclidean geometry they are parallel to the sides of the triangle. The
quadrangular involution $\sigma_{\QQ}$ of $h$
coincides with the conjugacy $\rho_h$ with respect to $\Phi$, so if $\Phi$ is
imaginary, or if $\Phi$ is real and $H$ is interior to $\Phi$, the fixed points
$I,J$ of $\sigma_{\QQ}$ are imaginary.

\begin{figure}[h]
\centering
\includegraphics[width=0.98\textwidth]
{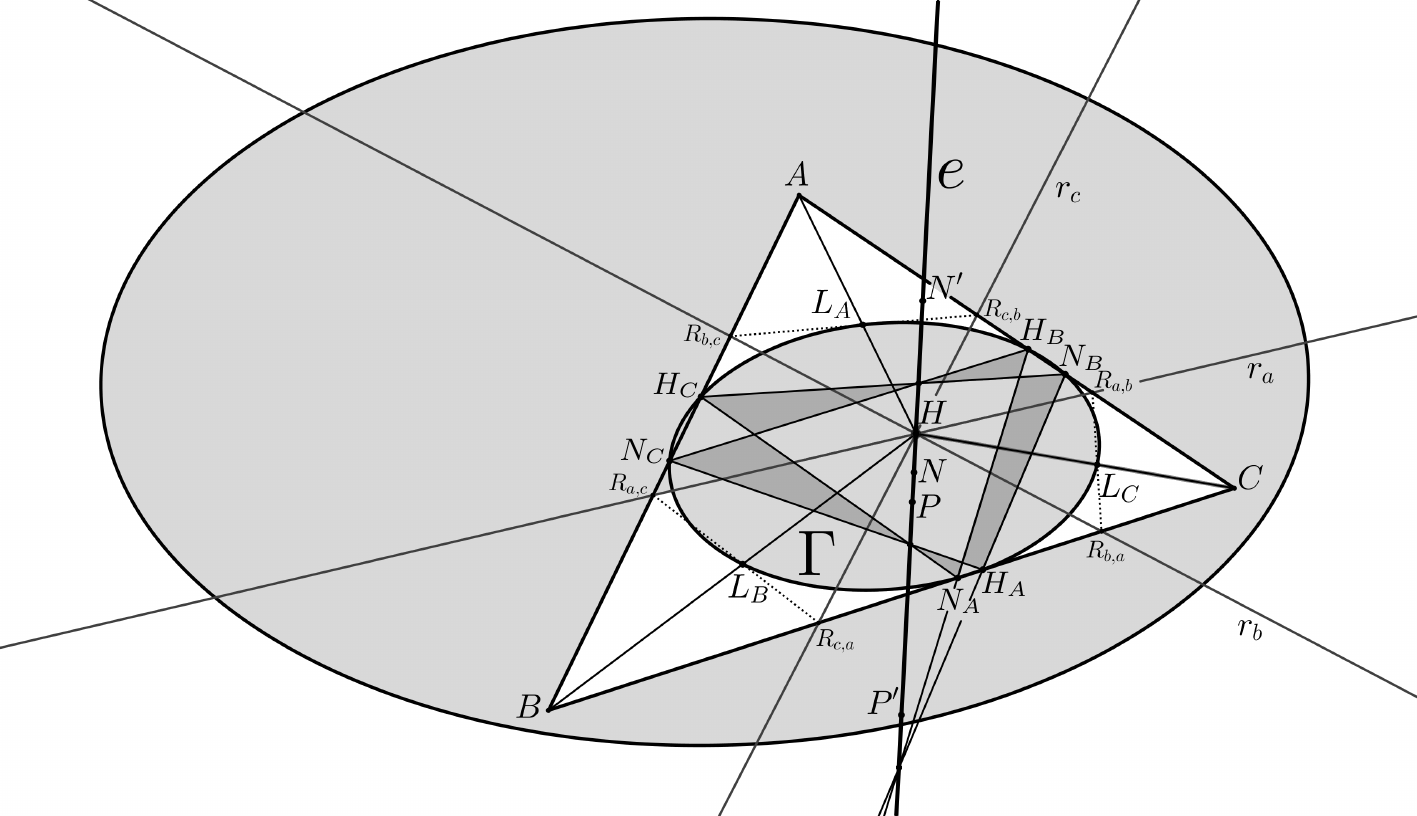}
\caption{The Euler-Wildberger line $e$ and the nine-point conic $\Gamma$ of a
hyperbolic triangle}
\label{Fig:Euler-line-nine-point-conic-hyperbolic}
\end{figure}

In analogy with euclidean geometry, we have
\begin{theorem}
The Euler-Wildberger line $e$ is the Pascal line of the hexagon 
$\SC{H}=H_AN_BH_CN_AH_BN_C$ (see Figures
\ref{Fig:Euler-line-nine-point-conic-hyperbolic} and
\ref{Fig:Euler-line-nine-point-conic-elliptic}).
\end{theorem}
\begin{proof}
It suffices to prove that the opposite sides $H_AN_B$ 
and $N_AH_B$ of the hexagon $\SC{H}$ inscribed in $\Gamma$ intersect at $e$. If
we consider the hexagon $AN_AH_BBN_BH_A$, with alternate vertices in the lines
$a,b$, by Pappus' Theorem the intersection points of opposite sides
$$AN_A\cdot BN_B=N,\quad N_AH_B\cdot N_BH_A,\quad H_BB\cdot H_A A=H$$
are collinear. This completes the proof.
\end{proof}

\begin{figure}[h]
\centering
\includegraphics[width=0.8\textwidth]
{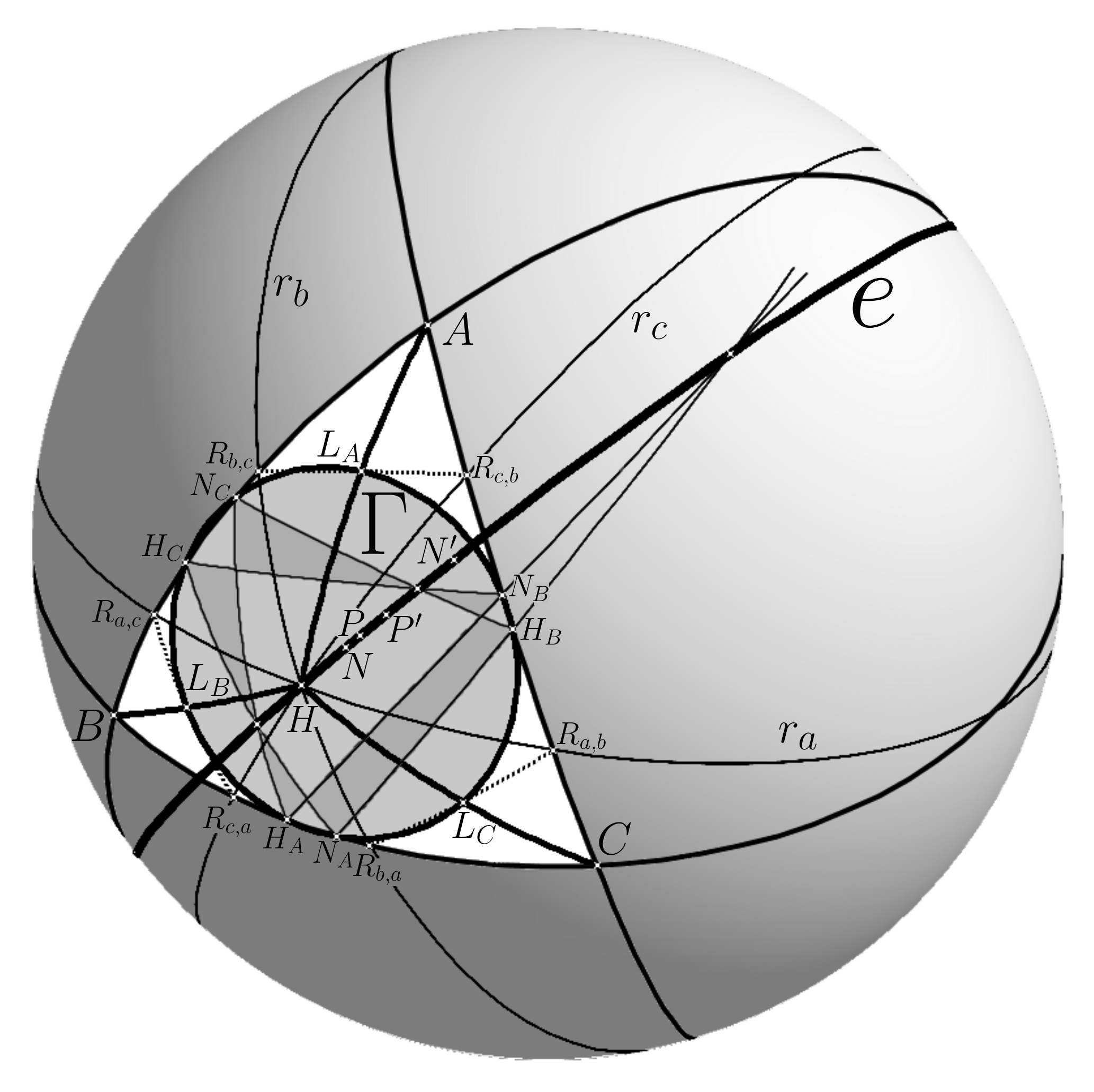}
\caption{The Euler-Wildberger line $e$ and the nine-point conic $\Gamma$ of an elliptic triangle}\label{Fig:Euler-line-nine-point-conic-elliptic}
\end{figure}

There are some more properties of the Euler-Wildberger line 
$e$ and the nine-point conic $\Gamma$ that we have checked experimentally
(working with a real conic $\Phi$ in GeoGebra \cite{GeoGebra}), but for which we
have no proofs. Among them:
\begin{enumerate}
	\item The midpoints of $\overline{HP}$ and the pole $E$ of $e$ with respect to $\Phi$ are the vertices of a self-polar triangle with respect to $\Phi$ and $\Gamma$.
	\item The line $e$ is a simmetry axis of $\Gamma$.
	\item When $\Gamma$ is an ellipse\footnote{$\Gamma$ is always an ellipse if $\Phi$ is imaginary. See \cite{Coolidge} for a classification of conics in the hyperbolic plane.}, its center is a midpoint of $\overline{HP}$ and the orthogonal line to $e$ through the center is also a simmetry axis of $\Gamma$.
\end{enumerate}
It should be interesting to find proofs for these statements and also to find
more analogies between the euclidean Euler line and
nine-point circle and their noneuclidean versions
proposed here or in~\cite{Akopyan-other-proposal}.

\chapter{Menelaus' Theorem and non-euclidean\\ trigonometry}\label{sec:Menelaus-Theorem}

The following theorem is classical Menelaus' Theorem as it is usually
stated in affine geometry.
\begin{theorem}[Menelaus' Theorem]
\label{thm:Menelaus-affine}Let $\TT=\widetriangle{XYZ}$
be a triangle in the affine plane. The points $X_{1},Y_{1},Z_{1}$
on the lines $x=YZ$, $y=ZX$, $z=XY$ respectively are collinear
if and only if
\begin{equation}
\frac{\brs{XZ_{1}}}{\brs{YZ_{1}}}\cdot\frac{\brs{YX_{1}}}{\brs{ZX_{1}}}\cdot\frac{\brs{ZY_{1}}}{\brs{XY_{1}}}=1.\label{eq:Menelaus_Affine}
\end{equation}
\end{theorem}
Identity~\eqref{eq:harmonic ratio equals cross ratio} allows
to make a projective interpretation of 
\hyperref[thm:Menelaus-affine]{Menelaus' Theorem}
by introducing the line at infinity as part of the figure (see \cite[vol. II, pp. 89-90]{V - Y}):
if $X_{\infty},Y_{\infty},Z_{\infty}$ are the points at infinity
of the lines $x,y,z$ respectively, then~\eqref{eq:Menelaus_Affine}
becomes
\begin{equation}
\left(XYZ_{1}Z_{\infty}\right)\left(YZX_{1}X_{\infty}\right)\left(ZXY_{1}Y_{\infty}\right)=1.\label{eq:Menelaus_Projective-1}
\end{equation}
In this projective version of~\eqref{eq:Menelaus_Affine}, the line
at infinity can be replaced with any other line of the projective
plane. Assume that $X_{1},Y_{1},Z_{1}$ are collinear, and let $s$ be the line
they belong to. If we consider another line $r$ and its intersection
points with $x$, $y$ and $z$:
\[
X_{0}=x\cdot r,\qquad Y_{0}=y\cdot r,\qquad Z_{0}=z\cdot r,
\]
by \hyperref[thm:Menelaus-affine]{Menelaus' Theorem} it is
\begin{equation}
\left(XYZ_{0}Z_{\infty}\right)\left(YZX_{0}X_{\infty}\right)\left(ZXY_{0}Y_{\infty}\right)=1.\label{eq:Menelaus_Projective-2}
\end{equation}
If we apply the identity~\eqref{eq:CR3} to each cross ratio 
in~\eqref{eq:Menelaus_Projective-1},
we have
\[
\left(XYZ_{1}Z_{0}\right)\left(XYZ_{0}Z_{\infty}\right)\left(YZX_{1}X_{0}\right)\left(YZX_{0}X_{\infty}\right)\left(ZXY_{1}Y_{0}\right)\left(ZXY_{0}Y_{\infty}\right)=1,
\]
and by~\eqref{eq:Menelaus_Projective-2} the previous identity turns
into
\begin{equation}
\left(XYZ_{1}Z_{0}\right)\left(YZX_{1}X_{0}\right)\left(ZXY_{1}Y_{0}\right)=1.\label{eq:Menelaus_Projective}
\end{equation}

\begin{figure}
\centering
\includegraphics[width=0.6\textwidth]
{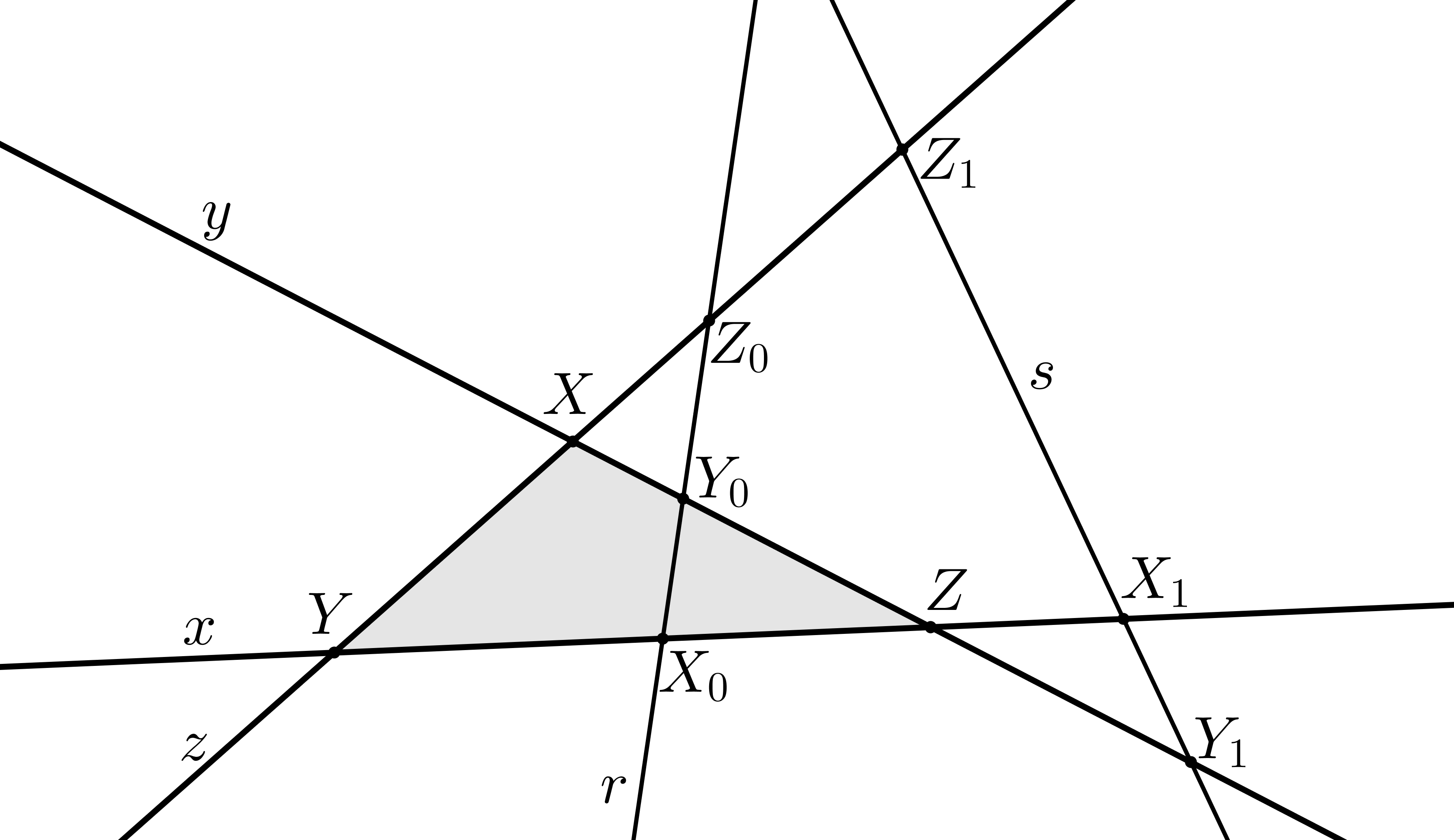}
\caption{Menelaus configuration}
\label{Menelaus_configuration} 
\end{figure}

We will say that the figure composed by the triangle $\widetriangle{xyz}$
and the two lines $r,s$ as above is a \emph{Menelaus configuration}
with triangle $\widetriangle{xyz}$ and transversals $r,s$. In order to
avoid degenerate cases, we will assume that the five lines involved
in a Menelaus' configuration are always in general position: they
are all distinct and no three of them are concurrent.

Thus, we have proved the following corollary of 
Theorem~\ref{thm:Menelaus-affine},
which will be enough for our purposes:
\begin{corollary}[Menelaus' Projective Formula]
\label{cor:Menelaus}For any Menelaus' configuration with triangle
$\TT=\widetriangle{xyz}$ and transversals $r,s$, if we label
the intersection points of the lines of the figure (with the only
exception of $r\cdot s$) as in Figure~\ref{Menelaus_configuration},
the identity~\eqref{eq:Menelaus_Projective} holds.
\end{corollary}

\section{Trigonometry of generalized right-angled triangles}
\label{sec:Trigonometry-for-right-angled}

\begin{figure}
\centering
\subfigure[spherical right-angled triangle]
{\includegraphics[width=0.9\textwidth]
{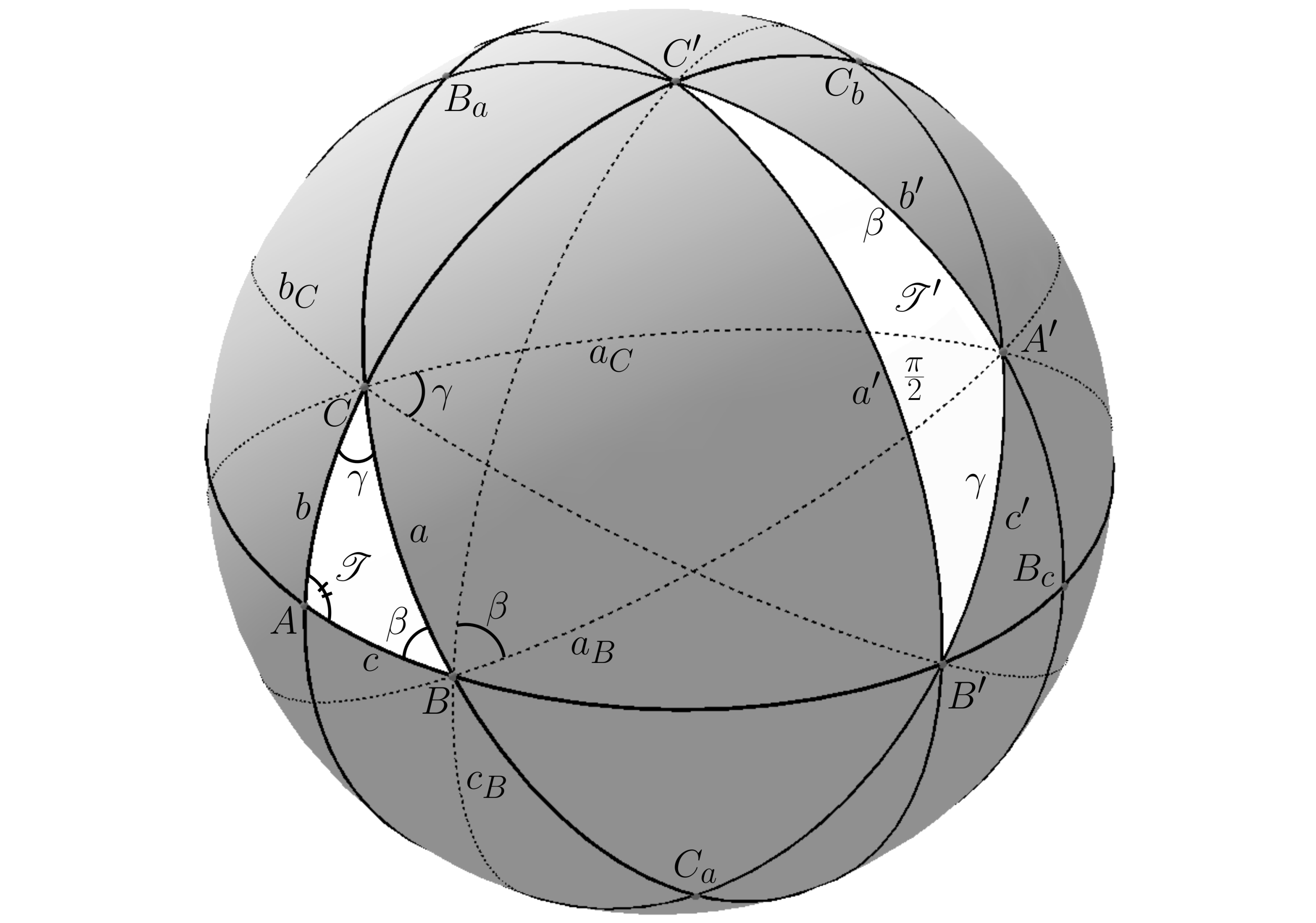}\label{sph_triangle} }
\\
\subfigure[hyperbolic right-angled triangle]
{\includegraphics[width=0.9\textwidth]
{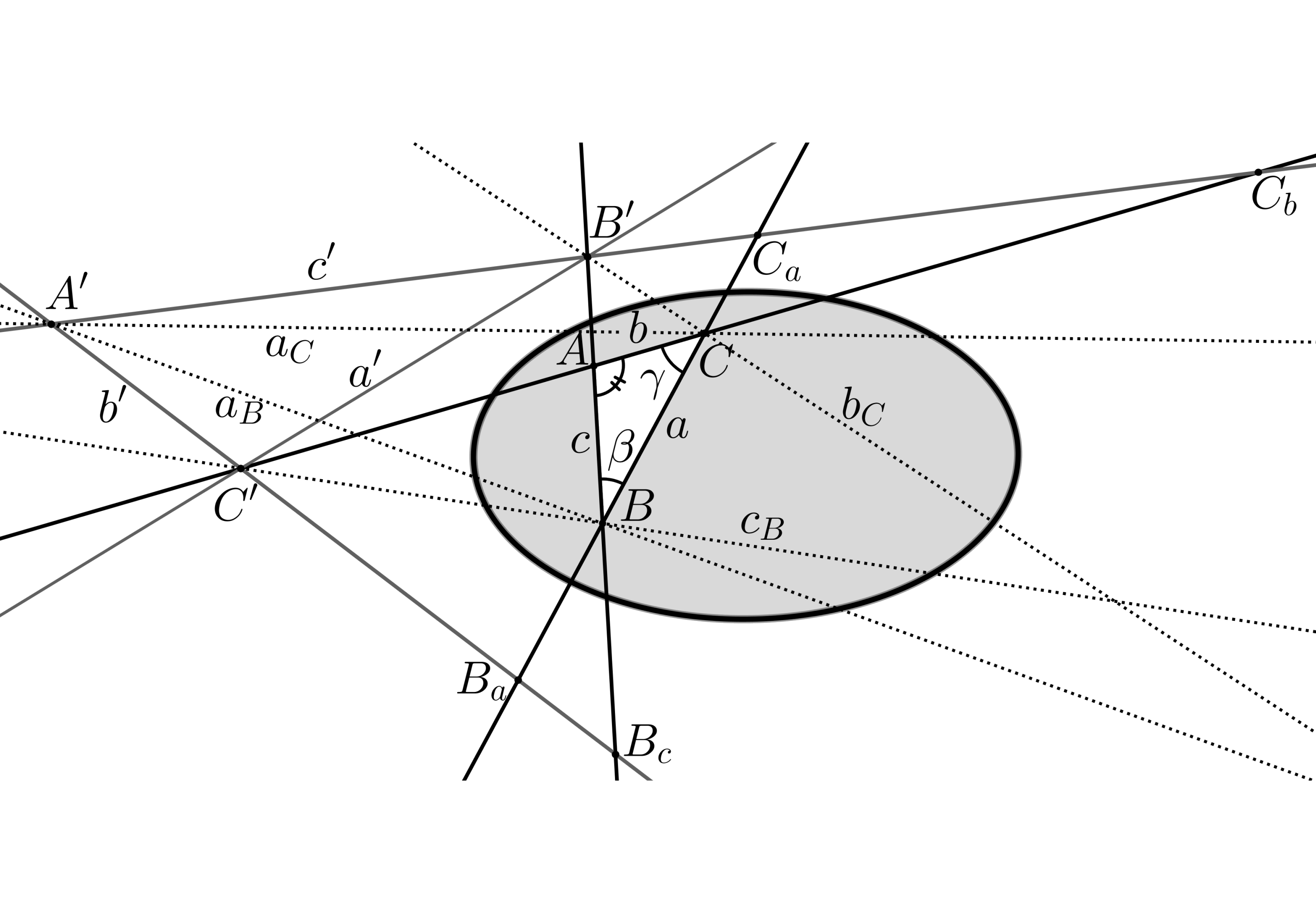}
\label{hyp_triangle} }

\caption{generalized right-angled triangles I}
\label{Fig:generalized-right-angled-triangles-I}
\end{figure}

Consider a projective triangle $\TT$ with vertices $A,B,C$ and its polar triangle $\TT'$ as before.
Assume from now on that $\TT$ is right-angled: $b$ and $c$ are conjugate to
each other; and
that $A\in\mathbb{P}$. Thus, $C'\in b$ and $B'\in c$. Taking the
conjugate points and lines
\[
\begin{split}B_{a}=a\cdot b',\quad B_{c}=c\cdot b',\quad a_{B}=BA',\quad c_{B}=BC',\\
C_{a}=a\cdot c',\quad C_{b}=b\cdot c',\quad a_{C}=CA',\quad b_{C}=CB',
\end{split}
\]
the following relations hold
\begin{equation*}
\begin{split}
A_{b}=b\cdot a'=C',\quad B_{a}=A'_{ b'},\quad B_{c}=C'_{ b'},\quad b_{A}=AB'=c,\\
A_{c}=c\cdot a'=B',\quad C_{a}=A'_{ c'},\quad C_{b}=B'_{ c'},\quad c_{A}=AC'=b.
\end{split}
\end{equation*}

\begin{figure}
\centering

\subfigure[hyperbolic Lambert quadrilateral]
{\includegraphics[width=0.9\textwidth]
{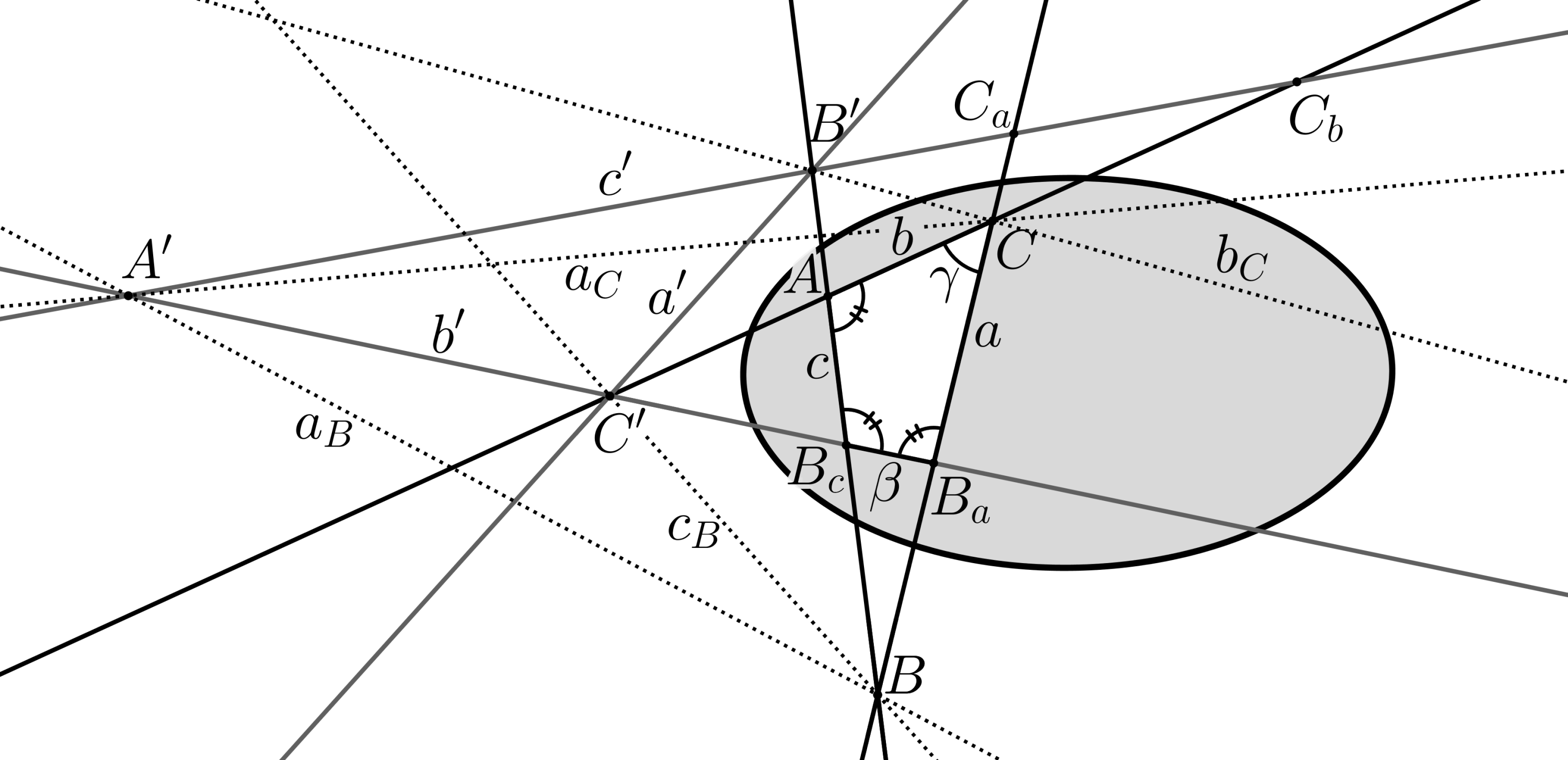}
\label{hyp_Lambert} }
\\
\subfigure[hyperbolic right-angled pentagon]
{\includegraphics[width=0.9\textwidth]
{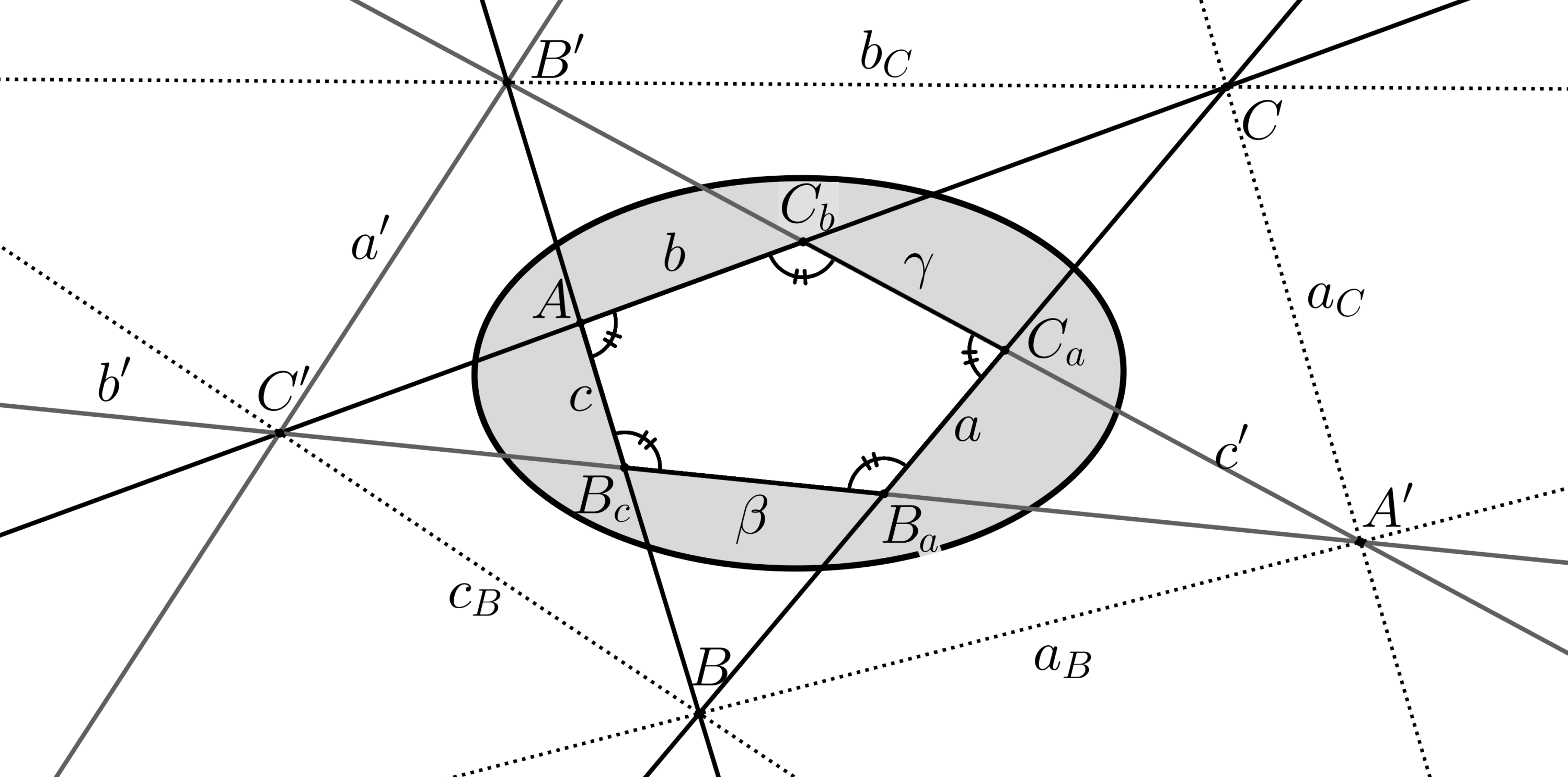}
\label{hyp_pentagon} }

\caption{generalized right-angled triangles II}
\label{Fig:generalized-right-angled-triangles-II}
\end{figure}

The projective figure composed by the triangles $\TT,\TT'$
and the rest of points and lines considered above has four different
geometric interpretations as generalized right-angled triangles. As
the lines $b$ and $c$ are conjugate to each other, they form a right
angle at the point $A$. When $\Phi$ is an imaginary conic (and
so $\mathbb{P}$ is the elliptic plane) $\TT$ is an
elliptic right-angled triangle (Figure~\ref{sph_triangle}). When $\Phi$ is a real conic, the triangles $\TT$ and $\TT'$
can produce three different hyperbolic polygons: (i) a right-angled
hyperbolic triangle when $B,C$ are also interior to $\Phi$ (Figure
\ref{hyp_triangle}); (ii) a Lambert quadrilateral, when one of the
vertices $B,C$ turns out to be exterior to $\Phi$ (Figure~\ref{hyp_Lambert});
and (iii) a right-angled pentagon, when both vertices $B,C$ are considered
exterior to $\Phi$ but the line $BC$ is secant to $\Phi$ (Figure
\ref{hyp_pentagon}). If the line $BC$ is exterior to $\Phi$, there
appears again a Lambert quadrilateral inside $\mathbb{P}$ and this
is the same as case (ii).

Our strategy for obtaining trigonometric relations for generalized
right-angled triangles will follow these steps:
\begin{enumerate}
\item Take a Menelaus' configuration with the five lines $a,b,c, b', c'$
by choosing three of the lines as the sides of the triangle and the
other two lines as the transversals of the configuration.
\item Apply 
\hyperref[cor:Menelaus]{Menelaus' Projective Formula} \eqref{eq:Menelaus_Projective} 
in order to obtain a formula relating three
cross ratios of points on the figure.
\item Translate
the three cross ratios of the formula just obtained into projective
trigonometric ratios associated to the sides of the triangles $\TT$
and $\mathscr{T'}$. We so obtain a \emph{projective trigonometric
formula} associated with the figure.
\item For every particular generalized right-angled triangle, translate
the projective trigonometric formula obtained in the previous step
into a non-euclidean trigonometric formula using 
Table~\ref{table:non-euclidean-segments}.
\end{enumerate}
After the fourth step, we will obtain a squared non-euclidean trigonometric formula
due to the presence of the square power in all the circular and hyperbolic
trigonometric ratios of Table~\ref{table:non-euclidean-segments}. As we will explain later,
the subjacent unsquared trigonometric formula will be the result of
removing all the square powers and choosing positive signs in the
trigonometric functions of the corresponding squared formula.

By an abuse of notation, we will denote each of the segments
$\ov{AB}$, $\ov{BC}$, $\ov{CA}$, $\ov{A'B'}$, $\ov{C'A'}$ with the name of the line that contains it.

\subsection{Non-euclidean Pythagorean Theorems}\label{sub:Non-euclidean-Pythagoras'-Theorem}

Take the Menelaus' configuration with triangle $\widetriangle{abc}$ and
transversals $ b', c'$ (Figure~\ref{Fig:projective_right_triangle_Menelaus_01}).
By \hyperref[cor:Menelaus]{Menelaus' Projective Formula} 
we have:
\begin{align}
\!\!\!\left(ABB_{c}B'\right)\left(BCB_{a}C_{a}\right)\left(CAC'C_{b}\right)=1 \underset{C'=A_{b}}{\overset{B'=A_{c}}{\Longleftrightarrow}} &\nonumber \\ 
\!\!\!\iff \left(ABB_{c}A_{c}\right)\left(BCB_{a}C_{a}\right)\left(CAA_{b}C_{b}\right)=1\overset{\eqref{eq:CR1}}{\Longleftrightarrow}& \nonumber\\
\!\!\!\iff  \mathbf{C}(c)\dfrac{1}{\mathbf{C}(a)}\mathbf{C}(b)=1 
\Longleftrightarrow  & \; \mathbf{C}(a)=\mathbf{C}(b)\mathbf{C}(c)\label{eq:trigo1}\tag{\textbf{T1}}
\end{align}

If we give the names $a,b,c, \beta, \gamma$ to the sides and angles
of generalized right-angled triangles as in Figures~\ref{Fig:generalized-right-angled-triangles-I} and~\ref{Fig:generalized-right-angled-triangles-II}, by
Table~\ref{table:non-euclidean-segments} identity~\eqref{eq:trigo1}
translates into:
\begin{itemize}
\item for an elliptic right-angled triangle
\begin{equation}
\cos^{2}a=\cos^{2}b\cos^{2}c.\label{eq:squared-formula-sph-Triangle}
\end{equation}

\item for a hyperbolic right-angled triangle 
\begin{equation}
\cosh^{2}a=\cosh^{2}b\cosh^{2}c,\label{eq:squared-formula-hyp-Triangle}
\end{equation}

\item for a Lambert quadrilateral
\begin{equation}
-\sinh^{2}a=\cosh^{2}b\left(-\sinh^{2}c\right),\label{eq:squared-formula-hyp-Lambert}
\end{equation}

\item for a right-angled pentagon
\begin{equation}
\cosh^{2}a=\left(-\sinh^{2}b\right)\left(-\sinh^{2}c\right),\label{eq:squared-formula-hyp-Pentagon}
\end{equation}
\end{itemize}
If we look for the unsquared versions of formulae $(\ref{eq:squared-formula-hyp-Triangle}-
\ref{eq:squared-formula-sph-Triangle})$, we must look at each figure
separately.

In the three hyperbolic figures there is no discussion if we assume,
as usual, that the length of a segment is always positive. In this
case, because hyperbolic sines and cosines are positive for positive
arguments, it must be
\[
\cosh a=\cosh b\cosh c
\]
for the hyperbolic right angled triangle (this formula is known as
the \emph{hyperbolic Pythagorean theorem}),
\[
\sinh a=\cosh b\sinh c
\]
for the Lambert quadrilateral, and
\[
\cosh a=\sinh b\sinh c
\]
for the right-angled pentagon.

The elliptic case is subtler than the hyperbolic ones because $\cos x$
can be positive or negative for $x\in(0,\pi)$. We need more geometric
information for deciding which of the two unsquared versions of~\eqref{eq:squared-formula-sph-Triangle}:
\[
\cos a=\cos b\cos c,\qquad\text{or}\qquad\cos a=-\cos b\cos c;
\]
is the correct one, or if both are correct but they apply to different
figures. Because elliptic segments have their length in $[0,\pi]$,
we will extend the usual terms for angles: \emph{acute}, \emph{right}
and \emph{obtuse}; to segments in the natural way.
The proof of the following proposition is left to the reader.
\begin{proposition}[elliptic triangles I]
\label{prop:elliptic-triangle-I-obtuse-sides}In the triangle $\TT$ of 
Figure~\ref{sph_triangle},
the hypotenuse $a$ is obtuse if and only if exactly one of the catheti
$b,c$ is obtuse. In particular, if the triangle $\TT$ has
an obtuse side, it has exactly $2$ obtuse sides.
\end{proposition}
Proposition~\ref{prop:elliptic-triangle-I-obtuse-sides} implies that
if one of the three ratios $\cos a$, $\cos b$, $\cos c$ is negative,
then exactly another one is also negative while the remaining one
is positive. Therefore, the correct formula for elliptic right-angled
triangles is
\[
\cos a=\cos b\cos c,
\]
which is known as the \emph{spherical Pythagorean theorem}.

\subsection{More trigonometric relations}\label{sub:More-trigonometric-relations}

\begin{sidewaysfigure}
\centering
\subfigure[Menelaus configuration 1]
{\includegraphics[width=0.45\textwidth]
{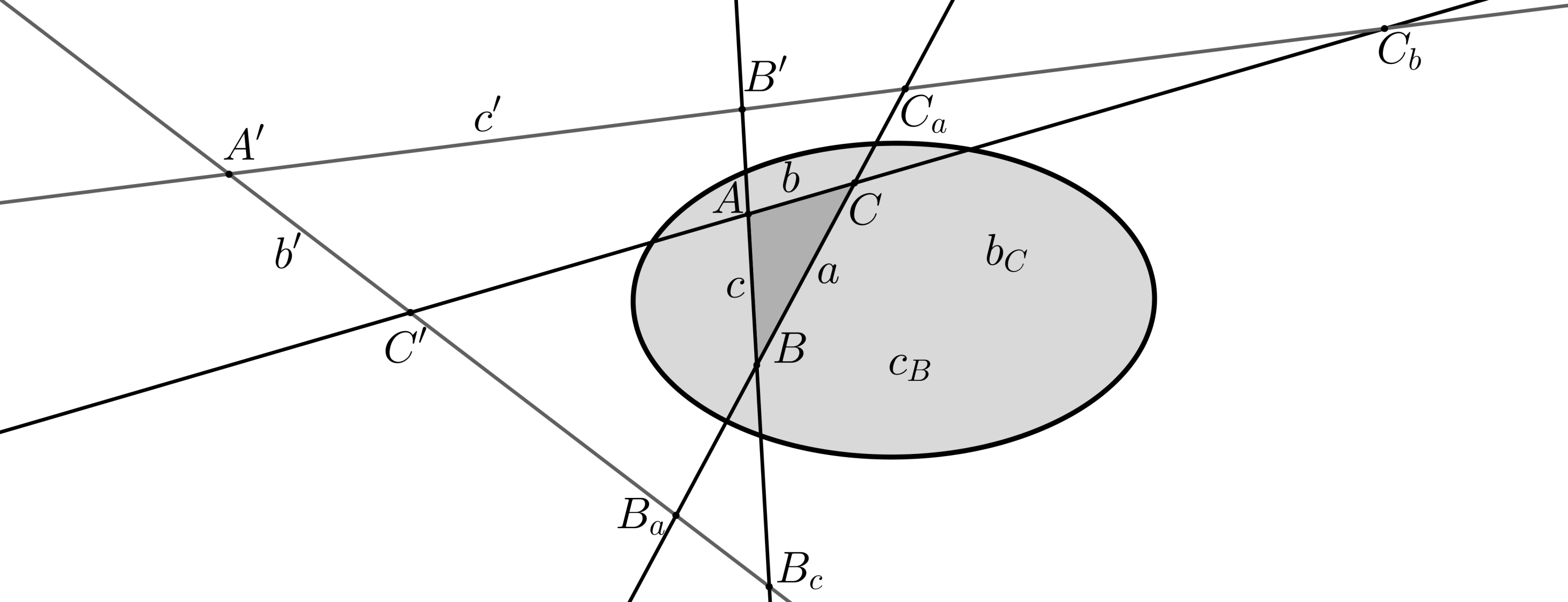}
\label{Fig:projective_right_triangle_Menelaus_01}}
\hfill
\subfigure[Menelaus configuration 2]
{\includegraphics[width=0.45\textwidth]
{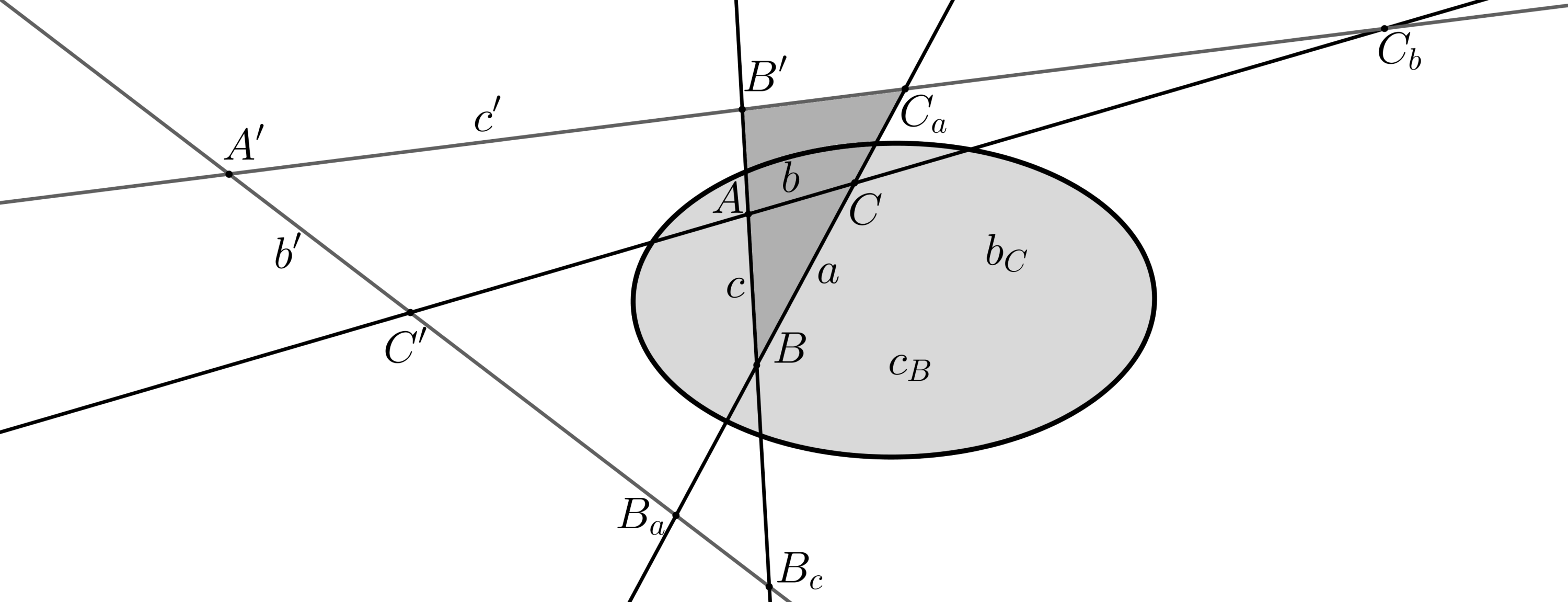}
\label{Fig:projective_right_triangle_Menelaus_02}}
\\
\subfigure[Menelaus configuration 3]
{\includegraphics[width=0.45\textwidth]
{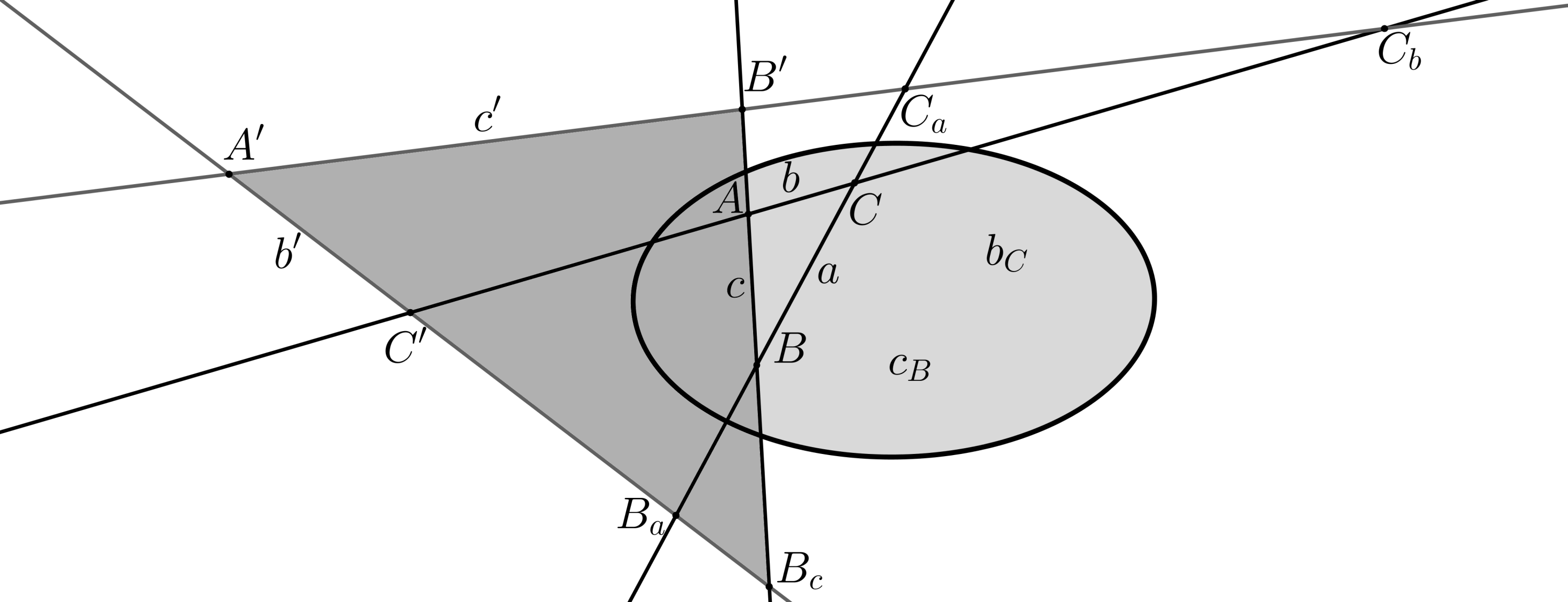}
\label{Fig:projective_right_triangle_Menelaus_03}}
\hfill
\subfigure[Menelaus configuration 4]
{\includegraphics[width=0.45\textwidth]
{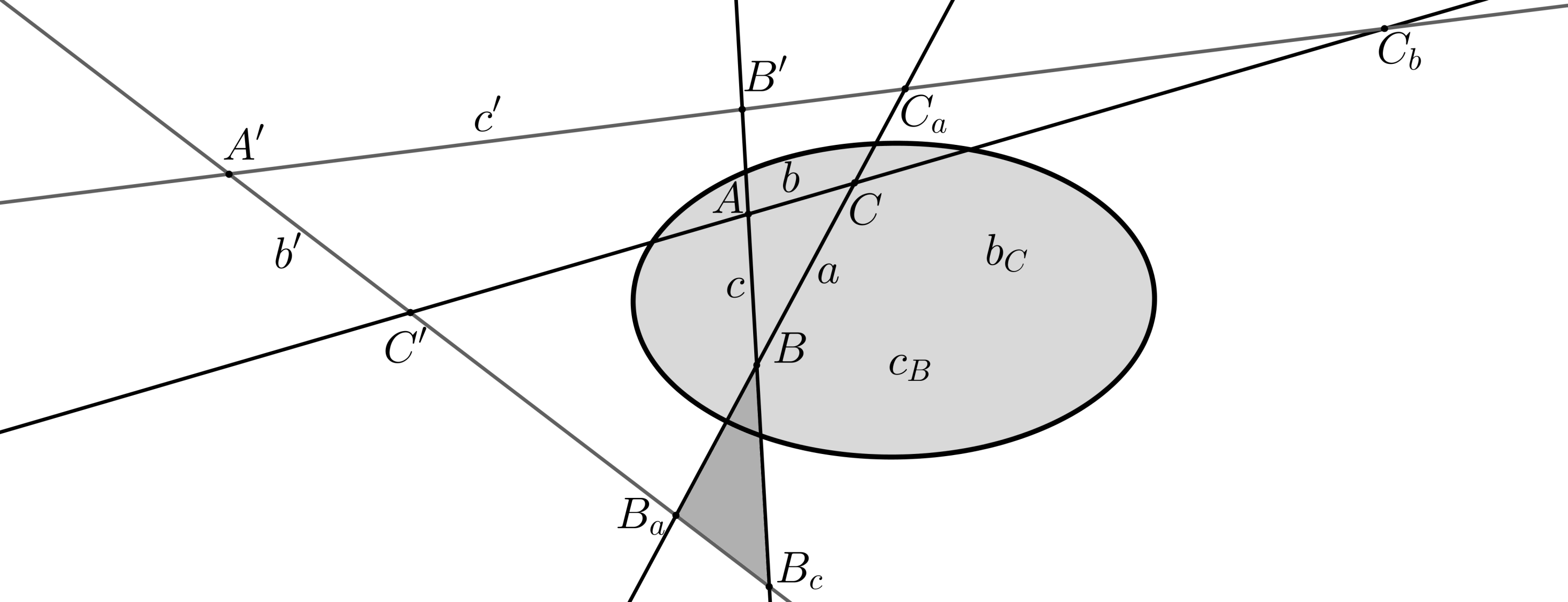}
\label{Fig:projective_right_triangle_Menelaus_04}}
\\
\subfigure[Menelaus configuration 5]
{\includegraphics[width=0.45\textwidth]
{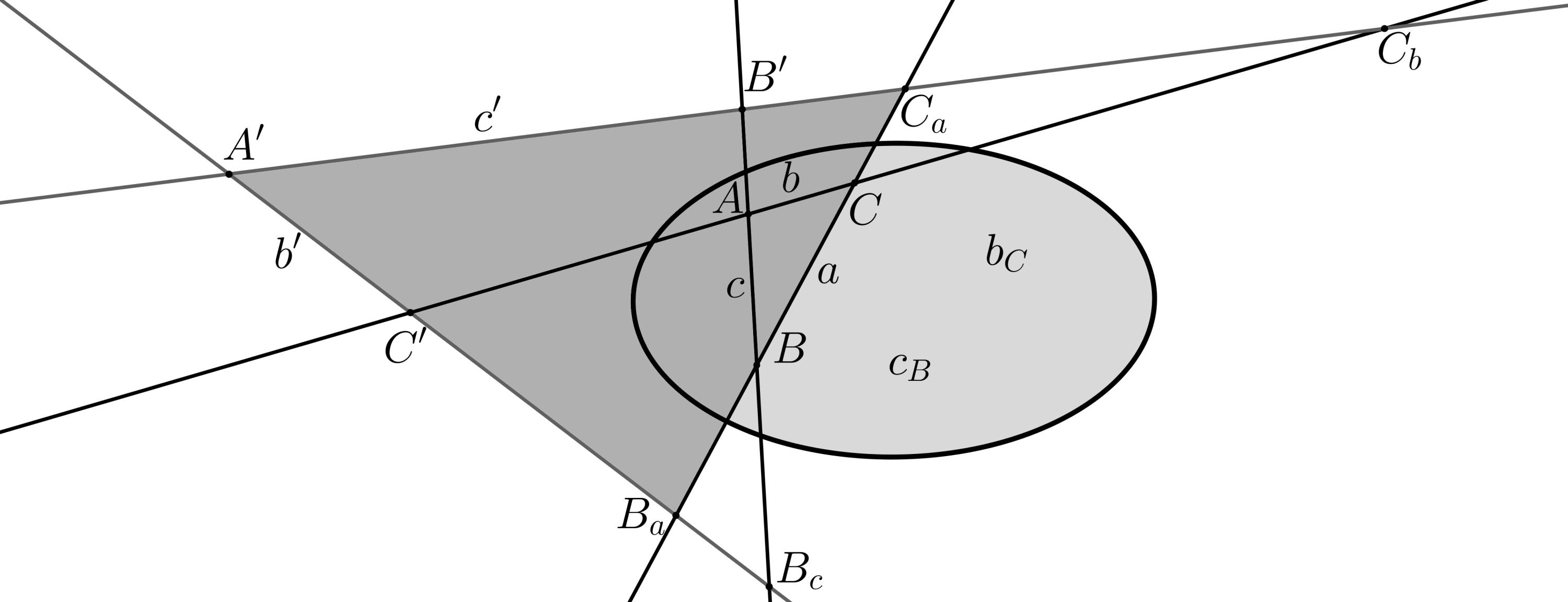}
\label{Fig:projective_right_triangle_Menelaus_05}}
\hfill
\subfigure[Menelaus configuration 6]
{\includegraphics[width=0.45\textwidth]
{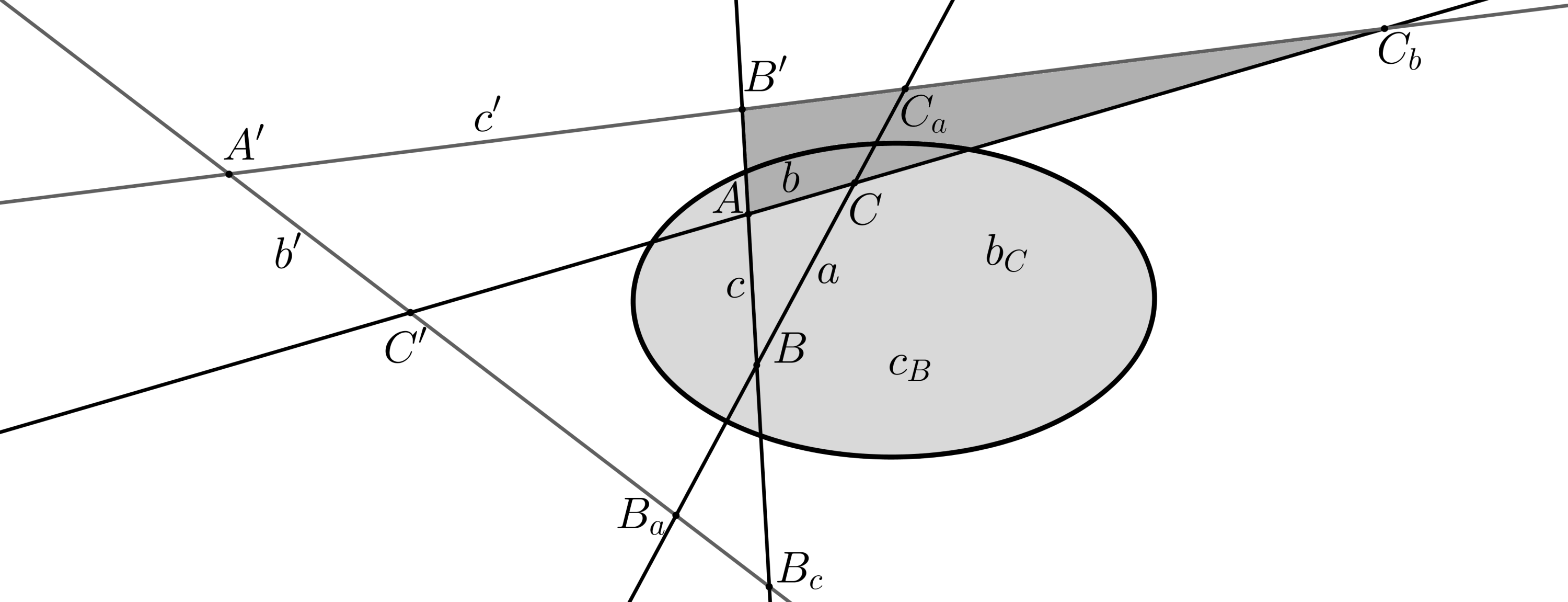}
\label{Fig:projective_right_triangle_Menelaus_06}}

\caption{Menelaus' configurations for a generalized right-angled triangle}
\label{Fig:generalized-right-angled-triangles-Menelaus-Configurations}
\end{sidewaysfigure}

There are $\binom{5}{2}=10$ different Menelaus' configurations that
we can build from the five lines $a,b,c, b', c'$. If we exclude
those relations that are equivalent after renaming the
lines of the figure, there rest the six configurations depicted in
Figure~\ref{Fig:generalized-right-angled-triangles-Menelaus-Configurations}.
\begin{itemize}
\item Taking the triangle $\widetriangle{ac c'}$ and the transversals $b, b'$
(see Figure~\ref{Fig:projective_right_triangle_Menelaus_02}), by 
\hyperref[cor:Menelaus]{Menelaus' Projective Formula}
we have:
\end{itemize}
\vspace{-0.5cm}\begin{align}
\!\!\!\!\!\!\left(C_{a}BCB_{a}\right)\left(BB'AB_{c}\right)\left(B'C_{a}C_{b}A'\right)=1  \iff& \nonumber\\
\!\!\!\!\!\!\iff \left(BC_{a}B_{a}C\right)\left(BA_{c}AB_{c}\right)\left(B'A'_{ c'}B'_{ c'}A'\right)=1\iff &\nonumber \\
\!\!\!\!\!\! \iff \dfrac{1}{\mathbf{S}(a)} \mathbf{S}(c)\dfrac{1}{\mathbf{S}( c')}=1 
\iff & \; \mathbf{S}(c)=\mathbf{S}(a) \mathbf{S}( c')\label{eq:trigo2}\tag{\textbf{T2}}
\end{align}

\begin{itemize}
\item Taking the triangle $\widetriangle{c b' c'}$ and the transversals $a,b$ 
(Figure~\ref{Fig:projective_right_triangle_Menelaus_03}),
we arrive to:
\end{itemize}
\vspace{-0.5cm}\begin{align}
\!\!\!\!\!\!\left(A'B'C_{b}C_{a}\right)\left(B'B_{c}AB\right)\left(B_{c}A'C'B_{a}\right)=1 \!\!\!\iff & \nonumber\\
\!\!\!\!\!\!\iff \left(A'B'B'_{ c'}A'_{ c'}\right)\left(A_{c}B_{c}AB\right)\left(C'_{ b'}A'C'A'_{ b'}\right)=1\!\!\!\iff &\nonumber \\
\!\!\!\!\!\!\iff \mathbf{C}( c')\dfrac{1}{\mathbf{C}(c)}\dfrac{1}{\mathbf{S}( b')}=1 \!\!\iff & \; \mathbf{C}( c')=\mathbf{C}(c) \mathbf{S}( b')\label{eq:trigo3}\tag{\textbf{T3}}
\end{align}

\begin{itemize}
\item Taking the triangle $\widetriangle{ac b'}$ and transversals $b,c'$ (Figure
\ref{Fig:projective_right_triangle_Menelaus_04}) we obtain:
\begin{equation}
\mathbf{T}(c)=\mathbf{T}(a) \mathbf{C}( b').\label{eq:trigo4}\tag{\textbf{T4}}
\end{equation}

\item Taking the triangle $\widetriangle{a b' c'}$ and transversals $b,c$ 
(Figure~\ref{Fig:projective_right_triangle_Menelaus_05}),
we arrive to:
\begin{equation}
\mathbf{C}(a)=\dfrac{1}{\mathbf{T}( b')}\dfrac{1}{\mathbf{T}( c')}.\label{eq:trigo5}\tag{\textbf{T5}}
\end{equation}

\item Finally, taking the triangle $\widetriangle{bc c'}$ and transversals $a,b'$ 
(Figure~\ref{Fig:projective_right_triangle_Menelaus_06}),
we get:
\begin{equation}
\mathbf{T}(c)=\mathbf{S}(b) \mathbf{T}( c').\label{eq:trigo6}\tag{\textbf{T6}}
\end{equation}

\end{itemize}
It can be seen that the projective trigonometric formulae~\eqref{eq:trigo4},
\eqref{eq:trigo5} and~\eqref{eq:trigo6} can be deduced from~\eqref{eq:trigo1},
\eqref{eq:trigo2} and~\eqref{eq:trigo3}.

If we apply Table~\ref{table:non-euclidean-segments} to expressions
\eqref{eq:trigo2} to~\eqref{eq:trigo6}, for each general right-angled
triangle we obtain a collection of squared non-euclidean trigonometric formulae
associated with each figure. In order to decide which is the unsquared
correct formula corresponding to each squared formula, we need to
use the geometric properties of each figure. In the right-angled hyperbolic
pentagon there will be no discussion because all the relevant magnitudes
of the figure are segments. For right-angled hyperbolic triangles
and Lambert quadrilaterals, it is enough to remark that:
\begin{proposition}
\label{prop:hyperbolic-triangles-Lambert-acute-angle}A hyperbolic
right-angled triangle and a Lambert quadrilateral cannot have an obtuse
angle.
\end{proposition}
Proposition~\ref{prop:hyperbolic-triangles-Lambert-acute-angle} is
a simple consequence of the fact that the sum of the angles of a hyperbolic
triangle is lower than $\pi$. For elliptic right-angled triangles,
besides Proposition~\ref{prop:elliptic-triangle-I-obtuse-sides} it
should be necessary also the properties listed in 
Proposition~\ref{prop:elliptic-triangles-II} (whose proof is left to the reader).
\begin{proposition}[elliptic triangles II]
\label{prop:elliptic-triangles-II}The right-angled triangle $\TT$
of Figure~\ref{sph_triangle} verifies:
\begin{itemize}
\item the angles of $\TT$ are equal to the sides of the polar triangle
$\TT'$ and vice versa;
\item one of the angles $\beta,\gamma$ of $\TT$ is obtuse (right)
if and only if its opposite side is also obtuse (right).
\end{itemize}
\end{proposition}

\begin{table}
\centering\renewcommand{\arraystretch}{1.2}
{\scriptsize
\begin{tabular}{clllll}
& \multicolumn{1}{c}{Elliptic} && \multicolumn{3}{c}{Hyperbolic} \tabularnewline
\cline{2-2}\cline{4-6}
&\multicolumn{1}{c}{right-angled} && \multicolumn{1}{c}{right-angled}  & \multicolumn{1}{c}{Lambert}  & \multicolumn{1}{c}{right-angled} \tabularnewline
 & \multicolumn{1}{c}{triangle} && \multicolumn{1}{c}{triangle} & \multicolumn{1}{c}{quadrilateral} & \multicolumn{1}{c}{pentagon}\tabularnewline
\ref{eq:trigo1} & $\cos a=\cos c\cos b$ && $\cosh a=\cosh c\cosh b$ & $\sinh a=\sinh c\cosh b$ & $\cosh a=\sinh c\sinh b$\tabularnewline
\ref{eq:trigo2} & $\sin c=\sin a\sin\gamma$ && $\sinh c=\sinh a\sin\gamma$ & $\cosh c=\cosh a\sin\gamma$ & $\cosh c=\sinh a\sinh\gamma$\tabularnewline
\ref{eq:trigo3} & $\cos\gamma=\cos c\sin\beta$ && $\cos\gamma=\cosh c\sin\beta$ & $\cos\gamma=\sinh c\sinh\beta$ & $\cosh\gamma=\sinh c\sinh\beta$\tabularnewline
\ref{eq:trigo4} & $\tan c=\tan a\cos\beta$ && $\tanh c=\tanh a\cos\beta$ & $\coth c=\coth a\cosh\beta$ & $\coth c=\tanh a\cosh\beta$\tabularnewline
\ref{eq:trigo5} & $\cos a=\cot\beta\cot\gamma$ && $\cosh a=\cot\beta\cot\gamma$ & $\sinh a=\coth\beta\cot\gamma$ & $\cosh a=\coth\beta\coth\gamma$\tabularnewline
\ref{eq:trigo6} & $\tan c=\sin b\tan\gamma$ && $\tanh c=\sinh b\tan\gamma$ & $\coth c=\sinh b\tan\gamma$ & $\coth c=\cosh b\tanh\gamma$\tabularnewline
\end{tabular}
}
\smallskip
\caption{trigonometric relations for generalized right-angled triangles}\label{Table:trigonometric_formulae}
\end{table}

Using Propositions~\ref{prop:elliptic-triangle-I-obtuse-sides},~\ref{prop:hyperbolic-triangles-Lambert-acute-angle}
and~\ref{prop:elliptic-triangles-II}, it can be seen that the (unsquared)
non-euclidean trigonometric translations of formulae $(\ref{eq:trigo2}-\ref{eq:trigo6})$ for each generalized right-angled triangle are
those listed in Table~\ref{Table:trigonometric_formulae} (we include
also the translations of~\eqref{eq:trigo1} for completeness).

\section{Trigonometry of generalized, non right-angled, triangles}
\label{sec:Trigonometry-for-generalized-triangles}

Let us consider again the triangle $\TT=\widetriangle{ABC}$ with sides
$a,b,c$ and its polar triangle $\TT'=\widetriangle{A'B'C'}$ with
sides $ a', b', c'$ as in \S~\ref{sec:triangle-notation}. We assume
again that $\TT$ is in general position with respect to $\Phi$
and $\TT'$, but we will not assume now that $\TT$ has
conjugate lines. Thus, $\TT$ could be an elliptic or hyperbolic
triangle, or it could compose with $\TT'$ any of the generalized
triangles depicted in Figures
\ref{Fig:hyperbolic_generalized_triangles_I}--\ref{Fig:hyperbolic_generalized_triangles_star_II}. 
We will see
how all the trigonometic formulae for these figures can be deduced
from the results of the previous section.

Consider the line $h_a=AA'$, and take the point $H_A=h_a \cdot a$ and its
conjugate point in $a$, which is the point $A_0$ (see  \S~\ref{sec:Desargues}).
For simplifying the notation, we denote now $X,X_a$ to the points $H_A,A_0$
respectively. Take also the point $X'=h_a \cdot a'$ (Figure
\ref{Fig:general_sine_rule}). The line
$h_a$ decomposes the Figure $\TT\cup\TT'$ into two
Menelaus' configurations with some common lines. We will use the following segments (see Figure
\ref{Fig:general_sine_rule}):
\[
p=\overline{AX},\quad a_{1}=\overline{BX},\quad a_{2}=\overline{CX},\quad a'_{1}=\overline{C'X'},\quad a'_{2}=\overline{B'X'}.
\]

\begin{figure}
\centering
\includegraphics[width=0.6\textwidth]{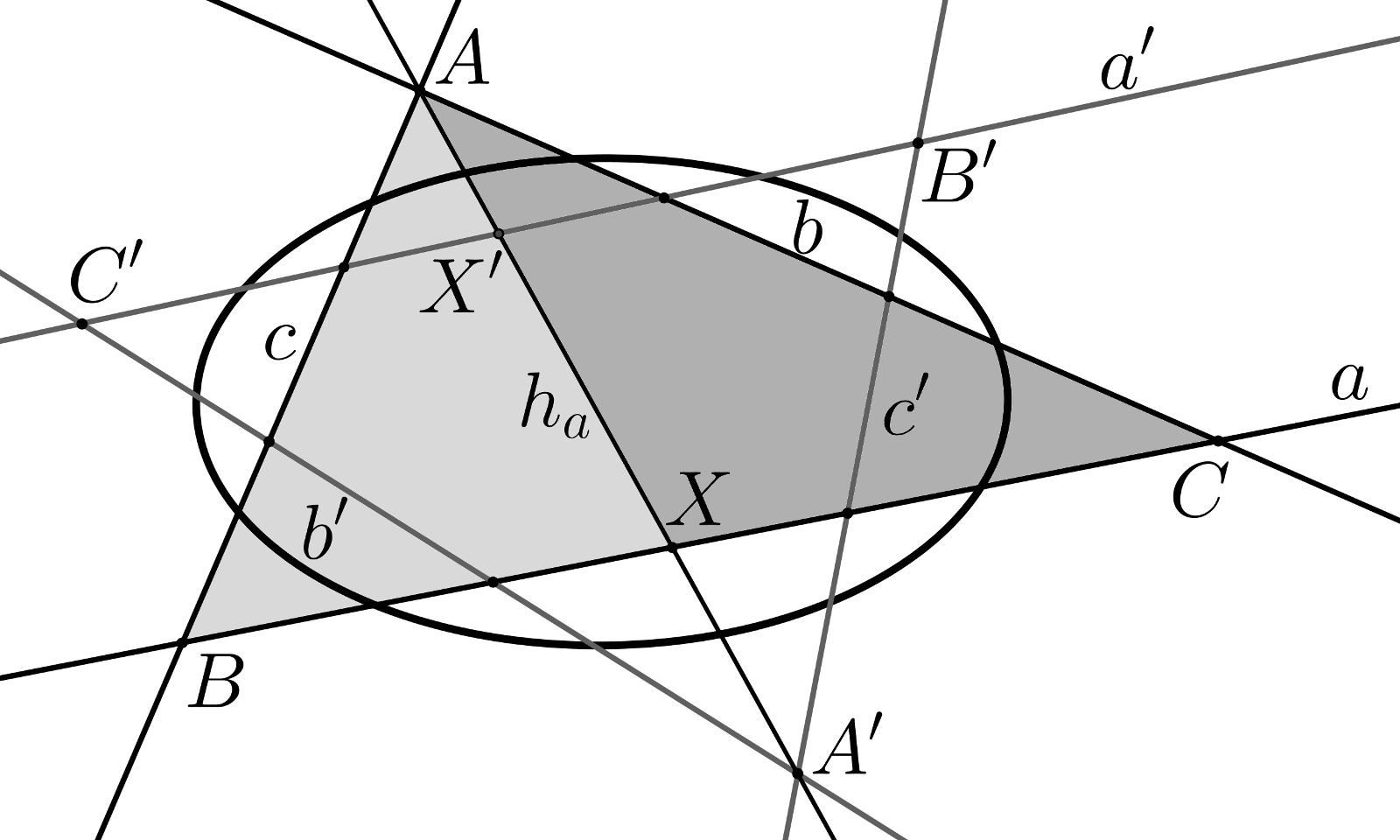}
\caption{projective triangle splitted into two right-angled projective triangles}\label{Fig:general_sine_rule}
\end{figure}

\subsection{The general (squared) law of sines.}\label{sub:The-general-sine-rule}

If we apply the identity~\eqref{eq:trigo2}
to the projective right-angled triangles $\widetriangle{AXB}$ and $\widetriangle{AXC}$
we obtain 
\[
\mathbf{S}(p)=\mathbf{S}(c) \mathbf{S}( b')\quad\text{and}\quad \mathbf{S}(p)=\mathbf{S}(b) \mathbf{S}( c'),
\]
respectively, and so the ratios $\mathbf{S}(b)/\mathbf{S}( b')$ and $\mathbf{S}(c)/\mathbf{S}( c')$
must coincide. Repeating the process, they must coincide also with
the ratio $\mathbf{S}(a)/\mathbf{S}( a')$. Therefore, the following identity holds
\begin{equation}
\dfrac{\mathbf{S}(a)}{\mathbf{S}( a')}=\dfrac{\mathbf{S}(b)}{\mathbf{S}( b')}=\dfrac{\mathbf{S}(c)}{\mathbf{S}( c')}.\label{eq:general-law-of-sines}
\end{equation}
We will call to identity~\eqref{eq:general-law-of-sines} the \emph{general
law of sines} (compare \cite[Theorem 2.6.20]{Buser} and \cite[Theorem 22.6]{Richter-Gebert}).

If we apply Table~\ref{table:non-euclidean-segments} to~\eqref{eq:general-law-of-sines}
for each generalized triangle, we obtain the law of sines associated
to each figure. For example, for an elliptic triangle 
(Figure~\ref{Fig:general_spherical_triangle}), the unsquared
translation of~\eqref{eq:general-law-of-sines} is 
\[
\dfrac{\sin a}{\sin\alpha}=\dfrac{\sin b}{\sin\beta}=\dfrac{\sin c}{\sin\gamma},
\]
while for a hyperbolic triangle (Figure~\ref{Fig:general_hyperbolic_triangle})
we have
\[
\dfrac{\sinh a}{\sin\alpha}=\dfrac{\sinh b}{\sin\beta}=\dfrac{\sinh c}{\sin\gamma}.
\]
For example, for the hyperbolic pentagon with four right angles 
(Figure~\ref{Fig:general_hyperbolic_pentagon})
we obtain
\[
\dfrac{\sinh a}{\sin\alpha}=\dfrac{\cosh b}{\sinh\beta}=\dfrac{\cosh c}{\sinh\gamma}.
\]

\subsection{The general (squared) law of cosines}

In order to obtain a projective trigonometric formula 
similar to the cosine rules of generalized triangles we need the
following lemma.
\begin{lemma}
\label{lem:cross-ratio-relations}The following relations hold:
\begin{align}
 & (BXCC_{a})^{2}=\dfrac{\mathbf{T}(a)}{\mathbf{T}(a_{2})},\label{eq:BCXXa-TBX-TCX}\\
 & \mathbf{C}(a_{1})=\left(\sqrt{\mathbf{C}(a)\mathbf{C}(a_{2})}+\sqrt{\mathbf{S}(a)\mathbf{S}(a_{2})}\right)^{2}.\label{eq:cosine-sum}
\end{align}
\end{lemma}
\begin{proof}
The proof relies in the properties~\eqref{eq:CR1},~\eqref{eq:CR2}
and~\eqref{eq:CR3} of cross ratios.

By applying~\eqref{eq:CR3} twice
and~\eqref{eq:CR1} once we have:
\[
(BXCC_{a})=(X_{a}XCC_{a}) (BX_{a}CC_{a})=\dfrac{1}{(XX_{a}CC_{a})} (B_{a}X_{a}CC_{a}) (BB_{a}CC_{a}).
\]
As cross ratios are invariant under conjugacy, it is
\[
(B_{a}X_{a}CC_{a})=(BXC_{a}C)=\dfrac{1}{(BXCC_{a})},
\]
and so we obtain
\[
(BXCC_{a})=\dfrac{1}{(BXCC_{a})} \dfrac{(BB_{a}CC_{a})}{(XX_{a}CC_{a})}\Rightarrow(BXCC_{a})^{2}=\dfrac{\mathbf{T}(a)}{\mathbf{T}(a_{2})}.
\]
By applying~\eqref{eq:CR3} again, we have:
\begin{align*}
\mathbf{C}(a_{1}) & =(BXX_{a}B_{a})=(BXC_{a}B_{a}) (BXX_{a}C_{a})=\\
 & =(CXC_{a}B_{a})(BCC_{a}B_{a}) (CXX_{a}C_{a})(BCX_{a}C_{a})=\\
 & =(CXC_{a}B_{a})  \mathbf{C}(a)  \mathbf{C}(a_{2}) (BCX_{a}C_{a}).
\end{align*}
Therefore
\[
(CXC_{a}B_{a})=(C_{a}X_{a}CB)=(BCX_{a}C_{a}).
\]
This implies:
\[
\mathbf{C}(a_{1})=\mathbf{C}(a)  \mathbf{C}(a_{2}) (BCX_{a}C_{a})^{2}=\mathbf{C}(a)  \mathbf{C}(a_{1}) [1-(BX_{a}CC_{a})]^{2}.
\]
Using \textbf{\eqref{eq:BCXXa-TBX-TCX}}, we obtain
\[
(BX_{a}CC_{a})=(XX_{a}CC_{a})(BXCC_{a})=-\mathbf{T}(a_{2})\sqrt{\dfrac{\mathbf{T}(a)}{\mathbf{T}(a_{2})}}=-\sqrt{\mathbf{T}(a)\mathbf{T}(a_{2})}.
\]
Then,
\begin{align*}
\mathbf{C}(a_{1}) & =\mathbf{C}(a)  \mathbf{C}(a_{2}) [1+\sqrt{\mathbf{T}(a)\mathbf{T}(a_{2})}]^{2}=\\
 & =\left(\sqrt{\mathbf{C}(a)\mathbf{C}(a_{2})}+\sqrt{\mathbf{S}(a)\mathbf{S}(a_{2})}\right)^{2}.
\end{align*}

\end{proof}

From the proof of Lemma~\ref{lem:cross-ratio-relations} it follows
that~\eqref{eq:BCXXa-TBX-TCX} and~\eqref{eq:cosine-sum} 
are valid\footnote{Formula~\eqref{eq:cosine-sum} is a projective version of the
trigonometric formulae for the angle sum $\cos (\alpha\pm \beta), \sin
(\alpha\pm \beta), \cosh (\alpha\pm \beta)$ and $\sinh (\alpha\pm \beta)$.}
for any three points $B,C,X$ lying on a line $a$ not tangent to
$\Phi$, independently of the figure $\TT\cup\TT'$.

If we apply~\eqref{eq:trigo1} to the projective right-angled triangles
$\widetriangle{ABX}$ and $\widetriangle{ACX}$, we obtain
\[
\mathbf{C}(c)=\mathbf{C}(p)  \mathbf{C}(a_{1})\qquad\text{and}\qquad \mathbf{C}(b)=\mathbf{C}(p)  \mathbf{C}(a_{2})
\]
respectively. It can be deduced
\[
\mathbf{C}(c)=\dfrac{\mathbf{C}(b)}{\mathbf{C}(a_{2})}  \mathbf{C}(a_{1}),
\]
and by~\eqref{eq:cosine-sum}, we arrive to
\begin{align*}
\mathbf{C}(c) & =\dfrac{\mathbf{C}(b)}{\mathbf{C}(a_{2})} \left(\sqrt{\mathbf{C}(a)\mathbf{C}(a_{2})}+\sqrt{\mathbf{S}(a)\mathbf{S}(a_{2})}\right)^{2}=\\
 & =\mathbf{C}(b) \left(\sqrt{\mathbf{C}(a)}+\sqrt{\mathbf{S}(a)\mathbf{T}(a_{2})}\right)^{2}.
\end{align*}
If we apply {\footnotesize~\eqref{eq:trigo4}} to the triangle $\widetriangle{ACX}$,
we obtain that $\mathbf{T}(a_{2})=\mathbf{T}(b)  \mathbf{C}( c')$, and so we have
\begin{equation}
\mathbf{C}(c)=\left(\sqrt{\mathbf{C}(a)\mathbf{C}(b)}+\sqrt{\mathbf{S}(a)\mathbf{S}(b)\mathbf{C}( c')}\right)^{2}.\label{eq:general-law-of-cosines}
\end{equation}

We will call to identity~\eqref{eq:general-law-of-cosines} the \emph{general
law of cosines} (compare \cite[Theorem 2.6.20]{Buser}).

For a given generalized triangle, there is a great difference between
the general law of sines~\eqref{eq:general-law-of-sines} and the
general law of cosines~\eqref{eq:general-law-of-cosines} when we
want to obtain their unsquared non-euclidean trigonometric translations. The unsquared
trigonometric translations of~\eqref{eq:general-law-of-sines} depend
on a simple choice of signs that can be done in an straightforward
way. On the other hand, the unsquared trigonometric
translations of~\eqref{eq:general-law-of-cosines} depend on a multiple
choice of signs due to the presence of the two square roots and the
square power in the right-hand side of~\eqref{eq:general-law-of-cosines}.

Every non-euclidean generalized triangle $T$ have six laws of cosines, one based
in
each side of $T$ and another one based in each non-right angle of
$T$. If $\TT$ and $\TT'$ are the two projective triangles,
polar to each other, that produce the non-euclidean polygon $T$,
these six cosine laws are the unsquared trigonometric translations
of~\eqref{eq:general-law-of-cosines}, when the segment $c$ varies
along the three sides of $\TT$ and the three sides of $\TT'$.
When $T$ is a hyperbolic generalized triangle, in \cite{Buser} it
is shown how to obtain all the cosine rules of $T$ by taking an orientation
on the elements (sides and non-right angles) of $T$ and assigning
some $\pm1$ coefficients to these elements. In our context
we would like to have another
\emph{projective general law of cosines} different from~\eqref{eq:general-law-of-cosines}
whose trigonometric translations are straightforward
and give the cosine rules for each (hyperbolic and elliptic) generalized
triangle.  We'll explore the existence of this formula in \S
\ref{sec:Cosine-Rule}.
In order to obtain
it, it would be desirable to express in projective terms
the actual (unsquared) trigonometric functions associated to the measurements
of segments and angles.

Nevertheless, every generalized triangle is the result of pasting
together two generalized right-angled triangles along a side (Figures
\ref{Fig:hyperbolic_generalized_triangles_I}--\ref{Fig:hyperbolic_generalized_triangles_star_II}).
This allows us to deduce the trigonometry of all generalized triangles
from the trigonometry of generalized right-angled triangles. This suffices for
presenting the whole non-euclidean trigonometry as a corollary of
\hyperref[thm:Menelaus-affine]{Menelaus' Theorem}.

\chapter{Carnot's Theorem and... Carnot's
Theorem?}\label{sec:Carnot's-theorem-and}

There are many geometric results which are known as \emph{Carnot's Theorem}.
One of them is a generalization of 
\hyperref[thm:Menelaus-affine]{Menelaus' Theorem}
(see Figure~\ref{Carnot_theorem}):
\begin{theorem}[Carnot's Theorem on affine triangles]
\label{thm:Carnot-affine}$ $

\noindent Let  $\wt{XYZ}$ be a triangle
in the affine plane. The six points $X_{1},X_{2}\in YZ$, $Y_{1},Y_{2}\in ZX$,
$Z_{1},Z_{2}\in XY$ on the sides of $\TT$ lie on a conic if and
only if
\begin{equation*}
\dfrac{\brs{X_{1}Y}}{\brs{X_{1}Z}}\cdot\dfrac{\brs{X_{2}Y}}{\brs{X_{2}Z}}\cdot\dfrac{\brs{Y_{1}Z}}{\brs{Y_{1}X}}\cdot\dfrac{\brs{Y_{2}Z}}{\brs{Y_{2}X}}\cdot\dfrac{\brs{Z_{1}X}}{\brs{Z_{1}Y}}\cdot\dfrac{\brs{Z_{2}X}}{\brs{Z_{2}Y}}=1\,.
\end{equation*}
\end{theorem}
Another ``Carnot's Theorem'' is Theorem~\ref{thm:Carnot_theorem_euclidean} below.

Let $T$ be a triangle
in euclidean plane with vertices $A,B,C$ and sides $a=BC,b=CA,c=AB$ as usual.
Let $A^{*},B^{*},C^{*}$
be three points on the lines $a,b,c$ respectively, and let
$a^{*},b^{*},c^{*}$ be the perpendicular lines to $a,b,c$ through
the points $A^{*},B^{*},C^{*}$ respectively. Let $a_1,a_2,b_1,b_2,c_1,c_2$ 
denote the euclidean lengths of the segments
$\ov{BA^*},\ov{CA^*},\ov{CB^*},\ov{AB^*},\ov{AC^*},\ov{BC^*}$, respectively 
(see Figure~\ref{Fig:Carnot_theorem_euclidean}).
\begin{theorem}[Carnot's Theorem on euclidean triangles]
\label{thm:Carnot_theorem_euclidean}$ $

\noindent The lines $a^{*},b^{*},c^{*}$
are concurrent if and only if
\begin{equation*}
a_1^2+b_1^2+c_1^2=a_2^2+b_2^2+c_2^2
\end{equation*}
\end{theorem}
This theorem can be proved by simple application of Pythagoras' Theorem. For a proof of Theorem~\ref{thm:Carnot-affine}, see 
See \cite[vol. II, p. 90]{V - Y}. We say that the points $A^*,B^*,C^*$ such that
$a^*,b^*,c^*$ are concurrent are \emph{Carnot points} of $T$.

A version of Theorem~\ref{thm:Carnot_theorem_euclidean} for hyperbolic triangles
is stated in \cite{Demirel-Soyturk}. Another versions for different generalized
triangles can be constructed using the non-euclidean versions of Pythagoras'
theorem (see \S\ref{sub:Non-euclidean-Pythagoras'-Theorem}), or its projective
version~\eqref{eq:trigo1}. What we will see in the following is that\emph{ the
non-euclidean shadows of (the projective version of)
Theorem~\ref{thm:Carnot-affine} are the non-euclidean
versions of Theorem~\ref{thm:Carnot_theorem_euclidean}
for generalized triangles}.

\begin{figure}
\centering
\subfigure[Carnot's theorem on projective triangles]
{\includegraphics[width=0.75\textwidth]{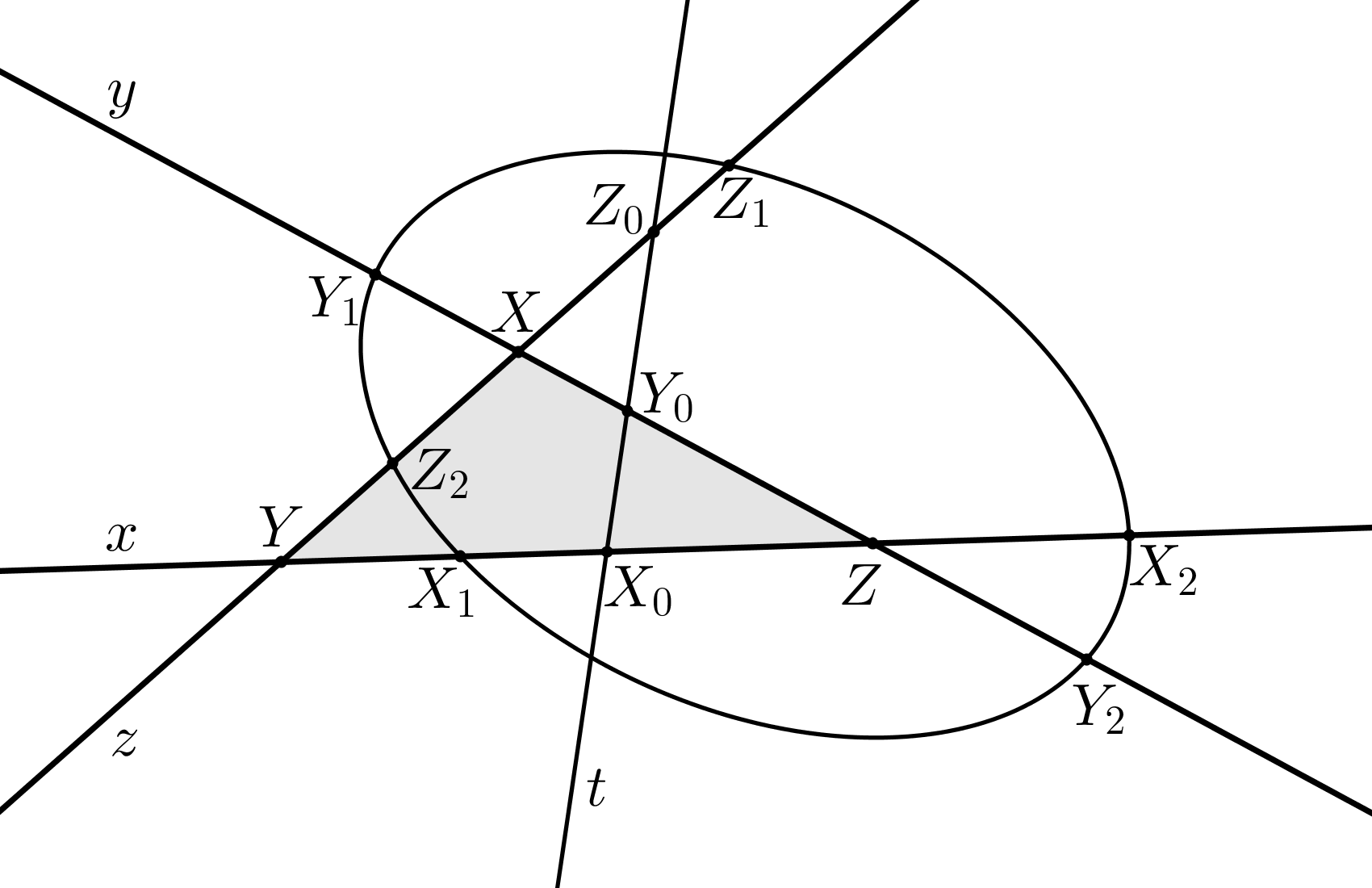}
\label{Carnot_theorem}}
\\
\subfigure[Carnot's theorem on euclidean triangles]
{\includegraphics[width=0.45\textwidth]{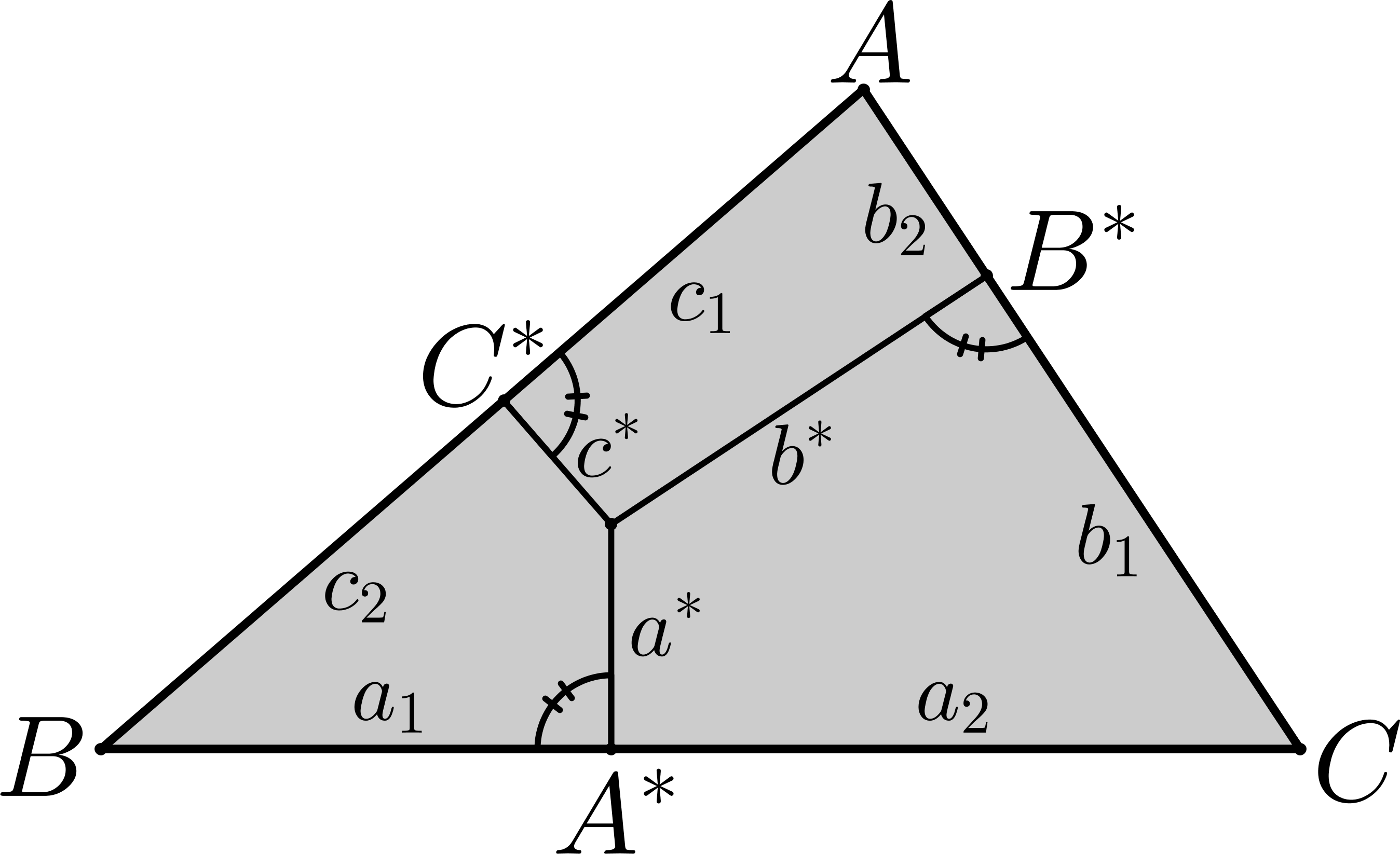}
\label{Fig:Carnot_theorem_euclidean}}

\caption{Carnot's theorems}
\label{Carnot_theorems} 
\end{figure}

As we did with \hyperref[thm:Menelaus-affine]{Menelaus' Theorem} 
in \S\ref{sec:Menelaus-Theorem}, 
we can make a projective
interpretation of Theorem~\ref{thm:Carnot-affine} by introducing the line at
infinity
as part of the figure. We
will use the same notation as in \S\ref{sec:Menelaus-Theorem}. Let
$x,y,z,t$ be four projective lines such that no three of
them are concurrent. Consider the projective triangle $\TT=\widetriangle{xyz}$
and take
the points
\begin{align*}
\begin{split}
X=y\cdot z,\quad Y=z\cdot x, \quad Z=x\cdot y,\\
X_{0}=x\cdot t, \quad Y_{0}=y\cdot t, \quad Z_{0}=z\cdot t.
\end{split}
\end{align*}
Then, Theorem~\ref{thm:Carnot-affine} turns into
\begin{theorem}[Carnot's Theorem on projective triangles]
\label{thm:Carnot-projective}$ $

\noindent Six points $X_{1},X_{2}\in x$,
$Y_{1},Y_{2}\in y$, $Z_{1},Z_{2}\in z$ on the sides of $\TT$ lie
on a conic if and only if
\[
\left(XYZ_{0}Z_{1}\right)\left(XYZ_{0}Z_{2}\right)\left(YZX_{0}X_{1}\right)
\left(YZX_{0}X_{2}\right)\left(ZXY_{0}Y_{1}\right)\left(ZXY_{0}Y_{2}\right)=1\,.
\]
\end{theorem}

\begin{exercise}
Prove Lemma~\ref{lem:midpoints-collinear} and Theorem
\ref{thm:midpoints-quadrilateral} based in Menelaus' and Carnot's Theorems
instead of \hyperref[thm:Pascal]{Pascal's Theorem}.
\end{exercise}

We will use the following corollary
instead of Theorem~\ref{thm:Carnot-projective}
itself. 
\begin{corollary}
\label{thm:Carnot-projective-Corollary}Let $\TT=\widetriangle{XYZ}$
be a triangle in the projective plane, and let $X_{1},X_{2}\in YZ$,
$Y_{1},Y_{2}\in ZX$,
$Z_{1},Z_{2}\in XY$ be six points on the sides of $\TT$
different from $X,Y,Z$ lying on a conic.
Let $X_{0}\in YZ$, $Y_{0}\in ZX$, $Z_{0}\in XY$ be another three
points on the sides of $\TT$. If the points $X_{0},Y_{0},Z_{0}$
are collinear then 
\begin{equation}
\!\!\!\!\!\!\left(XYZ_{0}Z_{1}\right)\!
\left(XYZ_{0}Z_{2}\right)\!
\left(YZX_{0}X_{1}\right)\!
\left(YZX_{0}X_{2}\right)\!
\left(ZXY_{0}Y_{1}\right)\!
\left(ZXY_{0}Y_{2}\right)=1\,.\label{eq:Carnot_projective}
\end{equation}

\end{corollary}

The converse of Corollary~\ref{thm:Carnot-projective-Corollary} is
not true, but we have the following result instead.
\begin{proposition}
\label{prop:Carnot-converse}Let $\TT=\widetriangle{XYZ}$ be a
projective triangle, and let $X_{1},X_{2}\in YZ$, $Y_{1},Y_{2}\in ZX$,
$Z_{1},Z_{2}\in XY$ be six points on the sides of $\TT$
different from $X,Y,Z$ lying on a conic.
Let $X_{0}\in YZ$, $Y_{0}\in ZX$, $Z_{0}\in XY$ be another three
points on the sides of $\TT$, and consider also the point $Z'_{0}=XY\cdot
X_{0}Y_{0}$.
If the identity~\eqref{eq:Carnot_projective} holds, then it is $Z_{0}=Z'_{0}$
(and so $X_{0},Y_{0},Z_{0}$ are collinear) or $Z_{0}$ is the harmonic
conjugate of $Z'_{0}$ with respect to $X,Y$.\end{proposition}
\begin{proof}
Assume that~\eqref{eq:Carnot_projective} is true. By Theorem
\ref{thm:Carnot-projective}
we have
\[
\left(XYZ'_{0}Z_{1}\right)\left(XYZ'_{0}Z_{2}\right)\left(YZX_{0}X_{1}\right)
\left(YZX_{0}X_{2}\right)\left(ZXY_{0}Y_{1}\right)\left(ZXY_{0}Y_{2}\right)=1\;.
\]
This relation, together with~\eqref{eq:Carnot_projective} imply that
\[
\left(XYZ'_{0}Z_{1}\right)\left(XYZ'_{0}Z_{2}\right)=\left(XYZ_{0}Z_{1}
\right)\left(XYZ_{0}Z_{2}\right)
\]
or equivalently by~\eqref{eq:cross_ratio_identities}
\begin{align*}
\left(XYZ'_{0}Z_{1}\right)\left(XYZ_{1}Z_{0}\right)&=
\left(XYZ_{0}Z_{2}\right)\left(XYZ_{2}Z'_{0}\right)\iff\\
\iff\left(XYZ'_{0}Z_{0}\right)&=\left(XYZ_{0}Z'_{0}\right)\iff
\left(XYZ'_{0}Z_{0}\right)=\pm1\,.
\end{align*}
If $\left(XYZ'_{0}Z_{0}\right)=+1$, because $X,Y$ are different points, it must
be $Z_{0}=Z'_{0}$,
and if $\left(XYZ'_{0}Z_{0}\right)=-1$, then $Z_{0}$ is the harmonic
conjugate of $Z'_{0}$ with respect to $X,Y$. 
\end{proof}

When dealing with our projective triangle $\TT$ and its polar triangle $\TT'$
with respect to $\Phi$ as in \S\ref{sec:triangle-notation},
we have the following:
\begin{theorem}
\label{thm:six-points-on-a-conic}
The points $A_b,A_c,B_a,B_c,C_a,C_b$
lie on a conic.\end{theorem}
\begin{proof}
The lines $a c' b a' c b'$ are, in this order, the sides
of the hexagon $\SC{H}$ with consecutive vertices $B_a,C_a,C_b,A_b,A_c,B_c$.
By \hyperref[thm:Chasles-Polar-triangle]{Chasles' Theorem}
the lines $AA'$, $BB'$, $CC'$ are concurrent,
and so, by \hyperref[thm:Desargues]{Desargues' Theorem}, the points $a\cdot a'$, $b\cdot b'$
and $c\cdot c'$ are collinear. Thus, the opposite sides of the
hexagon $\SC{H}$ intersect at three collinear points. By the converse
of \hyperref[thm:Pascal]{Pascal's Theorem}, 
the six vertices of the hexagon $\SC{H}$ must lie
on a conic.
\end{proof}

It can be deduced from the previous proof that this theorem holds for any pair
of perspective triangles\footnote{In fact, any pair of perspective triangles are
polar to each other with respect to a conic (see \cite[p. 65]{Cox Proj}).}.
When the triangle $\TT$ has two conjugate sides, the
conic given by Theorem~\ref{thm:six-points-on-a-conic} is a degenerate
conic (see Figure~\ref{Fig:generalized-right-angled-triangles}). Theorem
\ref{thm:six-points-on-a-conic}, together with
Corollary~\ref{thm:Carnot-projective-Corollary}, allows us to obtain
geometric information about a non-euclidean generalized triangle when
the collinear points $X_{0},Y_{0},Z_{0}$ of Corollary
\ref{thm:Carnot-projective-Corollary} are suitably
chosen.

For a similar notation to that of
Corollary~\ref{thm:Carnot-projective-Corollary} and Proposition
\ref{prop:Carnot-converse}, we will use now the names $A_1$, $A_2$, $B_1$,
$B_2$, $C_1$, $C_2$ for the points $B_a$, $C_a$, $C_b$, $A_b$, $A_c$, $B_c$,
respectively.
Let $A^{*},B^{*},C^{*}$ be three points lying on the sides
$a,b,c$ of the triangle $\TT$ respectively, and consider
the lines
\[
a^{*}=A^{*}A',\qquad b^{*}=B^{*}B',\qquad c^{*}=C^{*}C'\;,
\]
and the segments
\[
\begin{split}
a_{1}=\ov{BA^*},\quad  b_{1}=\ov{CB^*}, \quad c_{1}=\ov{AC^*},\\
a_{2}=\ov{CA^*}, \quad b_{2}=\ov{AB^*}, \quad c_{2}=\ov{BC^*}.
\end{split}
\]

\begin{theorem}
\label{thm:Carnot-projective-cosines}If the lines $a^{*},b^{*},c^{*}$ are
concurrent, the following relation holds:
\begin{equation}\label{eq:Carnot-projective-trigonometric-formula}
\mathbf{C}(a_1)\,\mathbf{C}(b_1)\,\mathbf{C}(c_1)=\mathbf{C}(a_2)\,\mathbf{C}
(b_2)\,\mathbf{C}(c_2)\,.
\end{equation}
\end{theorem}
Before proving Theorem~\ref{thm:Carnot-projective-cosines}, we will
prove the following lemma:
\begin{lemma}
\label{lem:cosines}Let $p$ be a projective line not tangent to $\Phi$, and let
$X,Y,Z$
be three points of $p$ not lying on $\Phi$. Then,
\[
\left(XYZ_{p}X_{p}\right)\left(XYZ_{p}Y_{p}\right)=\dfrac{\mathbf{C}(\ov{XZ})}{
\mathbf{C}(\ov{YZ})}\,.
\]
\end{lemma}
\begin{proof}
\[
\left(ZXZ_{p}Y_{p}\right)
\overset{(\ref{eq:invariance_cross-ratio-I})}{=}
\left(Z_{p}X_{p}ZY\right)
\overset{(\ref{eq:cross_ratio_identities})}{=}
\left(ZYZ_{p}X_{p}\right)\,.
\]
So we have
\[
\left(XYZ_{p}X_{p}\right)\left(XYZ_{p}Y_{p}\right)
\overset{\eqref{eq:cross_ratio_identities}}{=}\dfrac{\left(XYZ_{p}X_{p}\right)}
{\left(YXZ_{p}Y_{p}\right)}\overset{\eqref{eq:CR3}}{=}
\dfrac{\cancel{\left(ZYZ_{p}X_{p}\right)}
\left(XZZ_{p}X_{p}\right)}{\cancel{\left(ZXZ_{p}Y_{p}\right)}\left(YZZ_{p}Y_{p}
\right)}=\dfrac{\mathbf{C}(\ov{YZ})}{\mathbf{C}(\ov{XZ})}\,.
\]

\end{proof}

\begin{proof}[Proof of Theorem~\ref{thm:Carnot-projective-cosines}]
Let $H^*$ be the intersection point of the lines $a^{*},b^{*},c^{*}$,
and let $h^*$ be the polar line of $H^*$. Consider also the three points:
$$A_{0}^*=a\cdot h^*,\quad B_{0}^*=b\cdot h^*, \quad C_{0}^*=c\cdot h^*\,.$$
By Theorem~\ref{thm:six-points-on-a-conic} and Corollary
\ref{thm:Carnot-projective-Corollary} 
we have that
\begin{equation}
\!\!\!\!\!\left(ABC_{0}^*C_{1}\right)\!\left(ABC_{0}^*C_{2}\right)\!\left(BCA_{0
}^*A_{1}\right)\!\left(BCA_{0}^*A_{2}\right)\!\left(CAB_{0}^*B_{1}
\right)\!\left(CAB_{0}^*B_{2}\right)=1.\label{eq:Carnot_projective-1}
\end{equation}

Because $a^{*}=A^{*}A'=A'H^*$, we have that
$$\rho(a^{*})=\rho(A')\cdot\rho(H^*)=a\cdot h^*=A_{0}^*\,.$$
Moreover, as $A^{*}\in a^{*}$ it is $\rho(A^{*})\ni\rho(a^{*})$
and this implies that $A_0^*$ and $A^*$ are conjugate to each other. In the same
way, we have that $B_0^*,C_0^*$ are the conjugate points of $B^*,C^*$ in $b,c$
respectively.
Finally, by Lemma~\ref{lem:cosines}
\begin{align*}
\left(ABC_{0}^*C_{1}\right)\left(ABC_{0}^*C_{2}\right)&=\dfrac{\mathbf{C}(\ov{
AC^{*}})}{\mathbf{C}(\ov{BC^{*}})}=\dfrac{\mathbf{C}(c_1)}{\mathbf{C}(c_2)}\,,\\
\left(BCA_{0}^*A_{1}\right)\left(BCA_{0}^*A_{2}\right)&=\dfrac{\mathbf{C}(\ov{
BA^{*}})}{\mathbf{C}(\ov{CA^{*}})}=\dfrac{\mathbf{C}(a_1)}{\mathbf{C}(a_2)}\,,\\
\left(CAB_{0}^*B_{1}\right)\left(CAB_{0}^*B_{2}\right)&=\dfrac{\mathbf{C}(\ov{
CB^{*}})}{\mathbf{C}(\ov{AB^{*}})}=\dfrac{\mathbf{C}(b_1)}{\mathbf{C}(b_2)}\,.
\end{align*}
This completes the proof.
\end{proof}

Due to Proposition~\ref{prop:Carnot-converse}, identity
\eqref{eq:Carnot-projective-trigonometric-formula} is a necessary but not
sufficient condition for the concurrency of the lines $a^*,b^*,c^*$ in Theorem
\ref{thm:Carnot-projective-cosines}. In order to obtain a partial converse of
Theorem~\ref{thm:Carnot-projective-cosines} we need to introduce some new
notation. For any pair of points $X,Y$ in the projective plane such the line
$z=XY$ is not tangent to $\Phi$, the \emph{Carnot involution on $z$ with respect
to} $X,Y$ is the composition $\zeta_{XY}=\rho_z  \tau_{XY} \rho_z$, where
$\rho_z$ and $\tau_{XY}$ are the conjugacy with
respect to $\Phi$ and the harmonic conjugacy with respect to $X,Y$ respectively
as introduced in \S\ref{sec:Basics-of-projective}. For any other point $W$
on $z$, $\zeta_{XY}(W)$ is the \emph{Carnot conjugate} of $W$ with respect to
$X,Y$.
With the same notation as that introduced before Theorem
\ref{thm:Carnot-projective-cosines}, let denote now by $H^*$ the point $b^*\cdot
c^*$ and consider also the line $d^*=A'H^*$ and the point $D^*=a\cdot d^*$.
\begin{theorem}\label{thm:converse-Carnot-projective-cosines}
If the identity~\eqref{eq:Carnot-projective-trigonometric-formula} holds, then
it is $A^*=D^*$ or $A^*$ is the Carnot conjugate of $D^*$ with respect to $B,C$.
\end{theorem}
\begin{proof}
Let $A_0,B_0,C_0$ be the conjugate points of $A^*,B^*,C^*$ in $a,b,c$
respectively. Let $D_0$ be the conjugate point of $D^*$ in $a$. By the proof of
Theorem~\ref{thm:Carnot-projective-cosines}, it is $D_0=a\cdot B_0C_0$. By Lemma
\ref{lem:cosines}, if~\eqref{eq:Carnot-projective-trigonometric-formula} holds,
then~\eqref{eq:Carnot_projective-1} also holds. By Proposition
\ref{prop:Carnot-converse}, if~\eqref{eq:Carnot_projective-1} holds, it is
$A_0=D_0$ or $A_0$ is the harmonic conjugate $\tau_{BC}(D_0)$ of $D_0$ with
respect to $B,C$, and this implies that
$$A^*=\rho_a(A_0)=\rho_a(D_0)=D^*$$
or
$$A^*=\rho_a(A_0)=\rho_a(\tau_{BC}(D_0))=\rho_a(\tau_{BC}(\rho_a
(D^*)))=\zeta_{BC}(D^*)\,.$$
\end{proof}
\begin{figure}
\centering
\subfigure[]{\includegraphics[width=0.62\textwidth]
{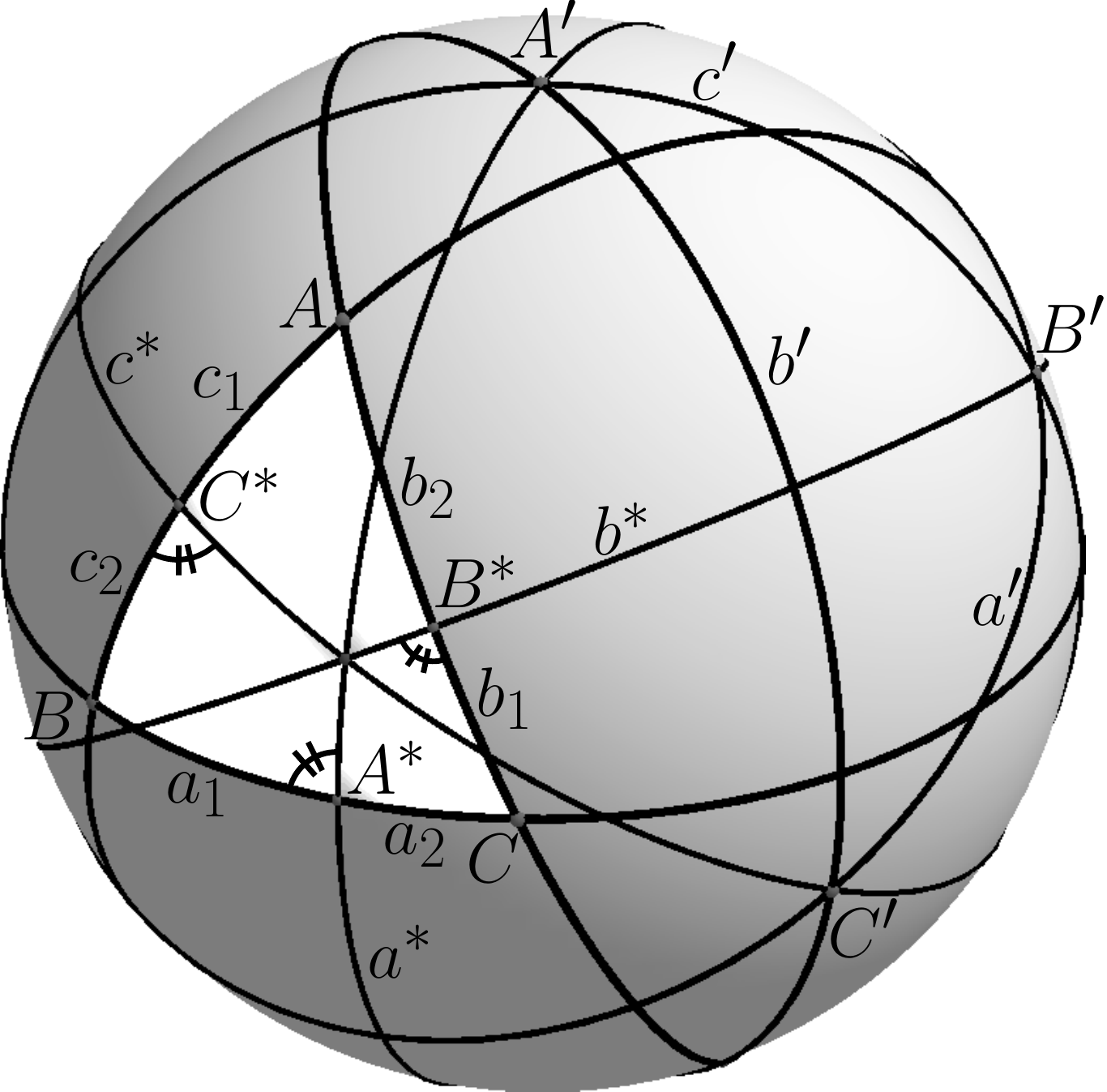}
\label{Fig:Carnot-spherical} }
\\
\subfigure[]{\includegraphics[width=0.93\textwidth]
  {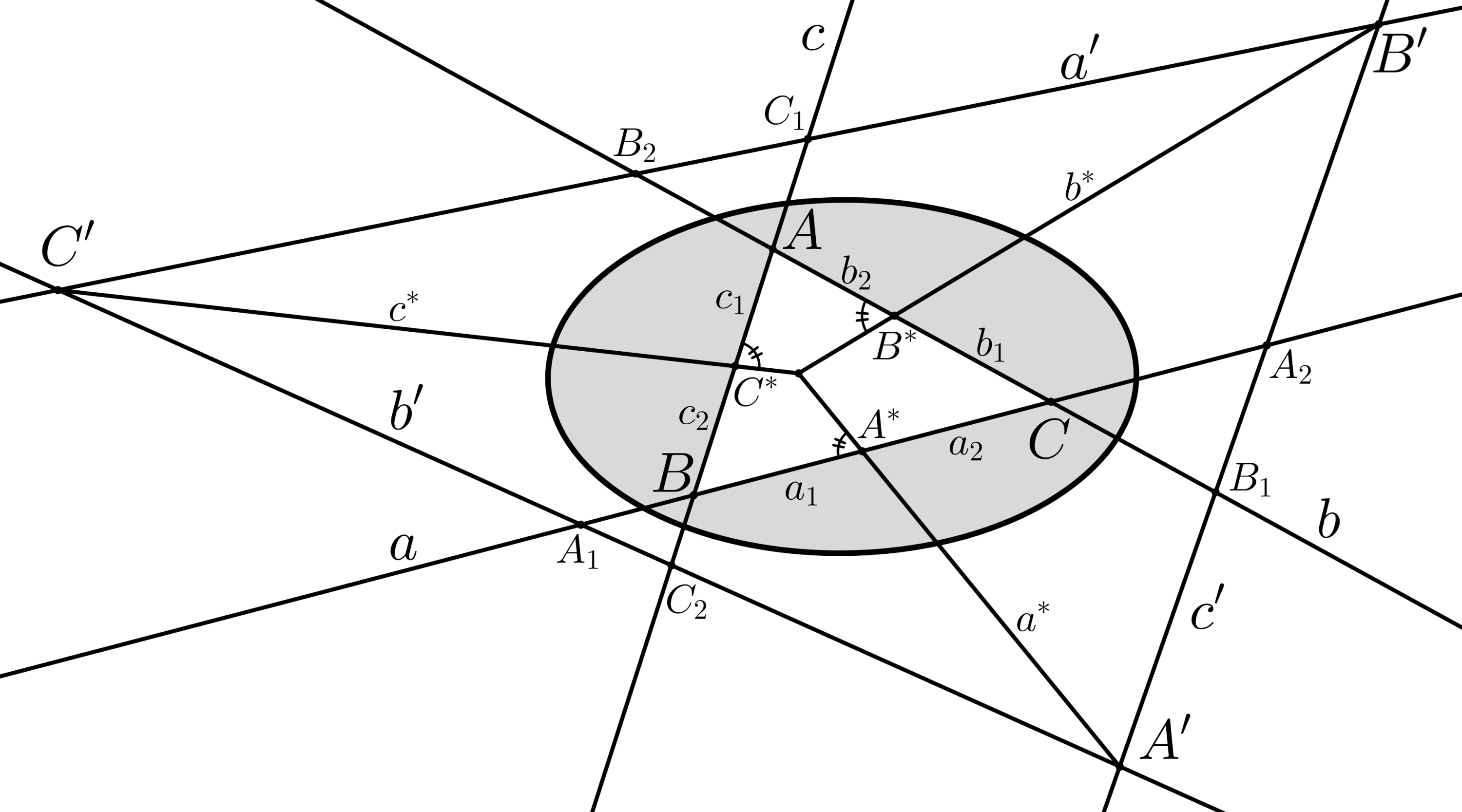}
\label{Fig:Carnot_hyp_triangle} }
\caption{Carnot's theorem on non-euclidean triangles}
\label{Fig:Carnot-triangles}

\end{figure}

If the vertices $A,B,C$ of the projective triangle $\TT$ and the points $A^{*}$,
$B^{*}$, $C^{*}$ lie in the non-euclidean plane $\mathbb{P}$, then $\TT$ is a
non-euclidean triangle and $a^{*},b^{*},c^{*}$ are the perpendicular lines to
$a,b,c$ through
$A^{*},B^{*},C^{*}$ respectively. Let denote also by $a_1,a_2,b_1,b_2,c_1,c_2$
the non-euclidean lengths in $\mathbb{P}$ of the segments
$a_1,a_2,b_1,b_2,c_1,c_2$, respectively. Theorem
\ref{thm:Carnot-projective-cosines} implies
the following (Figure~\ref{Fig:Carnot-spherical}):
\begin{theorem}
\label{thm:Carnot-spherical-triangle}If $\mathbb{P}$ is the elliptic plane and
the lines $a^{*},b^{*},c^{*}$
are concurrent, then
\begin{equation}
\cos a_1 \cos b_1 \cos c_1 = \cos a_2 \cos b_2 \cos c_2\,.
\label{eq:Carnot-elliptic-triangle}
\end{equation}
\end{theorem}
Theorem~\ref{thm:converse-Carnot-projective-cosines} shows how we can obtain
three points $A^*,B^*,C^*$ for which the identity
\eqref{eq:Carnot-elliptic-triangle} holds and such that the lines $a^*,b^*,c^*$
are not concurrent. If this is the case, we say that the points
$A^*,B^*,C^*$ are \emph{fake Carnot points} of $\TT$.
This cannot happen if $\TT$ is a triangle in the hyperbolic plane (Figure
\ref{Fig:Carnot_hyp_triangle}).
\begin{theorem}
\label{thm:Carnot-hyperbolic-triangle}Hyperbolic triangles have
no fake Carnot points: if $\mathbb{P}$ is the hyperbolic plane, the lines
$a^{*},b^{*},c^{*}$
are concurrent if and only if
\begin{equation}
\cosh a_1 \cosh b_1 \cosh c_1 = \cosh a_2 \cosh b_2 \cosh c_2\,.
\label{eq:Carnot-hyperbolic-triangle}
\end{equation}

\end{theorem}
\begin{proof}
If the lines $a^*,b^*,c^*$ are concurrent, then
\eqref{eq:Carnot-hyperbolic-triangle} is a consequence of Theorem
\ref{thm:Carnot-projective-cosines} and Proposition
\ref{prop:geometric-translations}.

If~\eqref{eq:Carnot-hyperbolic-triangle} holds, then
\eqref{eq:Carnot-projective-trigonometric-formula} also holds by Proposition
\ref{prop:geometric-translations}. Thus, by Theorem
\ref{thm:converse-Carnot-projective-cosines}, we have that $A^*$ equals $D^*$ or
the Carnot conjugate $\zeta_{BC}(D^*)$ of $D^*$ with respect to $B,C$, where
$D^*$ is, as in Theorem~\ref{thm:converse-Carnot-projective-cosines}, the
orthogonal projection of $H^*=b^*\cdot c^*$ into the line $a$.

The polar triangle $\TT'$ of $\TT$ divides the projective plane into four
triangular regions, and one of them contains the absolute conic $\Phi$ (see
Figure~\ref{Fig:Carnot_hyp_triangle}). Let us call $\TT'_\Phi$ to this region.
The point $H^*$ does not necessarily lie in $\mathbb{P}$, but because $B^*$ and
$C^*$ are in $\mathbb{P}$ it must lie in $\TT'_\Phi$. This implies that $D^*$
also lies in $\TT'_\Phi$, and we will see that in this case the Carnot conjugate
of $D^*$ with respect to $B,C$ cannot be in $\mathbb{P}$. If $U,V$ are the
intersection points of $a$ with $\Phi$, as it is mentioned in 
\S\ref{sec:Basics-of-projective} the conjugacy involution 
$\rho_a$ of $a$ 
coincides with the harmonic conjugacy
$\tau_{UV}$ with respect to $U,V$.

For any real projective line $x$, any two different real points $Y,Z$  on $x$ divide the line into two connected subsets (segments), say $x_0$ and $x_1$, and the harmonic conjugacy $\tau_{YZ}$ with respect to $Y,Z$ sends $x_0$ onto $x_1$ and vice versa (cf. Lemma~\ref{lemma:ABCD-negative-iff-AB-separate-CD}). Therefore, 
\begin{gather*}
D^*\in\mathbb{P}\Rightarrow \rho_a(D^*)\notin \mathbb{P}\Rightarrow \\
\Rightarrow \tau_{BC}(\rho_a(D^*)) \text{ belongs to the hyperbolic segment } \ov{BC}\subset\mathbb{P}\Rightarrow\\
\Rightarrow \zeta_{BC}(D^*)=\rho_a(\tau_{BC}(\rho_a(D^*)))\notin \mathbb{P}\,.
\end{gather*}
On the other hand, if $D^*$ does not lie in $\mathbb{P}$, because it is in $\TT'_\Phi$, its conjugate $\rho_a(D^*)$ lies in $\mathbb{P}$ but outside the hyperbolic segment $\ov{BC}$. Thus, $\tau_{BC}(\rho_a(D^*))$ is in the hyperbolic segment $\ov{BC}$, and so $\zeta_{BC}(D^*)$ lies outside $\mathbb{P}$. In any case, because $B^*$ and $C^*$ are in $\mathbb{P}$, it cannot be $\zeta_{BC}(D^*)\in\mathbb{P}$. Because $A^*$ belongs to $\mathbb{P}$, it must be $A^*=D^*$ and in consequence the lines $a^*,b^*,c^*$ are concurrent.

\end{proof}

\begin{figure}
\centering
\centering\includegraphics[width=0.6\textwidth]
{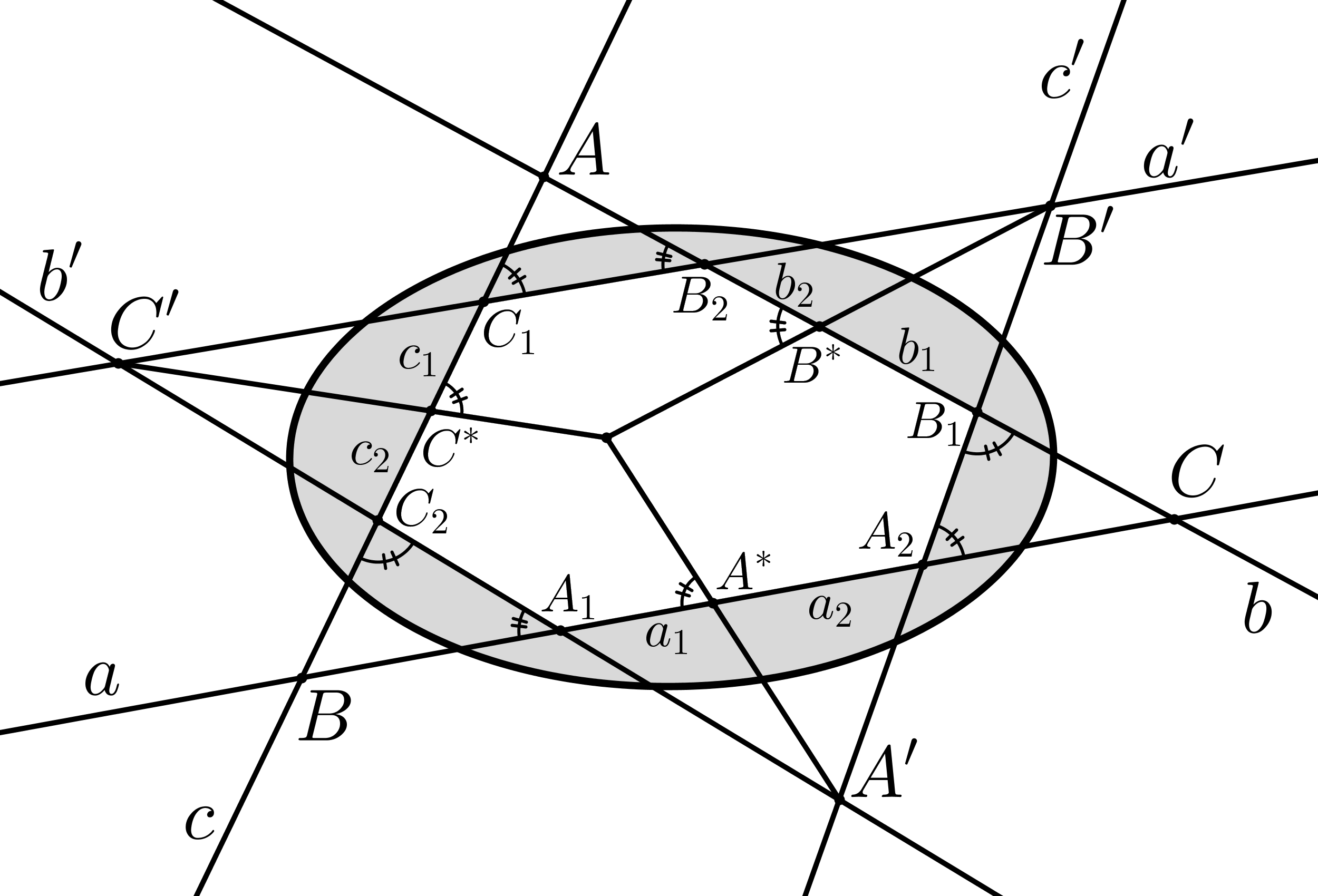}
\caption{Carnot's theorem on right-angled hexagons}
\label{Fig:Carnot_hexagon}
\end{figure}

Finally, we will present a version of Theorem~\ref{thm:Carnot_theorem_euclidean} for right-angled hyperbolic hexagons. Let $\SC{H}$ be a right-angled hexagon
in the hyperbolic plane with consecutive vertices $A_1,A_2,B_1,B_2,C_1,C_2$. Consider three points $A^{*},B^{*},C^{*}$ lying on the alternate sides $a=A_1A_2$, $b=B_1B_2$, $c=C_1C_2$ of $\SC{H}$ respectively, and let $a^{*},b^{*},c^{*}$ be the perpendicular lines to $a,b,c$ through $A^{*},B^{*},C^{*}$ respectively. Let $a_1,a_2,b_1,b_2,c_1,c_2$ denote now the hyperbolic lengths of the segments $\ov{A_1A^*},\ov{A_2A^*},\ov{B_1B^*},\ov{B_2B^*},\ov{C_1C^*},\ov{C_2C^*}$ (see Figure~\ref{Fig:Carnot_hexagon}). If we construct the projective triangle $\TT$ with vertices $A=b\cdot c$, $B=c\cdot a$ and $C=a\cdot b$, we can apply Theorem~\ref{thm:Carnot-projective-cosines} to $\TT$ and the points $A^*,B^*,C^*$. After translating the projective trigonometric ratios 
of~\eqref{eq:Carnot-projective-trigonometric-formula} to our construction using 
Proposition~\ref{prop:geometric-translations}, we obtain:
\begin{theorem}
\label{thm:Carnot-hyperbolic-hexagon}If the lines $a^{*},b^{*},c^{*}$
are concurrent, then
\begin{equation*}
\sinh a_1 \sinh b_1 \sinh c_1 = \sinh a_2 \sinh b_2 \sinh c_2\,.
\end{equation*}
\end{theorem}


\chapter{Where do laws of cosines come from?}
\label{sec:Cosine-Rule}

The projective trigonometric ratios $\CC,\SS,\TAN$ 
have the inconvenience of the presence of the square power 
in all their geometric translations 
(Table~\ref{table:non-euclidean-segments}). 
In \S\ref{sec:Menelaus-Theorem} we saw that 
they are sufficient to deal with figures whose trigonometric 
formulae are simple, and that they are not suitable to work with 
the law of cosines of a generalized non-euclidean triangle, 
for example. We are looking for projective expressions 
associated with segments or angles whose non-euclidean 
translation give the actual, \emph{unsquared}, 
trigonometric ratios of the corresponding segment or angle. 
This can be done using the midpoints of segments 
as accesory points. 
Another construction appears in the Appendix 
as Formula~\eqref{eq:angle-between-rays-cosine}.

\section{Midpoints and the unsquaring of projective trigonometric ratios}

\begin{figure}[t]
\centering
\begin{tabular}{lr}
\subfigure[]{\includegraphics[width=0.38\textwidth]
{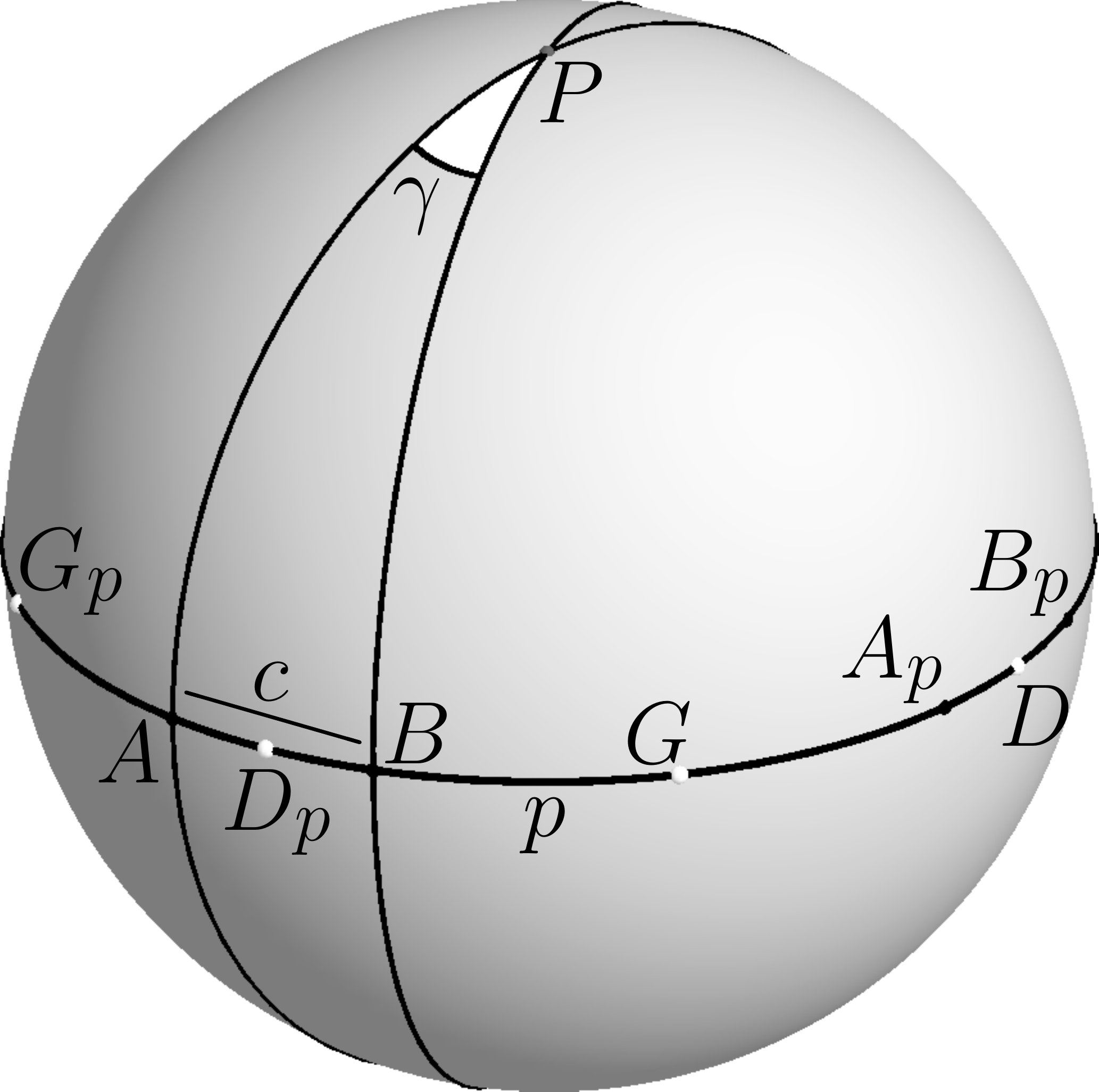}
\label{Fig:Laguerre-segment-spherical}}&
\subfigure[]{\includegraphics[width=0.55\textwidth]
{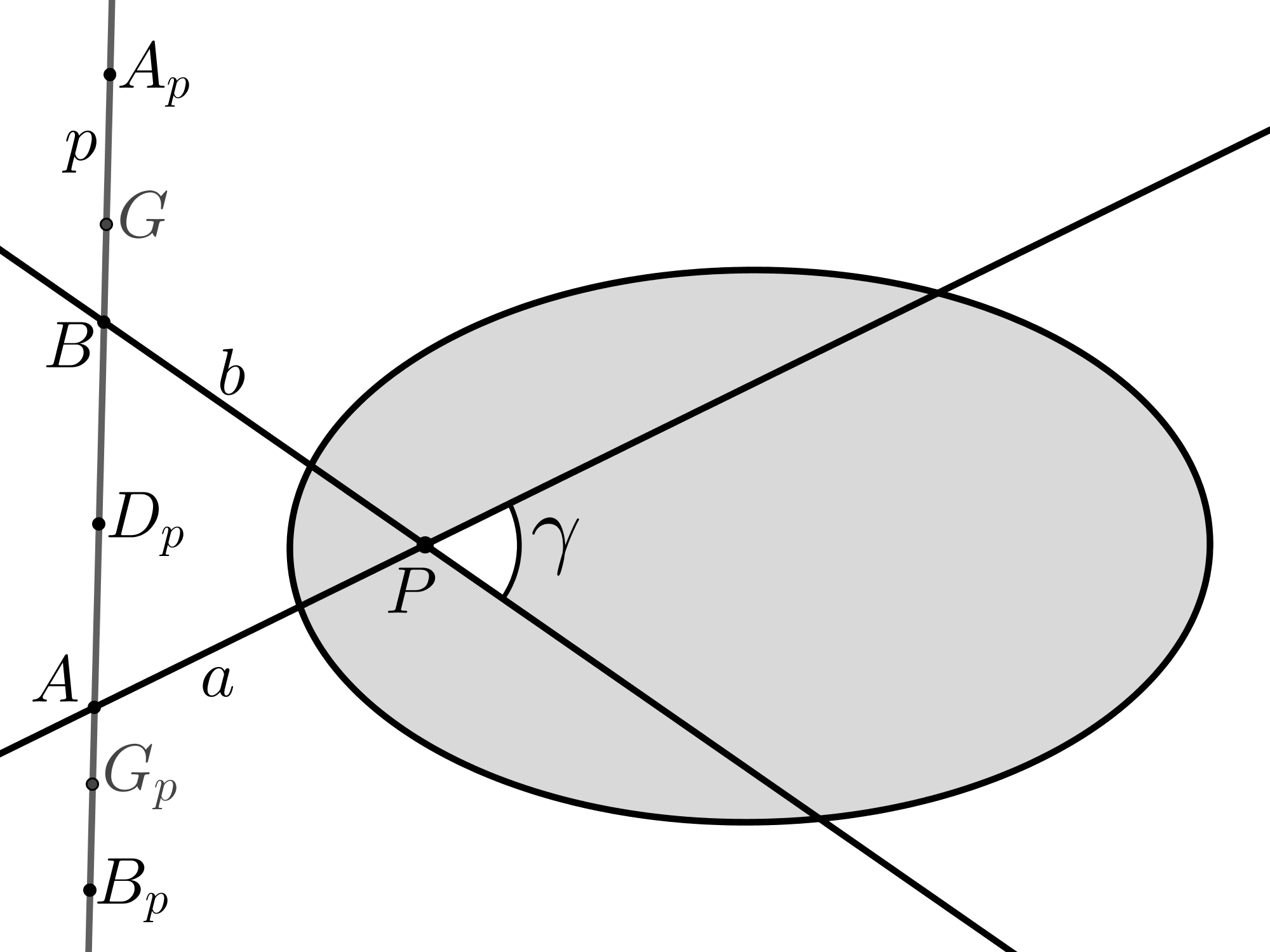}
\label{Fig:Laguerre-segment-hyperbolic-01}}
\end{tabular}\\
\caption{Projective segments and their
midpoints I}
\label{Fig:Laguerre-non-euclidean_segments_I}
\end{figure}

Let $A,B$ be two points not lying on $\Phi$ and such that the line $p$ joining
them is not tangent to $\Phi$. Let $D,D_{p}$ be the two midpoints of $\ov{AB}$.
The symmetry $\tau_{DD_p}$ of $p$ with respect to $D,D_{p}$ maps
$A$ and $A_{p}$ into $B$ and $B_{p}$ respectively. This
implies that 
\[
(ABDA_{p})=(BADB_{p}),
\]
and so it is
\[
(ABB_{p}A_{p})=(ABDA_{p})(ABB_{p}D)=(BADB_{p})(ABB_{p}D)=(ABB_{p}D)^{2}.
\]
Moreover,
\[
(ABB_{p}D)=(ABD_{p}D)(ABB_{p}D_{p})=-(ABB_{p}D_{p}).
\]
\begin{remark}\label{remark:midpoints-square-roots}
The previous identities imply that the two square roots of
$\CC(\ov{AB})=(ABB_{p}A_{p})$
are given by $(ABB_{p}D)$ and $(ABB_{p}D_{p})$.
\end{remark}

By definition, it is $\SS(\ov{AB})=\CC(\ov{AB_p})$. Therefore, the two square
roots of $\SS(\ov{AB})$ are given by
$(AB_pBG)$ and $(AB_pBG_p)$,
where $G,G_p$ are the two midpoints of $\ov{AB_p}$.

If we want to compute a square root of $\CC(\ov{AB})$ and a square root of
$\SS(\ov{AB})$, we must choose a midpoint of the segment 
$\ov{AB}$ and a midpoint of
the complementary segment $\ov{AB_p}$.
By Lemma~\ref{lem:midpoints-AB-midpoints-ApBp}, the midpoints of
$\ov{AB}$ are also the midpoints of $\ov{A_pB_p}$ and the midpoints of
$\ov{AB_p}$ are the midpoints of $\ov{A_pB}$. We say that the midpoints of
$\ov{AB_p}$ are the \emph{complementary midpoints} of $\ov{AB}$.

\begin{definition}\label{def:oriented-segment}
The segment $\ov{AB}$ is \emph{oriented} if we have chosen for it a 
midpoint and
a complementary midpoint as \emph{preferred midpoints}.
If $\ov{AB}$ is
oriented, we define its two associated 
\emph{vectors} $\overrightarrow{AB}$ and
$\overrightarrow{BA}$ as the \emph{ordered} pairs $(A,B)$ and $(B,A)$,
respectively, 
together with the preferred midpoints of $\ov{AB}$.
We define also the \emph{projective trigonometric ratios} of 
$\overrightarrow{AB}$ and $\overrightarrow{BA}$:
\begin{equation}\label{eq:def-unsquared-trigonometric-ratios}
\begin{aligned}
\cc(\overrightarrow{AB})=(ABB_pD)\,,\quad&\quad
\cc(\overrightarrow{BA})=(BAA_pD)\,,\\
\ss(\overrightarrow{AB})=(AB_pBG)\,,\quad&\quad 
\ss(\overrightarrow{BA})=(BA_pAG)\,.
\end{aligned}
\end{equation}
where $D$ and $G$ are respectively the preferred midpoint and the preferred
complementary midpoint of $\ov{AB}$.
\end{definition}
From now on, we will reserve the term \emph{projective trigonometric ratios} for
the functions $\cc$ and $\ss$ just defined, while the functions $\CC,\SS,\TAN$
defined in \S\ref{sec:Cayley-Klein-models-for} will be called \emph{squared
projective trigonometric ratios}.

Note that the functions $\cc$ and $\ss$ for oriented segments work as it is
expected in analogy with the circular and hyperbolic cosine and sine functions,
because it is:
\begin{align*}
\cc(\overrightarrow{AB})=(ABB_pD)
\overset{\tau_{DD_p}}{=}&(BAA_pD)=\cc(\overrightarrow{BA})\,,\\
\ss(\overrightarrow{AB})=(AB_pBG)
\overset{\tau_{DD_p}}{=}&(BA_pAG_p)=-\ss(\overrightarrow{BA})\,,
\end{align*}
and so $\cc$ works like an ``even'' function while $\ss$ behaves like an 
``odd'' function. For a simpler notation we use 
the expressions $\cc(AB),\ss(AB)$
instead of $\cc(\overrightarrow{AB}),\cc(\overrightarrow{AB})$.


\begin{figure}[t]
\centering
\subfigure[]{\includegraphics[width=0.48\textwidth]
{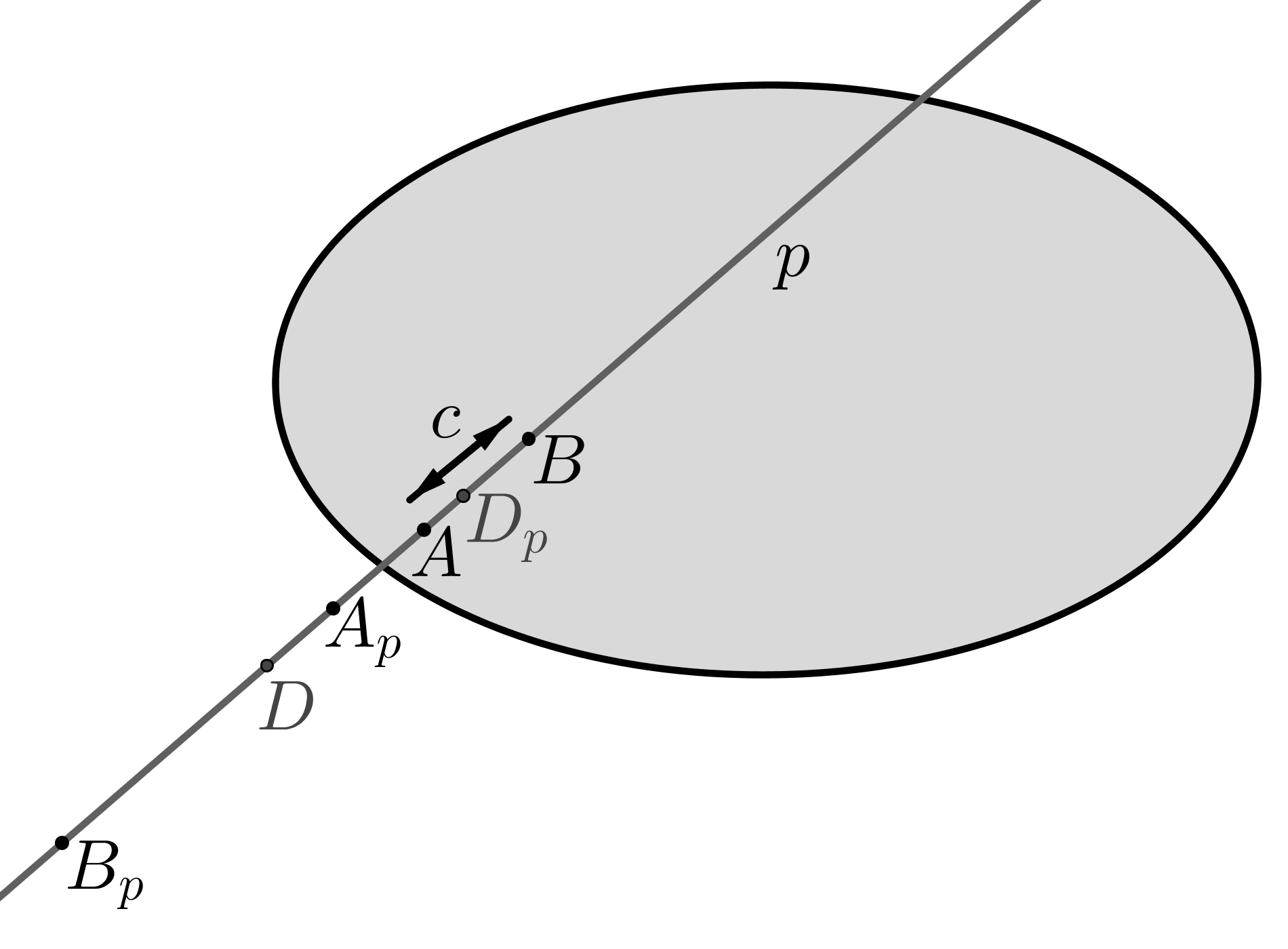}
\label{Fig:Laguerre-segment-hyperbolic-02}}
\subfigure[]{\includegraphics[width=0.48\textwidth]
{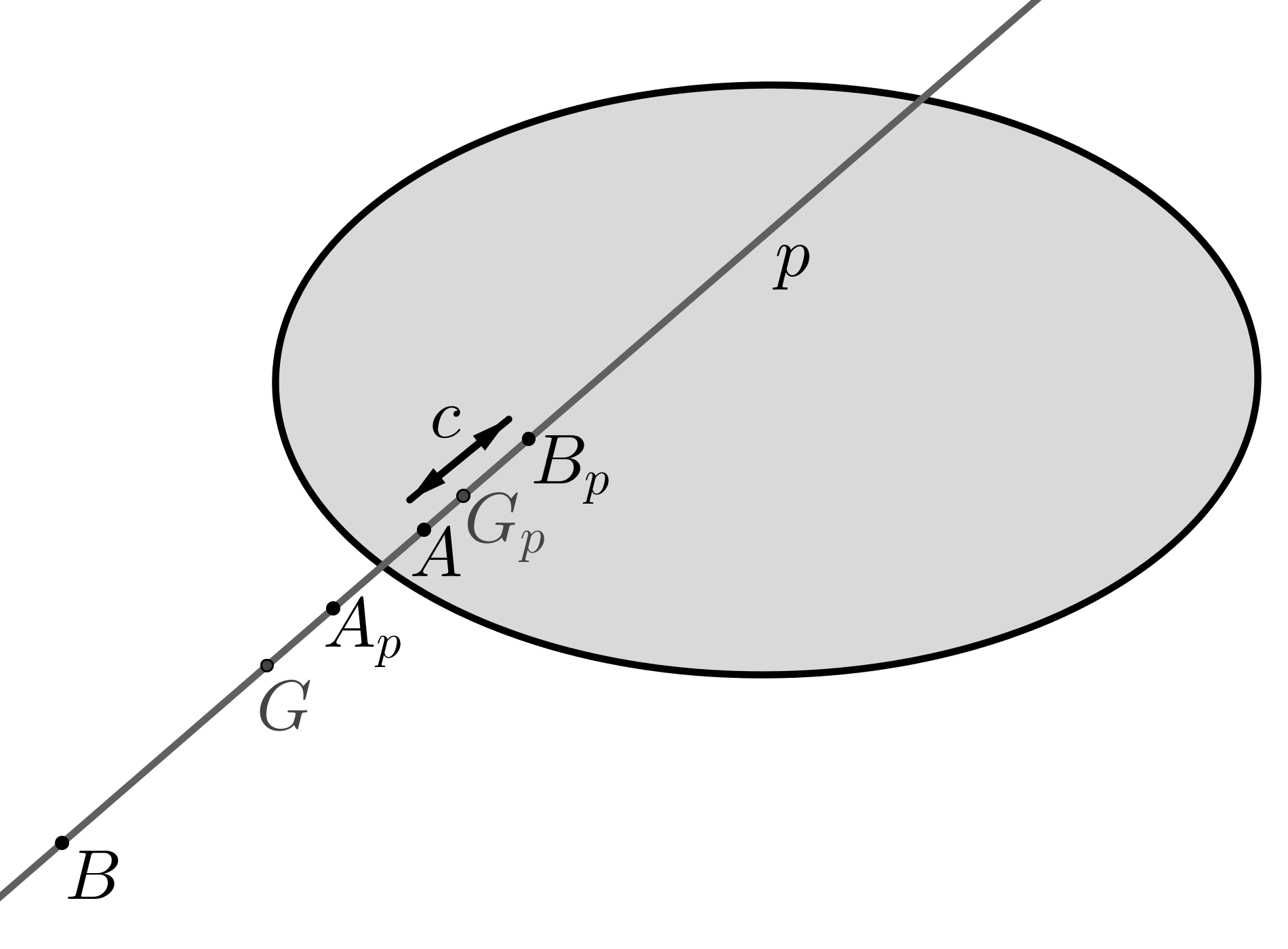}
\label{Fig:Laguerre-segment-hyperbolic-03}}\\
\subfigure[]{\includegraphics[width=0.48\textwidth]
{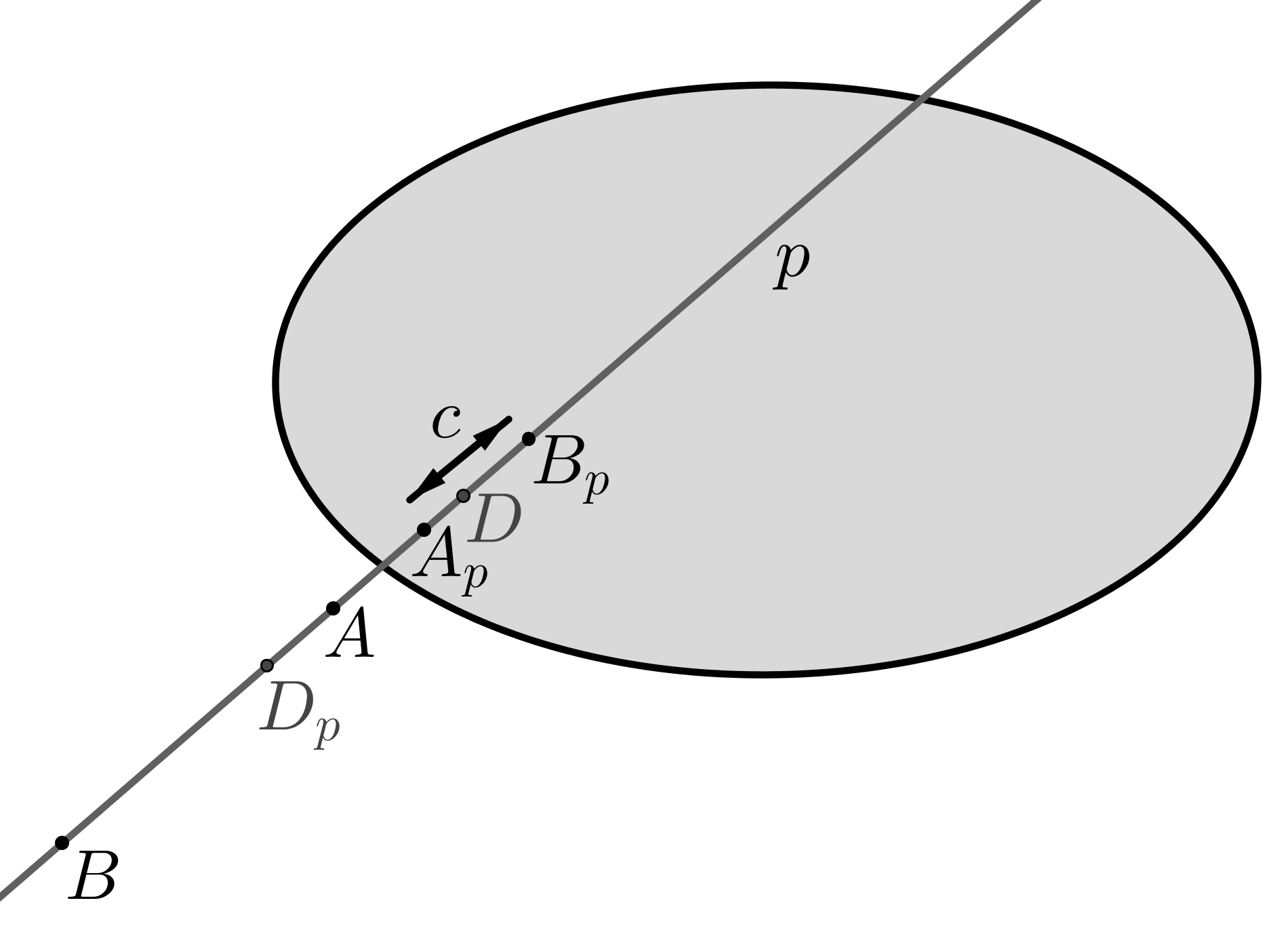}
\label{Fig:Laguerre-segment-hyperbolic-04}}
\caption{Projective segments and their midpoints II}
\label{Fig:Laguerre-non-euclidean_segments_II}
\end{figure}

If $\Phi$ is imaginary, the points $D,D_p,G,G_p$ are real points. If $A,B$
are the points depicted in Figure~\ref{Fig:Laguerre-segment-spherical}, and we
choose as preferred midpoints for them the points $D$ and $G$ of the same
figure, by Lemma~\ref{lemma:ABCD-negative-iff-AB-separate-CD} it is
$$\cc(AB)=(ABB_pD)>0\quad \text{and}\quad \ss(AB)=(AB_pBG)>0$$
Therefore, in this case it is
$\cc(AB)=\cos c$ and $\ss(AB)=\sin c$, where $c$ is the elliptic distance
between $A$ and $B$ as depicted in Figure~\ref{Fig:Laguerre-segment-spherical}.
If $a,b$ are the lines joining $A,B$ with the pole $P$ of $p$ and $\gamma$ is
the angle between $a$ and $b$ depicted in Figure
\ref{Fig:Laguerre-segment-spherical}, it is also $\cc(AB)=\cos \gamma$ and
$\ss(AB)=\sin \gamma$.

When $\Phi$ is real and $p$ is exterior to $\Phi$, the four points $D,D_p,G,G_p$
are again real points. In the situation of Figure
\ref{Fig:Laguerre-segment-hyperbolic-01}, if $D$ and $G$ are the preferred
midpoints of $\ov{AB}$, it is $\cc(AB)>0$ and also $\ss(AB)>0$. Let consider
again the lines $a,b$ joining $A,B$ with the pole $P$ of $p$. Then, because the
angle $\gamma$ in this figure is an acute angle (observe that the line $a$ and
its conjugate $a_P=PA_p$ form a right angle), it must be $\cc(AB)=\cos \gamma$
and $\ss(AB)=\sin \gamma$.

When $\Phi$ is a real conic and $p$ is secant to $\Phi$, there are some
possibilities depending on the relative positions of the points $A,B$ with
respect to $\Phi$. In this case, the midpoints of $\ov{AB}$ are real points if
and only if the complementary midpoints are imaginary. This will imply that for
any choice of preferred midpoints for the segment $\ov{AB}$, one of the
projective trigonometric ratios $\cc(AB),\ss(AB)$ will be real while the other
is	 pure imaginary. If we want to translate the projective trigonometric
ratios $\cc(AB),\ss(AB)$ into geometric trigonometric ratios associated with a
hyperbolic magnitude, we need to decide the \emph{sign} (positive or negative
when the number is real, \emph{positive imaginary} or \emph{negative
imaginary}\footnote{\label{footnote:positive-imaginary}
We use the terms \emph{positive imaginary} and
\emph{negative imaginary} in the obvious way: a positive (resp. negative)
imaginary number has the form $i\lambda$, with $\lambda\in\mathbb{R}$ positive
(resp. negative).} when the number is pure 
imaginary) of $\cc(AB)$ and $\ss(AB)$. The sign of the real number among
$\cc(AB)$ and $\ss(AB)$ can always be decided ``visually'', without doing
explicit computations, using Lemma~\ref{lemma:ABCD-negative-iff-AB-separate-CD}.
On the contrary, the sign of the pure imaginary number among $\cc(AB)$ and
$\ss(AB)$ cannot be decided in a simple way.

If both points $A,B$ are interior to $\Phi$, the two midpoints $D,D_p$ are real
points, while the two complementary midpoints $G,G_p$ are imaginary points.
Thus, we have no visual way for choosing a preferred complementary midpoint. In
the situation of Figure~\ref{Fig:Laguerre-segment-hyperbolic-02}, if $D$ is the
preferred midpoint of $\ov{AB}$, we have $\cc(AB)>0$ and so it is $\cc(AB)=\cosh
c$, where $c$ is the hyperbolic distance between $A$ and $B$. On the other hand,
if $G$ is a complementary midpoint of $\ov{AB}$, the cross ratio $(AB_pBG)$ must
be pure imaginary because by 
Lemma~\ref{lemma:ABCD-negative-iff-AB-separate-CD}
it is
$$(AB_pBG)^2=\SS(AB)=(AB_pBA_p)<0\,.$$
Thus, for any choice of preferred complementary midpoint it will be $\ss(AB)=\pm
i\sinh c$, and the correct sign of this equality cannot be decided without doing
explicit computations.

In the same way, if both points $A,B$ are exterior to $\Phi$, for any choice of
preferred midpoints it will be $\cc(AB)=\pm\cosh c$ and $\ss(AB)=\pm i\sinh c$,
where $c$ is the hyperbolic distance between the poins $A_p$ and $B_p$. For
example, in Figure~\ref{Fig:Laguerre-segment-hyperbolic-04}, if $D$ is the
preferred midpoint of the segment $\ov{AB}$ it is $\cc(AB)=\cosh c$.

If exactly one of the two points $A,B$ is interior to $\Phi$, then the midpoints
of $\ov{AB}$ are imaginary numbers while the complementary midpoints are real.
In Figure~\ref{Fig:Laguerre-segment-hyperbolic-03}, for example, if $G$ is the
preferred complementary midpoint of $\ov{AB}$ it is $\ss(AB)=\cosh c$ while for
any choice of a preferred midpoint of $\ov{AB}$ it will be $\cc(AB)=\pm i \sinh
c$, where $c$ is the hyperbolic distance between $A$ and $B_p$.

In our aim of
obtaining a projective version of the law of cosines, we will use the
functions $\cc,\ss$ just defined, but we have the annoying
problems of dealing with imaginary midpoints, and of deciding the sign of
projective trigonometric ratios when they are pure imaginary. As we will see, we
can
overcome these drawbacks: we will give a construction such that no imaginary
midpoint needs to be explicitly constructed and such that no sign of pure
imaginary 
projective trigonometric ratio needs to be explicitly computed.

\section{The magic midpoints. Oriented
triangles}\label{sec:magic-midpoints-oriented-triangles}

Consider a projective triangle $\TT$ and its polar triangle $\TT'$ as
in \S\ref{sec:triangle-notation}. The points $D,D_a$ and
$G,G_a$ are the midpoints and the complementary midpoints of $\ov{BC}$,
respectively. In the same way, the points $E,E_b$ and $H,H_b$ are  the midpoints
and the complementary midpoints of $\ov{CA}$, respectively, and $F,F_c$ and
$I,I_c$ are  the midpoints and the complementary midpoints of $\ov{AB}$,
respectively. As usual, the same notation but adding an apostrophe
$$D',D'_{a'},G',G'_{a'},E',E'_{b'},H',H'_{b'},F',F'_{c'},I',I'_{c'}$$
holds for the midpoints and complementary midpoints of $\TT'$.

In order to simplify notation, we say that the midpoints 
(complementary midpoints) of a segment whose 
endpoints are vertices of $\TT$ are midpoints 
(complementary midpoints) of the corresponding side 
of $\TT$ and also midpoints (complementary midpoints) of $\TT$.

If we want to apply our functions $\cc,\ss$ to the triangles $\TT$ and $\TT'$, 
we need to orient each side of both triangles in the sense of 
Definition~\ref{def:oriented-segment}. 
It should be interesting to do so in a coherent way. 
For the midpoints of $\TT$ and $\TT'$ it is not difficult to 
establish a choice criterion based on Theorem
\ref{thm:midpoints-quadrilateral}. 
We have assumed that $D,E,F$ are non-collinear, and this implies that
$D,E,F_c$ are collinear. 
There is a big difference between chosing $D,E,F$ or $D,E,F_c$ 
as preferred midpoints. We will say that the preferred 
midpoints of the sides of $\TT$ are \emph{coherently chosen} 
if they are not collinear, and the same for the preferred 
midpoints of $\TT'$. Thus, if we choose the midpoints $D,E,F$ and
$D',E',F'$ with the assumptions of \S\ref{sec:triangle-notation}
as preferred midpoints of their respective segments, they will be
coherently chosen.
The choice criterion is not as easy 
for the complementary midpoints of $\TT$, because they 
have not the same structure as the midpoints.

\begin{figure}
\centering
\includegraphics[width=0.8\textwidth]
{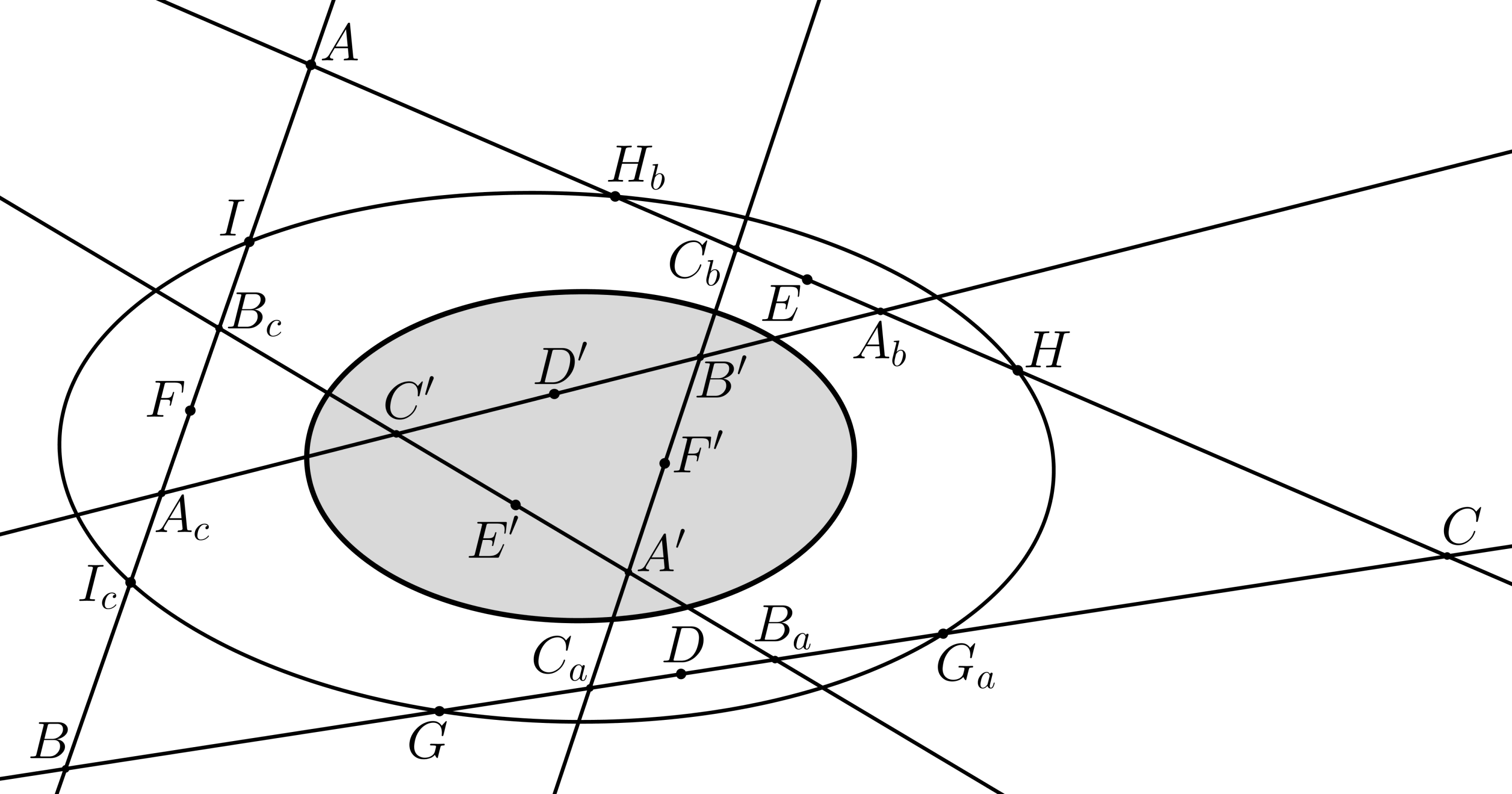}
\caption{Complementary midpoints lie on a
conic}\label{Fig:complementary-midponts-on-a-conic}
\end{figure}

\begin{theorem}\label{thm:complementary-midpoints-lie-on-conic}
The complementary midpoints of the triangle $\TT$ lie on a conic.
\end{theorem}
Before proving this theorem, we need the following lemma:
\begin{lemma}\label{lem:DDa-midpoints-of-GGa}
$D,D_a$ are the midpoints of $\ov{GG_a}$ and $G,G_a$ are the midpoints of
$\ov{DD_a}$.
\end{lemma}
\begin{proof}
By Lemma~\ref{lem:midpoints-harmonic-UV-AB}, it suffices to prove that
$(GG_aDD_a)=-1$.
The composition $\tau_{GG_a}\circ\tau_{DD_a}$ sends $B,B_a,C,C_a$ into
$B_a,B,C_a,C$ respectively.
Therefore, it is $\tau_{GG_a}\circ\tau_{DD_a}=\rho_a$, and so
$$\tau_{GG_a}(D)=\tau_{GG_a}(\tau_{DD_a}(D))=\rho(D)=D_a\Rightarrow
(GG_aDD_a)=-1\,.$$
\end{proof}

\begin{proof}[Proof of Theorem~\ref{thm:complementary-midpoints-lie-on-conic}]
Assume that the midponts $D,E,F$ of $\TT$ are not collinear  (see Figure
\ref{Fig:complementary-midponts-on-a-conic}). Then, $D_a,E_b,F_c$ are collinear.

By Lemma~\ref{lem:DDa-midpoints-of-GGa}, we have
$$(BCD_aG)\overset{\tau_{DD_a}}{=}(CBD_aG_a)\Rightarrow
(BCD_aG)(BCD_aG_a)=1\,.$$
In the same way, it is 
$$(CAE_bH)(CAE_bH_b)=1\quad\text{and} \quad (ABF_cI)(ABF_cI_c)=1\,.$$
Therefore,
$$(BCD_aG)(BCD_aG_a)(CAE_bH)(CAE_bH_b)(ABF_cI)(ABF_cI_c)=1\,.$$
The result now is a consequence of Carnot's Theorem on projective triangles
(Theorem~\ref{thm:Carnot-projective}).
\end{proof}

Thus, until now there is no remarkable difference between choosing $G,H,I$ or
$G,H,I_c$ as preferred complementary midpoints of $\TT$. We must make a deeper
exploration for finding a choice criterion for the preferred complementary
midpoins of $\TT$. This will need... some magic.

\begin{figure}
\centering
\includegraphics[width=0.9\textwidth]
{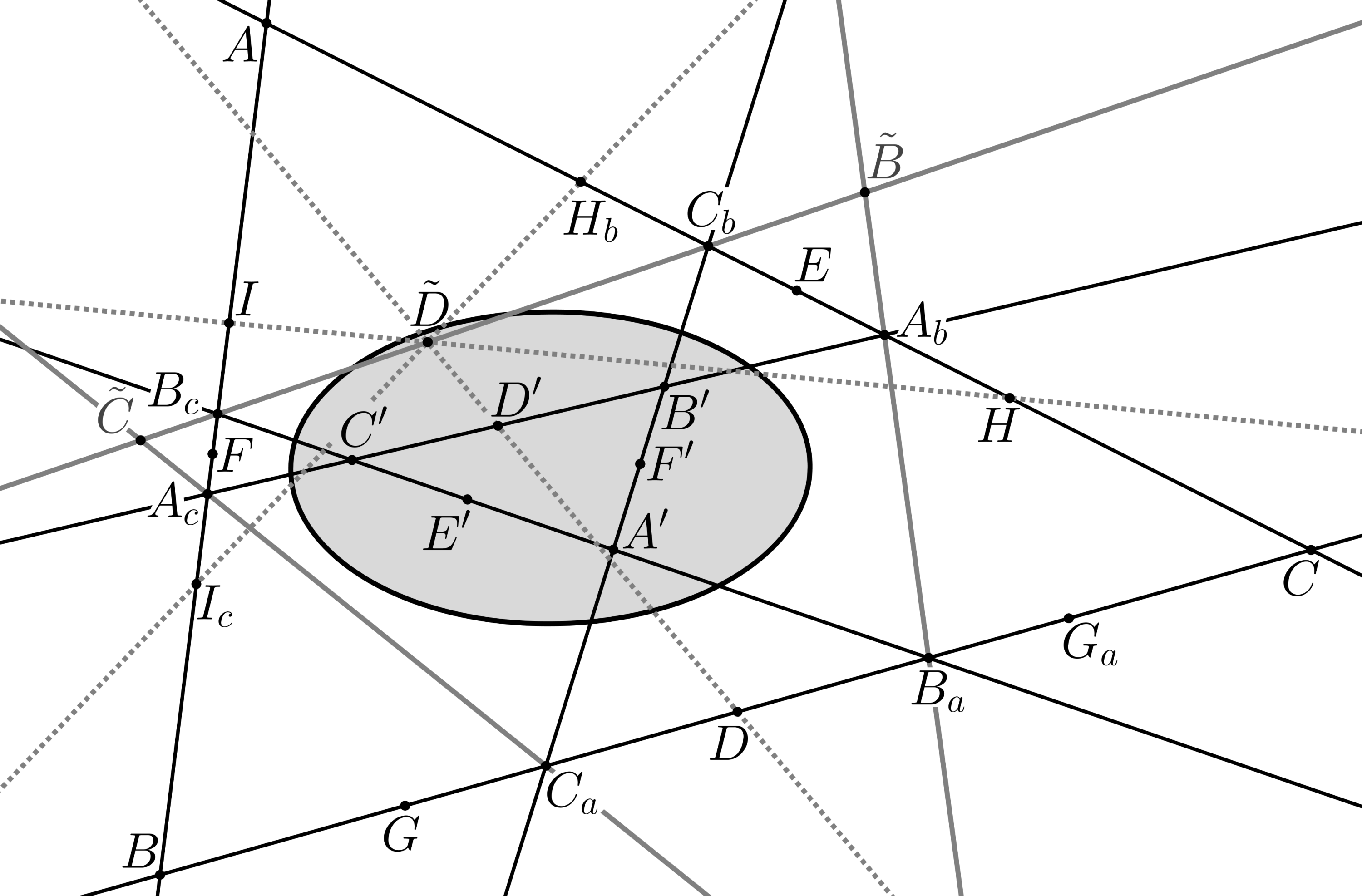}
\caption{Concurrence of lines at a magic midpoint}
\label{Fig:magic-midpoints-01}
\end{figure}

Let fix our sight now in the segment $\ov{B_cC_b}$. Its midpoints, and the
midpoints of $\ov{C_aA_c}$ and $\ov{A_bB_a}$ are very interesting because they
are ``in the middle of everywhere''. We need to introduce some new notation. Let
$\tia,\tib,\tic$ be the lines $B_cC_b,C_aA_c,A_bB_a$ respectively, let $\tiT$ be
the triangle with sides $\tia,\tib,\tic$, and consider the vertices of $\tiT$:
$$\tiA=\tib\cdot\tic,\quad\tiB=\tic\cdot\tia,\quad\tiC=\tia\cdot\tib\,.$$
The poles of $\tia,\tib,\tic$ are, respectively, the points
\[
A_{1}=b_{C}\cdot c_{B},\qquad B_{1}=c_{A}\cdot a_{C},\qquad C_{1}=a_{B}\cdot
b_{A},
\]
of \S\ref{sec:Desargues2}.
By the proof of Proposition
\ref{prop:pseudosymmetric-triangle-orthogonal-to-altitudes}, the lines $a,a'$
and $B_1C_1$ are concurrent, and this implies that their poles $A',A$ and $\tiA$
respectively are collinear. Equivalently, $\tiB$ is collinear with $B$ and $B'$
and $\tiC$ is collinear with $C$ and $C'$.
Let $\tiD,\tiD_{\tia}$ be the midpoints of $\ov{B_cC_b}$.

\begin{theorem}\label{thm:magic-midpoints}
The set $\{\tiD,\tiD_{\tia}\}$ coincides with:
    \begin{enumerate}[\bfseries I]
	\item the set of intersection points $\{HI\cdot H_bI_c,HI_c\cdot H_bI\}$;
	\item the set of intersection points $\{H'I'\cdot H'_{b'}I'_{c'},H'I'_{c'}\cdot H'_{b'}I'\}$; and
	\item the set of midpoints of $\ov{\tiB\tiC}$.
    \end{enumerate}
Moreover,
    \begin{enumerate}[\bfseries {IV}\alph{enumi}]
	\item if $\TT$ is not isosceles at $A$, the set $\{\tiD,\tiD_{\tia}\}$
		coincides with the set of intersection points 
		$\{DD'\cdot D_aD'_{a'},DD'_{a'}\cdot D_aD'\}$;
	\item if $\TT$ is isosceles at $A$, one of the points $\tiD,\tiD_{\tia}$
		coincides with
		$A_0$ while the other is collinear with $D$ and $D'$.		
    \end{enumerate}
\end{theorem}
In particular, $\tiD,\tiD_{\tia}$ are: (i)
the diagonal points different from $A$ of the quadrangle whose
vertices
are the complementary midpoints $H,H_b,I,I_c$ of $\ov{CA}$ and
$\ov{AB}$; (ii) the diagonal points different from $A'$ of the quadrangle whose
vertices are the complementary midpoints $H',H'_{b'},I',I'_{c'}$ of
$\ov{C'A'}$ and $\ov{A'B'}$; and (iii) if $\TT$ is not isosceles at $A$,
the diagonal points different from $A_0$ of the quadrangle whose
vertices are the midpoints $D,D_a,D',D'_{a'}$ of $\ov{BC}$ and
$\ov{B'C'}$ (Figure~\ref{Fig:magic-midpoints-01}).

\begin{proof}
We consider the triangle $\wt{AB_cC_b}$. Its midpoints are the vertices of a
complete quadrilateral. This proves \textbf{I}. An equivalent construction
proves \textbf{II}

We prove \textbf{IV} before \textbf{III}.
In order to prove \textbf{IV} we need:
\begin{claim}\label{claim:T-isosceles-a-a'-tia-concurrent}
  The triangle $\TT$ is isosceles at $A$ if and only if $a$, $a'$ and
  $\tia$ are concurrent.
\end{claim}
\begin{proof}[Proof of Claim~\ref{claim:T-isosceles-a-a'-tia-concurrent}]
If, $a,a',\tia$ are concurrent, it is
$$(ABB_cA_c)=(ACC_bA_b)\,.$$
Let $\pi_{A_0}$ be the perspectivity from $b$
onto $c$ with center at $A_0$. Because $A,C,C_b,A_b$ are four
different points, their cross ratio is a number different from $0,1$, and so it
has two different square roots which are also different from $0,1$. By Remark
\ref{remark:midpoints-square-roots}, the two square roots of $(ACC_bA_b)$
are given by $(ACC_bE)$ and $(ACC_bE_b)$, and so the two square
roots of $(ABB_cA_c)$ must be given by $(ABB_c \pi_{A_0}(E))$ and
$(ABB_c \pi_{A_0}(E_b))$. This implies that the set
$\{\pi_{A_0}(E),\pi_{A_0}(E_b)\}$
coincides with $\{F,F_c\}$. As we have assumed that $D,E,F$ are non-collinear
and that $D$ is different from $A_0$, this implies
that $EF$ and $E_bF_c$
pass also through $A_0$ and so $A_0=D_a$.

The previous argument can be reversed: if 
$D_a=A_0$, the line $EF$ passes through $A_0$.
By taking again the projection
$\pi_{A_0}$, we have that the cross-ratios
$(CAA_bE)$ and $(BAA_c F)$ must be equal, and so their
respective square powers $(CAA_bC_b)$ and $(BAA_c B_c)$ 
must be equal too. This
implies that
$\tia=B_cC_b$ passes also through $A_0$ and therefore $a,a',\tia$ are
concurrent.
\end{proof}

Therefore, if $\TT$ is not isosceles at $A$,
the triangle
$\TT_a=\wt{aa'\tia}$ is an actual triangle because its sides are not concurrent.
In this case, we will show that the midpoints of
$\TT_a$ coincide with the midpoints of the segments $\ov{BC}$,
$\ov{B'C'}$, $\ov{B_cC_b}$.

The vertices of the triangle $\TT_a$ are the points
$$A_0=a\cdot a', \quad J=a\cdot \tia ,\quad J'=a'\cdot \tia\,.$$
Let $D,D_a$ be the midpoints of $\ov{BC}$. 
We must prove that they are also the
midpoints of the segment $\ov{A_0J}$. 
The conjugate point of $A_0$ in $a$ is the
point $H_A$ where the altitude $h_a$ of $\TT$ intersects the side $a$. If we
consider the quadrangle $\QQ_2=\{A,A',B_c,C_b\}$, the quadrangular involution
$\sigma_{\QQ_2}$ induced by $\QQ_2$ on $a$ 
sends the points $B,C,J$ into the
points $C_a,B_a,H_A$ respectively and vice versa (Figure
\ref{Fig:magic-midpoints-03}). 
The composition $\rho_a\circ\sigma_{\QQ_2}$ sends
the points
$B,C,B_a,C_a$ into the points $C,B,C_a,B_a$, and so it coincides with the
symmetry $\tau_{DD_a}$ of $a$ with respect to $D$. Then, it is
$$\tau_{DD_a}(J)=\rho_a(H_A)=A_0\,,$$
and therefore, by Lemma~\ref{lem:midpoints-harmonic-UV-AB}, the points $D,D_a$
are also the midpoints of the segment $\ov{A_0J}$.

\begin{figure}
\centering
\includegraphics[width=\textwidth]{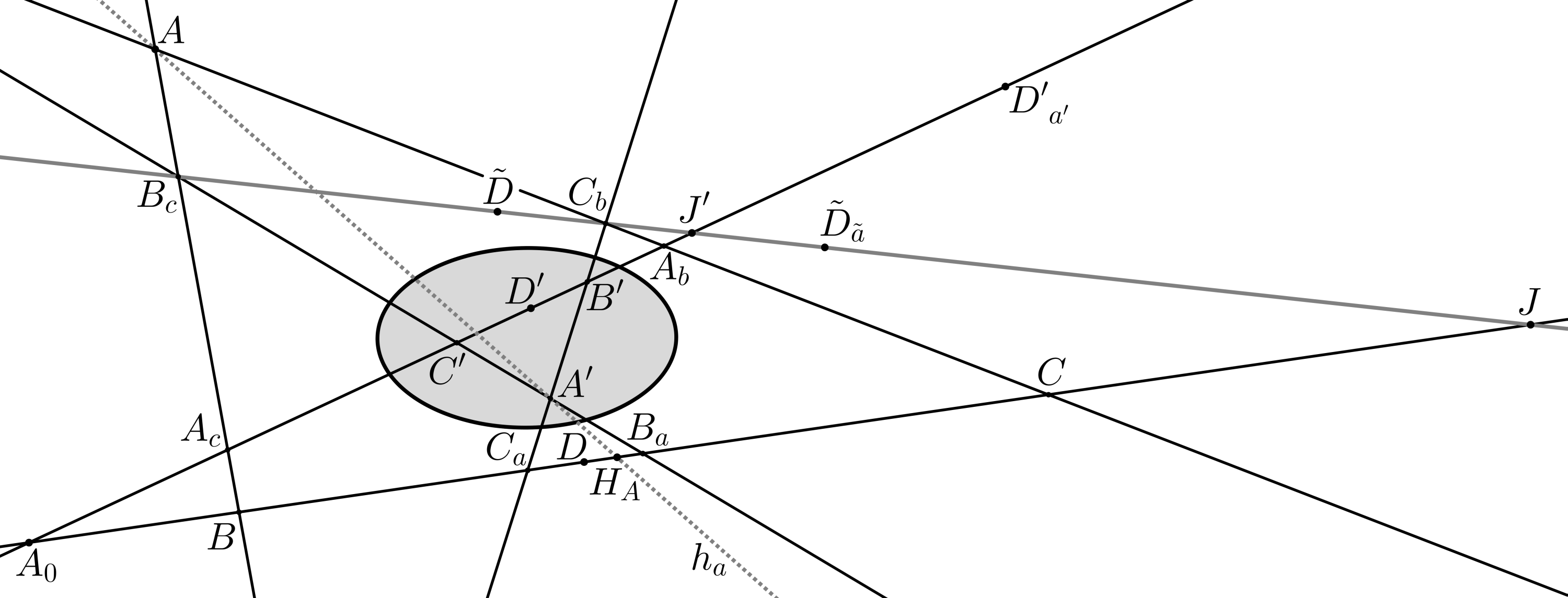}
\caption{Magic midpoints and ordinary midpoints of $\TT$ and
$\TT'$}\label{Fig:magic-midpoints-03}	
\end{figure}

In the same way, it can be proved that the midpoints $D',D'_{a'}$ 
of $\ov{B'C'}$ are also the midpoints of $\ov{A_0J'}$, 
and a similar argument can be used to conclude that 
the midpoints $\tiD,\tiD_{\tia}$ of $\ov{B_cC_b}$
are also the midpoints of $\ov{JJ'}$. 
The polar of $J$ is the line $A'A_1$ and the polar of $J'$ is the line $AA_1$.
If we consider the quadrangular involution $\sigma_{\QQ_3}$ on $\tia$ 
induced by the quadrangle $\QQ_3=\{A,B,C,A_1\}$, 
we have that $\sigma_{\QQ_3}$ sends $B_c,C_b,J$ 
into $\left(C_b\right)_{\tia}, \left(B_c\right)_{\tia}, J'_{\tia}$ 
respectively and vice versa. 
The composition $\rho_{\tia}\circ \sigma_{\QQ_3}$ 
coincides with the symmetry $\tau_{\tiD\tiD_{\tia}}$ of $\tia$ 
with respect to $\tiD$ and sends $J$ into $J'$. 
Thus, $\tiD,\tiD_{\tia}$ are the midpoints of $\ov{JJ'}$.

If $\TT$ is isosceles at $A$, the lines $a,a'$ and $\tia$ concur
at $A_0$, and the quadrangle $\QQ_4=\{A,A',B_0,C_0\}$
shows that $A_0$ and $AA'\cdot\tia=DD'\cdot \tia$ are the midpoints of
$\ov{B_cC_b}$.

\textbf{III}. Consider the quadrangle $\QQ_1=\{B,C,B',C'\}$. The quadrangular
involution
$\sigma_{\QQ_1}$ on $\tia$ sends
$$\left(B_c\right)_{\tia}=BC'\cdot\tia\,,\quad\tiB=BB'\cdot \tia\,,\quad J\,,$$
into
$$\left(C_b\right)_{\tia}=CB'\cdot \tia\,,\quad\tiC=CC'\cdot \tia\,,\quad
J'\,,$$
and vice versa. Thus, $\sigma_{\QQ_1}$ coincides with $\tau_{\tiD\tiD_{\tia}}$,
and we conclude that $\tiD,\tiD_{\tia}$ are the midpoints of $\ov{\tiB\tiC}$.
\end{proof}
\begin{remark}
By Lemma~\ref{lem:midpoints-quadrilateral-also-with-tangencies}, the statement
of the previous theorem remains true even if $\tia$ is tangent to $\Phi$.
\end{remark}

We will say that the triangle $\tiT$ is the \emph{magic triangle} of
$\TT$ and that the midpoints of the sides of $\tiT$ are the \emph{magic
midpoints} of $\TT$. By their construction, the magic triangle and the magic
midpoints of $\TT$ are also the magic triangle and the magic midpoints of
$\TT'$.

Theorem~\ref{thm:magic-midpoints} gives the key for defining a coherent
orientation on a
triangle in a purely projective way. The triangle $\TT$ is \emph{oriented} if we
have oriented each side of $\TT$ and each side of $\TT'$. Assume that $\TT$ is
oriented. Following our previous notation, let $D,E,F$ and $G,H,I$ be the
preferred midpoints and the preferred complementary midpoints of
$\ov{BC},\ov{CA},\ov{AB}$, respectively, and let $D',E',F'$ and $G',H',I'$ be
the preferred midpoints and the preferred complementary midpoints of
$\ov{B'C'},\ov{C'A'},\ov{A'B'}$, respectively.
\begin{definition}\label{def:coherently-oriented}
The triangle $\TT$ is \emph{coherently oriented} if (see Figure
\ref{Fig:magic_midpoints_spherical})
\begin{itemize}
\item $D,E,F$ are non-collinear and different from $A_0,B_0,C_0$ resp.;
\item $D',E',F'$ are non-collinear and different from $A_0,B_0,C_0$ resp.;
\end{itemize}
and there exist non-collinear magic midpoints $\tiD,\tiE,\tiF$ lying on
$\tia,\tib,\tic$ respectively such that
\begin{itemize}
\item $\tiD$ is the intersection point of $DD',HI$ and $H'I'$;
\item $\tiE$ is the intersection point of $EE',IG$ and $I'G'$;
\item $\tiF$ is the intersection point of $FF',GH$ and $G'H'$.
\end{itemize}
\end{definition}

In Figure~\ref{Fig:magic_midpoints_spherical} we have depicted $\TT$ as an
elliptic triangle. This is the only
situation where midpoints, complementary midpoints and magic midpoints of $\TT$
and $\TT'$  are all real points.

We will introduce some more properties about this construction. In particular,
we can complete a forgotten task of our to-do list: the proof of
Theorem~\ref{thm:DD'-EE'-FF'-concurrent}.

\begin{figure}
\centering
\includegraphics[width=0.75\textwidth]{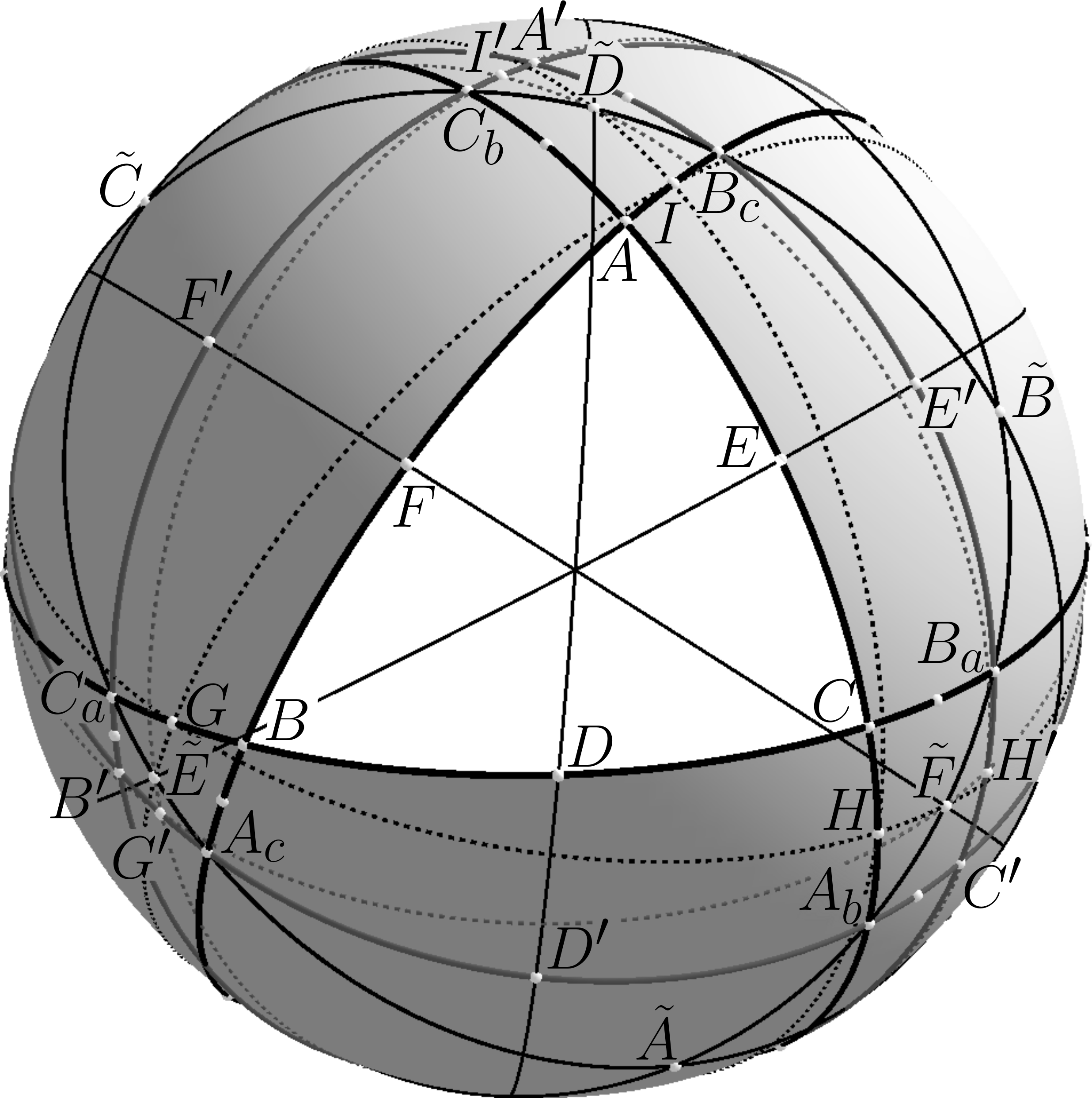}
\caption{coherently
oriented elliptic triangle}\label{Fig:magic_midpoints_spherical}
\end{figure}

\paragraph{Digression: the proof of Theorem~\ref{thm:DD'-EE'-FF'-concurrent}}
We need another classical
construction from projective geometry.
\begin{lemma}\label{lem:cross-ratio-coincides-Pappus-octogon}
Let $p,p'$ be two different projective lines, and let $X,Y,Z,T$ and
$X',Y',Z',T'$ be two tetrads of points on $p$ and on $p'$
respectively such that $(XYZT)=(X'Y'Z'T')$.
Then, the Pappus lines of the hexagons 
$XY'ZX'YZ'$, $XY'TX'YT'$, $XZ'TX'ZT'$ and
$YZ'TY'ZT'$ coincide.
\end{lemma}
The proof of this Lemma is left as an exercise. A look at \cite[p. 41]{Cox
Proj} would help.
      \begin{remark}\label{rem:DD'-Pappus-line}
      Assume that $\TT$ is coherently oriented. If $E,E_b,H,H_b$ and
      $F,F_c,I,I_c$ are taken as the tetrads of the previous Lemma, the resulting
      Pappus line is $DD'$ (see 
      Figure~\ref{Fig:magic-midpoints-hexagon-Pappus}).
      Similar constructions hold for $EE'$ and $FF'$.
      \end{remark}
\begin{proof}
      By Lemma~\ref{lem:DDa-midpoints-of-GGa}, the two tetrads are harmonic
      sets, and so they fullfil the hypothesis
      of Lemma~\ref{lem:cross-ratio-coincides-Pappus-octogon}. It suffices to
      remember that by the chosen notation it is 
      $EF_c\cdot E_bF=D$ and $HI\cdot
      H_bI_c=\tiD$.
\end{proof}

In Figure~\ref{Fig:magic-midpoints-hexagon-Pappus} we can see $DD'$ as the
Pappus line of the hexagon $EIHFH_bI_c$.
      \begin{figure}
	  \centering
	      \includegraphics[width=0.9\textwidth]
	      {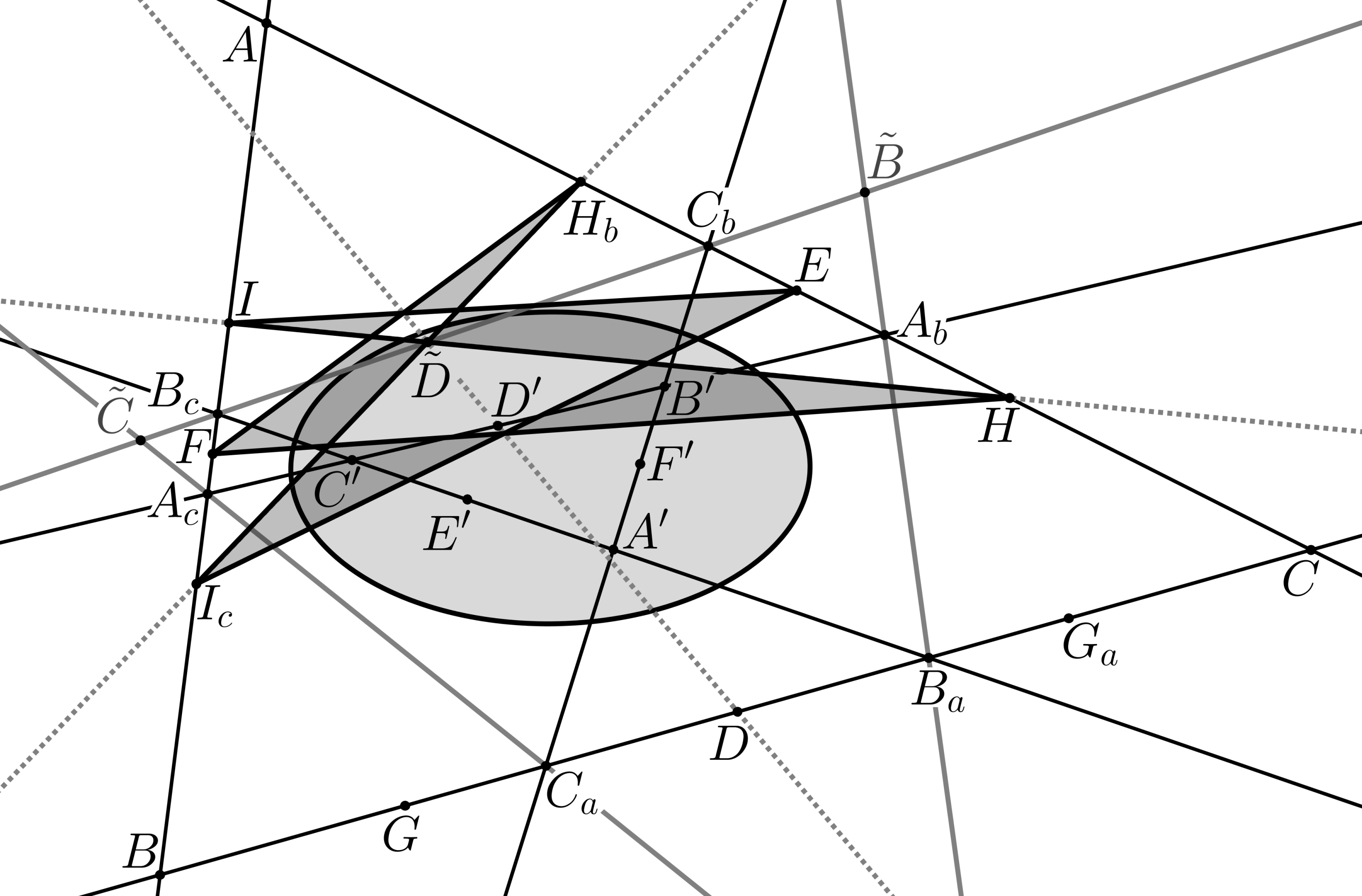}
	      \caption{$DD'$ as a Pappus line}
	  \label{Fig:magic-midpoints-hexagon-Pappus}
      \end{figure}

\noindent\begin{proof}[Proof of Theorem
\ref{thm:DD'-EE'-FF'-concurrent}]\label{proof:thm:DD'-EE'-FF'-concurrent}
We will divide the proof into some substeps, each of them presented as a
claim.

Let consider the lines $k=EI$ and $l=I_cD_a$ and the points
$K=a\cdot k$ and $L=b\cdot l$.
\begin{claim}\label{claim:DD'-EE'-FF'-concurrent-claim1}
The points $F,K,L$ are collinear  (see Figure
\ref{Fig:magic-midpoints-claim1-KL}).
\end{claim}
\begin{proof}[Proof of Claim~\ref{claim:DD'-EE'-FF'-concurrent-claim1}]
Let consider the quadrangle $\QQ=\{E,D_a,K,L\}$, and the involution
$\sigma_{\QQ}$
that it induces on $c$. We have that
\begin{align*}
EK\cdot c=I\,,\quad D_aL\cdot c=I_c\quad&\Rightarrow\quad
\sigma_{\QQ}(I)=I_c\,;\\
EL\cdot c=A\,;\quad D_aK\cdot c=B\quad&\Rightarrow\quad \sigma_{\QQ}(A)=B\,.
\end{align*}
This implies that $\sigma_{\QQ}$ must coincide with the symmetry
 $\tau_{FF_c}$ on
$c$ with respect to $F,F_c$. As $ED_a\cdot c=F$, 
it must be also $KL\cdot c=F$.
\end{proof}

\begin{figure}
\centering
\includegraphics[width=0.98\textwidth]
{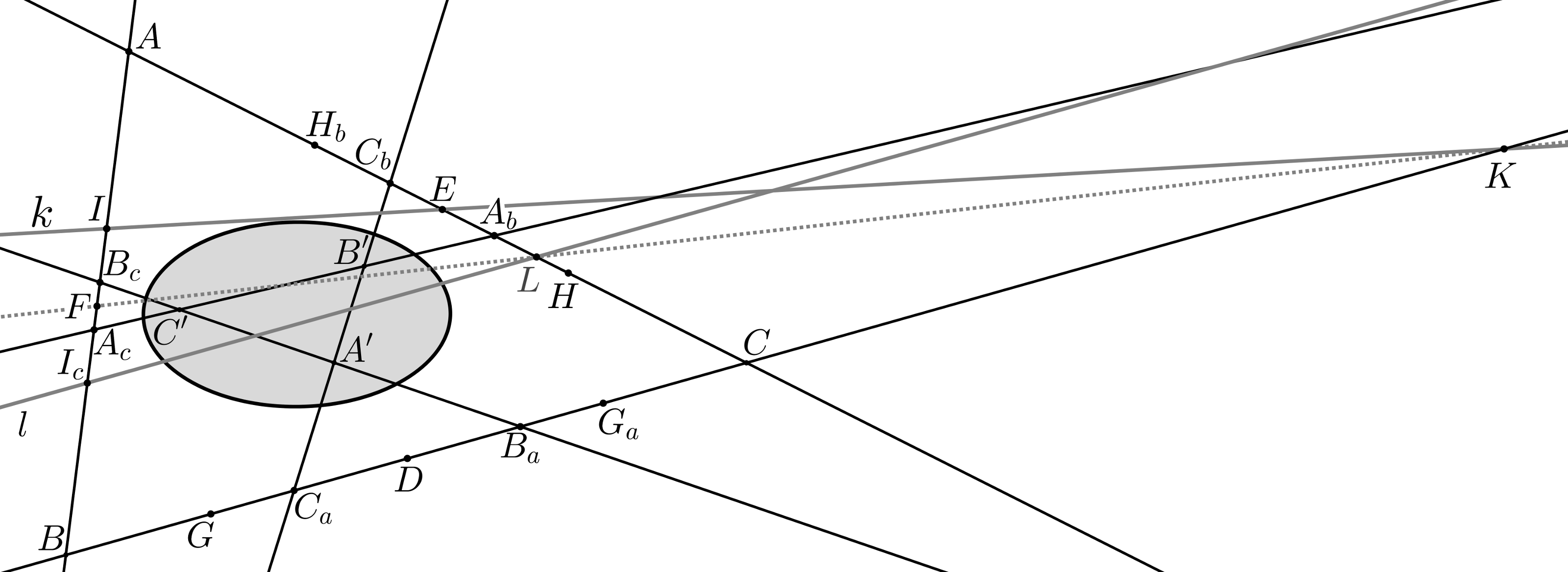}
\caption{The points $K,L$}\label{Fig:magic-midpoints-claim1-KL}
\end{figure}

Take the points $M=EI\cdot FH_b$ and $N=EF\cdot H_bI_c$.
\begin{claim}\label{claim:DD'-EE'-FF'-concurrent-claim2}
The points $B,M,N$ are collinear (see Figure
\ref{Fig:magic-midpoints-claim2-MN}).
\end{claim}
\begin{proof}[Proof of Claim~\ref{claim:DD'-EE'-FF'-concurrent-claim2}]
Take the triangles $\wt{BD_aI_c}$ and $\wt{MEH_b}$. The intersection points of
corresponding sides are:
\begin{align*}
BD_a\cdot ME=a\cdot k&=K\,;\\
D_aI_c\cdot EH_b= l\cdot b&=L\,;\text{ and}\\
I_cB\cdot H_bM=c\cdot FH_b&=F\,;
\end{align*}
which are collinear by Claim~\ref{claim:DD'-EE'-FF'-concurrent-claim1}. By
\hyperref[thm:Desargues]{Desargues' Theorem}, both triangles are perspective: the lines $BM,D_aE,I_cH_b$
are concurrent. In other words, the line $BM$ passes through
$D_aE\cdot I_cH_b=EF\cdot I_cH_b=N$.
\end{proof}
\begin{figure}
\centering
\includegraphics[width=0.98\textwidth]
{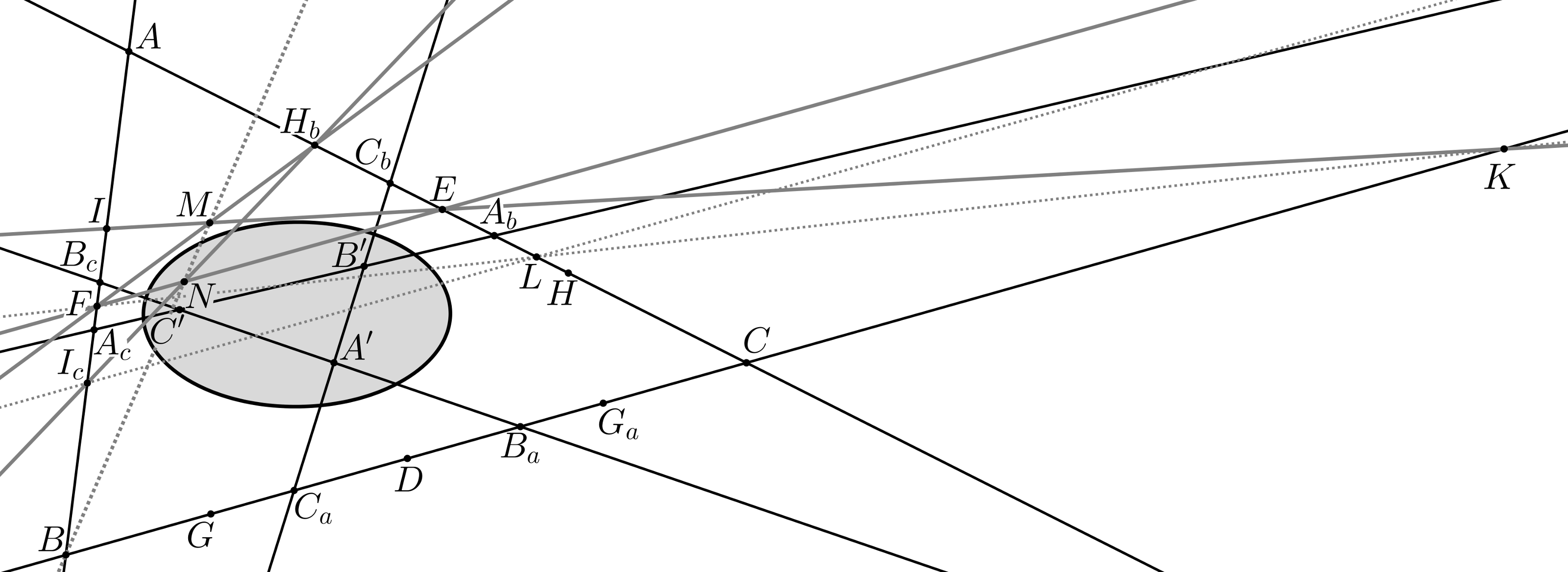}
\caption{The points $M,N$}\label{Fig:magic-midpoints-claim2-MN}
\end{figure}

Consider the point $S=DD'\cdot EE'$.
\begin{claim}\label{claim:DD'-EE'-FF'-concurrent-claim3}
The points $G_a,N,S$ are collinear.
\end{claim}
\begin{proof}[Proof of Claim~\ref{claim:DD'-EE'-FF'-concurrent-claim3}]
Consider the triangles $\wt{BFG_a}$ and $\wt{MES}$. By Remark
\ref{rem:DD'-Pappus-line}, the point $M$ lies on $\tid=DD'$ and the point
$M'=DI\cdot FG_a$ lies on $EE'$. This implies that the intersection points of
corresponding sides of both triangles are:
\begin{align*}
BF\cdot ME=c\cdot ME&=I\,;\\
FG_a\cdot ES&=M'\,;\text{ and}\\
G_aB\cdot SM=a\cdot \tid&=D\,.
\end{align*}
As we have just said, $M'\in DI$: by \hyperref[thm:Desargues]{Desargues' Theorem}, 
both triangles are
perspective. The lines $BM, FE,G_aS$ are concurrent, but by Claim
\ref{claim:DD'-EE'-FF'-concurrent-claim2} the intersection point $BM\cdot FE$
is $N$. This completes the proof.
\end{proof}

We can repeat step by step the proof of the previous claim but taking as
starting
point $S'=DD'\cdot FF'$ and considering the triangles $\wt{BEG_a}$ and
$\wt{MFS'}$ in order to conclude that the points $G_a,N,S'$ are collinear, and
this implies that $S$ and $S'$ must coincide. This completes the proof of
Theorem 
\ref{thm:DD'-EE'-FF'-concurrent}
\end{proof}

\begin{flushright}
\textbf{End of the digression}
\end{flushright}

We will use some of the notation from the proof of
Theorem 
\ref{thm:DD'-EE'-FF'-concurrent} in the proof of this small
lemma.

\begin{lemma}\label{lem:D-Da-midpoints}
The points $D,D_a$ are the midpoints of the segment whose endpoints are
$J_1:=a\cdot HI$ and $J_2=a\cdot H_bI_c$.
\end{lemma}
\begin{proof}
Consider the points $N=EF\cdot H_bI_c$ and $N'=EF\cdot HI$. Assume by
simplicity that $N$ and $N'$ are different points (the case $N=N'$ would follow
as a limit case). In the proof of Theorem~\ref{thm:DD'-EE'-FF'-concurrent} we
have seen that the point $N$ is collinear 
with $G_a$ and
$S$, where 
$S$ is the intersection point of the lines
$DD'$, $EE'$ and $FF'$. In an exactly similar way, it can be proved that $N'$
is collinear 
with $G$ and $S$.

Let recover the magic midpoint $\tiD=HI\cdot H_bI_c$, and consider the
quadrangle $\QQ=\{\tiD,Q,N,N'\}$ and the involution $\sigma_{\QQ}$ that it
induces
on $a$. We have that
$$\sigma_{\QQ}(D)=\sigma_{\QQ}(\tiA Q\cdot a)=NN'\cdot a=EF\cdot a=D_a\,,$$
and that $\sigma_{\QQ}(J_1)=G_a$ and that $\sigma_{\QQ}(J_2)=G$. Therefore,
$$(J_1 J_2 D D_a)=(G_a G D_a D)= -1\,.$$
The result now follows from Lemma~\ref{lem:midpoints-harmonic-UV-AB}.
\end{proof}
\begin{remark}
$D,D_a$ are also the midpoints of  the segment whose endpoints are $a\cdot
HI_c$ and $a\cdot H_bI$.
\end{remark}
\begin{remark}\label{rem:D-Da-non-collinear-I-Ic-H-Hb}
None of the points $D,D_a$ can be collinear with two of the
points $H,H_b,I,I_c$.
\end{remark}

\begin{figure}
\centering
\includegraphics[width=0.9\textwidth]
{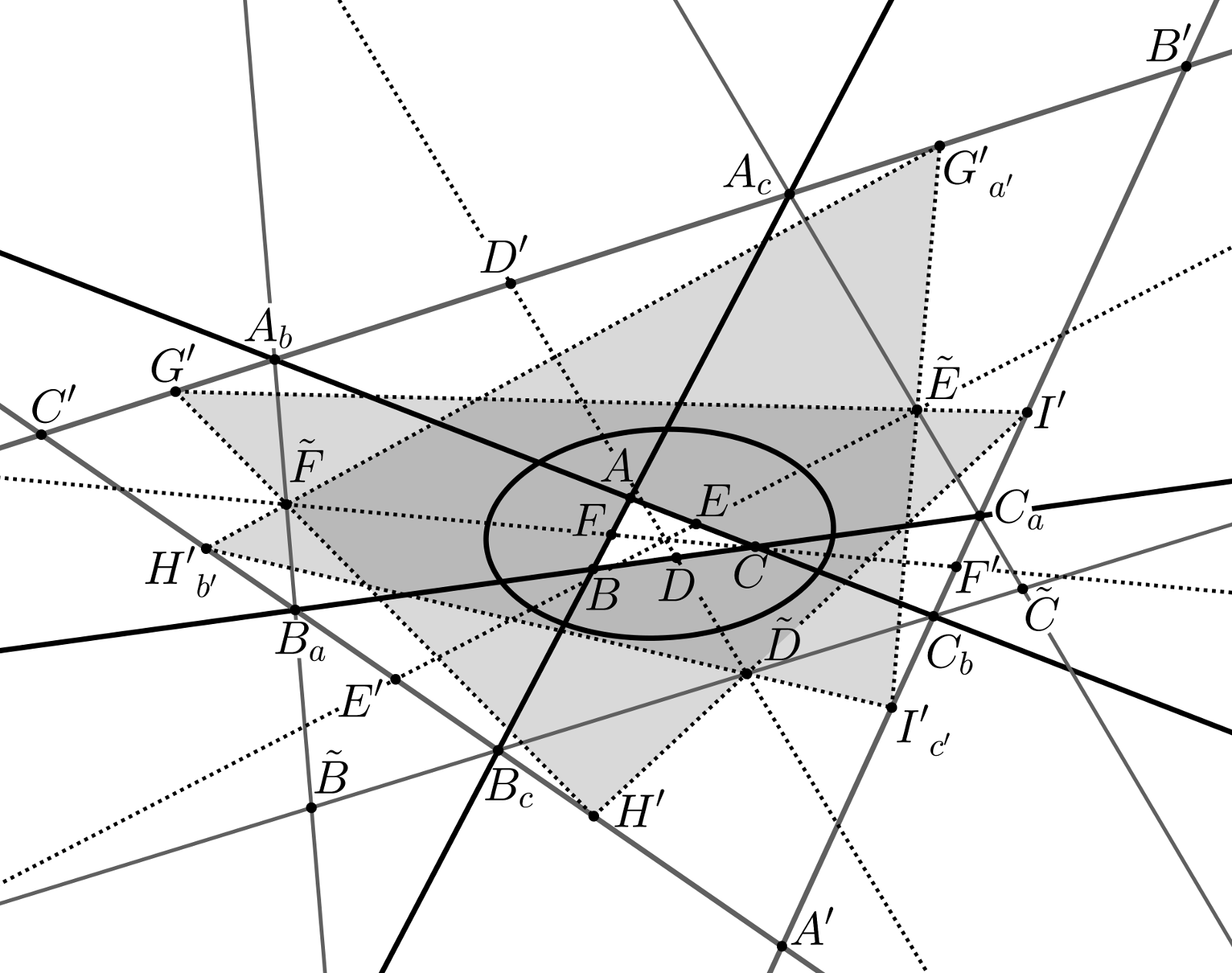}
\caption{Coherently oriented projective triangle}
\label{fig:cosine-rule-01-the-triangle}
\end{figure}

Definition~\ref{def:coherently-oriented} allows us to develop a complete
trigonometry for
generalized triangles. Along the next sections, we will assume that the
triangle $\TT$ is not right-angled and that it is coherently oriented, 
and we will use for the preferred
midpoints and complementary midpoints of $\TT$ and $\TT'$ the same notation as
in Definition~\ref{def:coherently-oriented}. A good exercise would be to
rewrite \S\ref{sec:Trigonometry-for-right-angled} using the functions
$\cc,\ss$ in these terms.

\section{The general law of sines}\label{sec:law-of-sines}

\begin{theorem}[Projective Law of sines]
If $\TT$ is coherently oriented, then
    \begin{equation}
	\dfrac{\ss(AB)}{\ss(A'B')}=\dfrac{\ss(BC)}{\ss(B'C')}=\dfrac{\ss(CA)}{\ss(C'A')}
	\label{eq:projective-law-of-sines}
    \end{equation}
\end{theorem}
\begin{proof}
Assume for instance that $\TT$ is the triangle depicted in 
Figure~\ref{fig:cosine-rule-01-the-triangle}.
If we apply \hyperref[cor:Menelaus]{Menelaus' Projective Formula} 
to the Menelaus' configuration
with triangle $\wt{AB_cC_b}$ and transversals $a$ and $HI$, we obtain
$$(AB_c B I)(B_cC_b J \tiD)(C_bA C H)=1\,,$$
where $J=a\cdot\tia$ as in \S\ref{sec:magic-midpoints-oriented-triangles}. 
Note that $\ss(AB)=(AB_cBI)$
and that $\ss(AC)=(AC_bCH)$. 
Therefore, the previous equality is equivalent to
$$\dfrac{\ss(AC)}{\ss(AB)}=(C_bB_cJ\tiD)\,.$$
On the other hand if we apply 
\hyperref[cor:Menelaus]{Menelaus' Projective Formula} to 
the Menelaus' configuration
with triangle $\wt{A'B_cC_b}$ and transversals $a'$ and $H'I'$, we obtain
$$(A'B_c C' H')(B_cC_b J' \tiD)(C_bA' B' I')=1\,,$$
where $J'=a'\cdot \tia$, and this last formula is equivalent to
$$\dfrac{\ss(A'C')}{\ss(A'B')}=(C_bB_cJ'\tiD)\,.$$
By the proof of Theorem~\ref{thm:magic-midpoints}, if $\TT$ is not 
isosceles at $A$ we know that $\tiD,\tiD_{\tia}$
are the midpoints of $\ov{JJ'}$. In this case, as $\tiD,\tiD_{\tia}$ are 
also the midpoints of $\ov{B_cC_b}$,
the simmetry $\tau_{\tiD\tiD_{\tia}}$ on $\tia$ with respect to $\tiD$ gives
$$(C_bB_cJ\tiD)=(B_cC_bJ'\tiD)=\dfrac{1}{(C_bB_cJ'\tiD)}\,.$$
Thus
$$\dfrac{\ss(AC)}{\ss(AB)}=\dfrac{\ss(A'B')}{\ss(A'C')}\,.$$
This argument also works If $\TT$ is isosceles at $A$, 
because in this case it is $(C_bB_cJ\tiD)=(C_bB_cJ'\tiD)=-1$.

As, $\ss(AC)=-\ss(CA)$ and $\ss(A'C')=-\ss(C'A')$, we obtain
$$\dfrac{\ss(CA)}{\ss(C'A')}=\dfrac{\ss(AB)}{\ss(A'B')}\,.$$
The rest of equalities if~\eqref{eq:projective-law-of-sines} are proved 
in the same way.
\end{proof}

\section{The general law of cosines}\label{sec:unsquared-cosine-rule}
For proving our main result, we will need an extra help from
some
other classic theorems not mentioned before. 
After the previous chapters, a reader familiar with 
the classical theorems of affine and projective geometry would 
have the feeling that ``someone's missing''.
\begin{theorem}[Ceva's Theorem\footnote{Although this theorem is historically
attributed to the seventeenth-century italian
matematician Giovanni Ceva, it is known that it was proved before by
the arab mathematician
Yusuf Al-Mu'taman ibn H\H{u}d, king of Zaragoza in the eleventh
century.}]
Let $\TT=\wt{XYZ}$ be a projective triangle, and let $X_1,Y_1,Z_1$ be three
points on the lines $YZ,ZX,XY$ respectively. Let $r$ be a line not incident with
$X$, $Y$, or $Z$, and consider the points $X_0=r\cdot YZ$, $Y_0=r\cdot ZX$, and
$Z_0=r\cdot XY$ (Figure~\ref{fig:Ceva-Van-Aubel}). The lines $XX_1,YY_1,ZZ_1$
are concurrent if and only if
(compare~\eqref{eq:Menelaus_Projective})
\begin{equation}\label{eq:ceva-projective-formula}
(XYZ_1Z_0)(YZX_1X_0)(ZXY_1Y_0)=-1
\end{equation}
\end{theorem}
We make a projective interpretation of Ceva's Theorem exactly as we did with
\hyperref[thm:Menelaus-affine]{Menelaus' Theorem} 
in \S\ref{sec:Menelaus-Theorem}. It is usually said
that Menelaus' and Ceva's
theorems are \emph{dual} to each other, but perhaps it should be said
that
they are \emph{harmonic} to each other because harmonic conjugacy 
provides the equivalence between
both theorems:
\begin{lemma}
Let $\TT=\wt{XYZ}$ be a projective triangle, let $X_1,Y_1$ and $Z_1$ be three
points on the lines $YZ,ZX$ and $XY$ respectively, and let $X_2,Y_2$ 
and $Z_2$ be their harmonic conjugates with respect to 
$Y$ and $Z$, $Z$ and $X$ and $X$ and $Y$, respectively. 
The lines $XX_1,YY_1$ and $ZZ_1$ are concurrent if and only if
the points $X_2,Y_2$ and $Z_2$ are collinear.
\end{lemma}
The proof of this lemma follows from 
Lemma \ref{lem:harmonic-sets-on-the-sides}, and it is left as an exercise.

\begin{figure}
\centering
\includegraphics[width=0.95\textwidth]
{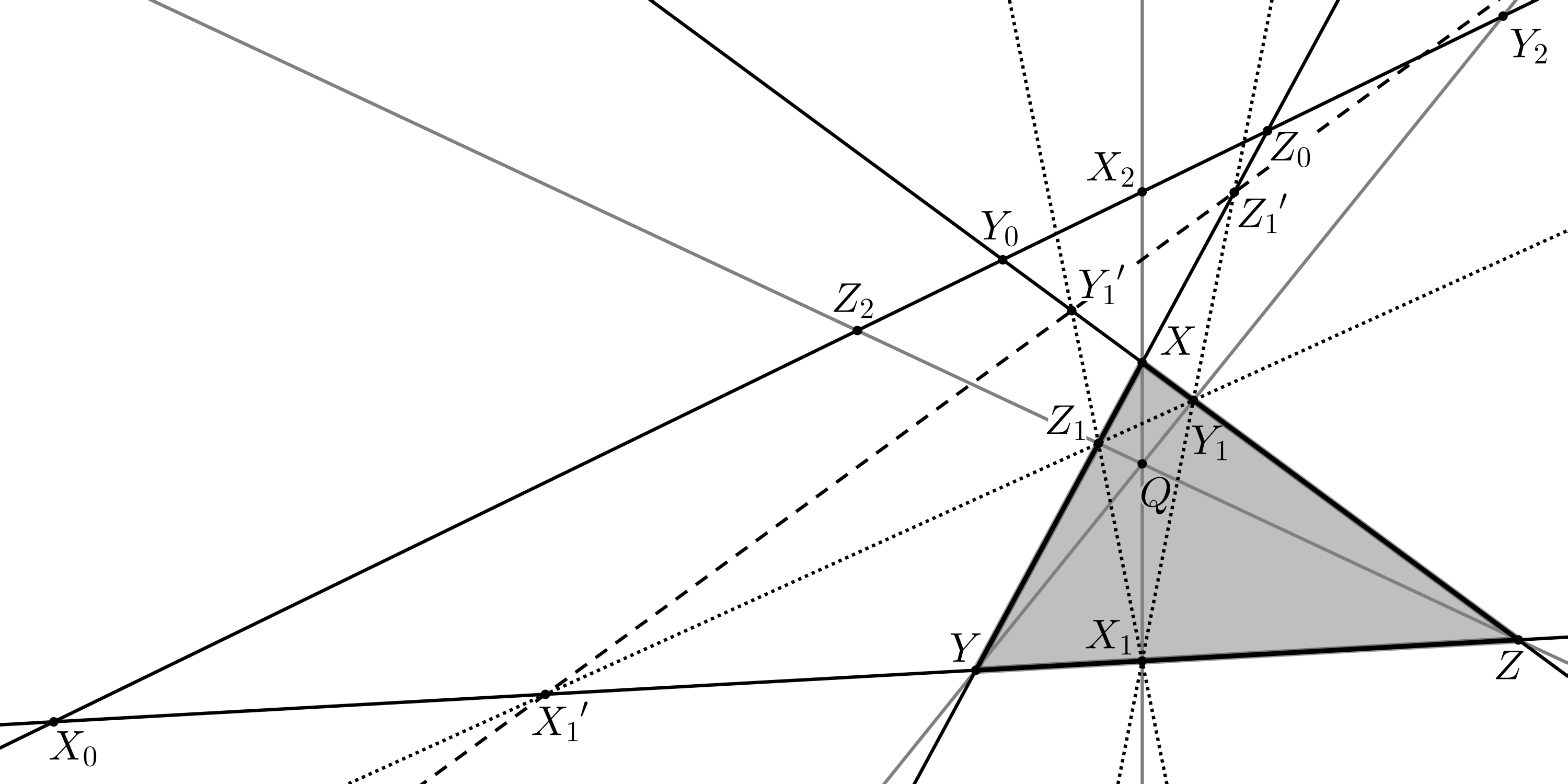}
\caption{Ceva's and Van Aubel's Theorems}
\label{fig:Ceva-Van-Aubel} 
\end{figure}

We will use Ceva's Theorem for proving another theorem, closely related to
Menelaus' and Ceva's ones. As before,
we propose a projective interpretation of the classical affine theorem.
    \begin{theorem}[Van Aubel's Theorem on cevians]\label{thm:Van-Aubel}
	Let $\TT=\wt{XYZ}$ be a projective triangle, and let $X_1,Y_1,Z_1$ be three
	points on the lines $YZ$, $ZX$, $XY$ respectively. Let $r$ be a line not
	incident with
	$X$, $Y$, or $Z$, and consider the points $X_2=r\cdot XX_1$, $Y_0=r\cdot XZ$,
	and
	$Z_0=r\cdot XY$  (Figure~\ref{fig:Ceva-Van-Aubel}). If the lines
	$XX_1,YY_1,ZZ_1$ are concurrent and $Q$ is their
	concurrence point, then
	    \begin{equation}\label{eq:Van-Aubel-Projective-Fromula}
		(XX_1QX_2)=(X Y Z_1 Z_0)+(X Z Y_1 Y_0)\,.
	    \end{equation}
    \end{theorem}
    \begin{proof}
Consider the points
$$X_0=r \cdot YZ\,,\quad Y_2=r\cdot YY_1\,,\quad\text{and}\quad Z_2=r\cdot
ZZ_1\,.$$
By Ceva's Theorem, identity~\eqref{eq:ceva-projective-formula} holds.
By projecting from $Y$, we have
$$(XX_1QX_2)=(Z_0X_0Y_2X_2)\,,$$
and by projecting from $Q$ we have
\begin{align*}
(XYZ_1Z_0)&=(X_2Y_2Z_2Z_0)\,,\\
(YZX_1X_0)&=(Y_2Z_2X_2X_0)\,,\\
(ZXY_1Y_0)&=(Z_2X_2Y_2Y_0)\,.
\end{align*}
Ceva's identity~\eqref{eq:ceva-projective-formula} implies that
$$(Y_2Z_2X_2X_0)=\dfrac{-1}{(X_2Y_2Z_2Z_0)(Z_2X_2Y_2Y_0)}$$
By applying cross-ratio identities~\eqref{eq:cross_ratio_identities}, we
obtain
\begin{align*}
(Z_0X_0Y_2X_2)&=(X_2Y_2X_0Z_0)=(X_2Y_2Z_2Z_0)(X_2Y_2X_0Z_2)=\\
&=(X_2Y_2Z_2Z_0)[1-(Z_2Y_2X_0X_2)]=\\
&=(X_2Y_2Z_2Z_0)[1-(Y_2Z_2X_2X_0)]=\\
&=(X_2Y_2Z_2Z_0)[1-\dfrac{-1}{(X_2Y_2Z_2Z_0)(Z_2X_2Y_2Y_0)}]=\\
&=(X_2Y_2Z_2Z_0)+\dfrac{1}{(Z_2X_2Y_2Y_0)}=(X_2Y_2Z_2Z_0)+(X_2Z_2Y_2Y_0)=\\
&=(XYZ_1Z_0)+(XZY_1Y_0)\,.
\end{align*}
    \end{proof}

Now we are ready to state and prove our main result:

\begin{theorem}[Projective law of cosines]\label{thm:cosine-rule}
    If the triangle $\TT$ is coherently oriented, then:
	\begin{equation}\label{eq:projective-law-of-cosines}
	    \cc(BC)=-\ss(AB)\ss(CA)\cc(B'C')-\cc(AB)\cc(CA)\,.
	\end{equation}
\end{theorem}

\begin{proof}

Instead of~\eqref{eq:projective-law-of-cosines}, we will prove the equivalent dual formula
    \begin{equation}\label{eq:projective-law-of-cosines-dual}
	\cc(B'C')=-\ss(A'B')\ss(C'A')\cc(BC)-\cc(A'B')\cc(C'A')\,.
    \end{equation}

We can assume that $\TT$ is the triangle depicted in Figure
\ref{fig:cosine-rule-01-the-triangle}. In that figure, $\TT$ is depicted as a
hyperbolic triangle, and we have depicted also the polar triangle $\TT'$, the
magic triangle $\tiT$, the preferred midpoints of $\TT$ and $\TT'$, the
complementary midpoints of $\TT'$ and the magic midpoints $\tiD,\tiE,\tiF$ of
$\TT$. The figure misses the
complementary midpoints of $\TT$ because they are imaginary.

We have that
\begin{align*}
\cc(A'B')&=\cc(B'A')=(B'A'C_aF')\overset{\rho_{c'}}{=}(C_bC_aA'F'_{c'})=\\
&=-(C_bC_aA'F')=-(A'F'C_bC_a)\,,
\end{align*}
and so
$$\cc(A'B')\cc(A'C')=\cc(A'B')\cc(C'A')=-(A'F'C_bC_a)(C'A'B_aE')\,.$$
By applying 
\hyperref[cor:Menelaus]{Menelaus' Projective Formula}~\eqref{eq:Menelaus_Projective}
to
the Menelaus configuration with triangle $\wt{C'A'F'}$ and
transversals $B_aC_b$ and $E'C_a$, we get
$$(C'A'B_aE')(A'F'C_bC_a)(F'C'ZY)=1\,,$$
where $Z=B_aC_b\cdot F'C'$ and $Y=E'C_a\cdot F'C'$ (Figure
\ref{fig:cosine-rule-02-1st-Menelaus}), and therefore
$$\cc(A'B')\cc(C'A')=-(C'F'ZY)\,.$$

\begin{figure}
\centering
\includegraphics[width=0.9\textwidth]
{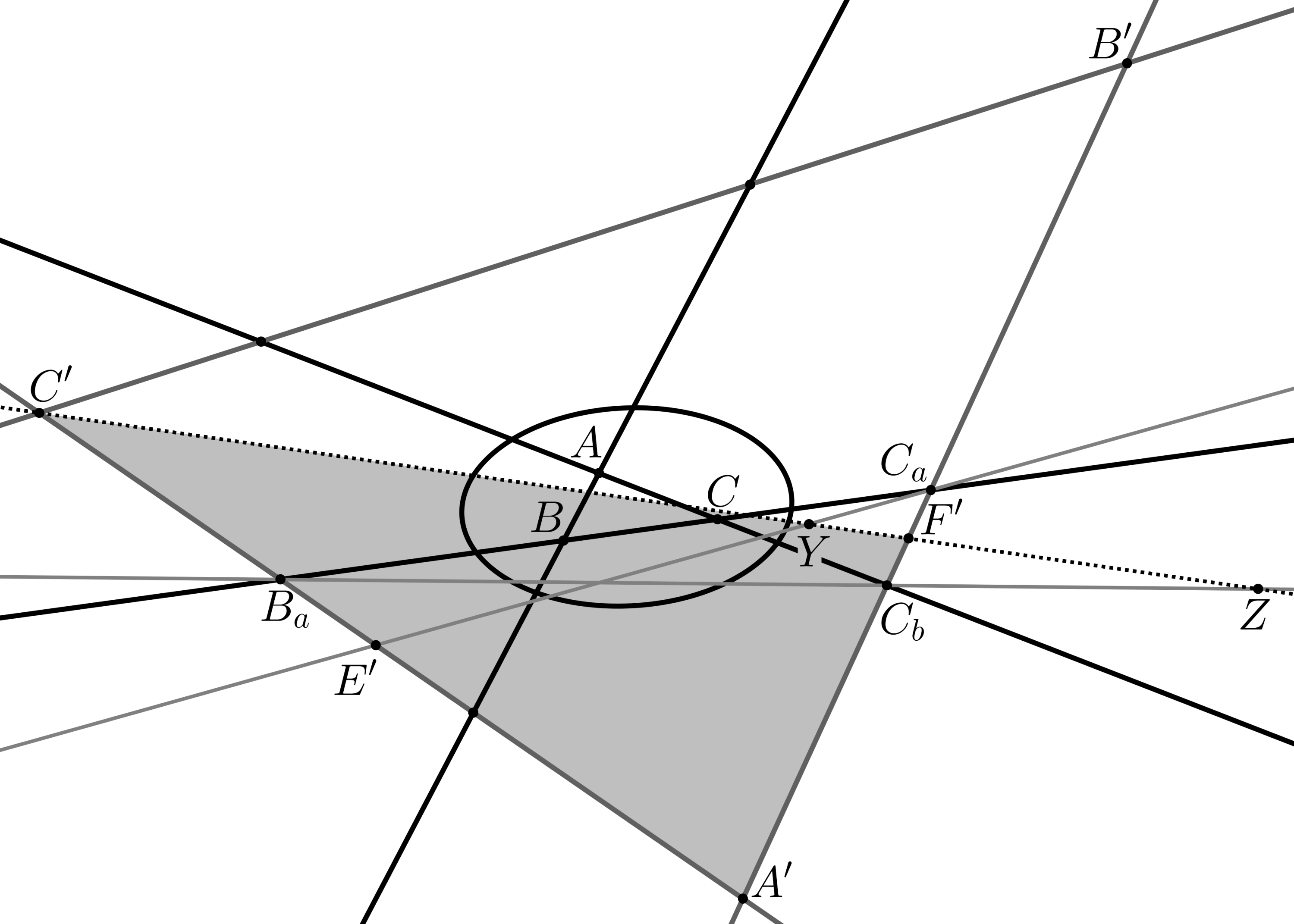}
\caption{\hyperref[cor:Menelaus]{Menelaus' Projective Formula} 1, the points $Y,Z$}
\label{fig:cosine-rule-02-1st-Menelaus}
\end{figure}

On the other hand, we have that
$$\cc(BC)=(BCC_aD)=(B_aC_aCD_a)=-(B_aC_aCD)\,,$$
and so
$$-\ss(C'A')\cc(BC)=(C'B_aA'H')(B_aC_aCD)\,.$$
We apply \hyperref[cor:Menelaus]{Menelaus' Projective Formula} to
the Menelaus configuration with triangle  $\wt{C'B_aC_a}$ and
transversals $A'C$ and $H'D$ in order to obtain
$$(C'B_aA'H')(B_aC_aCD)(C_aC'XW)=1\,,$$
where $X=A'C\cdot C_aC'$ and $W=H'D\cdot C_aC'$ 
(Figure~\ref{fig:cosine-rule-03-2nd-Menelaus}). Therefore, it is
$$-\ss(C'A')\cc(BC)=(C'C_aXW)\,.$$

\begin{figure}
\centering
\includegraphics[width=0.9\textwidth]
{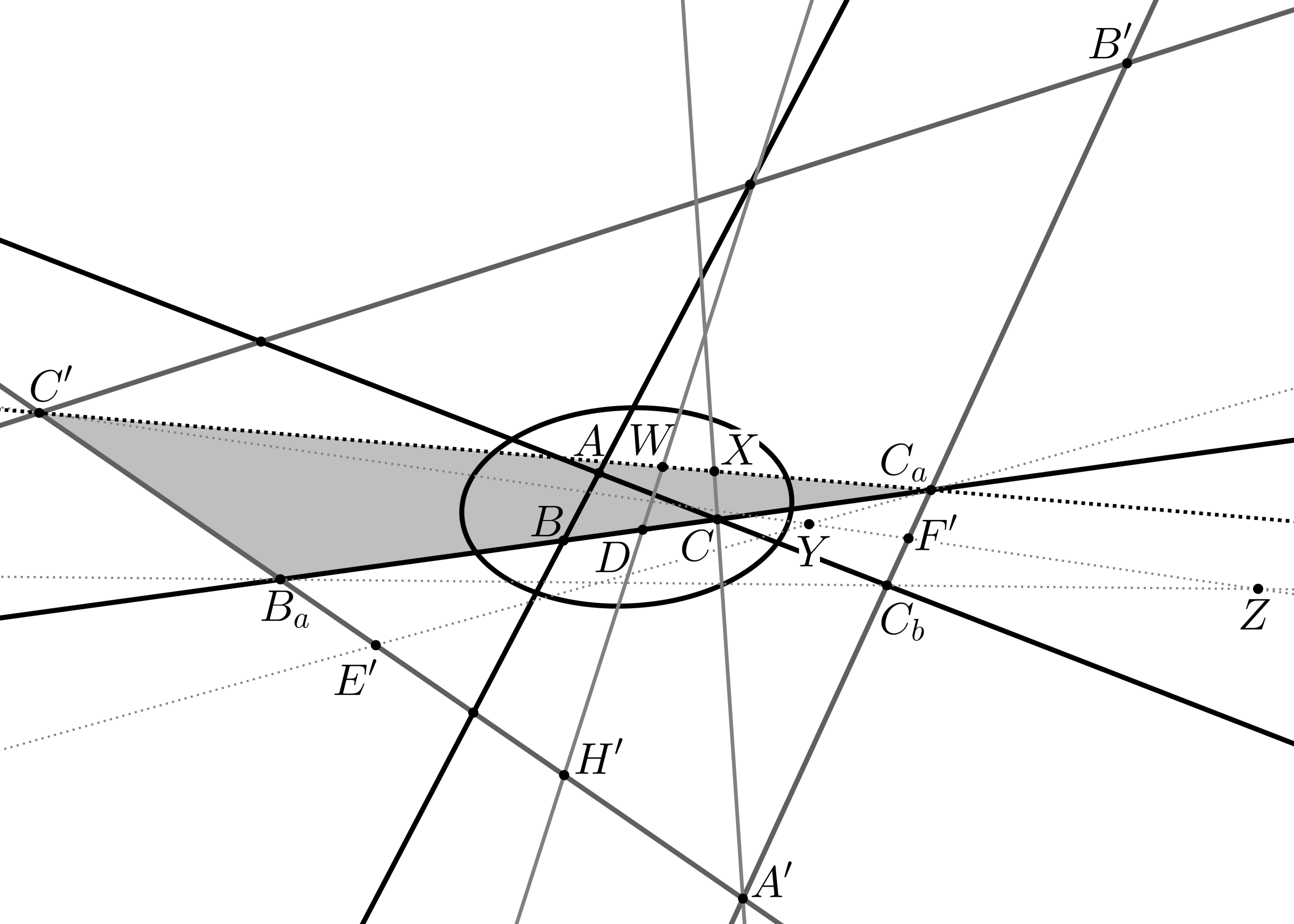}
\caption{\hyperref[cor:Menelaus]{Menelaus' Projective Formula} 2, the points $W,X$}
\label{fig:cosine-rule-03-2nd-Menelaus}
\end{figure}

Finally, we have that
$$
\ss(A'B')=(A'C_bB'I')\overset{\tau_{I'I'_{c'}}}{=}(C_bA'C_aI')=(C_aI'C_bA')\,,
$$
and then
$$-\ss(A'B')\ss(C'A')\cc(BC)=(C_aI'C_bA')(C'C_aXW)\,.$$
By applying \hyperref[cor:Menelaus]{Menelaus' Projective Formula} to
the Menelaus configuration with triangle $\wt{C'C_aI'}$ and
transversals $XC_b$ and $WA'$ we obtain
$$(C'C_aXW) (C_aI'C_bA')(I'C'TS)=1\,,$$
where $T=XC_b\cdot I'C'$ and $S=WA'\cdot I'C'$ (Figure
\ref{fig:cosine-rule-04-3rd-Menelaus}). Therefore, it is
$$-\ss(A'B')\ss(C'A')\cc(BC)=(C'I'TS)\,.$$

\begin{figure}
\centering
\includegraphics[width=0.9\textwidth]
{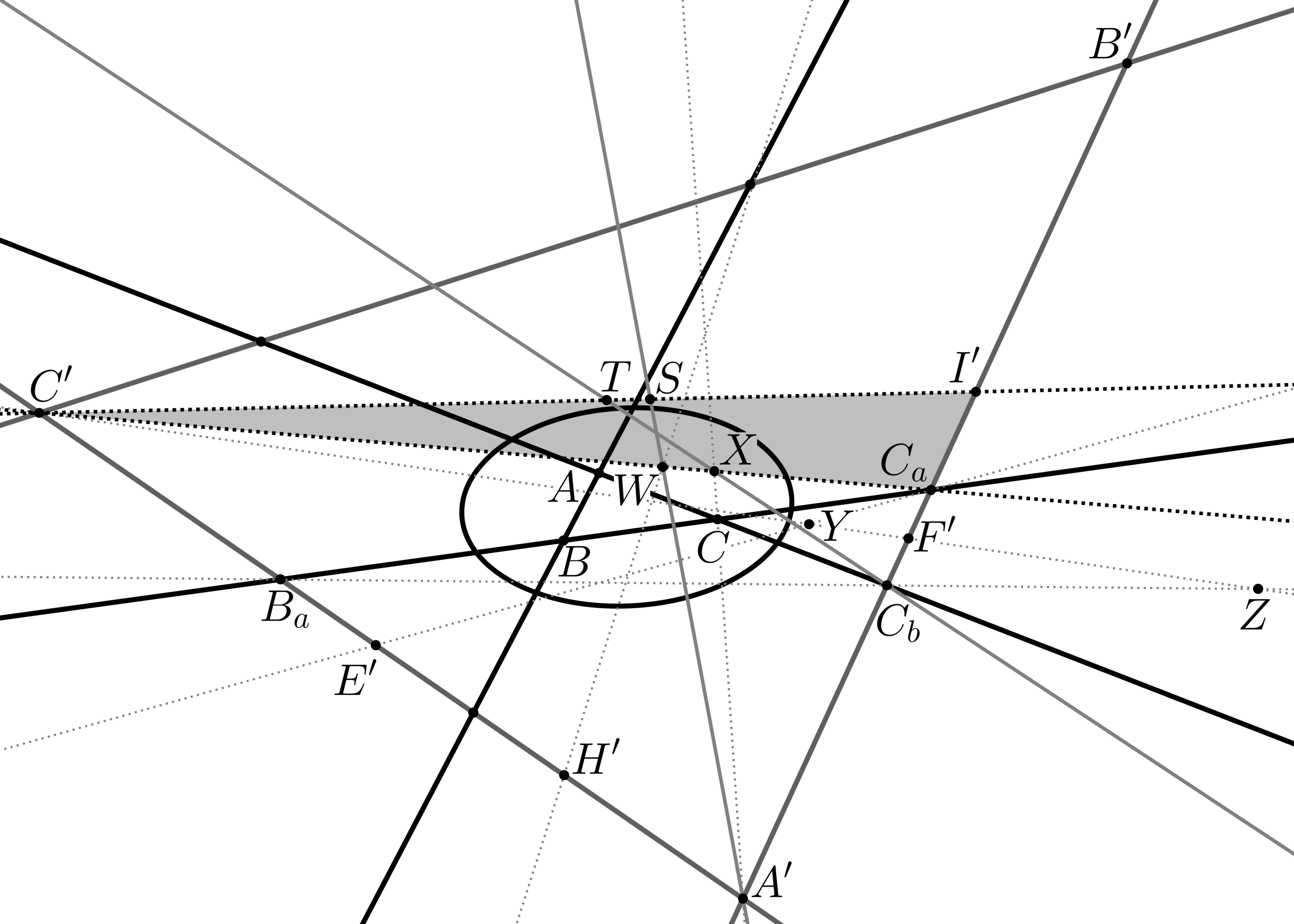}
\caption{\hyperref[cor:Menelaus]{Menelaus' Projective Formula} 3, the points $S,T$}
\label{fig:cosine-rule-04-3rd-Menelaus}
\end{figure}

The proof of Theorem \ref{thm:cosine-rule} relies now in the proof of the
identity
$$(C'B'A_bD')=(C'I'TS)+(C'F'ZY)\,.$$
We will need to prove some small claims before doing so.

\begin{figure}
\centering
\includegraphics[width=0.98\textwidth]
{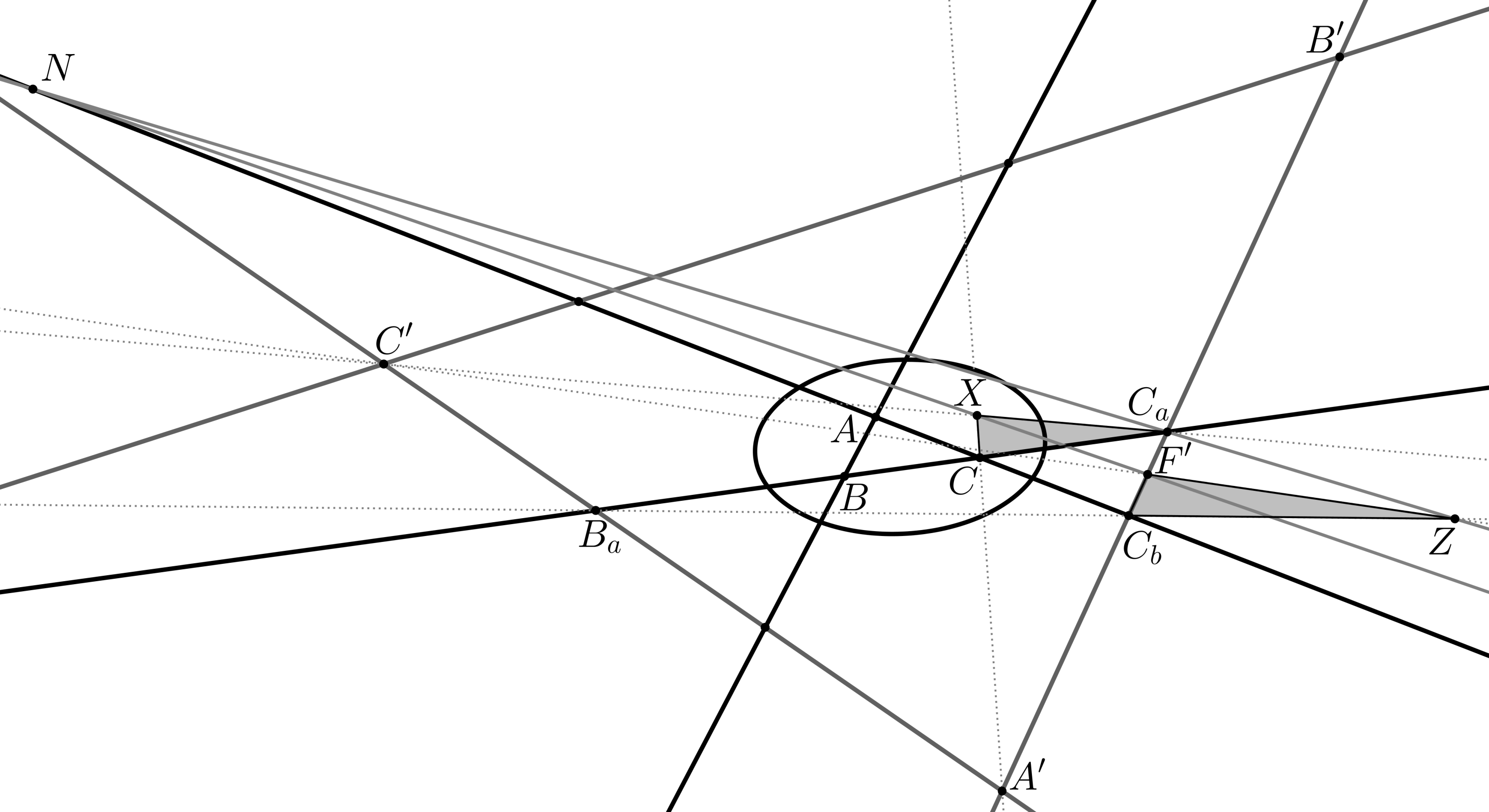}
\caption{The point $N$ lying at $b$}
\label{fig:cosine-rule-05-Claim-1-N}
\end{figure}

\begin{claim}\label{claim:concurrency-b-ZI'-F'T}
The intersection point $R$ of the lines $ZI'$ and $F'T$ lies on $b$.
\end{claim}
\begin{proof}[Proof of Claim~\ref{claim:concurrency-b-ZI'-F'T}]
Consider the triangles $\wt{CC_aX}$ and $\wt{C_bZF'}$. The intersection points
of the sides of both triangles are
$$CC_a\cdot C_bZ=B_a\,,\quad C_aX\cdot ZF'=C'\,,\quad XC\cdot F'C_b=A'\,,$$
which are collinear. By \hyperref[thm:Desargues]{Desargues' Theorem}, 
the lines $CC_b=b$,
$C_aZ$ and $XF'$ are concurrent: let $N$ be their intersection point (Figure
\ref{fig:cosine-rule-05-Claim-1-N}).

\begin{figure}
\centering
\includegraphics[width=0.98\textwidth]
{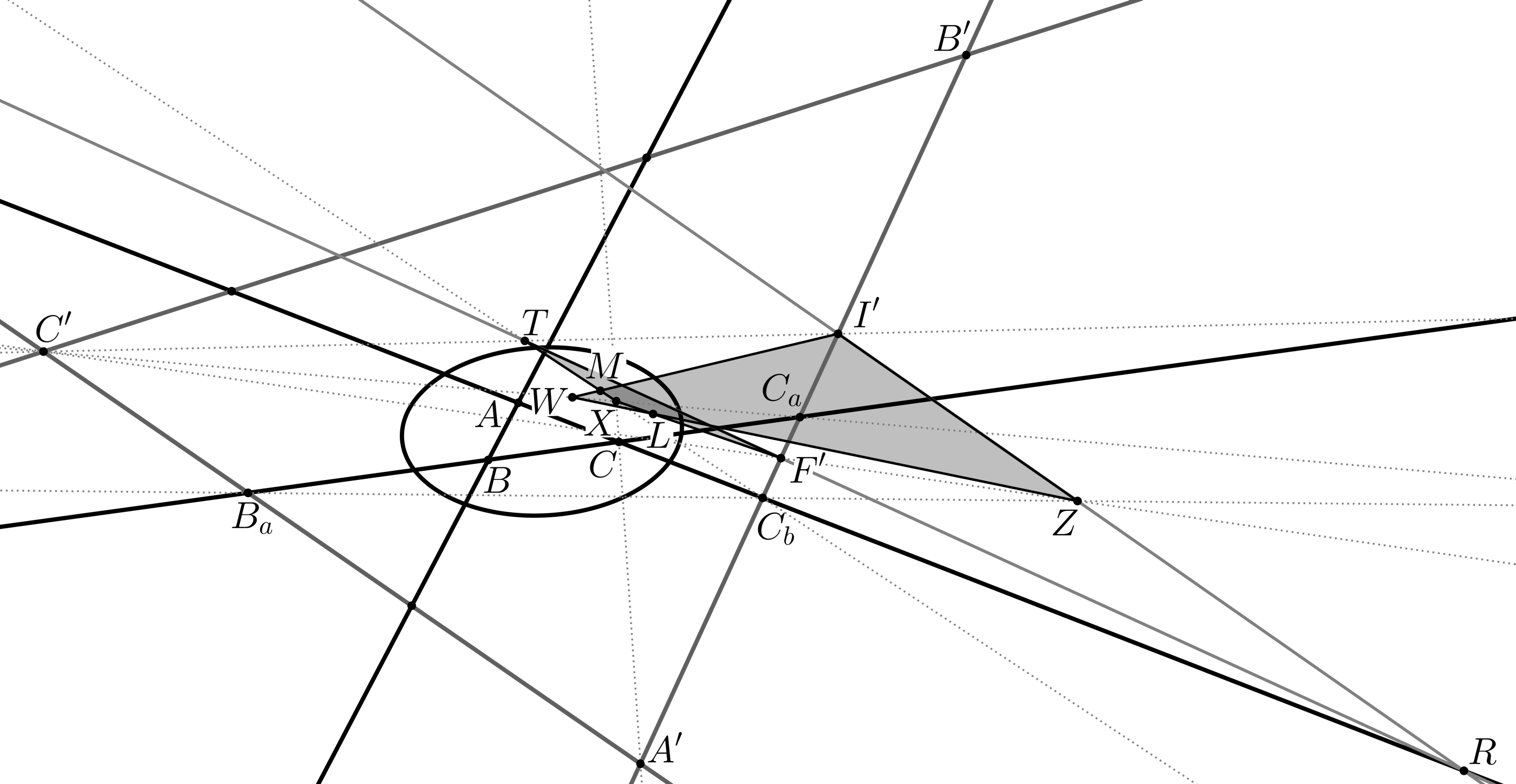}
\caption{The points $M, L$ and $R$}
\label{fig:cosine-rule-05-Claim-1-R}
\end{figure}

The triangles $\wt{WI'Z}$ and $\wt{XTF'}$ are perspective from $C'$ (Figure
\ref{fig:cosine-rule-05-Claim-1-R}), therefore the points $M=I'W\cdot TX$,
$L=ZW\cdot XF'$ and $R$ are
collinear.

Finally, the triangles $\wt{MLX}$ and $\wt{I'ZC_a}$ are perspective from $W$
(Figure~\ref{fig:cosine-rule-05-Claim-1-R}) and so the points
$$R=ML\cdot I'Z\,,\quad N=LX\cdot ZC_a \,,\quad C_b=XM\cdot C_aI'$$
are collinear.
\end{proof}

\begin{claim}\label{claim:concurrency-C'F'-CaD'-CbE'}
The lines $C'F'$, $C_aD'$ and $C_bE'$ are concurrent. The lines $C'F'$, $C_bD'$
and $C_aE'$ are also concurrent. In other words, the line $C'F'$ is a diagonal
line of the quadrangle $\{C_a,C_b,E',D'\}$.
\end{claim}
\begin{proof}[Proof of Claim~\ref{claim:concurrency-C'F'-CaD'-CbE'}]
Let $Y^*$ be the intersection point of the lines $C_aE'$ and $C_bD'$, and
consider
the quadrilateral $\QQ=\{C',D',Y^*,E'\}$. The quadrangular involution
$\sigma_{\QQ}$ on $c'$ sends the points $A'$ and $B'$ into the
points $C_b$ and $C_a$ respectively and vice versa, and so it coincides with
$\tau_{I'I'_{c'}}$. Then,
$$C'Y^*\cdot c'=\tau_{I'I'_{c'}}(D'E'\cdot c')=\tau_{I'I'_{c'}}(F'_{c'})=F'\,.$$
Therefore, $Y^*$ must coincide with $Y$ (Figure
\ref{fig:cosine-rule-05-Claim-2}).

The rest of the Claim can be proved in a similar way.
\end{proof}

\begin{figure}
\centering
\includegraphics[width=0.8\textwidth]
{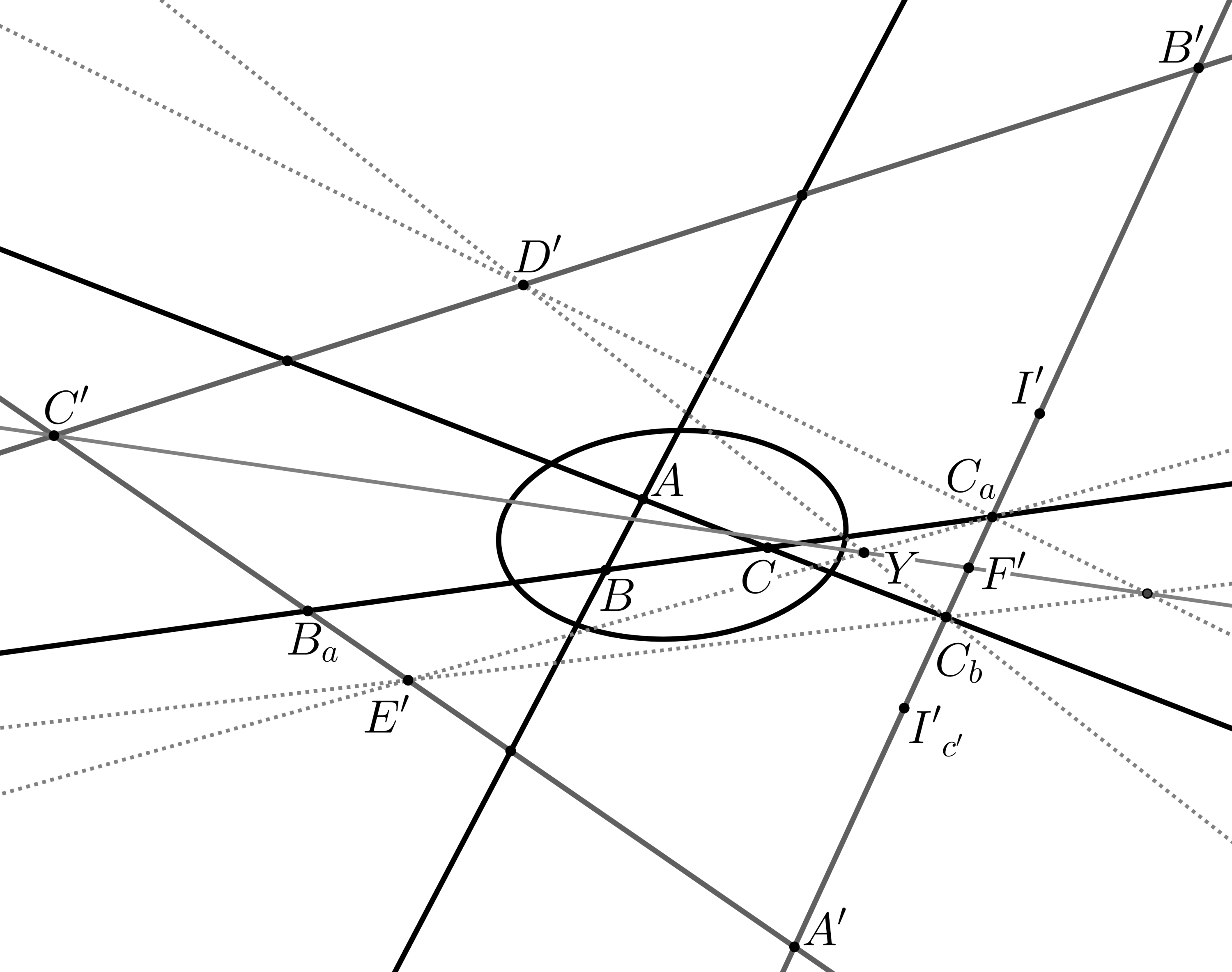}
\caption{Concurrency of lines in Claim
\ref{claim:concurrency-C'F'-CaD'-CbE'}}
\label{fig:cosine-rule-05-Claim-2}
\end{figure}

It is interesting to remark that, because $H',D$ are midpoints of the segments
$\ov{C'B_a},\ov{B_aC_a}$ respectively, the point $W$ is a midpoint of
$\ov{C'C_a}$. By considering the triangle $\wt{B'C'C_a}$, the line $D'W$ must
intersect $c'$ at a midpoint of $\ov{B'C_a}$. Thus, it must be $D'W\cdot c'=I'$
or $D'W\cdot c'=I'_{c'}$.
\begin{claim}\label{claim:D'W-pass-through-I'c'}
The points $D',W,I'_{c'}$ are collinear.
\end{claim}

\begin{proof}
Let $d$ be the line joining the points $D,D'$ and $\tiD$, and take the
points:
$$D_1^*=d\cdot c',\quad D_2^*=d\cdot b',\quad D_3^*=d\cdot C_aC'.$$
Let $\pi_1$ be the perspectivity from $c'$ onto $b'$ through the point $D_1'$,
let $\pi_2$ be the perspectivity from $b'$ onto $C_aC'$ through the point $D$,
and let $\pi_3$ be the perspectivity from $C_aC'$ onto $c'$ through the point
$D'$. Let consider also the projectivity $\pi$ of $c'$ onto itself given by the
composition
$\pi_3\circ\pi_2\circ\pi_1$.

Because the line $d$ joins the points $D,D'$ and $\tiD$, it is
$$D_1^*\overset{\pi_{1}}{\longmapsto}D_2^*\overset{\pi_{2}}{\longmapsto}
D_3^*\overset{\pi_{1}}{\longmapsto}D_1^*\,,$$
and so $D_1^*$ is a fixed point of $\pi$.

On the other hand, looking at the midpoints of the triangle $\wt{C'B_aC_a}$, it
is
$$I'\overset{\pi_{1}}{\longmapsto}
H'\overset{\pi_{2}}{\longmapsto}W\,,$$
and
$$I'_{c'}\overset{\pi_{1}}{\longmapsto}
H'_{b'}\overset{\pi_{2}}{\longmapsto}W'\,,$$
where $W'$ is the conjugate point of $W$ in $C_aC'$.
Looking at the midpoints of the triangle $\wt{B'C'C_a}$, it can be
$$\pi_3(W)=I'_{c'}\quad\text{and}\quad\pi_3(W')=I'$$
or
$$\pi_3(W)=I'\quad\text{and}\quad\pi_3(W')=I'_{c'}\,.$$
In the latter case, it turns out that $I',I'_{c'}$ are also fixed points of
$\pi$. If the point $D_1^*$ coincides with $I'$ or $I'_{c'}$, because
$\tiD\in\tid$ it would also be $H'_{b'}\in \tid$ or $H'\in\tid$, and this is
impossible by Remark~\ref{rem:D-Da-non-collinear-I-Ic-H-Hb}. Therefore,
$D_1^*$ is different from $I',I'_{c'}$, and so
the map $\pi$ must be the identity map on $c'$.

\begin{figure}
\centering
\includegraphics[width=0.8\textwidth]
{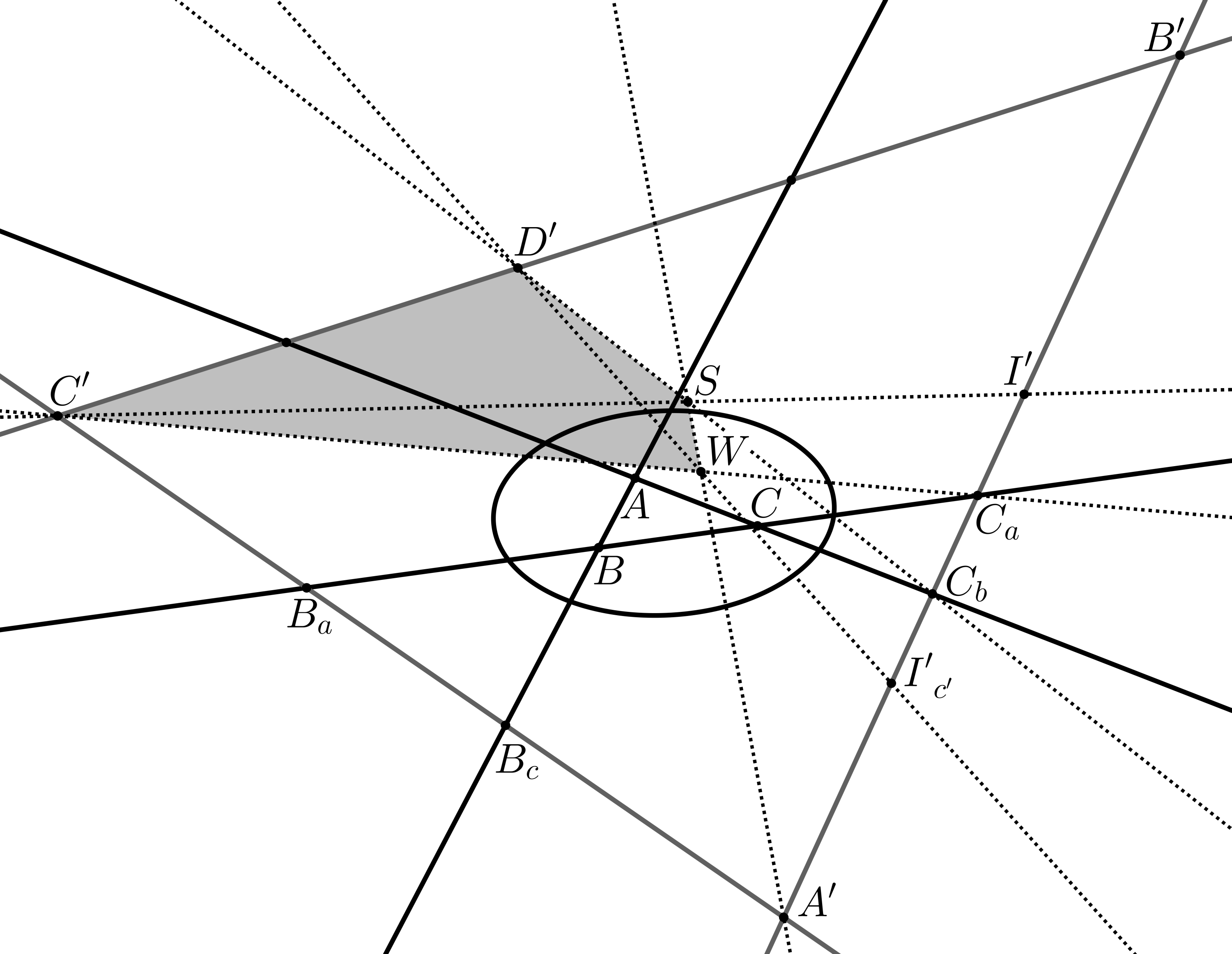}
\caption{The point $S$ lies in $C_bD'$}
\label{fig:cosine-rule-06-Claim-4}
\end{figure}

If we consider the point $K=C'\tiD\cdot c'$, we have
$$K\overset{\pi_{1}}{\longmapsto} C' \overset{\pi_{2}}{\longmapsto} C'
\overset{\pi_{3}}{\longmapsto} B'\,.$$
If $K=B'$, it should be $\tiD\in a'$ and also $D\in a'$, and this is not
possible because $\TT$ is coherently oriented. Therefore, the points $K$ and
$B'$ are different
and
$\pi$ cannot be the identity map. Thus, it must be $\pi_3(W)=I'_{c'}$ and
$\pi_3(W')=I'$.
\end{proof}

\begin{sidewaysfigure}
\centering
\includegraphics[width=0.9\textwidth]
{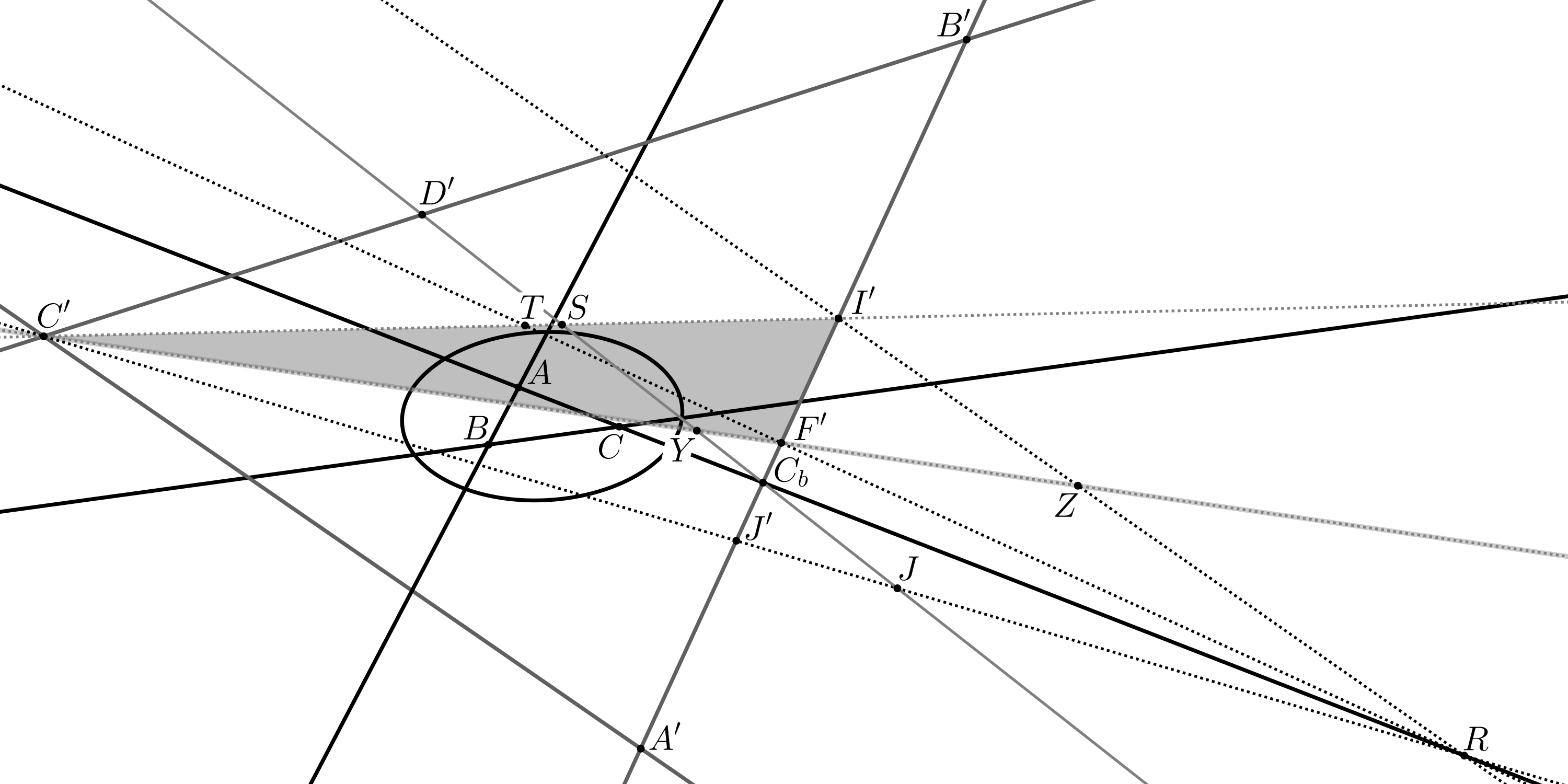}
\caption{Application of Van Aubel's Theorem}
\label{fig:cosine-rule-06-Claim-4}
\end{sidewaysfigure}

\begin{claim}\label{claim:S-Y-lie-in-CbD'}
The point $S$ lies in $C_bD'$.
\end{claim}
\begin{proof}[Proof of the Claim]
By Claim~\ref{claim:D'W-pass-through-I'c'}, the line $D'W$ passes through
$I'_{c'}$. If we consider the quadrangle $\QQ=\{C',D',S,W\}$, the quadrilateral
involution
$\sigma_{\QQ}$ on $c'$ sends $A',I'$ into $B',I'_{c'}$
respectively and vice versa. The involution $\sigma_{\QQ}$
coincides with the
symmetry $\tau_{F'F'_{c'}}$ of $c'$ with respect to $F'$ and it must send $C_a$
into
$C_b$. Therefore, it is $D'S\cdot c'=C_b$.
\end{proof}

Consider the line $j_0=C'R$ and the points $J_0,J'_0\in j$ given by
$$J_0=C_bD'\cdot j_0\quad \text{and} \quad J'_0=c'\cdot j_0\,.$$
By projecting since $C_b$ the line $a'$ onto $j_0$, we have
$$(C'B'A_bD')=(C'J'_0RJ_0)\,.$$
Let consider the triangle $\wt{C'F'I'}$ and its cevians that intersect at $R$.
Thus, $J'_0$ is the basepoint of the cevian through $C'$, $T$ is the
basepoint of the cevian through $F'$, and $Z$ is the basepoint of the
cevian through $I'$. Because $Y,S$ are collinear with $C_b$ and $D'$ (Claims
\ref{claim:S-Y-lie-in-CbD'} and~\ref{claim:concurrency-C'F'-CaD'-CbE'}), the
points $Y,S,J_0$ are collinear, and therefore, by 
\hyperref[thm:Van-Aubel]{Van Aubel's Theorem} it is
$$(C'J'_0RJ_0)=(C'F'ZY)+(C'I'TS)\,.$$
This completes the proof of Theorem~\ref{thm:cosine-rule}
\end{proof}

\section{Some examples}

Finally, we will see that the projective laws of sines and cosines \emph{work}
for obtaining elliptic and hyperbolic trigonometric formulae.
We show how they can be used
to deduce the corresponding geometric laws for any generalized triangle.
The first thing that we must do for any of of those figures
is to draw the points among the
midpoints, complementary midpoints and magic midpoints 
of $\TT$ and $\TT'$ that are real points, and to label them following 
the requirements of Definition~\ref{def:coherently-oriented}.
We will assume always that the projective triangles depicted 
are coherently oriented.
Then, we have to translate the projective 
trigonometric ratios of the projective law of cosines~\eqref{eq:projective-law-of-cosines} 
into circular or hyperbolic trigonometric ratios 
concerning the different magnitudes of the figure. Each figure
will be characterized by six geometric magnitudes that we will
call $a,b,c,\alpha,\beta,\gamma$. 
The magnitudes $a,b,c$ denote always a segment length, while
$\alpha,\beta$ or $\gamma$ could denote an angular measure
or a segment length.
After that, we must deduce the sign to be put in the different terms of the
geometric laws of cosines so obtained. 
When all the points on a cross-ratio are real 
points, its sign is deduced using 
Lemma~\ref{lemma:ABCD-negative-iff-AB-separate-CD}.
The main problem is to deduce the sign of a pure imaginary cross ratio
(see the footnote on page \pageref{footnote:positive-imaginary}), which appears
when one of the midpoints or complementary midpoints is imaginary. 
We cannot deduce
the sign of an imaginary without doing explicit computations. We didn't
any explicit computation in the whole book (just some tricky manipulations
with cross-ratios, indeed), and we will not start now. As we will see, 
using only the ``visible'' (real) points of each figure, we can deduce 
the relative sign of an imaginary cross 
ratio of~\eqref{eq:projective-law-of-cosines}
with respect to another one. This will be done using
the Projective Law of sines or using 
\hyperref[thm:Menelaus-affine]{Menelaus' Theorem}
together with Lemma~\ref{lemma:ABCD-negative-iff-AB-separate-CD},
and it will be enough for our purposes. 

\subsection{Elliptic triangles}

\begin{figure}
\centering
\includegraphics[width=0.6\textwidth]
{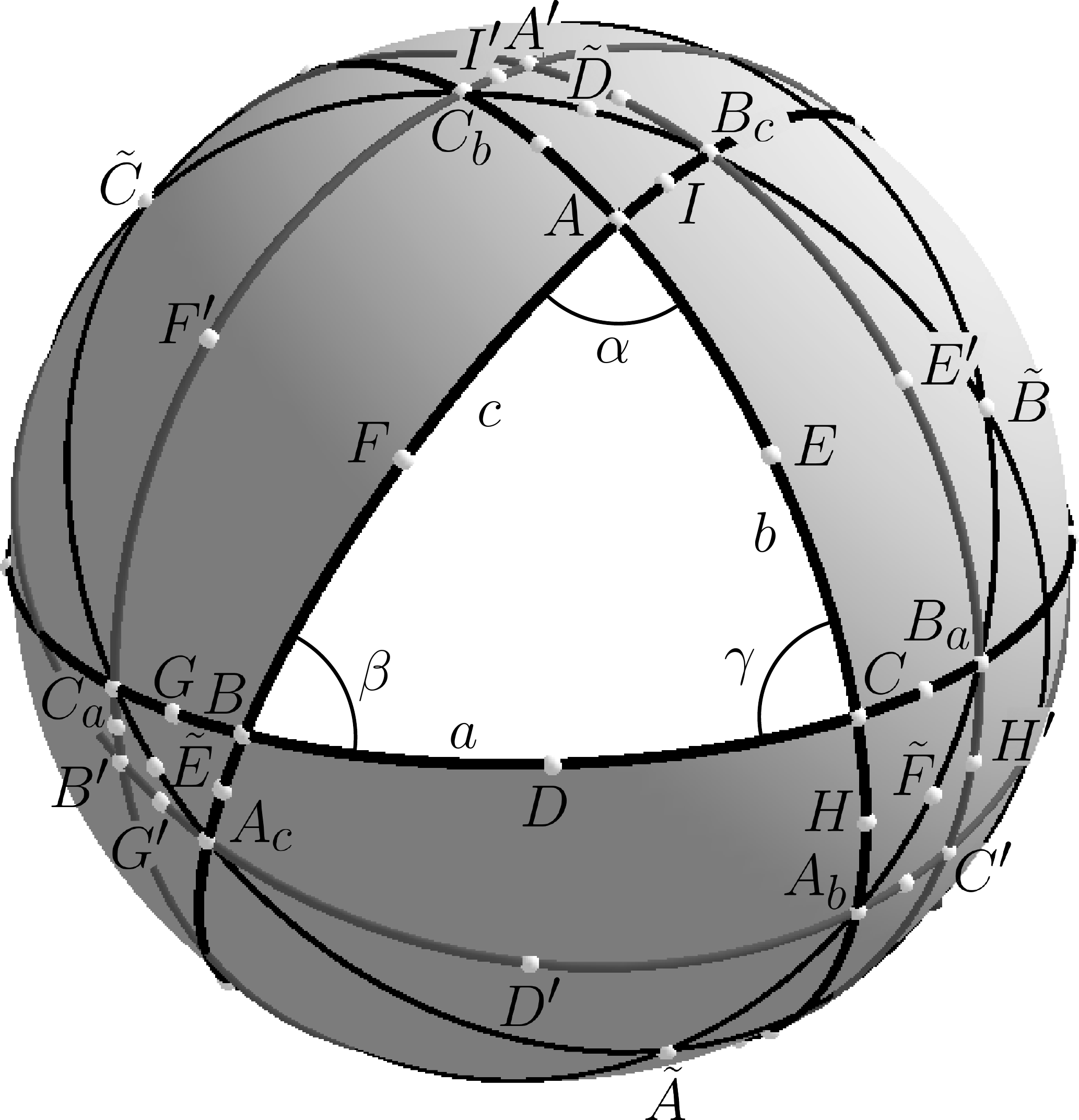}
\caption{Elliptic triangle}
\label{fig:cosine-rules-sph-triangle}
\end{figure}

Let $\TT$ be the elliptic triangle depicted in
Figure~\ref{fig:cosine-rules-sph-triangle}.

By applying Lemma~\ref{lemma:ABCD-negative-iff-AB-separate-CD}
to the cross ratios involving the Projective Law of sines we have:
    \[
    \begin{array}{rll}
	\ss(AB)&=(AB_cBI)<0&\Longrightarrow \ss(AB)=-\sin c\\
	\ss(BC)&=(BC_aCG)<0&\Longrightarrow \ss(BC)=-\sin a\\
	\ss(CA)&=(CA_bAH)<0&\Longrightarrow \ss(CA)=-\sin b\\
	\ss(A'B')&=(A'C_bB'I')<0&\Longrightarrow \ss(A'B')=-\sin \gamma\\
	\ss(B'C')&=(B'A_cC'G')<0&\Longrightarrow \ss(B'C')=-\sin \alpha\\
	\ss(C'A')&=(C'B_aA'H')<0&\Longrightarrow \ss(C'A')=-\sin \beta\,.\\
     \end{array}
    \]
The law of sines for an elliptic triangle is:
$$\dfrac{\sin a}{\sin \alpha}=\dfrac{\sin b}{\sin \beta}=\dfrac{\sin c}{\sin \gamma}\,.$$

  We have also:
    \[
    \begin{array}{rll}
	  \cc(AB)&=(ABB_cF)<0&\Longrightarrow \cc(AB)=-\cos c\\
	  \cc(BC)&=(BCC_aD)<0&\Longrightarrow \cc(BC)=-\cos a\\
	  \cc(CA)&=(CAA_bE)<0&\Longrightarrow \cc(CA)=-\cos b\\
	  \cc(A'B')&=(A'B'C_bF')>0&\Longrightarrow \cc(A'B')=\cos \gamma\\
	  \cc(B'C')&=(B'C'A_cD')>0&\Longrightarrow \cc(B'C')=\cos \alpha\\
	  \cc(C'A')&=(C'A'B_aE')>0&\Longrightarrow \cc(C'A')=\cos \beta\,.\\
     \end{array}
    \]
    
Therefore, from the Projective Law of cosines~\eqref{eq:projective-law-of-cosines}, we get
$$\cos a=\sin c \sin b \cos \alpha + \cos c \cos b\,,$$
and from its dual~\eqref{eq:projective-law-of-cosines-dual}:
      $$\cos \alpha =\sin \gamma \sin \beta \cos a -\cos \gamma \cos \beta\,.$$

\subsection{Hyperbolic triangles}

\begin{figure}
\centering
\includegraphics[width=0.9\textwidth]
{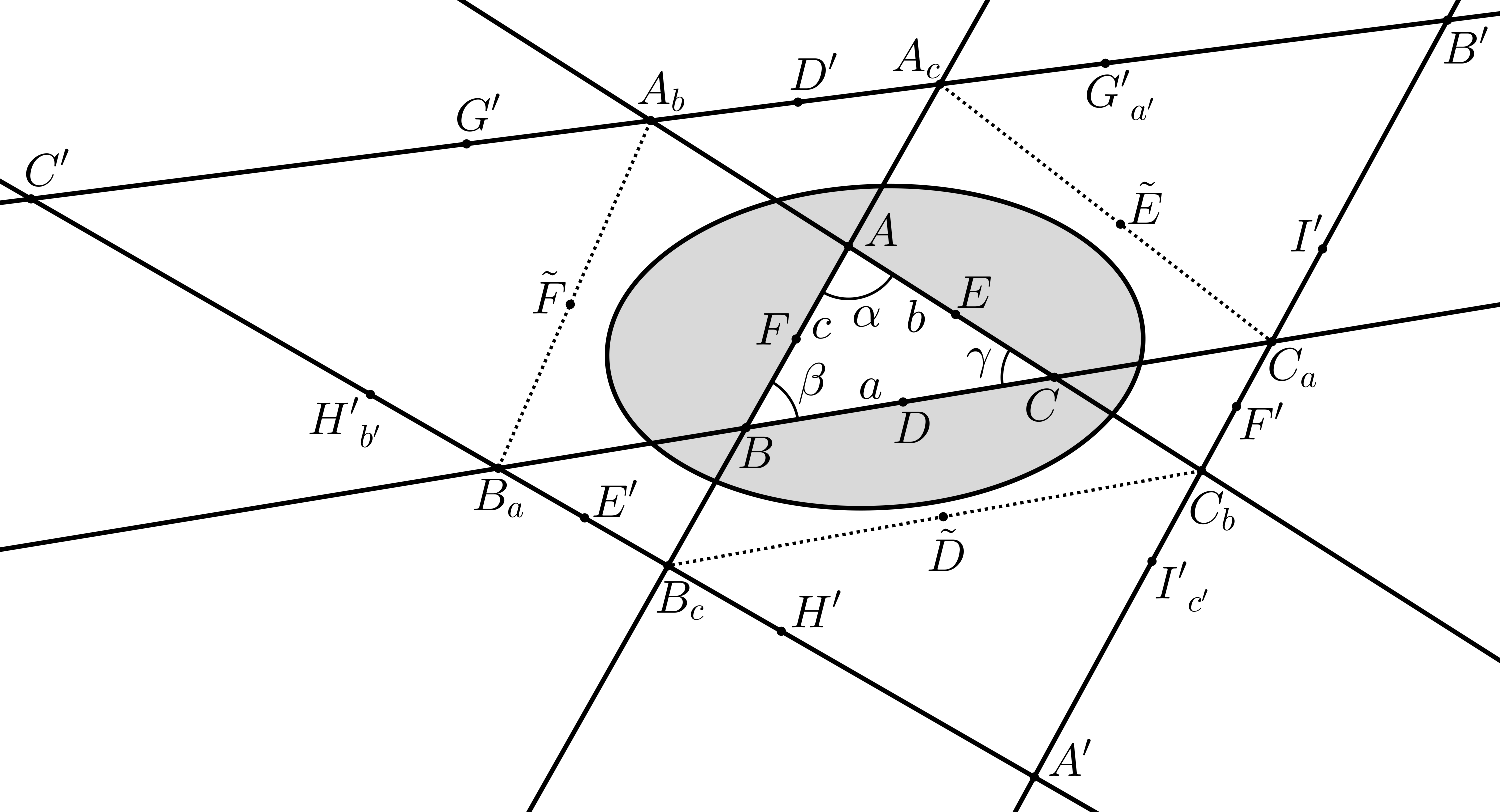}
\caption{Hyperbolic triangle}
\label{fig:cosine-rules-hyp-triangle}
\end{figure}

Consider now the hyperbolic triangle given by $\TT$ in 
Figure~\ref{fig:cosine-rules-hyp-triangle}. The sides and angles
of $\TT$ have the same names as in the elliptic case. We have
    \[
    \begin{array}{rll}
	\ss(AB)&=(AB_cBI)=\text{imaginary}&\Longrightarrow \ss(AB)=\pm i\sinh c\\
	\ss(BC)&=(BC_aCG)=\text{imaginary}&\Longrightarrow \ss(BC)=\pm i\sinh a\\
	\ss(CA)&=(CA_bAH)=\text{imaginary}&\Longrightarrow \ss(CA)=\pm i\sinh b\\
	\ss(A'B')&=(A'C_bB'I')>0&\Longrightarrow \ss(A'B')=\sin \gamma\\
	\ss(B'C')&=(B'A_cC'G')>0&\Longrightarrow \ss(B'C')=\sin \alpha\\
	\ss(C'A')&=(C'B_aA'H')>0&\Longrightarrow \ss(C'A')=\sin \beta\,.\\
     \end{array}
    \]
From the Projective Law of sines~\eqref{eq:projective-law-of-sines}
we deduce
$$\dfrac{\sinh a}{\sin \alpha}=\dfrac{\sinh b}{\sin \beta}
=\dfrac{\sinh c}{\sin \gamma}\,.$$
In particular, $\ss(AB)$, $\ss(BC)$ and $\ss(CA)$ have all the same 
imaginary sign: all of them are  positive 
imaginary or all of them are negative imaginary; and  
the product of any two of these ratios is a negative real number.

We have, also
    \[
    \begin{array}{rll}
	\cc(AB)&=(ABB_cF)<0&\Longrightarrow \cc(AB)=-\cosh c\\
	\cc(BC)&=(BCC_aD)<0&\Longrightarrow \cc(BC)=-\cosh a\\
	\cc(CA)&=(CAA_bE)<0&\Longrightarrow \cc(CA)=-\cosh b\\
	\cc(A'B')&=(A'B'C_bF')>0&\Longrightarrow \cc(A'B')=\cos \gamma\\
	\cc(B'C')&=(B'C'A_cD')>0&\Longrightarrow \cc(B'C')=\cos \alpha\\
	\cc(C'A')&=(C'A'B_aE')>0&\Longrightarrow \cc(C'A')=\cos \beta\,.\\
     \end{array}
    \]

Thus, the Projective Law of 
cosines~\eqref{eq:projective-law-of-cosines} gives
      $$\cosh a=-\sinh c \sinh b \cos \alpha + \cosh c \cosh b\,,$$
while its dual~\eqref{eq:projective-law-of-cosines-dual}
      $$\cos \alpha =\sin \gamma \sin \beta \cosh a -\cos \gamma \cos \beta\,.$$

\subsection{Right-angled hexagons}

\begin{figure}
\centering
\includegraphics[width=0.7\textwidth]
{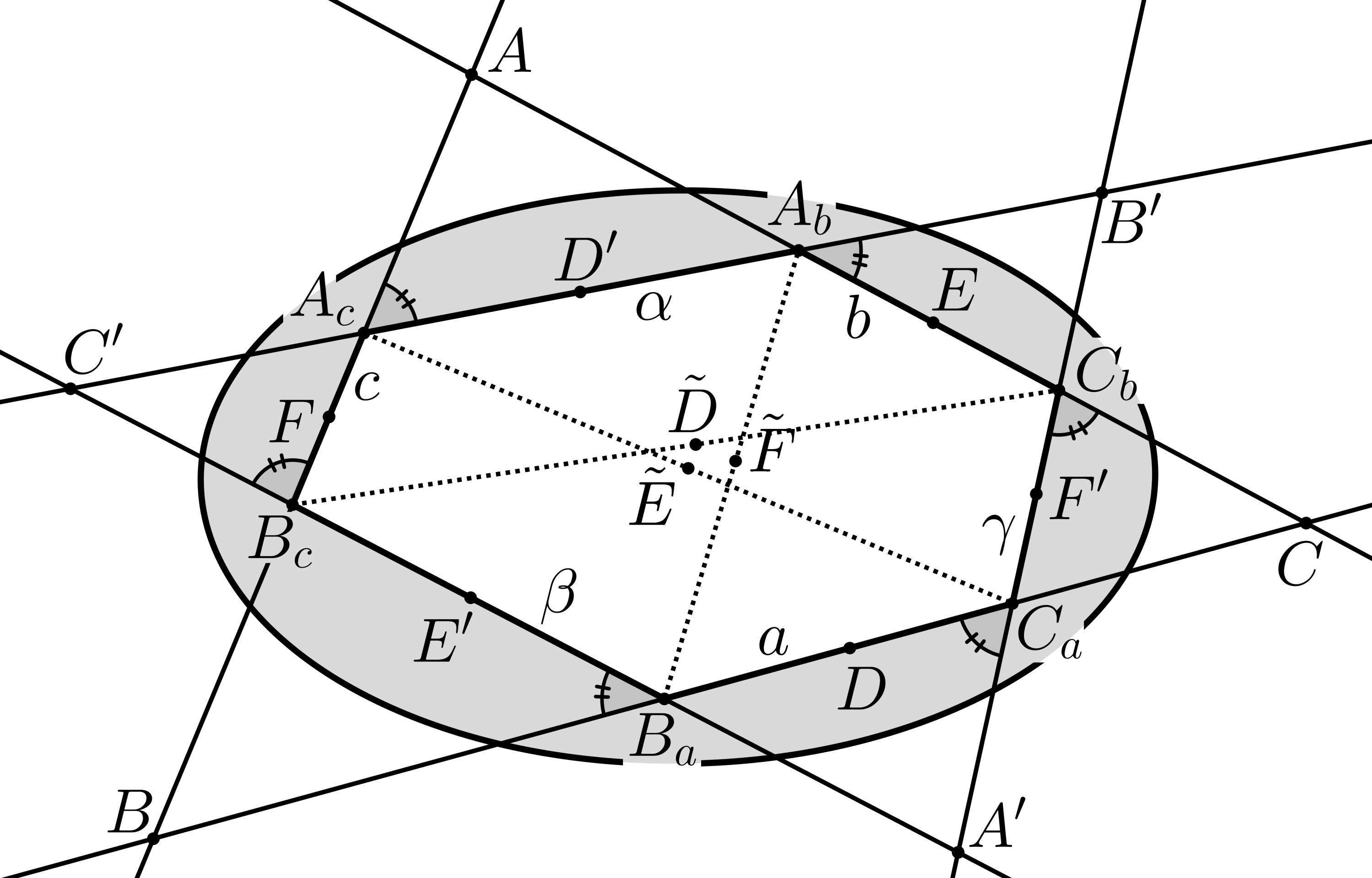}
\caption{Right-angled hexagon}
\label{fig:cosine-rules-Right-angled-hexagon}
\end{figure}

Consider now a right-angled hexagon with sides
$a,\gamma,b,\alpha,c,\beta$ generated by $\TT$ and
$\TT'$ in the hyperbolic plane as in 
Figure~\ref{fig:cosine-rules-Right-angled-hexagon}. In this case,
all the complementary midpoints of $\TT$ and $\TT'$ are
imaginary.

As the hyperbolic sines are always positive, the law of sines for this figure
must be
    $$\dfrac{\sinh a}{\sinh \alpha}=\dfrac{\sinh b}{\sinh \beta}=
    \dfrac{\sinh c}{\sinh \gamma}\,.$$

Let study now the Projective Law of 
cosines~\eqref{eq:projective-law-of-cosines}.
By Lemma~\ref{lemma:ABCD-negative-iff-AB-separate-CD}, we know
that the four real cross ratios of this formula: $\cc(BC)$, $\cc(B'C')$, 
$\cc(AB)$ and $\cc(CA)$ are positive. This implies that
the product $\ss(AB)\ss(CA)$ must be negative. Otherwise,
the two members of~\eqref{eq:projective-law-of-cosines} would have
different sign, which is impossible. Therefore, the law of cosines
for a right-angled hexagon is
      $$\cosh a=\sinh c \sinh b \cosh \alpha - \cosh c \cosh b\,.$$

\subsection{Quadrilateral with consecutive right angles}

\begin{figure}
\centering
\includegraphics[width=0.9\textwidth]
{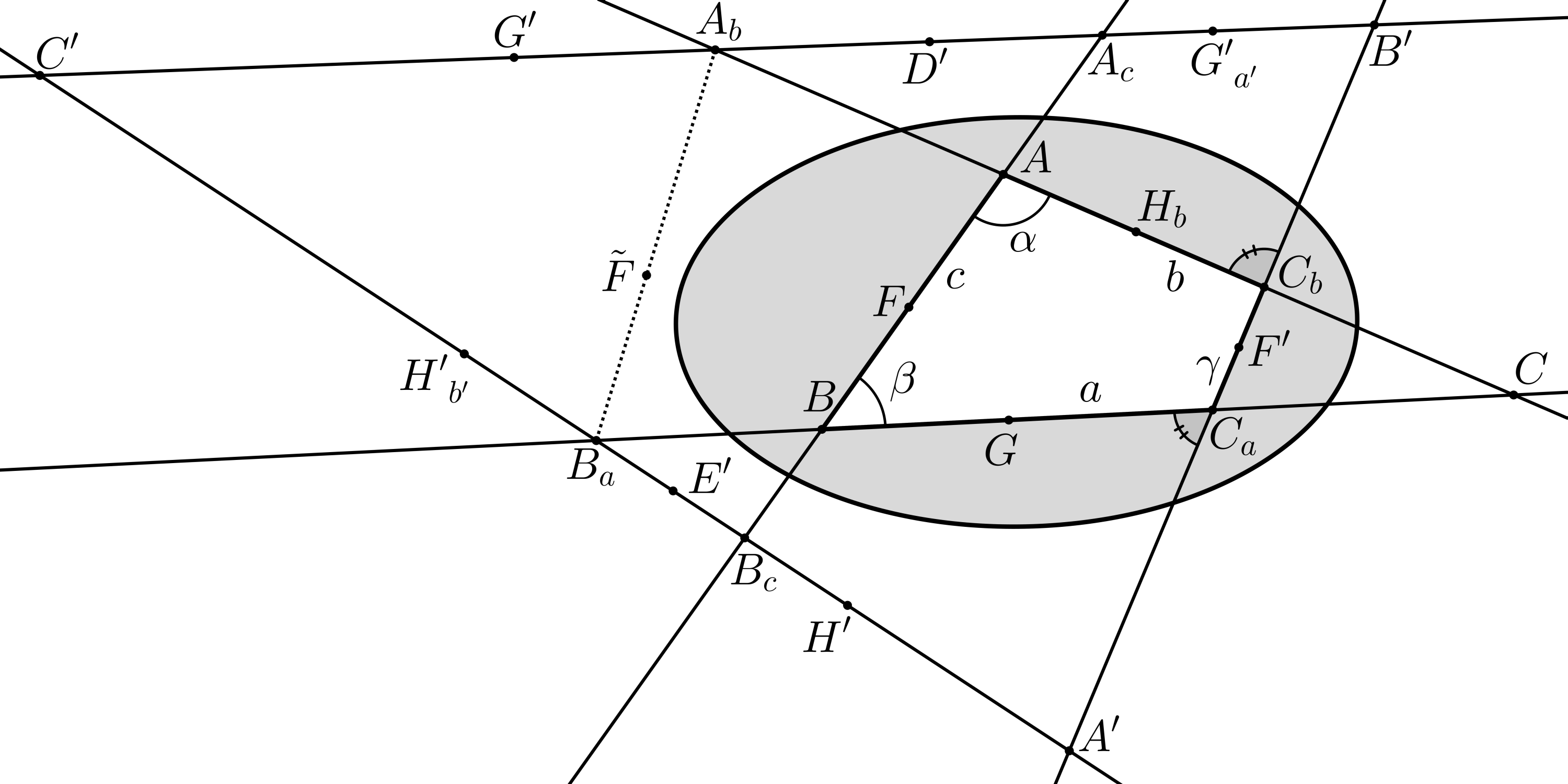}
\caption{Quadrilateral with two consecutive right angles}
\label{fig:cosine-rules-quadrilateral-01}
\end{figure}

We study now the quadrilateral with consecutive right angles
depicted in Figure~\ref{fig:cosine-rules-quadrilateral-01}.
We know that
    \[
    \begin{array}{rll}
	\ss(AB)&=(AB_cBI)=\text{imaginary}&\Longrightarrow \ss(AB)=\pm i\sinh c\\
	\ss(BC)&=(BC_aCG)<0&\Longrightarrow \ss(BC)=-\cosh a\\
	\ss(CA)&=(CA_bAH)<0&\Longrightarrow \ss(CA)=-\cosh b\\
	\ss(A'B')&=(A'C_bB'I')=\text{imaginary}&\Longrightarrow \ss(A'B')=\pm i\sinh \gamma\\
	\ss(B'C')&=(B'A_cC'G')>0&\Longrightarrow \ss(B'C')=\sin \alpha\\
	\ss(C'A')&=(C'B_aA'H')>0&\Longrightarrow \ss(C'A')=\sin \beta\,.\\
     \end{array}
    \]
The Projective Law of sines~\eqref{eq:projective-law-of-sines}
tells to us that $\ss(AB)$ and $\ss(A'B')$ have different imaginary sign,
as their quotient is a negative real number. The law of sines for this figure is
    $$\dfrac{\cosh a}{\sin \alpha}=\dfrac{\cosh b}{\sinh \beta}=
    \dfrac{\sinh c}{\sinh \gamma}\,.$$

For the projective cosines, we have:
    \[
    \begin{array}{rll}
 	\cc(AB)&=(ABB_cF)<0&\Longrightarrow \cc(AB)=-\cosh c\\
	\cc(BC)&=(BCC_aD)=\text{imaginary}&\Longrightarrow \cc(BC)=\pm\sinh a\\
	\cc(CA)&=(CAA_bE)=\text{imaginary}&\Longrightarrow \cc(CA)=\pm\sinh b\\
	\cc(A'B')&=(A'B'C_bF')>0&\Longrightarrow \cc(A'B')=\cosh \gamma\\
	\cc(B'C')&=(B'C'A_cD')>0&\Longrightarrow \cc(B'C')=\cos \alpha\\
	\cc(C'A')&=(C'A'B_aE')>0&\Longrightarrow \cc(C'A')=\cos \beta\,.\\
    \end{array}
    \]

Let us start now with the dual Projective Law of 
cosines~\eqref{eq:projective-law-of-cosines-dual}. The four real cross ratios
of this formula: $\cc(B'C')$, $\ss(C'A')$, $\cc(A'B')$ and $\cc(C'A')$ are positive.
This implies that the product $\ss(A'B')\cc(BC)$ is a negative real number, and then
$\ss(A'B')$ and $\cc(BC)$ have the same imaginary sign. The formula
that we obtain from~\eqref{eq:projective-law-of-cosines-dual} is
    $$\cos \alpha = \sinh \gamma \sin \beta \sinh a - \cosh \gamma \cos \beta\,.$$

We return now to the Projective Law of 
cosines~\eqref{eq:projective-law-of-cosines}. 
In this formula we have three unknown signs: those
of $\cc(BC)$, $\ss(AB)$ and $\cc(CA)$.

We have seen by
the Projective law of sines that $\ss(AB)$ and $\ss(A'B')$ 
have different imaginary sign, and now we know also that 
$\ss(A'B')$ and $\cc(BC)$ have the same imaginary sign. Thus,
$\cc(BC)$ and $\ss(AB)$ have different imaginary sign.

On the other hand, if we apply Menelaus Projective formula
to the triangle $\wt{ABC}$ with transversals $c'$ and $DE$,
we get
    $$(ABC_0F_c)(BCC_aD)(CAC_bE)=1\,.$$
As $(ABC_0F_c)>0$ by 
Lemma~\ref{lemma:ABCD-negative-iff-AB-separate-CD},
we conclude that 
    $$\cc(BC)=(BCC_aD)\quad\text{and}\quad 
    \cc(CA)=(CAA_bE)=(ACC_bE)$$
have the same imaginary sign. Therefore~\eqref{eq:projective-law-of-cosines}
gives
    $$\sinh a=-\sinh c \cosh b \cos \alpha + \cosh c \sinh b\,.$$

There are other two laws of cosines for this figure which are not 
equivalent to those given above. These are the laws of cosines
given by the segments $\ov{AB}$ and $\ov{A'B'}$.
The first of them is
    $$\cc(AB)=-\ss(BC)\ss(CA)\cc(A'B')-\cc(BC)\cc(CA)\,.$$
As $\cc(BC)$ and $\cc(CA)$ have the same imaginary sign, their product is
a negative real number. So,
    $$\cosh c=\cosh a \cosh b\cosh \gamma-\sinh a\sinh b\,.$$
    
The last law of cosines we are interested in is the dual of the previous one:
    $$\cc(A'B')=-\ss(B'C')\ss(C'A')\cc(AB)-\cc(B'C')\cc(C'A')\,.$$
All the cross-ratios involved here are real, and they give
    $$\cosh \gamma=\sin \alpha \sin \beta \cosh c -\cos \alpha\cos \beta\,.$$

\appendix
\setcounter{secnumdepth}{0}
\chapter{Laguerre's formula for rays}\label{sec:Laguerre}

As it is mentioned in \S\ref{sec:Cayley-Klein-models-for}, 
Laguerre's formula~\eqref{eq:Laguerre-formula}
for computing angles in Cayley-Klein models is valid for angles between
lines, not for angles between rays, and it cannot distinguish an angle
from its supplement. 
We propose a modification of Laguerre's formula that allows us
to compute angles between rays. Although our formula fits perfectly
in the projective model for the hyperbolic plane, it can be adapted
also to the elliptic and euclidean models.

A projective conic as a 1-dimensional object is
equivalent to a projective line or a pencil of lines through a point.
The key for this similarity is given by Steiner's theorem.
\begin{theorem}[Steiner's Theorem]
\label{thm:Steiner}Let $X,Y$ be two points on the nondegenerate
conic $\Theta$. The map from the pencil of lines through $X$ to
the pencil of lines through $Y$ that maps the line $XP$ into the
line $YP$ for all $P\in\Theta$ is a projectivity.\end{theorem}
As a corollary of Steiner's Theorem, we obtain the following result (which is also known as ``Chasles' Theorem''):
\begin{corollary}
\label{cor:Steiner-cross-ratio-over-conics-coincide}Let $A,B,C,D$
be four points on the conic $\Theta$, and let $X,Y$ be another two points on the same conic. If $a,b,c,d$ are the four lines joining the point $X$ with $A,B,C,D$ respectively, and $a',b',c',d'$ are the four lines joining $Y$  with $A,B,C,D$ respectively, then
\begin{equation*}
(a\, b\, c\, d)=(a'\, b'\, c'\, d').
\end{equation*}
\end{corollary}
This corollary allows us to define the cross ratio of four points on
a conic as the cross ratio of the lines joining them with any other point on the conic.
\begin{definition}
Given four points $A,B,C,D$ on the nondegenerate conic $\Theta$,
we define their cross ratio \emph{over} $\Theta$, and we denote it
by $(ABCD)_{\Theta}$, by the equality
\[
(ABCD)_{\Theta}:=(a\, b\, c\, d),
\]
where $a,b,c,d$ are the lines joining $A,B,C,D$, respectively, with any other point $X$ of $\Theta$.
\end{definition}

Assume now that $\mathbb{P}$ is the hyperbolic plane and
take a point $P\in\mathbb{P}$. Consider two lines $a,b$ through
$P$, and take the intersection points $A_{1},A_{2}$ and $B_{1},B_{2}$
of $a$ and $b$ with $\Phi$ respectively. The hyperbolic line, that
we also denote by $a$, determined by the projective line $a$ is
divided by $P$ into two rays $a_{1},a_{2}$ starting from $P$ and
ending at $A_{1},A_{2}$ respectively. In the same way, the line $b$
is divided by $P$ into the two rays $b_{1},b_{2}$ starting at $P$
and ending at $B_{1},B_{2}$ respectively. Let $u,v$ be the two tangent
lines to $\Phi$ through $P$, and let $U,V$ be their contact points
with $\Phi$, respectively. The line $p=UV$ is the polar line of $P$. If $A,B$ are the intersection points with $p$ of the lines $a,b$ respectively, the two diagonal points $Q,Q_p$ of the quadrangle $\QQ=\{A_1,A_2,B_1,B_2\}$ different from $P$ are the midpoints of the segment $\ov{AB}$. If, for example, it is $Q=p\cdot A_1B_2$, by projecting from $B_2$ we have
$$(UVQB)=(UVA_1B_1)_{\Phi},$$
and so by~\eqref{eq:midpoints-02} it is
$$(u\,v\,a\,b)=(UVAB)=(UVQB)^2=(UVA_1B_1)^2_\Phi.$$
This suggest the following definition:
\begin{definition}
The angle between the rays $a_{1},b_{1}$ is given by the formula
\begin{equation}
\widehat{a_{1}b_{1}}=\dfrac{1}{i}\log(UVA_{1}B_{1})_{\Phi}.
\label{eq:Laguerre-formula-for-Rays}
\end{equation}
\end{definition}
With this definition, angles between rays take values 
between $0$ and $2\pi$ as expected.
If we denote by $\alpha$ the angle $\widehat{a_{1}b_{1}}$, by applying 
\eqref{eq:lemma-4ABBpAp-UVAB}, it can be seen that
$$\cos \alpha=2(A_1B_1B_2A_2)_\Phi-1.$$
Moreover, if we draw the conjugate lines 
$a_P,b_P$ of $a,b$ at $P$ and we label their intersection 
points with $\Phi$ as in Figure~\ref{Fig:Laguerre-02}
(\footnote{We label the points $A'_1,A'_2,B'_1,B'_2$ in such a way that the lines 
$A_1B_2,B_1A_2,A'_1B'_1,A'_2B'_2$ are concurrent.}), it can be proved that
\begin{equation}\label{eq:angle-between-rays-cosine}
\cos \alpha=(A_1B_1B'_{1}A'_{1})_\Phi.
\end{equation}
Formula~\eqref{eq:angle-between-rays-cosine} shows that 
\eqref{eq:Laguerre-formula-for-Rays} works, because it gives 
a positive cosine for acute angles and a negative cosine 
for obtuse angles
(cf. Lemma~\ref{lemma:ABCD-negative-iff-AB-separate-CD} for conics).
If the angle $\alpha$ between the rays $a_1,b_1$ is acute as 
in Figure~\ref{Fig:Laguerre-02}, the points $A_1,B_1$ do not 
separate the points $A'_1,B'_1$. In the same figure, the angle 
between the rays $a_1$ and $b_2$ is obtuse because 
$A'_1,A'_2$ separate $A_1$ from $B_2$, and in this case 
\eqref{eq:angle-between-rays-cosine} will give a 
negative cosine for $\widehat{a_1b_2}$.

Formula~\eqref{eq:Laguerre-formula-for-Rays} is valid for euclidean, hyperbolic
and elliptic geometries if we replace the absolute conic $\Phi$ with any circle
centered at $P$. In euclidean geometry, $U,V$ are Poncelet's circular points at infinity.
\begin{figure}
\centering
\includegraphics[width=0.55\textwidth]{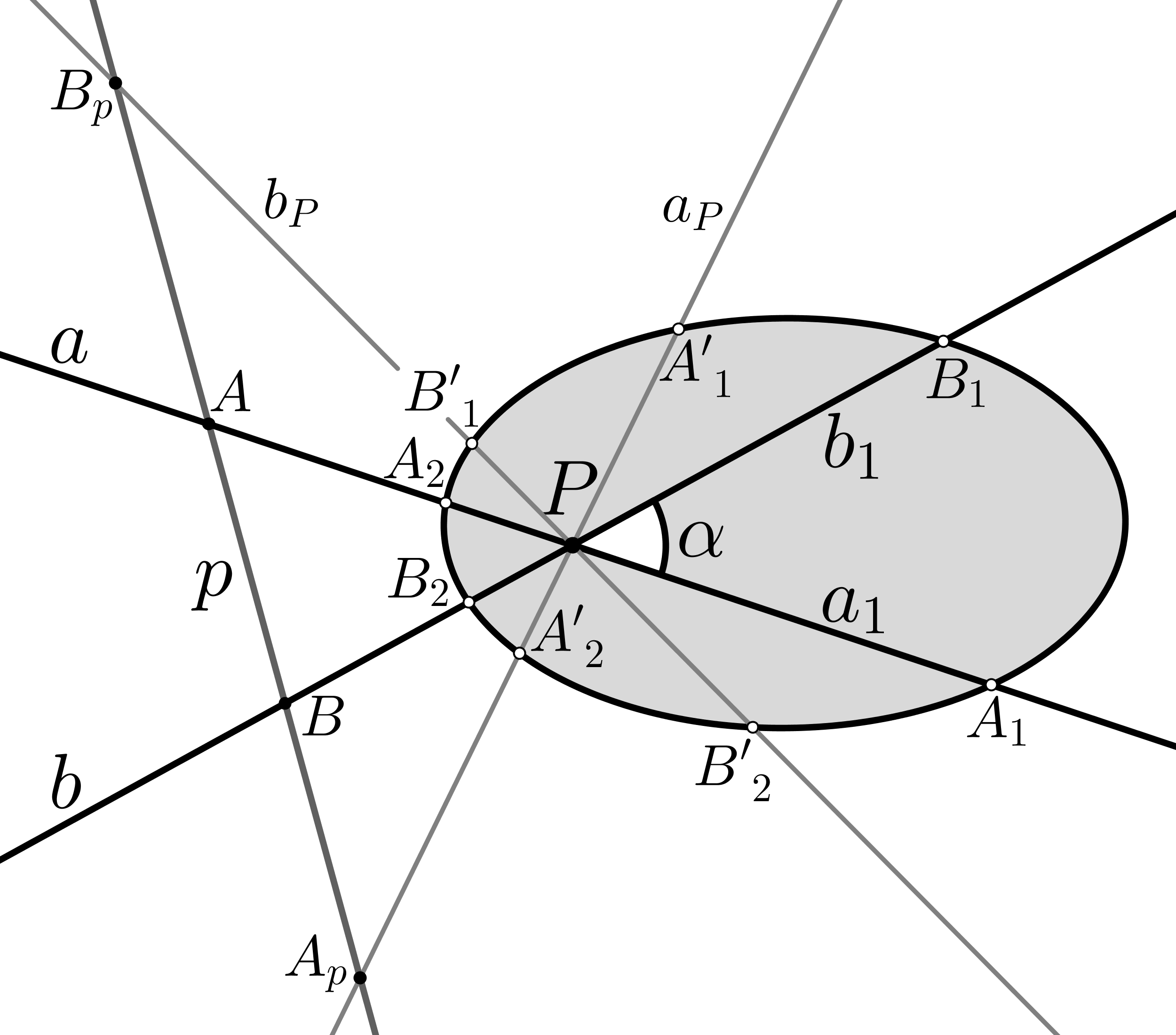}
\caption{angle between rays}\label{Fig:Laguerre-02}
\end{figure}

\end{document}